\documentclass[12pt, leqno]{article}
\usepackage{amsmath,amsfonts,amsthm,amscd,amssymb,verbatim} 
\usepackage{graphicx}
\usepackage[all]{xy}
\usepackage{hyperref}
\usepackage{ulem} 
\usepackage{xcolor}
\numberwithin{equation}{section}

\theoremstyle{plain}
\newtheorem{thm}{Theorem}[section]
\newtheorem{defn}[thm]{Definition}
\newtheorem{prop}[thm]{Proposition}
\newtheorem{lem}[thm]{Lemma}
\newtheorem{cor}[thm]{Corollary}

\newtheorem{propdefn}[thm]{Proposition-Definition}

\newtheorem{notation}[thm]{Notation}

\theoremstyle{definition}
\newtheorem{rem}[thm]{Remark}

\renewcommand{\b}{\bullet}

\newcommand{\N}{{\Bbb N}}
\newcommand{\Z}{{\Bbb Z}}
\newcommand{\Q}{{\Bbb Q}} 

\newcommand{\C}{{\Bbb C}}

\newcommand{\G}{{\Bbb G}}

\newcommand{\F}{{\Bbb F}}

\newcommand{\Af}{{\Bbb A}}

\renewcommand{\Pr}{{\Bbb P}}

\newcommand{\Spec}{{\mathrm{Spec}}\,}

\newcommand{\Spf}{{\mathrm{Spf}}\,}

\newcommand{\lra}{\longrightarrow}
\newcommand{\ra}{\rightarrow}
\newcommand{\hra}{\hookrightarrow}

\newcommand{\lla}{\longleftarrow}

\newcommand{\isom}{\overset{\sim}{=}}

\newcommand{\wt}[1]{\widetilde{#1}}

\newcommand{\ul}[1]{\underline{#1}}
\newcommand{\ol}[1]{\overline{#1}}
\newcommand{\os}{\overset}

\newcommand{\dR}{{\mathrm{dR}}}

\newcommand{\et}{{\mathrm{et}}}

\newcommand{\crys}{{\mathrm{crys}}}

\newcommand{\Hom}{{\mathrm{Hom}}}

\newcommand{\Ker}{{\mathrm{Ker}}}

\newcommand{\Coker}{{\mathrm{Coker}}}

\newcommand{\triv}{{\mathrm{triv}}}

\newcommand{\Vector}{{\mathrm{Vec}}}

\newcommand{\Fil}{{\mathrm{Fil}}}
\newcommand{\id}{{\mathrm{id}}}

\newcommand{\Rep}{{\mathrm{Rep}}}
\newcommand{\End}{{\mathrm{End}}}
\newcommand{\gp}{{\mathrm{gp}}}

\newcommand{\Strat}{{\mathrm{Str}}}

\newcommand{\dlog}{{\mathrm{dlog}}\,}
\newcommand{\Ext}{{\mathrm{Ext}}}
\newcommand{\an}{{\mathrm{an}}}

\newcommand{\cA}{{\cal A}}
\newcommand{\cB}{{\cal B}}
\newcommand{\cC}{{\cal C}}
\newcommand{\cD}{{\cal D}}
\newcommand{\cE}{{\cal E}}

\newcommand{\cI}{{\cal I}}
\newcommand{\cJ}{{\cal J}}

\newcommand{\cM}{{\cal M}}

\newcommand{\cO}{{\cal O}}

\newcommand{\cS}{{\cal S}}

\newcommand{\cU}{{\cal U}}

\newcommand{\cX}{{\cal X}}

\newcommand{\wh}{\widehat}
\renewcommand{\wt}{\widetilde}

\newcommand{\e}{{\bold e}}

\renewcommand{\End}{{\mathrm{End}}}

\newcommand{\MC}{{\mathrm{MC}}}
\newcommand{\MIC}{{\mathrm{MIC}}}
\newcommand{\MICn}{{\mathrm{MIC}^{\rm n}}}
\newcommand{\Nf}{{\mathrm{N}_f}}
\newcommand{\NfMIC}{{\mathrm{N}_f\MIC}}

\newcommand{\NfMICn}{{\mathrm{N}_f\MICn}}
\newcommand{\Nfs}{{\mathrm{N}_{f_s}}}
\newcommand{\NfsMIC}{{\mathrm{N}_{f_s}\MIC}}
\newcommand{\Crys}{{\mathrm{Crys}}}
\renewcommand{\inf}{{\mathrm{inf}}}
\newcommand{\str}{{\mathrm{str}}}
\newcommand{\StrCrys}{\Strat\Crys}
\newcommand{\NfStrCrys}{\mathrm{N}_f\StrCrys}
\newcommand{\NfStrCrysn}{\mathrm{N}_f\StrCrys^{\rm n}}
\newcommand{\sTW}{{\bold s}_{\mathrm{TW}}}
\newcommand{\s}{{\bold s}}
\newcommand{\cEn}{{\cal E}^{\rm n}}
\newcommand{\equ}[1]{\equiv_{#1}}

\renewcommand{\thefootnote}{\fnsymbol{footnote}}

\begin{document}
\title{Comparison of relatively unipotent log de Rham fundamental groups}
\author{Bruno Chiarellotto\footnote{Dipartimento di Matematica 'Tullio Levi-Civita",  
Universit\`a degli Studi di Padova, Via Trieste, 63, 35121 Padova, Italy. 
E-mail address: chiarbru@math.unipd.it}, \,
Valentina Di Proietto\footnote{ 
College of Engineering, Mathematics
and Physical Sciences,
University of Exeter,
EX4 4RN, Exeter,
United Kingdom. E-mail address: V.Di-Proietto@exeter.ac.uk
} \, and 
Atsushi Shiho\footnote{
Graduate School of Mathematical Sciences, 
University of Tokyo, 3-8-1 Komaba, Meguro-ku, Tokyo 153-8914, Japan. 
E-mail address: shiho@ms.u-tokyo.ac.jp} 
\setcounter{footnote}{-1}
}
\date{}
\maketitle

\begin{abstract}
In this paper, we prove compatibilities of various definitions of 
relatively unipotent log de Rham fundamental groups for certain 
proper log smooth integral 
morphisms of fine log schemes of characteristic zero. 
Our proofs are purely algebraic. As an application, we give 
a purely algebraic 
calculation of the monodromy action on the unipotent log de Rham fundamental 
group of a stable log curve. As a corollary we give a purely algebraic proof to 
the transcendental part of  Andreatta--Iovita--Kim's article: obtaining in this way a complete algebraic criterion for good reduction for curves. 
\end{abstract}

\tableofcontents

\section*{Introduction}

Unipotent fundamental groups have been used  to obtain several results  from number theory to algebraic geometry of hyperbolic curves: from finiteness of integral points (\cite{ki}, \cite{ha}) to the study of  good/bad reduction in families (\cite{oda2}, \cite{aik}). In all these studies, the motivic nature of unipotent fundamental groups plays an important role.   
Namely,  the  realizations (classical, \'etale, de Rham and crystalline)  of unipotent fundamental groups and comparison theorems between them appear in crucial way in the proofs. Moreover,  in the case of semistable reduction, the comparison theorems are `logarithmic and relative in nature'  in the sense that, in the formulation of comparison theorems, the log scheme defined by the reduction naturally appears 
as well as the monodromy action on unipotent fundamental group.

To explain the motivation of our study, let us recall the results in  
\cite{oda2}, \cite{aik}. In \cite{oda2}, Oda proved the following two results. 
\begin{enumerate}
\item For a family of proper hyperbolic complex curves $f: X \lra \Delta^*$ over a puntured disc $\Delta^*$ 
with semistable reduction at the center, the outer monodromy action of $\pi_1(\Delta^*) \cong \Z$ on 
the classical unipotent fundamental group of the generic fiber of $f$ is trivial if and only if 
$f$ has good reduction at the center. 
\item For a proper hyperbolic curve $f: X \lra \Spec K$ over a discrete valuation field $K$ of 
mixed characteristic $(0,p)$ 
with semistable reduction, the outer monodromy action of $\pi_1(\Spec K)$ on 
the $l$-adic unipotent fundamental group $(l \not= p)$ of the geometric generic fiber of $f$ is unramified if and only if 
$f$ has good reduction. 
\end{enumerate}
Also, in \cite{aik}, Andreatta--Iovita--Kim proved the following result. 
\begin{enumerate}
\item[(3)] For a proper hyperbolic curve $f: X \lra \Spec K$ over a discrete valuation field $K$ 
of mixed characteristic $(0,p)$ 
with semistable reduction and a section $\iota$, the monodromy action of $\pi_1(\Spec K)$ on 
the $p$-adic unipotent fundamental group of the geometric generic fiber of $f$ 
induced by $\iota$ is crystalline if and only if $f$ has good reduction. 
\end{enumerate}
Oda proved the claim (2) by reducing to (1). In fact, he
considers  a semistable family of hyperbolic curves 
over $R[[z]]$ (where $R$ is a certain discrete valuation ring) which contains the original 
semistable family as the locus $z = \pi$ (where $\pi$ is a uniformizer 
of the valuation ring $O_K$ of $K$) and then he compares  the $l$-adic and classical 
unipotent fundamental groups. As a matter of fact, 
Andreatta--Iovita--Kim proved  (3) again by reducing to (1):  
they use   relative $p$-adic Hodge theory 
(for the above semistable family over $R[[z]]$),  
to get the claim on the unipotent log de Rham  
fundamental group
and compare it with the classical unipotent fundamental group. 
Thus the main computation of the monodromy action 
is done in the situation given in (1), i.e. via  transcendental methods. 

The final purpose of this article, which will be completed in Section 6, 
is to show that Oda's proof of (1) is, in some sense, 
motivic.  Namely, (a suitably modified version of) Oda's proof can be done 
for unipotent log de Rham fundamental group in purely algebraic way. 
Thus, by combining with the $p$-adic Hodge part of \cite{aik}, we can give 
a purely algebraic proof of the result (3) of Andreatta--Iovita--Kim. 

In order to achieve such a result we have to study unipotent log de Rham fundamental group. There are several methods to define such a group, 
and each method has its own advantage. 
In the situation of (1), Oda described the fundamental group of 
the generic fiber of $f$ in terms of graph of groups associated to the dual graph and 
the fundamental groups of irreducible components of the central fiber, and he proved (1) 
by using this description and some concrete 
computation of  the monodromy action on certain paths. 
In order to do a similar argument for 
unipotent log de Rham fundamental groups, we need the notion of paths, and 
for this purpose, the definition of them  
via Tannakian category (\cite{Del89}, \cite[2.1]{laz}) is suitable. 
On the other hand, we need to calculate concretely  
the unipotent (log) de Rham fundamental group of curves in purely algebraic way, 
and for this purpose, the definition of them via minimal model (\cite{gm}, \cite{naho}, \cite{nagm}) is suitable. 
Also, to finish a purely algebraic proof of 
 the result (3) of Andreatta--Iovita--Kim, we need to use their 
definition of fundamental group, which is based on the construction of 
unipotent fundamental group  given by Hadian (\cite{ha}). 

So we need to compare the three definitions of 
unipotent log de Rham fundamental group above. Note also that we need to prove 
the comparison in relative setting because we need the compatibility of 
monodromy actions. This comparison results, which are interesting in themselves, 
are the subject  of the first five sections of this article. 
We prove them in much more general setting than  it will be needed in the last section.  
As a matter of fact,  in order to achieve the comparison of the three definitions as introduced before, we will need to introduce another one. Each of these  approaches  acquires its own value according to the point of view we want to take and no one seems to us more natural than the others.  For this reason we decided not to give ``{\it the definition}" of relatively unipotent log de Rham fundamental group and we did not favour one with respect to the others: we only proved the equivalence of all of them.

We give a more detailed description of each sections.  In Section 1, we give some preliminaries on modules with integrable log connection: we will introduce the   geometric setting which we consider in the first five sections and  we will define various categories related to the category of modules with integrable log connection which will be used later. In particular, we will show the relations between modules with integrable log connection, stratifications and crystals in this relative setting. 
 We will see later how this relative stratification structure will induce the monodromy action on 
 the unipotent log de Rham fundamental groups.  
In Section 2, we will use the categories defined in the previous section to give our first definition (via Tannakian formalism)  of the unipotent log de Rham fundamental group together with the monodromy action (as a log connection). 
We also  prove  that the definition of monodromy action is the same as 
that one coming from the unipotent version of homotopy exact sequence, although 
a part of the proof will be  postponed in the next section. This is essentially a log version of 
a result of Lazda \cite[Theorem 1.6]{laz}. 
In Section 3, we will introduce the second definition of unipotent log de Rham fundamental group: we will use   the notion of pointed universal object. This is a generalization to higher dimension of the definition of Hadian \cite{ha} and Andreatta--Iovita--Kim \cite{aik}, 
as well as a log version of the construction  which appeared in Lazda's paper \cite[pp.13--15]{laz}. 
We prove that this new  definition is equivalent to that one given in Section 2.
In Section 4, we will give the third definition of unipotent log de Rham fundamental group: we will
use a generalization of the relative minimal model theory of Navarro Aznar (\cite{nagm}). 
In Section 5, we will prove that this third definition is equivalent to the previous ones.  This will be given by introducing the fourth definition, which will be shown to be isomorphic to the third and the first in an independent way. This fourth definition is given by means of  relative bar construction. 
In the absolute case, the equivalence  between the construction via minimal model and that  one via bar construction has been proved by Bloch--Kriz \cite{bk} and the equivalence between the construction via bar constrution and that  one via Tannakian formalism has been proved by Terasoma \cite{terasoma}. The equivalence of monodromy actions follows from the fact that they are defined via relative stratification. 
Finally in Section 6, we restrict ourselves to the case of stable log curves over standard log point and we calculate the monodromy action on unipotent log de Rham fundamental groups of such curves.
We prove that this monodromy action is non-trivial (as an outer action) when the given 
stable log curve is not smooth in classical sense. Using the result of this type 
instead of Oda's one, 
we end  Section  6 with  a purely algebraic proof of the aforementioned result of Andreatta--Iovita--Kim. 

\section*{Acknowledgements}

First of all we would like to thank the two referees for their suggestions and remarks which helped us to improve the article. In particular, we would like to acknowledge the basic help of the first referee in the proof of the crucial result of Subsection 6.7: his suggestion considerably shortened the proof. We are also grateful to the first referee for the subject of Remark \ref{ref.abs}. We would like to thank the second referee for pointing out to us the need to write up the correspondence between nilpotent Lie groups and unipotent group schemes in our relative setting, which is important but missing in the first version. 

The first author is partly supported by MIUR-PRIN 2015 ``Number Theory and Arithmetic 
Algebraic Geometry'' and MIUR-PRIN 2017 "Geometric, algebraic and analytic methods in Arithmetic".  The second author would like to thank Sinan \"Unver for a useful discussion 
about tangential points. The third author is partly supported by 
JSPS KAKENHI (Grant Numbers 25400008, 17K05162, 15H02048, 18H03667 and 18H05233). 
For the preparation of this work, discussions and  revisions were done during various and reciprocal visits 
of the authors at the University of Padova, the Free University of Berlin and the University of Tokyo. 
We would like to thank these institutions for the hospitality.

\section*{Conventions}
We use freely the notions concerning log (formal) schemes which are written in 
\cite{ka}, \cite{il}, \cite[II 5]{agt}, \cite{ogu}, \cite{tsuji19}. In particular (formal) schemes and sheaves are considered in the 
etale site, unless otherwise stated. 
A (formal) scheme is naturally regarded as a fine log (formal) scheme endowed with trivial 
log structure. For example, for a ring $R$, `the fine log scheme $\Spec R$' means 
the fine log scheme $(\Spec R, \text{triv.~log str.})$. 
For a fine log (formal) scheme $X$, denote 
its underlying (formal) scheme by $X^{\circ}$, 
the log structure of $X$ by $\cM_X$ 
and the structure homomorphism of monoids $\cM_X \lra \cO_{X^{\circ}} =: \cO_X$ by 
$\alpha_X$. A log (formal) scheme is called Noetherian if so is its  
underlying scheme.  Throughout this paper, all the log (formal) schemes are 
assumed to be Noetherian, unless otherwise stated. 
A morphism of fine log schemes is called of finite type/proper  
if so is its underlying morphism of schemes. 
A morphism of fine log schemes $f:X \lra Y$ is called 
strict if the morphism of log structures $f^*\cM_Y \lra \cM_X$ induced by $f$ is an isomorphism. Cartesian diagrams of fine log (formal) schemes are considered in 
the category of fine log (formal) schemes, unless otherwise stated. \par 

For a scheme or a ring $T$ and $r,s \in \N$, we denote the log structure on 
$\Af^r_T = \ul{\Spec}_T \cO_T[t_1, ..., t_r]$ associated to 
the monoid homomorphism $\N^{s} \lra \cO_T[t_1, ..., t_r]$ sending 
$e_i \,(1 \leq i \leq \textrm{min}(r,s))$ to $t_i$ and 
sending $e_i \,(r+1 \leq i \leq s)$ to $0$ if $s > r$ 
(where $\{e_i\}_{i=1}^{s}$ is the canonical basis of $\N^{s}$) 
by $\cM_{r,s}$, and denote the log scheme $(\Af^r_T, \cM_{r,s})$ by 
$\Af_T^{r,s}$. For a field $k$, 
we call $\Af_k^{0,1}$ the standard log point over $k$. 

For a log scheme $T$ and $r \in \N$, 
we denote the log formal scheme consisting of the formal scheme 
$\ul{\Spf}_T \cO_T[[t_1, ..., t_r]]$
and the log structure associated to the monoid homomorphism 
$\cM_T \os{\alpha_T}{\lra} \cO_T \hra \cO_T[[t_1, ..., t_r]]$ 
by 
$\wh{\Af}_T^{r}$.

For a morphism $h: Y \lra Z$ of fine log schemes of finite type, we denote by $\Omega^1_{Y/Z}$ the sheaf of relative log differentials. 

\section{Preliminaries}

In this section, after giving some preliminaries on modules with integrable 
log connection, we give a geometric setting which we consider 
throughout this paper, and we define various categories related to 
the category of modules with integrable log connection
 which we use later. \par 
For a morphism $h: Y \lra Z$ of fine log schemes of finite type, 
we denote the category of locally free $\cO_Y$-modules of finite rank 
(resp. coherent $\cO_Y$-modules) with 
log connection on $Y/Z$ by $\MC(Y/Z)$ 
(resp. $\wt{\MC}(Y/Z)$), and denote the full subcategory 
consisting of modules with integrable connection by 
$\MIC(Y/Z)$ (resp. $\wt{\MIC}(Y/Z)$). 
When $Z = \Spec R$ for 
some ring $R$, we denote it simply by $\MIC(Y/R)$ (resp. $\wt{\MIC}(Y/R)$). 
Note that, when $Y = Z$, $\MIC(Y/Z)$ (resp. $\wt{\MIC}(Y/Z)$) is just 
the category of locally free $\cO_Y$-modules of finite rank 
(resp. coherent $\cO_Y$-modules). 
Note also that, even when $\Omega^1_{Y/Z}$ is a locally free $\cO_Y$-module 
of finite rank (which is the case 
we are interested in), $\MIC(Y/Z)$ is a rigid tensor cartegory but 
not necessarily an abelian category, 
due to the possible existence of non-trivial log structure on $Y$. 
On the other hand, under the same assumption on $Y/Z$, 
$\wt{\MIC}(Y/Z)$ is an abelian category, but not necessarily a 
rigid tensor category (\cite[Remark 2.10]{dps}). \par 
For a commutative diagram 
\begin{equation}\label{eq:comm}
\begin{CD}
Y' @>{a}>> Y \\ 
@V{h'}VV @VhVV \\ 
Z' @>{b}>> Z, 
\end{CD}
\end{equation}
the pull-back functors 
$$ a_{\dR}^*: \MIC(Y/Z) \lra \MIC(Y'/Z'), \quad 
a_{\dR}^*: \wt{\MIC}(Y/Z) \lra \wt{\MIC}(Y'/Z')$$ 
are defined in natural way. In particular, the functors 
$$ h_{\dR}^*: \MIC(Z/Z) \lra \MIC(Y/Z), \quad 
h_{\dR}^*: \wt{\MIC}(Z/Z) \lra \wt{\MIC}(Y/Z)$$ are defined. 

An object $E := (E,\nabla)$ in $\wt{\MIC}(Y/Z)$ gives rise to the 
de Rham complex $E \otimes_{\cO_Y} \Omega^{\b}_{Y/Z}$. 
Using it, we define 
the $i$-th relative de Rham cohomology by 
$R^ih_{\dR *}E := R^ih_*(E \otimes_{\cO_Y} \Omega^{\b}_{Y/Z})$. 
When $h$ is proper, it is a coherent $\cO_Z$-module 
(but not necessarily locally free), and so we have the functor 
$$ R^ih_{\dR *}: \wt{\MIC}(Y/Z) \lra \wt{\MIC}(Z/Z). $$
In the following, $R^0h_{\dR *}E$ is denoted simply by $h_{\dR *}E$. 

The adjointness of the functors 
$h^*$, $h_{*}$ for $\cO_?$-modules ($? = Y$ or $Z$) induces the 
morphisms of functors $h_{\dR}^* h_{\dR *} \lra \id, 
\id \lra h_{\dR *}h_{\dR}^*$. 
For $E \in \MIC(Z/Z)$ and $(E',\nabla') \in \wt{\MIC}(Y/Z)$, 
we have 
the projection formula 
\begin{equation*}
R^ih_{\dR *}(h_{\dR}^*E \otimes (E',\nabla')) \cong 
E \otimes_{\cO_Z} R^ih_{\dR *}(E',\nabla'). 
\end{equation*} 
Also, when we have a Cartesian diagram 
\begin{equation*}
\begin{CD}
Y' @>{a}>> Y \\ 
@V{h'}VV @VhVV \\ 
Z' @>b>> Z 
\end{CD}
\end{equation*} 
and $E := (E,\nabla) \in \wt{\MIC}(Y/Z)$, $h$ is separated and integral (so the diagram is a cartesian diagram also for the underling schemes) and when $b$ is flat or 
when $E \otimes_{\cO_Y} \Omega^i_{Y/Z} \,(i \in \N)$ and 
$R^ih_{\dR *}E \,(i \in \N)$ are flat over $\cO_Z$, then we have the base change isomorphism 
$$ b^*R^ih_{\dR *}E \os{\cong}{\lra} R^i{h'}_{\dR *}a_{\dR}^*E. $$
Indeed, we have the base change quasi-isomorphism in derived category 
$$ Lb^*Rh_*(E \otimes_{\cO_Y} \Omega^{\bullet}_{Y/Z} ) \os{\cong}{\lra} 
Rh'_*(a_{\dR}^*E \otimes_{\cO_{Y'}} \Omega^{\bullet}_{Y'/Z'}) $$
by the proof similar to \cite[Theorem 7.8]{berthelotogus}, namely, 
by the base change result in affine case and cohomological descent. 
Then, by taking the $i$-th cohomology on both hand sides, 
we obtain the required base change isomorphism.

For an object $(E,\nabla)$ in $\MIC(Y/Z)$ and a geometric point $y$ of $Y$, 
the map $\nabla: E \lra E \otimes_{\cO_Y} \Omega^1_{Y/Z}$ induces 
a $k(y)$-linear map 
$\nabla_y: E|_y \lra E|_y \otimes_{\Z} \ol{\cM}_{Y,y}^{\rm gp}/
\ol{\cM}_{Z,h(y)}^{\rm gp}$, 
where $E|_y := E \otimes_{\cO_Y} k(y)$, 
$\ol{\cM}_{Y,y}^{\rm gp} := \cM_{Y,y}^{\rm gp}/\cO_{Y,y}^{\times}$, 
$\ol{\cM}_{Z,h(y)}^{\rm gp} := \cM_{Z,h(y)}^{\rm gp}/\cO_{Z,h(y)}^{\times}$. 
For a $\Z$-linear map 
$\psi: \ol{\cM}_{Y,y}^{\rm gp}/\ol{\cM}_{Z,h(y)}^{\rm gp} \lra \Z$,  
the composite 
$(\id \otimes \psi) \circ \nabla_y: E|_y \lra E|_y$ is called the residue map of 
$(E,\nabla)$ at $y$ along $\psi$. 
We say that $(E,\nabla)$ has nilpotent residues at $y$ 
if, for 
any $\Z$-linear map 
$\psi$ as above, the residue map  
$\psi \circ \nabla_y$ is nilpotent as endomorphism on $E|_y$ \cite[Definition 2.11]{dps}. 
Also, in this article, we say that $(E,\nabla)$ has nilpotent residues if 
it has nilpotent residues at any geometric point $y$ of $Y$.
We denote the full subcategory of $\MIC(Y/Z)$ which consists of 
objects having nilpotent residues by $\MICn(Y/Z)$. 

Note that, when we are given the commutative diagram \eqref{eq:comm}, 
the pull-back functor 
$a_{\dR}^*$ induces the functor $\MICn(Y/Z) \lra \MICn(Y'/Z')$, which 
is also denoted by $a_{\dR}^*$. Note also that, when the log structures of 
$Y$ and $Z$ are defined in Zariski topology, the residue map can be defined 
at closed points $y$ (not geometric points) and so 
the nilpotence of residues
can be checked at closed points. 

\begin{rem}\label{rem:nilpnonclosed1}
In \cite[Definition 2.11]{dps}, the definition of 
`having nilpotent residues' is different from the one given above: 
In fact, $(E,\nabla)$ has nilpotent residues 
in the sense in \cite{dps} if it has nilpotent residues 
at any geometric point $y$ over any closed point of $Y$. 
In this article, we denote the full subcategory of $\MIC(Y/Z)$ which consists of 
objects having nilpotent residues in the sense of \cite{dps} by $\MICn(Y/Z)'$. 

The definition in this article has an advantage that the functoriality 
for pullbacks is valid for any commutative diagram \eqref{eq:comm}. 
On the other hand, the definition in \cite{dps} has an advantage that 
it would be easier to prove that an object has nilpotent residues. 

By definition, there exists the canonical inclusion 
$\MICn(Y/Z) \subseteq \MICn(Y/Z)'$, and it is an interesting 
problem if this inclusion is an equality. 
We do not know the answer to this problem in general case, but in some cases, 
the above question has the affirmative answer. 
See \cite[Remark 2.12]{dps} and Remark \ref{rem:nilpnonclosed2} below.
\end{rem}
 
So far, we worked in the category of fine log schemes, but sometimes we need 
to work with fine log formal schemes. 
For a morphism $h: Y \lra Z$ topologically of finite type between 
Noetherian fine log formal schemes, 
we can define the category $\MIC(Y/Z)$ (resp. $\wt{\MIC}(Y/Z)$) 
of locally free $\cO_Y$-modules of finite rank 
(resp. coherent $\cO_Y$-modules) with 
integrable log connection on $Y/Z$. (In the definition of integrable log 
connection, we use the continuous log differential module $\Omega^1_{Y/Z}$.) 
We can define the pull-back functor and 
the $i$-th relative de Rham cohomology 
$R^ih_{\dR *}E := R^ih_*(E \otimes_{\cO_Y} \Omega^{\b}_{Y/Z})$ 
\, $(E \in \wt{\MIC}(Y/Z))$ also in this situation. 
We often use the following proposition later. 

\begin{prop}\label{prop:fib}
Let $Y, Z$ be fine log schemes over $\Q$, 
$Y'$ a fine log formal scheme over $\Q$ and 
let $Y' \os{h'}{\lra} Y \os{h}{\lra} Z$ be morphisms over $\Q$ 
such that $h$ is 
of finite type with $\Omega^1_{Y/Z}$ locally free 
and that, etale locally on $Y$, 
$h'$ is isomorphic to the projection 
$\wh{\Af}^d_{Y} \lra Y$ for some 
$d$. $($For the definition of $\wh{\Af}^d_{Y}$, see Convention.$)$
Then, for any $E \in \MIC(Y/Z)$, we have the canonical isomorphism 
$$ R^ih_{\dR *}E \os{=}{\lra} R^i(h \circ h')_{\dR *}({h'}^*_{\dR} E) \quad (i \in \N). $$
\end{prop}

\begin{proof}
Let us denote by 
$E \otimes_{\cO_Y} \Omega^{\b}_{Y/Z}$ 
(resp. ${h'}^*E \otimes_{\cO_{Y'}} \Omega^{\b}_{Y'/Z}$)  
the de Rham complex associated to $E$ (resp. ${h'}^*_{\dR} E$). 
Then it suffices to prove that the canonical morphism 
$$ E \otimes_{\cO_Y} \Omega^{\b}_{Y/Z} \lra 
Rh'_*({h'}^*E \otimes_{\cO_{Y'}} \Omega^{\b}_{Y'/Z}) 
$$
is a quasi-isomorphism. Since we can work etale locally on $Y$, we may assume that $h'$ is the projection 
$\wh{\Af}^d_{Y} \lra Y$. 
Then we have the equality 
$$ {h'}^*E \otimes_{\cO_{Y'}} \Omega^{\b}_{Y'/Z} = 
{\rm Tot}({h'}^*(E \otimes_{\cO_Y} \Omega^{\b}_{Y/Z}) 
\otimes_{\cO_{\wh{\Af}^d_Y}} 
\Omega^{\b}_{\wh{\Af}^d_Y/Y}), $$
and by using the projection formula, we see that it suffices to prove 
that the canonical injective homomorphism 
$$ \iota: \cO_{Y} \lra {h'}_*
\Omega^{\b}_{\wh{\Af}^d_Y/Y} $$
is a quasi-isomorphism. 
This is standard, but we give a proof for the convenience 
of the reader. Let $\pi: \Omega^{\b}_{\wh{\Af}^d_Y/Y} \lra 
\cO_{Y}$ be the homomorphism of `taking constant term'. Then 
$\pi \circ \iota$ is the identity, and so it suffices to prove that 
$\iota \circ \pi$ is homotopic to the identity map on $\Omega^{\b}_{\wh{\Af}^d_Y/Y}$. 
A homotopy is given by 
$$ 
t_1^{i_1} \cdots t_d^{i_d} dt_{j_1} \wedge \cdots \wedge dt_{j_m} 
\mapsto 
\begin{cases}
(i_1+1)^{-1}t_1^{i_1+1} \cdots t_d^{i_d} dt_{j_2} \wedge \cdots \wedge dt_{j_m} &
\text{(if $j_1=1$)}, \\ 
0 & \text{(otherwise)}, 
\end{cases}
$$
where $t_1, \dots, t_d$ are the coordinate of $\wh{\Af}^d_Y$ and 
$i_1, \dots, i_d \geq 0$, $1 \leq j_1 < \cdots < j_m \leq d$. 
Note that the first element on the right hand side is well-defined 
because the characteristic of $Y$ is zero. So the proof of the proposition is finished. 
\end{proof}

Next we introduce the notion of stratifications and crystals, and 
relate them to modules with integrable connections. 
To do so, first we recall the definition of 
log infinitesimal neighborhoods.

\begin{propdefn}[{\cite[Remark 5.8]{ka}}]\label{propdef:loginfnbd}
Let $m \in \N$. 
Let $\cC$ be the category of closed immersions 
{\rm \cite[(3.1)]{ka}} $X \hra Y$ of log schemes such that 
$\cM_X$ is fine and $\cM_Y$ is coherent {\rm \cite[(2.1)]{ka}}, and let 
$\cC_m$ be the category of exact closed immersions of fine log 
schemes $X \hra Y$ such that $X^{\circ}$ is defined in $Y^{\circ}$ 
by an ideal $J$ with $J^{m+1} = 0$. Then, the canonical inclusion 
functor $\cC_m \hra \cC$ has a right adjoint $\alpha: \cC \lra \cC_m$ 
such that $\alpha(X \hra Y)$ has the form $X \hra Y'$. We call the 
fine log scheme $Y'$ the $m$-th log infinitesimal neighborhood 
of $X$ in $Y$. 
\end{propdefn}

For the concrete construction of log infinitesimal neighborhoods, 
see \cite[Remark 5.8]{ka}. 

\begin{rem}\label{rem:loginfnbd}
In this remark, we keep the notation of Proposition-Definition \ref{propdef:loginfnbd}. 
Also, we denote the fiber product in the category of log schemes 
by $\times'$, and keep to denote the fiber product in the category of fine 
log schemes by $\times$. It is known that the category of coherent log schemes 
is closed under finite inverse limits of log schemes \cite[Proposition 2.6]{ka}, and 
the canonical inclusion functor 
$$ \text{(fine log schemes)} \hra \text{(coherent log schemes)} $$
has a right adjoint 
$$ \beta: \text{(coherent log schemes)} \lra \text{(fine log schemes)} $$
(\cite[Proposition 2.7]{ka}). So, if $\{X_{\lambda}\}_{\lambda}$ is 
a finite diagram of fine log schemes and if $X$ is its inverse limit in the category of 
log schemes, the inverse limit of the diagram $\{X_{\lambda}\}_{\lambda}$ 
in the category of fine log schemes is given by $\beta(X)$. Also, we see from 
the construction of the functor $\beta$ in \cite[Proposition 2.7]{ka} that, 
if we define $\cC'$ to be the category of closed immersions of fine log schemes, 
the canonical inclusion functor $\cC' \hra \cC$ has a right adjoint 
(denoted also by $\beta$) given by $(X \hra Y) \mapsto (X \hra \beta(Y))$. 
Moreover, it is known by the construction given in \cite[(5.6), Remark 5.8]{ka} that 
the functor $\alpha: \cC \lra \cC_m$ factors as 
$\cC \os{\beta}{\lra} \cC' \os{\alpha'}{\lra} \cC_m$. 

From these facts, we see the following: for a morphism 
$X \lra Y$ of fine log schemes, we have the equalities 
\begin{align*}
\alpha (X \hra Y \times'_X \cdots \times'_X Y) 
& = 
\alpha' \circ \beta (X \hra Y \times'_X \cdots \times'_X Y) \\ & 
= \alpha'(X \hra \beta (Y \times'_X \cdots \times'_X Y) ) \\ 
& = \alpha'(X \hra Y \times_X \cdots \times_X Y) \\ & 
= \alpha(X \hra Y \times_X \cdots \times_X Y). 
\end{align*}
So, when we consider the $m$-th log infinitesimal neighborhood of 
a diagonal embedding for a morphism of fine log schemes, 
it does not matter in which category we take the fiber product. 
\end{rem}

Now we define the notion of stratifications and crystals.

\begin{defn}\label{def:strat}
Let $k$ be a field of characteristic zero and 
let $h:Y \lra Z$ be a morphism of fine log schemes of finite type over $k$. 
For $m,j \in \N$, let $Y^m(j)$ be the $m$-th log infinitesimal 
neighborhood of $Y$ in $Y \times_Z \cdots \times_Z Y$ $((j+1)$ times$)$ and 
let $p^m_j: Y^m(1) \lra Y \,(j=1,2), p^m_{j,j'}: Y^m(2) \lra Y^m(1) \,
(1 \leq j < j' \leq 3)$ be projections. 
Then a stratification on a coherent $\cO_Y$-module $E$ 
$($with respect to $h)$ is 
a compatible family of $\cO_{Y^m(1)}$-linear isomorphisms 
$\{\epsilon^m: {p^m_2}^*E \lra 
{p^m_1}^*E \}_m$ with $\epsilon_0 = \id$ 
and ${p^m_{1,3}}^*(\epsilon^m) = 
{p^m_{1,2}}^*(\epsilon^m) \circ {p^m_{2,3}}^*(\epsilon^m)$. 
We denote the category of locally free $\cO_Y$-modules of finite rank 
$($resp. coherent $\cO_Y$-modules$)$ endowed with stratification 
$($with respect to $h)$ by $\Strat(Y/Z)$ $($resp. $\wt{\Strat}(Y/Z))$. 
\end{defn}

\begin{defn}\label{def:inf-strat}
Let $k$ and $h:Y \lra Z$ be as in Definition \ref{def:strat}. \\
$(1)$ \, We define the infinitesimal site $(Y/Z)_{\inf}$ of $Y/Z$ in 
the following way$:$ An object is given by a triple $(U, T, i)$, 
where $U$ is a $($Noetherian$)$ fine log scheme over $Y$, 
$T$ is a $($Noetherian$)$ fine log scheme of finite type over $Z$ and 
$i$ is a nilpotent exact closed immersion 
$U \lra T$ over $Z$. 
A morphism is defined in natural way and a covering is the one naturally induced 
from the etale topology on $T$. \\
$(2)$ \, We define the stratifying site $(Y/Z)_{\str}$ of $Y/Z$ 
as the full subcategory of $(Y/Z)_{\inf}$ consisting of triples 
$(U, T, i)$ which admit etale locally a morphism $T \lra Y$ over $Z$ 
extending $U \lra Y$. \\
$(3)$ \, We define the sheaf $\cO_{Y/Z}$ on $(Y/Z)_{\inf}$ or on $(Y/Z)_{\str}$
by $\cO_{Y/Z}((U,T,i)) := \Gamma(T,\cO_T)$. \\ 
$(4)$ \, For a sheaf $E$ on $(Y/Z)_{\inf}$ $($resp. $(Y/Z)_{\str})$ and 
an object $(U,T,i)$ in $(Y/Z)_{\inf}$ $($resp. $(Y/Z)_{\str})$, we denote 
the sheaf on $T_{\et}$ induced by $E$ by $E_T$. \\ 
$(5)$ \, We call a sheaf $E$ of $\cO_{Y/Z}$-modules 
on $(Y/Z)_{\inf}$ $($resp. $(Y/Z)_{\str})$ a crystal 
if, for any morphism $\varphi: (U',T',i) \lra (U,T,i)$ in 
$(Y/Z)_{\inf}$ $($resp. $(Y/Z)_{\str})$, the morphism 
$\varphi^*E_T := \cO_{T'} \otimes_{\varphi^{-1}\cO_T} \varphi^{-1}E_T 
\lra E_{T'}$ induced by the definition of $E$ is an isomorphism. \\
$(6)$ \, We call a crystal $E$ on $(Y/Z)_{*} \, (* \in \{\inf, \str\})$ 
locally free $($resp, coherent$)$ if $E_T$ is a locally free $\cO_T$-module 
of finite rank $($resp. coherent $\cO_T$-module$)$ for any $(U,T,i)$. 
We denote the category of locally free $($resp. coherent$)$ crystals on 
$(Y/Z)_*$ by $\Crys((Y/Z)_*)$ $($resp. $\wt{\Crys}((Y/Z)_*))$. 
\end{defn}

When $h: Y \lra Z$ is log smooth, $(Y/Z)_{\str} = (Y/Z)_{\inf}$ by 
the infinitesimal lifting property of log smooth morphisms, and so 
we have an equality $\Crys((Y/Z)_{\str}) = \Crys((Y/Z)_{\inf})$ (resp. $\wt{\Crys}((Y/Z)_{\str}) = \wt{\Crys}((Y/Z)_{\inf})$). In general, 
the inclusion $(Y/Z)_{\str} \hra (Y/Z)_{\inf}$ induces  
the natural functor $\Crys((Y/Z)_{\inf}) \lra \Crys((Y/Z)_{\str})$ (resp. $\wt{\Crys}((Y/Z)_{\inf}) \lra \wt{\Crys}((Y/Z)_{\str})$). 
The categories $\Crys((Y/Z)_*)$, $\wt{\Crys}((Y/Z)_*)$ ($* \in \{\inf, \str\}$) are contravariantly 
functorial with respect to $h: Y \lra Z$: Indeed, if we are given 
a commutative diagram \eqref{eq:comm} and 
an object $(U, T, i)$ in $(Y'/Z')_{*}$, it is also regarded as an object in 
$(Y/Z)_{*}$. Hence, if we are given an object $E$ in $\Crys((Y/Z)_*)$ (resp. 
$\wt{\Crys}((Y/Z)_*)$), we can define a sheaf $h^*E$ on 
 $(Y'/Z')_{*}$ by putting 
$(h^*E)((U,T,i)) := E((U,T,i))$, and we can check that 
$h^*E$ is in fact an object in $\Crys((Y'/Z')_*)$ (resp. $\wt{\Crys}((Y'/Z')_*)$). 

We explain the relation among 
modules with integrable log connection, stratifications and crystals. 
Let $k, h:Y \lra Z$ be as in as in Definition \ref{def:strat} and 
fix the symbol $* \in \{\str, \inf\}$. Let us take an object 
$(U,T,i)$ in $(Y/Z)_{*}$. If we denote the 
the $m$-th log infinitesimal 
neighborhood of $T$ in $T \times_Z \cdots \times_Z T$ ($(j+1)$ times) 
by $T^m(j)$, the triple $(U,T^m(j), U \os{i}{\ra} T \os{\Delta}{\ra} T^m(j))$ 
(where $\Delta$ is the morphism induced by the diagonal morphism 
$T \hra T \times_Z \cdots \times_Z T$) is an object in $(Y/Z)_{*}$. 
Therefore, if we are given an object $E$ in $\wt{\Crys}((Y/Z)_{*})$, 
we obtain naturally an object of the form $(E_T, \{\epsilon^m\}_m)$ 
in $\wt{\Strat}(T/Z)$ by defining $\epsilon^m$ to be the isomorphism 
$$ {p_2^m}^*E_T \os{\cong}{\lra} E_{T^m(1)} \os{\cong}{\lla} {p_1^m}^*E_T $$
induced by the structure of the crystal $E$. 
(Here $p_j^m: T^m(1) \lra T$ are the projections.) When $E$ is locally free, 
$E_T$ is locally free. Thus we obtain functors 
\begin{equation}\label{eq:trinity1}
\Crys((Y/Z)_{*}) \lra \Strat(T/Z), \qquad 
\wt{\Crys}((Y/Z)_{*}) \lra  \wt{\Strat}(T/Z). 
\end{equation}
Also, when we are given an object $(E, \{\epsilon^m\}_m)$ 
in $\wt{\Strat}(T/Z)$, we can define a structure of a log connection $\nabla: 
E \lra E \otimes_{\cO_T} \Omega^1_{T/Z}$ by putting 
$$ 
\nabla(e) := \epsilon^1({p_2^1}^*(e)) - {p_1^1}^*(e)
\in \Ker({p_1^1}^*E \lra E) \cong E \otimes_{\cO_T} \Omega^1_{T/Z}, 
$$
where the last isomorphism is induced by the canonical isomorphism 
$\Ker(\cO_{T^1(1)} \lra \cO_T) \cong \Omega^1_{T/Z}$ 
\cite[Proposition 3.2.5]{shihocrysI}.  
(One can prove the Leibniz rule 
for $\nabla$ in the same way as the proof of \cite[Proposition 2.9]{berthelotogus}.) 
Thus we obtain functors 
\begin{align}\label{eq:trinity2} 
& \Strat(T/Z) \lra \MC(T/Z), \qquad   
\wt{\Strat}(T/Z) \lra \wt{\MC}(T/Z). 
\end{align}
The functors we defined above are contravariantly functorial with respect to 
$h:Y \lra Z$ and $(U,T,i)$. Concerning the above functors 
in the case $U = Y$, we have the following proposition. 

\begin{prop}\label{prop:trinity}
Let $k$, $h:Y \lra Z$,  be as in Definition \ref{def:strat} and $(U,T,i)\in (Y/Z)_*$. Assume that $U=Y$ 
and that $\Omega^1_{T/Z}$ is a locally free $\cO_T$-module of finite rank. 
Then the functors \eqref{eq:trinity1} for $* = \str$ and the functors \eqref{eq:trinity2} induce the equivalences of categories 
\begin{align}
& \Crys((Y/Z)_{\str}) \os{\cong}{\lra}  \Strat(T/Z) \os{\cong}{\lra} 
\MIC(T/Z), \label{eq:tri1} \\ 
& \wt{\Crys}((Y/Z)_{\str}) \os{\cong}{\lra} \wt{\Strat}(T/Z) 
\os{\cong}{\lra} \wt{\MIC}(T/Z).  \label{eq:tri2}
\end{align}
Furthermore, when $T \lra Z$ is log smooth, 
the functors \eqref{eq:trinity1} for $* = \inf$ induce 
the equivalences of categories 
\begin{equation}\label{eq:tri3}
\Crys((Y/Z)_{\inf}) \os{\cong}{\lra} \Strat(T/Z), \quad  
\wt{\Crys}((Y/Z)_{\inf}) \os{\cong}{\lra} \wt{\Strat}(T/Z).  
\end{equation}
\end{prop}

\begin{proof}
We see that \eqref{eq:tri1} follows from \eqref{eq:tri2} and that 
the first equivalence in \eqref{eq:tri3} follows from the second equivalence 
by restricting to the subcategories of locally free objects. 
So it suffices to prove the equivalences \eqref{eq:tri2} 
and the second equivalence in \eqref{eq:tri3}. 
The strategy of the proof is similar to that in \cite[\S 2]{berthelotogus}, 
\cite[\S 3.2]{shihocrysI} and \cite[\S 1.2]{shihocrysII}. 

First we prove the first equivalence in \eqref{eq:tri2}. 
To do so, it suffices to construct a quasi-inverse 
\begin{equation}\label{eq:tri4}
\wt{\Strat}(T/Z) \lra \wt{\Crys}((Y/Z)_{\str})
\end{equation}
of the second functor in \eqref{eq:trinity1} for $* = \str$. 
Let $(E,\{\epsilon^m\}_m)$ be an object in $\wt{\Strat}(T/Z)$ and take 
an object $(U',T',i')$ in $(Y/Z)_{\str}$. By definition, etale locally on $T'$, 
there exists a morphism $\varphi_1: T' \lra Y$ over $Z$ extending 
$U' \lra Y$. Then we define the etale sheaf $E_{T'}$ on $T'$ by 
$E_{T'} := (i \circ \varphi_1)^*E$. This definition works only etale locally on $T'$
and depends on the choice of $\varphi_1$. 
If there exists another morphism 
$\varphi_2: T' \lra Y$ as $\varphi_1$, we see that the morphism 
$(i \circ \varphi_1) \times (i \circ \varphi_2):  T' \lra T \times_Z T$ factors as 
$$ T' \os{\varphi_{12}}{\lra} T^m(1) \lra T \times_Z T $$
for some $m$, where the second morphism is induced by the diagonal map 
$T \hra T \times_Z T$. 
Thus we have the isomorphism 
$$ \varphi_{12}^*\epsilon^m: (i \circ \varphi_2)^*E \os{\cong}{\lra} 
(i \circ \varphi_1)^*E, $$
and we see that it satisfies the cocycle condition. Hence we can define 
the etale sheaf $E_{T'}$ globally on $T'$ by gluing the local definition 
by etale descent. Also, we see that the definition of $E_{T'}$ does not depend on 
any choice, and that the sheaf on $(Y/Z)_{\str}$ defined by 
$(U',T',i') \mapsto E_{T'}(T')$ is a crystal. Thus we have defined the functor 
\eqref{eq:tri4}, and we can check that this is in fact a quasi-inverse of 
\eqref{eq:trinity1} for $* = \str$. 

Next we prove the second equivalence in \eqref{eq:tri3}. 
We construct a quasi-inverse 
\begin{equation}\label{eq:tri01}
\wt{\Strat}(T/Z) \lra \wt{\Crys}((Y/Z)_{\inf})
\end{equation}
of the functor \eqref{eq:trinity1} for $* = \inf$. 
The construction is similar to that in the previous paragraph: 
For an object $(E,\{\epsilon^m\}_m)$ in $\wt{\Strat}(T/Z)$ and  
an object $(U',T',i')$ in $(Y/Z)_{\inf}$, there exists a morphism  
$\varphi_1: T' \lra T$ over $Z$ extending 
$U' \lra Y \os{i}{\lra} T$ etale locally on $T'$, by the log smoothness of $T$ over $Z$. 
Then we define the etale sheaf $E_{T'}$ on $T'$ by 
$E_{T'} := \varphi_1^*E$ locally on $T'$, and define $E_{T'}$ globally by gluing local definition 
using $\epsilon^m$ for some $m$ as above. 
Then the sheaf on $(Y/Z)_{\str}$ defined by 
$(U',T',i') \mapsto E_{T'}(T')$ is a crystal and thus 
we have defined the functor 
\eqref{eq:tri01}. We can check also that this is a quasi-inverse of 
\eqref{eq:trinity1} for $* = \inf$. 

Finally we prove the second equivalence in \eqref{eq:tri2}. 
To do so, first we prove that the essential image of the second functor in \eqref{eq:trinity2} 
is contained in $\wt{\MIC}(T/Z)$, thus induces the functor 
\begin{equation}\label{eq:tri5}
\wt{\Strat}(T/Z) \lra \wt{\MIC}(T/Z). 
\end{equation}
To prove this, we may work etale locally. So we assume that 
$\Omega^1_{T/Z}$ has a basis of the form $\{ \dlog x_i \}_{i=1}^r$ for some 
$x_1, \dots, x_r \in \cM_T$. Then there exists a unique element 
$u_i \in \varprojlim_m \Ker(\cO_{T^m(1)}^{\times} \lra \cO_T^{\times})$ for each  
$1 \leq i \leq r$ satisfying 
\begin{equation}\label{eq:20190128}
({p_2^m}^*(x_i))_m = ({p_1^m}^*(x_i))_m \cdot u_i \in \varprojlim_m \cM_{T^m(1)}. 
\end{equation}
We put $\xi_i := \sum_{n=1}^{\infty} (-1)^{n-1} (u_i-1)^n/n \in 
\varprojlim_m \cO_{T^m(1)}$ and we will denote the images of $u_i, \xi_i$ in 
$\cO_{T^m(1)}$ by the same letters, by abuse of notation. 
Recall that the canonical isomorphism 
$\Ker(\cO_{T^1(1)} \lra \cO_T) \cong \Omega^1_{T/Z}$ in 
\cite[Proposition 3.2.5]{shihocrysI} sends $\xi_i=u_i -1$ on the left hand side to 
$\dlog x_i$ on the right hand side. Hence $\{\xi_i\}_{i=1}^r$ is a basis of 
$\Ker(\cO_{T^1(1)} \lra \cO_T)$ and so $\{1, \xi_1, \dots, \xi_r\}$ is a basis of 
$\cO_{T^1(1)}$. 

We prove that, for any $m \geq 1$, the elements 
$\xi^q \,(|q| \leq m)$ form an $\cO_T$-basis of $\cO_{T^m(1)}$. 
(Here we use multi-index notation.)
To do so, first recall that the diagonal morphism 
$T \lra T \times_Z T$ factors etale locally as 
$$ T \lra (T \times_Z T)' \lra T \times_Z T $$
with the first arrow (resp. the second arrow) 
an exact closed immersion 
(resp. a log etale morphism) and that, if we put 
$\cI := \Ker(\cO_{(T \times_Z T)'} \lra \cO_T)$, 
$\cO_{T^m(1)}$ is defined as 
$\cO_{T^m(1)} := \cO_{(T \times_Z T)'}/\cI^{m+1}$. 
From this description, we see easily that the set 
$\{\xi^q\}_{|q| \leq m}$ generates $\cO_{T^m(1)}$ as $\cO_T$-module.
So it suffices to show that, assuming the linear independence of 
the elements $\xi^q \,(|q| \leq m)$ in $\cO_{T^m(1)}$ over $\cO_T$, 
the elements $\xi^q \,(|q| \leq 2m)$ in $\cO_{T^{2m}(1)}$ are also linearly independent 
over $\cO_T$. 
By the universality of log infinitesimal neighborhood, 
there exist canonical morphisms 
$\delta^m: T^m(1) \times_T T^m(1) \lra T^{2m}(2) \, (m \in \N)$ such that 
$\{\delta^m\}_m$ is an isomorphism as inductive system of log schemes. 
Denote the composite 
$T^m(1) \times_T T^m(1) \os{\delta^m}{\lra} T^{2m}(2) \os{p_{1,3}^{2m}}{\lra} T^{2m}(1)$ 
by $\ol{\delta}^m$. Thus $\ol{\delta}^m$ is the morphism induced by the composite 
$$ (T\times_Z T) \times_T (T \times_Z T) = T \times_Z T \times_Z T \longrightarrow T \times_Z T, $$
where the last morphism is the projection to the first and the third factors.

Then the morphism 
$$ {\ol{\delta}^m}^*: \cO_{T^{2m}(1)} \lra \cO_{T^m(1)} \otimes_{\cO_T} 
\cO_{T^m(1)} $$ 
sends $u_i$ to $u_i \otimes u_i$: This fact is proven in \cite[p.562]{shihocrysI} 
but we provide a proof here for the convenience of the reader. 
If we denote the map $\varprojlim_m \cM_{T^m(1)} \longrightarrow \varprojlim_m \cM_{T^m(1) \times_T T^m(1)}$ 
induced by the $i$-th projections $T^m(1) \times_T T^m(1) \lra T^m(1) \, (m \in \N)$ 
by $\pi_i$ (for $i=1,2$), we see from the characterization 
of $\ol{\delta}^m$ given above that the map of projective limit of log structures 
$$(\ol{\delta}^m)_m^*: \varprojlim_m\cM_{T^{2m}(1)} \lra \varprojlim_m\cM_{T^m(1) \times_T T^m(1)} $$
induced by $(\ol{\delta}^m)_m \, (m \in \N)$ 
sends $(p_1^{2m *}(x_i))_m$ to $\pi_1^*(p_1^{m *}(x_i))_m$ and 
$(p_2^{2m *}(x_i))_m$ to $\pi_2^*(p_2^{m *}(x_i))_m$. 
Then, by pulling back the equality \eqref{eq:20190128} via $(\ol{\delta}^m)_m$, we obtain the equality 
\begin{equation*}
\pi_2^*({p_2^m}^*(x_i))_m = \pi_1^*({p_1^m}^*(x_i))_m \cdot ({\ol{\delta}^m}^*(u_i))_m \in \varprojlim_m \cM_{T^m(1)}. 
\end{equation*}
On the other hand, by pulling back the equality \eqref{eq:20190128} via the first and the 
second projections $T^m(1) \times_T T^m(1) \lra T^m(1) \, (m \in \N)$, we obtain the 
equalities 
\begin{align*}
& \pi_1^*({p_2^m}^*(x_i))_m = \pi_1^*({p_1^m}^*(x_i))_m \cdot (u_i \otimes 1) \in \varprojlim_m \cM_{T^m(1)}, \\  
& \pi_2^*({p_2^m}^*(x_i))_m = \pi_2^*({p_1^m}^*(x_i))_m \cdot (1 \otimes u_i) \in \varprojlim_m \cM_{T^m(1)}. 
\end{align*}
Combining these three equalities, we obtain the equalities ${\ol{\delta}^m}^*(u_i) = u_i \otimes u_i \,(m \in \N)$, 
as required. 

Since ${\ol{\delta}^m}^*$ sends $u_i$ to $u_i \otimes u_i$, it sends 
$\xi_i$ to $1 \otimes \xi_i + \xi_i \otimes 1$. 
(This follows from the equality of formal power series 
$$ \sum_{n=1}^{\infty} (-1)^{n-1}(xy + x + y)^n/n 
= \sum_{n=1}^{\infty}(-1)^{n-1}x^n/n + 
\sum_{n=1}^{\infty}(-1)^{n-1}y^n/n. )$$
So, for $q =(q_1, ..., q_r) \in \N^r$ with $|q| \leq 2m$, 
the morphism ${\ol{\delta}^m}^*$
sends $\xi^q := \prod_{i=1}^r \xi_i^{q_i}$ to 
$(1 \otimes \xi + \xi \otimes 1)^q := \prod_{i=1}^r (1 \otimes \xi_i + \xi_i \otimes 1)^{q_i}$, 
which is an $\N_{>0}$-linear combination of elements of the form 
$\xi^a \otimes \xi^b \,(a,b \in \N^r, a+b = q, |a|, |b| \leq m)$. 
If we put $S_q := \{(a,b) \in \N^r \times \N^r \,|\, a+b = q, |a|, |b| \leq m\}$, 
it is non-empty and $S_q$'s ($q \in \N^r, |q| \leq 2m$) are all disjoint. 
Thus the elements $(1 \otimes \xi + \xi \otimes 1)^q$\,($q \in \N^r, |q| \leq 2m$) 
in $\cO_{T^m(1)} \otimes_{\cO_T} \cO_{T^m(1)}$ 
are linearly independent over $\cO_T$ and thus the elements 
$\xi^q \,(|q| \leq 2m)$ in $\cO_{T^{2m}(1)}$ are also linearly independent 
over $\cO_T$, as required. 

Note that the argument in the previous paragraph shows that 
the map ${\ol{\delta}^m}^*$ is an injection into a direct summand as $\cO_T$-module. 

Let $(E,\{\epsilon^m\}_m)$ be an object in $\wt{\Strat}(T/Z)$, let $e \in E$ and 
write the image of $1 \otimes e \in \cO_{T^2(1)} \otimes_{\cO_T} E = {p_2^2}^*E$ 
by $\epsilon^2$ as 
$$ \epsilon^2(1 \otimes e) = e \otimes 1 + 
\sum_{i=1}^r \partial_i(e) \otimes \xi_i + \sum_{1 \leq i \leq j \leq r} \partial_{ij}(e) \otimes \xi_i\xi_j 
\in E \otimes_{\cO_T} \cO_{T^2(1)} = {p_1^2}^*E. $$
The cocycle condition for 
$\{\epsilon^m\}_m$ implies that the morphism 
\begin{equation}\label{eq:tri6}
{\ol{\delta}^1}^*(\epsilon^2): \cO_{T^1(1)} \otimes_{\cO_T} \cO_{T^1(1)} \otimes_{\cO_T} E 
\lra E \otimes_{\cO_T} \cO_{T^1(1)} \otimes_{\cO_T} \cO_{T^1(1)} 
\end{equation}
is equal to the composite  
\begin{align}
\cO_{T^1(1)} \otimes_{\cO_T} \cO_{T^1(1)} \otimes_{\cO_T} E 
& \os{\id \otimes \epsilon^1}{\lra}  
\cO_{T^1(1)} \otimes_{\cO_T} E \otimes_{\cO_T} \cO_{T^1(1)} 
\label{eq:tri7} \\ 
& \os{\epsilon^1 \otimes \id}{\lra}
E \otimes_{\cO_T} \cO_{T^1(1)} \otimes_{\cO_T} \cO_{T^1(1)}. \nonumber 
\end{align}
The image of $1 \otimes e$ by the morphism 
${\ol{\delta}^1}^*(\epsilon^2)$ in \eqref{eq:tri6} is equal to  
$$ e \otimes 1 \otimes 1 + \sum_{i=1}^r \partial_i(e) \otimes (1 \otimes \xi_i + \xi_i \otimes 1) + 
\sum_{1 \leq i \leq j \leq r} \partial_{ij}(e) \otimes (\xi_i \otimes \xi_j + \xi_j \otimes \xi_i), $$
while that by the composite morphism in \eqref{eq:tri7} is calculated as 
\begin{align*}
&  (\epsilon^1 \otimes \id)\circ(\id \otimes \epsilon^1) 
(1 \otimes 1 \otimes e) \\ = \, & 
(\epsilon^1 \otimes \id) 
(1 \otimes e \otimes 1 + \sum_{i=1}^r 1 \otimes \partial_i(e) \otimes \xi_i) \\ 
= \, & 
e \otimes 1 \otimes 1 + 
\sum_{i=1}^r \partial_i(e) \otimes \xi_i \otimes 1 + 
\sum_{i=1}^r \partial_i(e) \otimes 1 \otimes \xi_i + 
\sum_{i,j=1}^r \partial_j(\partial_i(e)) \otimes \xi_j \otimes \xi_i. 
\end{align*}
Thus we have $\partial_i(\partial_j(e)) = \partial_{ij}(e) = \partial_j(\partial_i(e))$ for 
$1 \leq i < j \leq r$. On the other hand, if we denote the image of 
$(E,\{\epsilon^m\}_m)$ by the second functor in \eqref{eq:trinity2} 
by $(E,\nabla)$, we have 
$\nabla(e) = \sum_{i=1}^r \partial_i(e) \otimes \dlog x_i $. So we can calculate 
its curvature as 
\begin{align*}
\nabla \circ \nabla(e) & = 
\nabla(\sum_{i=1}^r \partial_i(e) \otimes \dlog x_i) = 
\sum_{i,j} \partial_j(\partial_i(e)) \otimes \dlog x_j \wedge \dlog x_i \\ 
& = 
\sum_{i<j} (\partial_i(\partial_j(e)) - 
\partial_j(\partial_i(e))) \otimes \dlog x_i \wedge \dlog x_j = 0. 
\end{align*}
So $(E,\nabla)$ is integrable and hence the functor 
\eqref{eq:tri5} is defined. 

Finally we prove that the functor \eqref{eq:tri5} is an equivalence. 
To do so, it suffices to construct its quasi-inverse etale locally. 
If we are given an object $(E,\nabla) \in \wt{\MIC}(T/Z)$ and write 
$\nabla$ locally as $\nabla(e) = \sum_{i=1}^r \partial_i(e) \otimes \dlog x_i$, 
we define the maps $\{\epsilon^m: \cO_{T^m(1)} \otimes_{\cO_T} E \lra E \otimes_{\cO_T} \cO_{T^m(1)}\}_m$ by 
$$ \epsilon^m(1 \otimes e) = \sum_{|q| \leq m} (\partial^q(e)/q!) \otimes \xi^q. $$
They are isomorphisms because the inverses are defined by 
$e \otimes 1 \mapsto \sum_q \xi^q \otimes ((-\partial)^q(e)/q!)$.   
Also, we can check the coincidence of the maps 
$$ {\ol{\delta}^m}^*(\epsilon^{2m})
: \cO_{T^m(1)} \otimes_{\cO_T} \cO_{T^m(1)} \otimes_{\cO_T} E 
\lra E \otimes_{\cO_T} \cO_{T^m(1)} \otimes_{\cO_T} \cO_{T^m(1)}, $$
\begin{align*} 
\cO_{T^m(1)} \otimes_{\cO_T} \cO_{T^m(1)} \otimes_{\cO_T} E 
& \os{\id \otimes \epsilon^m}{\lra}  
\cO_{T^m(1)} \otimes_{\cO_T} E \otimes_{\cO_T} \cO_{T^m(1)} \\ 
& \os{\epsilon^m \otimes \id}{\lra}
E \otimes_{\cO_T} \cO_{T^m(1)} \otimes_{\cO_T} \cO_{T^m(1)}. 
\end{align*}
Since $\{\delta^m: T^m(1) \times_T T^m(1) \lra T^{2m}(2)\}_m$ is an isomorphism 
as inductive system of log schemes, 
the coincidence of the above maps 
implies the cocycle condition for the maps $\{\epsilon^m\}_m$. 
Hence it defines a stratification on $E$. Moreover, we see from the coincidence of the 
above two maps that the map $\epsilon^{2m}$ is determined uniquely by the map 
$\epsilon^m$, because the map ${\ol{\delta}^m}^*$ is an injection into a 
direct summand as $\cO_T$-module. So we see that the above stratification  $\{\epsilon^m\}_m$ is the unique one on $E$ which gives rise to the connection $\nabla$. 
Hence the functor 
$(E,\nabla) \mapsto (E,\{\epsilon^m\}_m)$ gives a local quasi-inverse of 
\eqref{eq:tri5}. So \eqref{eq:tri5} is an equivalence and the proof of the proposition 
is now finished. 
\end{proof}
 
\begin{rem}\label{rem:trinity}
Let $k, h: Y \lra Z$ be as in Definition \ref{def:strat}, 
let $U$ be a fine log scheme over $Y$ and let 
$i: U \hra T$ be an exact closed immersion over $Z$ into a 
Noetherian fine log formal scheme $T$ such that $U$ is a scheme of 
definition of $T$. If we denote by $T_n$ the 
closed fine log scheme of $T$ defined by $(\Ker\,i^*)^n$ and by $i_n$
the canonical exact closed immersion $U \hra T_n$, 
$(U,T_n, i_n)$ is an object in $(Y/Z)_{\inf}$ and thus we have the functors 
\begin{align*}
& \Crys((Y/Z)_{\inf}) \lra \MC(T_n/Z), \quad 
\wt{\Crys}((Y/Z)_{\inf}) \lra \wt{\MC}(T_n/Z) 
\end{align*}
by \eqref{eq:trinity2} $\circ$ \eqref{eq:trinity1}.
Because this functor is compatible with respect to $n$, 
they induce the functors 
\begin{align*}
& \Crys((Y/Z)_{\inf}) \lra \varprojlim_n \MC(T_n/Z) \cong \MC(T/Z), \\ 
& \wt{\Crys}((Y/Z)_{\inf}) \lra \varprojlim_n \wt{\MC}(T_n/Z) \cong \wt{\MC}(T/Z).  
\end{align*}
If $T$ admits etale locally a morphism $T \lra Y$ over $Z$ compatible with 
$U \lra Y$, we can also define the functors 
\begin{align*}
& \Crys((Y/Z)_{\str}) \lra \varprojlim_n \MC(T_n/Z) \cong \MC(T/Z), \\  & 
\wt{\Crys}((Y/Z)_{\str}) \lra \varprojlim_n \wt{\MC}(T_n/Z) \cong \wt{\MC}(T/Z).  
\end{align*}
in the same way. 
The functors we constructed in this remark are contravariantly functorial with respect to 
$h:Y \lra Z$ and $(U,T,i)$. 
\end{rem}

We have the following topological invariance property for the 
categories $\Crys((Y/Z)_{\inf})$, $\wt{\Crys}((Y/Z)_{\inf})$. 

\begin{prop}\label{prop:topinv}
Let $k, h: Y \lra Z$ be as in Definition \ref{def:strat} such that 
$h$ is log smooth, and let 
$i: Y' \hra Y$ be a nilpotent exact closed immersion. Then the natural 
functors 
$$ \Crys((Y/Z)_{\inf}) \lra \Crys((Y'/Z)_{\inf}), \quad 
\wt{\Crys}((Y/Z)_{\inf}) \lra \wt{\Crys}((Y'/Z)_{\inf}) $$ 
are equivalences. 
\end{prop}

\begin{proof}
The triple $(Y',Y,i)$ is an object in $(Y'/Z)_{\inf}$ 
and the triple $(Y,Y,\id)$ is an object in $(Y/Z)_{\inf}$. 
So, by Proposition \ref{prop:trinity}, 
$\Crys((Y'/Z)_{\inf})$ is equivalent to $\Strat(Y/Z)$ and it is 
equivalent to $\Crys((Y/Z)_{\inf})$. 
The same argument works for the categories with tilde. 
\end{proof}

Now, let $k$ be a field of characteristic zero and 
consider the diagram of fine log schemes 
\begin{equation}\label{0}
\xymatrix {
X \ar[r]^f & S \ar[r]^(0.36){g} \ar @/^4mm/[0,-1]^{\iota} & \Spec k
} 
\end{equation}
satisfying the following conditions. 
\begin{enumerate}
\item[(A)] 
$f$ is a proper log smooth integral morphism and 
$g$ is a separated 
morphism of finite type. $\iota$ is a section of $f$. 
\item[(B)] 
$\Omega^1_{S/k}$ is a locally free $\cO_S$-module of finite rank. 
\end{enumerate}

Then, by assumption (B) and Proposition \ref{prop:trinity}, 
we have the equivalences 
\begin{equation}\label{s}
\MIC(S/k) \cong \Strat(S/k), \quad 
\wt{\MIC}(S/k) \cong \wt{\Strat}(S/k). 
\end{equation}
The same equivalences hold also for $X/k$, because $\Omega^1_{X/k}$ is locally free and there exist 
equivalences
\begin{equation}\label{x}
\MIC(X/k) \cong \Crys((X/k)_{\str}), \quad 
\wt{\MIC}(X/k) \cong \wt{\Crys}((X/k)_{\str}), 
\end{equation}
again by Proposition \ref{prop:trinity}. 
Also, since $f$ is log smooth, there exist  
equivalences
\begin{equation}\label{xs}
\MIC(X/S) \cong \Crys((X/S)_{\inf}), \quad 
\wt{\MIC}(X/S) \cong \wt{\Crys}((X/S)_{\inf}). 
\end{equation}

For $m,j \in \N$, let $S^m(j)$ be the $m$-th log infinitesimal 
neighborhood of $S$ in $S \times_k \cdots \times_k S$ ($(j+1)$ times). 
Then we have the category $\Crys((X/S^m(j))_{\inf})$ (resp. $\wt{\Crys}((X/S^m(j))_{\inf})$)
of locally free crystals (resp. coherent) on $(X/S^m(j))_{\inf}$. 
Let $p^m_j: S^m(1) \lra S \,(j=1,2)$, 
$p^m_{j,j'}: S^m(2) \lra S^m(1) \,(1 \leq j < j' \leq 3)$ be 
the projections and 
let 
\begin{align*}
&{p^{m,*}_{j,\inf}}: \Crys((X/S)_{\inf}) \lra \Crys((X/S^m(1))_{\inf}), \\
&{p^{m,*}_{j,j',\inf}}:  \Crys((X/S^m(1))_{\inf}) \lra \Crys((X/S^m(2))_{\inf})\\
(\textrm{resp.} \;&{p^{m,*}_{j,\inf}}: \wt{\Crys}((X/S)_{\inf}) \lra \wt{\Crys}((X/S^m(1))_{\inf}), \\
&{p^{m,*}_{j,j',\inf}}:  \wt{\Crys}((X/S^m(1))_{\inf}) \lra \wt{\Crys}((X/S^m(2))_{\inf})\;)
\end{align*}
be the associated pull-back functors. Using these, 
we define the categories $\Strat\Crys(X/k)$, 
$\wt{\Strat\Crys}(X/k)$ as follows: 

\begin{defn}\label{strcrys}
Let the notations be as above. Then we define 
$\Strat\Crys(X/k)$ $($resp. $\wt{\Strat\Crys}(X/k))$ 
to be the category of pairs $(E,\{\epsilon^m\}_m)$, 
where $E \in \Crys((X/S)_{\inf})$ 
$($resp. $E \in \wt{\Crys}((X/S)_{\inf}))$ 
and 
$\{\epsilon^m: {p^{m,*}_{2,\inf}}E \lra {p^{m,*}_{1,\inf}}E\}_m$ is a compatible family of 
isomorphisms in $\Crys((X/S^m(1))_{\inf}) \, (m \in \N)$ $($resp. $\wt{\Crys}((X/S^m(1))_{\inf}) \, (m \in \N))$ with $\epsilon^0 = \id$ 
and ${p^{m,*}_{1,3,\inf}}(\epsilon^m) = 
{p^{m,*}_{1,2,\inf}}(\epsilon^m) \circ {p^{m,*}_{2,3,\inf}}(\epsilon^m)$. 
\end{defn}

Then we have the following. 

\begin{prop}\label{crysprop1}
Let the notations be as above. Then 
we have equivalences of categories 
$\MIC(X/k) \cong \Strat\Crys(X/k)$, 
$\wt{\MIC}(X/k) \cong \wt{\Strat\Crys}(X/k)$. 
\end{prop}

\begin{proof}
We only prove the second equivalence 
$\wt{\MIC}(X/k) \cong \wt{\Strat\Crys}(X/k)$, 
because then we can deduce the first one by checking 
that the equivalence respects the local freeness. 
Also, by the equivalence \eqref{x}, it suffices to prove the equivalence 
$\wt{\Crys}((X/k)_{\str}) \cong \wt{\Strat\Crys}(X/k)$. \par 
First 
we construct the functor $\wt{\Crys}((X/k)_{\str}) \lra \wt{\Strat\Crys}(X/k)$. 
Let $E$ be an object of $\wt{\Crys}((X/k)_{\str})$. 
Then it naturally induces an object $E/S$ in $\wt{\Crys}((X/S)_{\inf})$
by the contravariant functoriality. 
(Note that $(X/S)_{\str} = (X/S)_{\inf}$ because $X$ is log smooth over $S$.) 
Let us take an object $T := (U \hra T \os{\varphi}{\ra} S^m(1))$ 
of $(X/S^m(1))_{\inf}$. Denote the object 
$(U \hra T \os{p_j^m \circ \varphi}{\ra} S)$ of $(X/S)_{\inf} = (X/S)_{\str}$ 
by $T_j$ and the object $(U \hra T \ra \Spec k)$ in 
$(X/k)_{\str}$ by $T_0$. Then the value of ${p_{j,\inf}^{m,*}}(E/S)$ at $T$ is 
equal to that of $E/S$ at $T_j$ and this is further equal to 
that of $E$ at $T_0$, which is independent of $j$. So we have the 
natural isomorphism ${p_{2,\inf}^{m,*}}(E/S) \os{\cong}{\lra} {p_{1,\inf}^{m,*}}(E/S)$. 
If we denote it by $\epsilon^m$, we see easily that 
$E \mapsto (E/S, \{\epsilon^m\}_m)$ defines the functor 
$\wt{\Crys}((X/k)_{\str}) \lra \wt{\Strat\Crys}(X/k)$. \par 
Next we construct the functor 
$\wt{\Strat\Crys}(X/k) \lra \wt{\Crys}((X/k)_{\str})$. 
Let $(E,\{\epsilon^m\}_m)$ be an object of 
$\wt{\Strat\Crys}(X/k)$ and let $T := 
(U \os{i}{\hra} T \os{\varphi}{\ra} \Spec k)$ be an object in 
$(X/k)_{\str}$. Then, there exists a morphism 
$\psi:T \lra S$ compatible with $i, \varphi$ in suitable sense, etale  
locally on $T$. So we can define 
the value $E_T$ of $E$ at $T$ locally. Also, 
if we have two morphisms $\psi, \psi':T \lra S$ as above, 
$\psi \times \psi': T \lra S \times_k S$ factors through 
$S^m(1)$ for some $m$, and so $\epsilon^m$ defines 
the isomorphism of the two definitions of $E_T$ 
(via $\psi$ and $\psi'$) 
which satisfies the 
cocycle condition. So we can define $E_T$ globally on $T$ by etale descent. 
This construction for objects $T$ in $(X/k)_{\str}$ induces 
the functor
$\wt{\Strat\Crys}(X/k) \lra \wt{\Crys}((X/k)_{\str})$. \par 
Finally, it is easy to see that the two functors we constructed 
are the inverses of each other. 
\end{proof}

Note that there exist the canonical pullback functors 
$$ 
f^*: \Strat(S/k) \lra \Strat\Crys(X/k), \quad 
\iota^*: \Strat\Crys(X/k) \lra \Strat(S/k) 
$$ 
which are equivalent to the functors 
$$ 
f^*_{\dR}: \MIC(S/k) \lra \MIC(X/k), \quad 
\iota^*_{\dR}: \MIC(X/k) \lra \MIC(S/k). 
$$

We give a definition of Gauss--Manin connection in our setting. 
For $j=1,2$, let us denote by $X^m_j$
the fiber product $X \times_{S,p_j^m}S^m(1)$ 
and denote by $q_j^m$, $f_j^m$ the projections 
$X_j^m \lra X$, $X_j^m \lra S^m(1)$, respectively. 
Also, let 
$\wh{X}^m$ be the log formal tube of $X$ in $X_1^m \times_{S^m(1)} X_2^m$, namely, 
the direct limit of the $l$-th log infinitesimal neighborhoods of 
$X$ in $X_1^m \times_{S^m(1)} X_2^m$ with respect to $l$ (cf. \cite[0.9]{cf}). 
This is a Noetherian fine log formal scheme. Let 
$\wh{f}^m: \wh{X}^m \lra S^m(1)$ be the map induced by $f_1^m, f_2^m$ and 
let $\wh{q}_j^m: \wh{X}^m \lra X$ be the composite of 
the projection $\wh{X}^m \lra X_j^m$ and $q_j^m$. 

Let $(E,\nabla)$ be an object in $\wt{\MIC}(X/k)$ and denote 
its image in $\wt{\MIC}(X/S)$ by $(\ol{E},\ol{\nabla})$. Thanks to hypothesis (B) and the proof of Proposition \ref{prop:trinity} we know that $p_j^m$ is flat, hence we have the diagram
\begin{align}\label{eq:gm}
{p_j^m}^*R^if_{\dR *}\ol{E} \lra 
R^if^m_{j,\dR *} {q_j^m}_{\dR}^*\ol{E} \lra 
R^i\wh{f}^m_{\dR *} (\wh{q}_j^m)_{\dR}^* \ol{E} 
\end{align}
in which the first map is the base change isomophism 
for $j=1,2$. Then we have the following: 

\begin{prop} \label{basechange}
Let the notations be as above. \\
$(1)$ \, We have the canonical isomorphism 
$(\wh{q}_{2}^m)_{\dR}^* \ol{E} \cong (\wh{q}_1^m)_{\dR}^* \ol{E}$. \\ 
$(2)$ \, The morphisms in the diagram \eqref{eq:gm} are isomorphisms. 
\end{prop}

\begin{proof}
(1) \, The functor 
$(\wh{q}_{2}^m)_{\dR}^*: \wt{\MIC}(X/S) \lra \wt{\MIC}(\wh{X}^m/S^m(1)) \hra 
\wt{\MC}(\wh{X}^m/S^m(1))$ is rewritten as the composite 
$$ \wt{\MIC}(X/S) \cong \wt{\Crys}((X/S)_{\inf}) \os{{p_{j,\inf}^{m,*}}}{\lra} 
\wt{\Crys}((X/S^m(1))_{\inf}) \lra \wt{\MC}(\wh{X}^m/S^m(1)), $$
where the last arrow 
is the one in Remark \ref{rem:trinity}. So it suffices to prove 
the existence of canonical isomorphism 
${p_{2,\inf}^{m,*}}^*\ol{E} \cong {p_{1,\inf}^{m,*}}\ol{E}$. It follows from the fact that 
the composite 
$$ \wt{\MIC}(X/k) \cong \wt{\Crys}((X/k)_{\str}) \lra 
\wt{\Crys}((X/S)_{\inf}) \os{{p_{j,\inf}^{m,*}}}{\lra} 
\wt{\Crys}((X/S^m(1))_{\inf})$$ is independent of 
$j$ and that ${p_{j,\inf}^{m,*}}\ol{E}$ is nothing but the image 
of $E$ by this composite. \\
(2) \, 
It sufficies to prove that the second map is an isomorphism. Since both $f_1^m, f_2^m$ are log smooth lift of $f$, they are locally 
isomorphic \cite[Proposition 3.14]{ka}. Also, $X \hra X_j^m$ is nilpotent. So 
$\wh{X}^m$ is locally isomorphic to the log formal tube of 
$X_j^m$ in $X_j^m \times_{S^m(1)} X_j^m$. Therefore, 
the projection $\wh{X}^m \lra X_j^m$ is etale locally isomorphic to 
the projection 
$\wh{\Af}^d_{X_j^m} \lra X_j^m$,  
where $d$ is the rank of $\Omega^1_{X/S}$. So, by 
Proposition \ref{prop:fib}, the second morphism in \eqref{eq:gm} 
is also an isomorphism. 
\end{proof}

If we denote the composite of the morphisms in \eqref{eq:gm} for $j=2$ and 
the inverse of morphisms in \eqref{eq:gm} for $j=1$ by 
$\eta^m: {p_2^m}^*R^if_{\dR *}\ol{E} \os{\cong}{\lra} 
{p_1^m}^*R^if_{\dR *}\ol{E}$, we see that the pair 
$R^if_{\dR *}E := (R^if_{\dR *}\ol{E}, \{\eta^m\}_m)$ defines 
an object in $\wt{\Strat}(S/k) \cong \wt{\MIC}(S/k)$. 
(To check the cocycle condition, we need to work on pullbacks of 
$f:X \lra S$ to $S^m(2)$, and we leave the reader to write out the detail.) 
This is our definition of the Gauss--Manin connection 
on relative de Rham cohomology. 

\begin{rem}\label{gmexseq}
In this remark, we explain the compatibility of several morphisms 
and the construction of Gauss--Manin connection. 

(1) \, Suppose that we are given a Cartesian diagram 
\begin{equation*}
\begin{CD}
X' @>{g'}>> X \\ 
@V{f'}VV @VfVV \\ 
S' @>g>> S 
\end{CD}
\end{equation*}
and let 
$E := (E,\nabla)$ be an object in $\wt{\MIC}(X/k)$, 
whose image in $\wt{\MIC}(X/S)$ we denote by 
$\ol{E} := (\ol{E},\ol{\nabla})$. If $g$ is flat or  if $\ol{E}$ and all $
R^if_{\dR *}\ol{E} \,(i \in \N)$ are flat over $\cO_S$, the base change isomorphism 
$$ 
g^*_{\dR} R^i f_{\dR *} \ol{E} \lra R^if'_{\dR *} {g'}_{\dR}^*\ol{E} 
$$ 
in $\wt{\MIC}(S'/S')$
is compatible with the maps as in \eqref{eq:gm} by functoriality. 
So the above morphism induces the isomorphism 
\begin{equation*}
\begin{CD}
g^*_{\dR} R^i f_{\dR *} E \lra R^if'_{\dR *} {g'}_{\dR}^*E 
\end{CD}
\end{equation*}
in $\wt{\MIC}(S'/k)$. 

(2) \, Let $E := (E,\nabla) \in \MIC(S/k)$, 
$E' := (E',\nabla') \in \wt{\MIC}(X/k)$ and denote the image 
of $E$ (resp. $E'$) in $\MIC(S/S)$ (resp. $\wt{\MIC}(X/S)$) by 
$\ol{E}$ (resp. $\ol{E}' := (\ol{E}',\ol{\nabla}')$). 
Then the isomorphism 
$$ 
R^if_{\dR *}(f^*_{\dR}\ol{E} \otimes \ol{E}') \cong \ol{E} \otimes 
R^if_{\dR *}\ol{E}' $$
of projection formula in $\wt{\MIC}(S/S)$ is compatible with the maps in \eqref{eq:gm} by functoriality. 
So the above morphism induces the isomorphism 
\begin{equation}\label{projk}
\begin{CD}
R^if_{\dR *}(f^*_{\dR}E \otimes E') \cong E \otimes 
R^if_{\dR *}E'
\end{CD}
\end{equation} 
in $\wt{\MIC}(S/k)$. 

(3) \, 
Let $0 \lra E' \lra E \lra E'' \lra 0$ be an exact sequence in 
$\wt{\MIC}(X/k)$ and denote its image in $\wt{\MIC}(X/S)$ by 
$0 \lra \ol{E}' \lra \ol{E} \lra \ol{E}'' \lra 0$. Then we have the 
associated long exact sequence 
$$ \cdots \lra R^if_{\dR *}\ol{E}' \lra 
 R^if_{\dR *}\ol{E} \lra  R^if_{\dR *}\ol{E}'' \lra  R^{i+1}f_{\dR *}\ol{E}' \lra \cdots 
$$ 
in $\wt{\MIC}(S/S)$ and it is compatible with the maps in 
\eqref{eq:gm} by functoriality. 
Thus the above exact sequence is enriched to 
the long exact sequence of Gauss--Manin connections 
$$ \cdots \lra R^if_{\dR *}E' \lra 
 R^if_{\dR *}E \lra  R^if_{\dR *}E'' \lra  R^{i+1}f_{\dR *}E' \lra \cdots 
$$ 
in $\wt{\MIC}(S/k)$. 
\end{rem}

\begin{rem}\label{gm}
Note that our definition of Gauss--Manin connection here, which has crystalline flavor, 
is not a priori the same as the usual definition given by the Katz--Oda spectral sequence. 
Berthelot proved in \cite[V Proposition 3.6.4]{BE} the equivalence of two definitions 
under certain assumption, when the base scheme is killed by a power of some prime number $p$. 
We will prove later the coincidence of two definitions in certain cases, 
partly using Berthelot's result. 
(See Remarks \ref{rem:gmcoin}, \ref{rem:wn-lazda}, \ref{rem:coin2}.) 
We expect that the two definitions coincide in general, but 
we will not prove it. 
\end{rem}

In the sequel, we assume the following conditions on the morphism $f$ in 
\eqref{0}. 
\begin{enumerate}
\item[(C)] 
For any $i$, $R^if_{\dR *}(\cO_X,d)$ (endowed with Gauss--Manin connection) 
belongs to $\MICn(S/k)$. 
\item[(D)] 
$f_{\dR *}(\cO_X,d) = (\cO_S, d)$, $g_{\dR *}(\cO_S,d) = k$. 
\end{enumerate}
Also, let $\NfMIC(X/k)$ (resp. $\NfMICn(X/k)$) 
be the full subcategory of $\MIC(X/k)$ 
consisting of the objects which are iterated 
extensions of objects in $f_{\dR}^*\MIC(S/k)$ (resp. $f_{\dR}^*\MICn(S/k)$). 
Note that, since $\MIC(S/k)$ is not necessarily an abelian category 
due to the possible existence of non-trivial log structure on $S$, 
we cannot expect that $\NfMIC(X/k)$ is an abelian category when 
the log structure on $S$ is non-trivial. 



For any 
$(E,\nabla) \in \NfMIC(X/k)$ (resp. $\NfMICn(X/k)$), 
the morphism of functors $f_{\dR}^*f_{\dR *} \lra \id$ induces 
the injection 
\begin{equation}\label{eq:f^*f_*->id}
f_{\dR}^*f_{\dR *}(E,\nabla) \hra (E,\nabla) 
\end{equation}
onto the maximal subobject of $(E,\nabla)$ belonging to the category 
$f_{\dR}^* \MIC(S/k)$ (resp. $f_{\dR}^*\MICn(S/k)$): 
Indeed, if $(E,\nabla)$ belongs to  
$f_{\dR}^* \MIC(S/k)$ (resp. $f_{\dR}^*\MICn(S/k)$), 
\eqref{eq:f^*f_*->id} is an isomorphism by 
\eqref{projk} for $i=0$ and the assumption (D), and 
the injectivity of the map \eqref{eq:f^*f_*->id} in general case is proven by 
induction on the number of iterations of extension. 
Moreover, if we have $f_{\dR}^*(F,\nabla_F) \hra (E,\nabla)$ for some 
$(F,\nabla_F) \in \MIC(S/k)$ (resp. $\MICn(S/k)$), 
we have 
$$ f_{\dR}^*(F,\nabla_F) = f^*_{\dR}f_{\dR *} f^*_{\dR}(F,\nabla_F) 
\hra f_{\dR}^*f_{\dR *}(E,\nabla). $$ 

The above result remains true if we replace $X/k$ by $X/S$: 
We define $\NfMIC(X/S)$ as the full subcategory of 
$\MIC(X/S)$ consisting of the objects which is an iterated 
extension of objects in $f_{\dR}^*\MIC(S/S)$. 
Then, for any $(E,\nabla) \in \NfMIC(X/S)$, we have the injection 
$$ f_{\dR}^*f_{\dR *}(E,\nabla) \hra (E,\nabla) $$
onto the maximal subobject of $(E,\nabla)$ belonging to the category 
$f_{\dR}^* \MIC(S/S)$.  
\\ \quad \par 

Next, we define the full subcategories 
$\NfStrCrys(X/k)$, $\NfStrCrysn(X/k)$ 
of $\Strat\Crys(X/k)$ as follows. 

\begin{defn}\label{def:1.11}
Let the notations be as above. \\
$(1)$ \, 
Let $\NfStrCrys(X/k)$ be the full subcategory of $\Strat\Crys(X/k)$ 
consisting of the objects which are iterated extensions of objects in 
$f^*\Strat(S/k)$. \\ 
$(2)$ \, 
Let $\NfStrCrysn(X/k)$ be the full subcategory of $\NfStrCrys(X/k)$ 
consisting of the objects which are iterated extension of 
objects in the essential image of 
$$ \MICn(S/k) \hra \MIC(S/k) \cong \Strat(S/k) \os{f^*}{\lra} 
\Strat\Crys(X/k). $$
\end{defn}

It is obvious from the definition that 
$\NfStrCrys(X/k)$ (resp. $\NfStrCrysn(X/k)$) corresponds to 
$\NfMIC(X/k)$ (resp. $\NfMICn(X/k)$) via the equivalence 
$\MIC(X/k) \cong \Strat\Crys(X/k)$
in Proposition \ref{crysprop1}.

\bigskip 

Now we introduce notation and setting which will be in force 
all along the rest of the paper. 

\begin{notation}\label{notation} 
 We fix a field $k$ of characteristic zero and 
the diagram of fine log schemes \eqref{0} satisfying the conditions 
${\rm (A), (B), (C), (D)}$ above and moreover the following condition ${\rm (E)}:$ 
\begin{enumerate}
\item[\rm (E)] $\MICn(S/k)$ is an abelian subcategory of $\wt{\MIC}(S/k)$. 
\end{enumerate}
Also, let $s \hra S$ be an exact closed immersion over $k$ 
from a fine log scheme $s$ such that the structure morphism 
$s \lra {\rm Spec}\,k$ 
at the level of underlying schemes is the identity. 
 We denote the morphism $s \hra S$ also by $s$ by abuse 
of notation, and denote the composite $\iota \circ s$ by $x:s \hra X$. 
We denote the fiber product $s \times_S X$ by $X_s$ and the projection 
$X_s \lra s$ by $f_s$. Then $x$ induces the closed immersion 
$s \hra X_s$, which we denote also by $x$. 
\end{notation}

Under the notation above, 
$R^if_{\dR *}$ defines the functor
\begin{equation}\label{locally_free_dR}
R^if_{\dR *}: \NfMICn(X/k) \lra \MICn(S/k). 
\end{equation}
Indeed, if $(E,\nabla)$ belongs to $f_{\dR}^*\MICn(S/k)$,  
$R^if_{\dR *} E$ belongs to $\MICn(S/k)$ by \eqref{projk} and the assumption (C), and 
one can prove it for general $(E,\nabla) \in \NfMICn(X/k)$ by  
induction on the number of iterations of 
extensions, using the exact sequence in Remark \ref{gmexseq}(3), 
the assumption (E) and the fact that $\MICn(S/k)$ is closed under extension 
in $\wt{\MIC}(S/k)$. In particular, for 
$E := (E,\nabla) \in \MICn(S/k)$ and 
$E' := (E',\nabla') \in \NfMICn(X/k)$, the isomorphism 
\eqref{projk} of projection formula is an isomorphism in 
$\MICn(S/k)$. 

Moreover, under the notation above, 
$\MICn(S/k)$ is actually a neutral 
Tannakian category 
with fiber functor $s_{\dR}^*: \MICn(S/k) \lra \MICn(s/s) = \Vector_k$. 
Also, we have $g_{\dR *}(\cO_S,d) = k$ by the assumption (D) and this 
and the assumption (E) imply, by a similar argument to 
\cite[Proposition 1.4.3]{dp}
that $\NfMICn(X/k)$ is a neutral Tannakian category with 
fiber functor $x_{\dR}^*: \NfMICn(X/k) \lra \MICn(s/s) = \Vector_k$. \par 
Let $\NfsMIC(X_s/s)$ be the subcategory of $\MIC(X_s/s)$ consisting of 
objects which are iterated extensions of objects in 
$f_{s,\dR}^*\MIC(s/s) = f_{s,\dR}^*\Vector_k$. 
Then we have 
$$ f_{s \dR *}(\cO_{X_s},d) = s_{\dR}^*f_{\dR *}(\cO_X,d) = k $$
by assumptions (C) and (D), and we see that it is also 
a neutral Tannakian category with 
fiber functor $x_{\dR}^*: \NfsMIC(X_s/s) \lra \MICn(s/s) = \Vector_k$. 

Finally, we give a useful sufficient condition for the diagram 
\eqref{0} to satisfy the conditions (A), (B), (C), (D) and (E). 

\begin{prop}\label{prop:abcde-rs}
Suppose that we are given a diagram 
\eqref{0} satisfying the following conditions: \\ 
$(1)$ \, $S$ is geometrically connected, 
$f$ is a proper log smooth integral saturated {\rm (\cite[II.5.18]{agt},  
\cite[Definition II.2.10]{tsuji19})}
morphism between fs log schemes with geometrically reduced, geometrically connected fibers 
and $g$ is a  separated morphism of finite type. 
$\iota$ is a section of $f$. \\
$(2)$ \, Etale locally around 
any closed point $t$ of $S$, there exists a strict etale 
morphism $S \lra \Af^{r,s}_k$ over $k$ 
for some $r,s \in \N$. \\
Then it satisfies the conditions {\rm (A), (B), (C), (D)} and {\rm (E)}. 
\end{prop}

\begin{rem}\label{rem:nilpnonclosed2}
In this remark, we prove that, in the situation of Proposition \ref{prop:abcde-rs}, 
the inclusion of categories $\MICn(S/k) \subseteq \MICn(S/k)'$ is, actually, an  
equality ($\MICn(S/k)'$ is  defined in Remark \ref{rem:nilpnonclosed1}). 
The proof given here is a modification of the proof in \cite[Remark 2.12]{dps}. 

We need to prove that, for an object $(E,\nabla)$ in $\MICn(S/k)'$ and 
a geometric point $y$ of $S$ which is not necessarily over a closed point, 
$(E,\nabla)$ has nilpotent residues at $y$. 
Since the assertion is etale local, we may assume that 
there exists an strict etale morphism 
$\varphi: S \lra \Af^{r,s}_k$. 
 For $I \subseteq \{1, \dots, \min(r,s)\}$, let  
$\Af^{r,s}_I$ be the closed log subscheme $\bigcap_{i \in I}\{t_i = 0\}$ of $\Af^{r,s}$ 
where $t_1, \dots, t_r$ are the coordinates of $\Af^{r,s}$, and put  
$\Af_I^{r,s,\circ} := \Af^{r,s}_I \setminus \bigcup_{J \supsetneq I} \Af^{r,s}_J$, 
$S_I^{\circ} := \varphi^{-1}(\Af_I^{r,s,\circ})$. 
Then $S = \sqcup_{I \subseteq \{1, \dots, \min(r,s)\}} S_I^{\circ}$ set-theoretically.  
Also, we have the equality $\overline{\cM}^{\rm gp}_S|_{S_I^{\circ}} = \Z^{r_I}_{S^{\circ}_I}$, 
where $\overline{\cM}_S^{\rm gp} := \cM_S^{\rm gp}/\mathcal{O}_S^*$ 
and $r_I := |I| + \max(s-r,0)$. 

Now take any object $(E,\nabla)$ and $y$ as above. 
Let $I$ be the subset of $\{1, \dots, \min(r,s)\}$ such that the image of $y$ in $S$ belongs to 
$S_I^{\circ}$. Then $\nabla$ induces a linear map 
$$\rho_I: E|_{S^{\circ}_I} \rightarrow E|_{S^{\circ}_I} \otimes 
\overline{\cM}_S^{\rm gp}|_{S_I^{\circ}} = E|_{S^{\circ}_I} \otimes 
\Z^{r_I}_{S^{\circ}_I}.$$ 
For a geometric point $z$ of $S^{\circ}_I$, 
we denote the specialization of $\rho_I$ to $z$ by 
$$ \rho_z: E|_z \rightarrow E|_z \otimes 
\overline{\cM}^{\rm gp}_{S,z} = E|_z \otimes 
\Z^{r_I}. $$
To prove the assertion, we should prove that, for any $\Z$-linear map 
$\psi: \overline{\cM}_{S,y}^{\rm gp} = \Z^{r_I} \longrightarrow \Z$, the map 
$({\rm id} \otimes \psi) \circ \rho_y: E|_y \rightarrow E|_y$ is nilpotent. 
By replacing $S$ by a small affine open log subscheme containing the image of $y$ in $S$, 
we may assume that $E|_{S^{\circ}_I}$ is free of rank $c$. 

For $1 \leq i \leq r_I$, 
we denote the $i$-th projection 
$\overline{\cM}_S^{\rm gp}|_{S_I^{\circ}} = \Z^{r_I}_{S_I^{\circ}} \longrightarrow \Z_{S_I^{\circ}}$
by $\pi_i$ and denote its fiber 
$\overline{\cM}_{S,y}^{\rm gp} = \Z^{r_I} \longrightarrow \Z$ at $y$ by 
$\pi_{i,y}$. Then, by the integrability of $\nabla$, it suffices to prove the nilpotence of 
$\rho_{y,i} := ({\rm id} \otimes \pi_{i,y}) \circ \rho_y: E|_y \longrightarrow E|_y$ for any $i$. 
Since this map is the specialization of the map 
$$ \rho_{I,i}:= ({\rm id} \otimes \pi_i) \circ \rho_I: E|_{S^{\circ}_I} \longrightarrow 
E|_{S^{\circ}_I}, $$
we are reduced to proving the nilpotence of the map $\rho_{I,i}$. 
Choose for any closed point of $S^{\circ}_I$ a geometric point over it and 
let $T$ be the set of such geometric points. Then, by assumption
that $(E,\nabla)$ belongs to $\MICn(S/k)'$, 
for any $z \in T$ and for any $1 \leq i \leq r_I$, the map $\rho_{z,i}$ is nilpotent. 
Hence, as an element in ${\rm End}(E|_z)$, $\rho_{z,i}^{c+1}$ is equal to $0$. 
On the other hand, since $S^{\circ}_I$ is reduced and Jacobson and $E|_{S^{\circ}_I}$ is free, 
the natural map 
${\rm End}(E|_{S^{\circ}_I}) \to \prod_{z \in T} {\rm End}(E|_z)$ 
is injective. Hence $\rho_{I,i}^{c+1}$ is equal to $0$ in ${\rm End}(E|_{S^{\circ}_I})$. 
So the proof of the assertion is finished. 

By the same argument, we see the equality 
$\MICn(S/k) = \MICn(S/k)'$ when $S$ admits etale locally 
a set-theoretical decomposition $S = \sqcup_{i} S_i$ by reduced 
locally closed subschemes $S_i$ such that, for each $i$, the equality 
$\ol{\cM}^{\rm gp}_S |_{S_i} = \Z^{r_i}_{S_i}$ holds for some $r_i$.  
\end{rem}

\begin{proof} {\it (of Proposition  \ref{prop:abcde-rs})}
It is easy to see that the conditions (A), (B) are satisfied. 
The first condition in (D) follows from 
the assumption that $f$ is proper and has 
geometrically reduced, geometrically connected fibers. 
We can prove the second condition in (D) 
using the geometric connectedness of $S$ and the 
local description of $S$ in (2). 

We prove the condition (E). To do so, it suffices to work 
etale locally on $S$ and 
thus we may assume 
the existence of 
a strict etale morphism $\varphi:S \lra \Af^{r,s}_k$ over $k$ as in (2). 
Let $D := \bigcup_{i=1}^{\min(s,r)} \varphi^{-1}(\{t_i = 0\})$, 
where $t_1, ..., t_r$ is the canonical coordinate of $\Af^{r,s}_k$ and 
let $S'$ be the log scheme whose underlying scheme is the same as 
that of $S$ and whose log structure is associated to the normal 
crossing divisor $D$. 
Then, if $s \leq r$, we have $S' = S$ and it is well-known that 
$\MICn(S'/k)$ is an abelian subcategory of $\wt{\MIC}(S'/k)$. 
(This follows from G\'erard--Levelt theory \cite[\S 3]{gl}, as shown in \cite[Lemma 3.1.6, Corollary 3.1.7]{cailotto} for example. 
Also, one can prove it directly in the same way as \cite[Lemma 3.2.14]{kedlayaI}.)
When $s > r$, the log structure $\cM_S$ of $S$ is the one associated to 
the monoid homomorphism $\cM_{S'} \oplus \N^{s-r} \lra \cO_{S'}$ induced by 
$\alpha_{S'}: \cM_{S'} \lra \cO_{S'}$ and the monoid homomorphism  
$\N^{s-r} \lra \cO_{S'}$ sending the standard basis of $\N^{s-r}$ to zero. 
From this, we see that the category 
$\MICn(S/k)$ (resp. $\wt{\MIC}(S/k)$) is the category of objects in 
$\MICn(S'/k)$ (resp. $\wt{\MIC}(S'/k)$) endowed with 
$(s-r)$ commuting nilpotent endomorphisms (resp. commuting endomorphisms). 
Thus $\MICn(S/k)$ is an abelian subcategory of $\wt{\MIC}(S/k)$  
also in this case. So the condition (E) is satisfied. 

In the following, we prove the condition (C). 
By standard argument, we may reduce to the case where 
$k$ is finitely generated over $\Q$. 
Take a closed point $t$ of $S$ and we want to prove that 
$R^if_{\dR *}(\cO_X,d)$ is locally free around $t$ and has nilpotent residues 
at a geometric point over $t$. Since we may work etale locally around $t$ and enlarge $k$, we can assume 
that $S$ is affine, $t$ is $k$-rational and that 
$g$ factors as $S \os{\varphi}{\lra} \Af^{r,s}_k \lra \Spec\,k$, 
where $\varphi$ is a strict etale morphism such that the inverse image 
of the origin in $\Af^{r,s}_k$ by $\varphi$ is equal to $t$. Hence we have a diagram 
\begin{equation}\label{eq:ttdiag1}
X \os{f}{\lra} S \os{\varphi}{\lra} \Af^{r,s}_k \lra \Spec k. 
\end{equation}
Take an affine connected scheme $\Spec A'$ smooth of finite type over $\Z$ 
with ${\rm Frac}\,A' = k$ and a diagram 
\begin{equation}\label{eq:ttdiag2}
X_{A'} \os{f_{A'}}{\lra} S_{A'} \os{\varphi_{A'}}{\lra} \Af^{r,s}_{A'} \lra \Spec A'  
\end{equation}
which satisfy the following: 
\begin{itemize}
\item The pull-back of the diagram \eqref{eq:ttdiag2} by the map 
$\alpha: \Spec\,k \lra \Spec\, A'$ induced by the inclusion 
$A' \hra {\rm Frac}\,A' = k$ is the diagram \eqref{eq:ttdiag1}. 
\item 
The inverse image $t_{A'}$ 
of the origin in $\Af^{r,s}_{A'}$ by $\varphi_{A'}$ is isomorphic to $\Spec A'$.  
\item 
The morphism $f_{A'}$ is proper log smooth, integral and saturated.
\item
The morphism $\varphi_{A'}$ is strict etale.
\end{itemize}

Take a prime number $p$ such that 
$\Spec (A' \otimes_{\Z} \F_p)$ is connected and non-empty, 
let $A$ be the $p$-adic completion of $A' \otimes_{\Z} \Z_p$ and denote the pullback of 
the diagram \eqref{eq:ttdiag2} by the canonical morphism 
$\beta: \Spec A \lra \Spec A'$ by 
\begin{equation}\label{eq:ttdiag3}
X_{A} \os{f_{A}}{\lra} S_{A} \os{\varphi_{A}}{\lra} \Af^{r,s}_{A} \lra \Spec A.  
\end{equation}
Also, let us take a morphism $F: \Spec A \lra \Spec A$ over $\Spec \Z_p$ 
which is a lift of the absolute 
Frobenius on $\Spec (A \otimes_{\Z} \F_p)$. Let $k_0$ be the perfection 
of ${\rm Frac}\,(A \otimes_{\Z} \F_p)$ and put 
$W := W(k_0)$. Then there exists a morphism 
$\gamma: \Spec W \lra \Spec A$ 
which is compatible with Frobenius morphism on $\Spec W$ and the morphism $F$ 
on $\Spec A$. 
Denote the pullback of the diagram \eqref{eq:ttdiag3} by 
$\gamma$ by 
\begin{equation}\label{eq:ttdiag4}
X_{W} \os{f_{W}}{\lra} S_{W} \os{\varphi_{W}}{\lra} \Af^{r,s}_{W} \lra \Spec W.  
\end{equation}
Denote its special fiber by 
\begin{equation}\label{eq:ttdiag5}
X_{0} \os{f_{0}}{\lra} S_{0} \os{\varphi_{0}}{\lra} \Af^{r,s}_{k_0} \lra \Spec k_0,   
\end{equation}
generic fiber by 
\begin{equation}\label{eq:ttdiag6}
X_{k_W} \os{f_{k_W}}{\lra} S_{k_W} \os{\varphi_{k_W}}{\lra} \Af^{r,s}_{k_W} \lra \Spec k_W 
\end{equation}
and its $p$-adic completion by 
\begin{equation}\label{eq:ttdiag7}
\wh{X}_{W} \os{\wh{f}_{W}}{\lra} \wh{S}_{W} \os{\wh{\varphi}_{W}}{\lra} 
\wh{\Af}^{r,s}_{W} \lra \Spf W.  
\end{equation}
We denote the canonical open immersion $\Spec k_W \hra \Spec W$ by 
$\delta$. 

Since $k$ is the fraction field of $A'$, we have a morphism 
$\epsilon: \Spec k_W \lra \Spec k$ with 
$\alpha \circ \epsilon = \beta \circ \gamma \circ \delta$.

Denote the inverse image of the origin by 
$\varphi_A, \varphi_W, \varphi_{k_W}$ by 
$t_A, t_W, t_{k_W}$ respectively. Then 
we have the following: 
\begin{equation*}
\begin{CD}
t @>>> t_{A'} @<<< t_A @<<< t_W @<<< t_{k_W} \\ 
@| @| @| @| @| \\ 
\Spec k @>{\alpha}>> \Spec A' @<{\beta}<< \Spec A 
@<{\gamma}<< \Spec W @<{\delta}<< \Spec k_W. 
\end{CD}
\end{equation*}

Let 
$\psi: B := \Gamma(S_W,\cO_{S_W}) \lra W$ be the 
homomorphism induced by the inclusion $t_W \hra S_W$ and 
let $\mathfrak{m}$ be the kernel of $\psi_{\Q}: B_{\Q} \lra k_W$, where 
$_{\Q}$ stands for $\otimes_{\Z} \Q$. 
Then we can regard  $R^if_{k_W, \dR *} (\cO_{X_{k_W}},d)$ as 
a finitely generated $B_{\Q}$-module with integrable connection, 
and we should prove that it is free at $\mathfrak{m}$ and that it has nilpotent 
residues at the closed point of $S_{k_W}$ 
corresponding to $\mathfrak{m}$. (Indeed, by Remark \ref{rem:nilpnonclosed2}, 
the nilpotence of residues can be checked at geometric points over closed 
points. Also, noting that the log structure on $S_W$ 
is the pullback of that on $\Af^{r,s}_{W}$, the log structure of 
$S_W$ is defined in Zariski topology and so the nilpotence of residues 
can be checked 
at the closed point corresponding to $\mathfrak{m}$, not at a geometric point above it. 
See the paragraph before Remark \ref{rem:nilpnonclosed1}.) 

Put $\wh{B} := \Gamma(\wh{S}_W,\cO_{\wh{S}_W})$. Then 
$\psi: B = \Gamma(S_W,\cO_{S_W}) \lra W$ induces the 
homomorphism $\wh{\psi}: \wh{B} \lra W$. Let 
$\wh{\mathfrak{m}}$ be the kernel of $\wh{\psi}_{\Q}: \wh{B}_{\Q} \lra k_W$. 
Then, to show the claim in the previous paragraph, 
it suffices to prove that 
$E := R^if_{0, \crys *} \cO_{X_0/\wh{S}_W} \otimes \Q$ 
(where $f_{0, \crys}$ denotes 
the morphism of topoi $(X_0/\wh{S}_W)_{\crys}^{\sim} \lra \wh{S}_{W,\et}^{\sim}$), 
regarded as a finitely generated $\wh{B}_{\Q}$-module with 
integrable connection (in rigid analytic sense), 
is free at $\wh{\mathfrak{m}}$ and that it has nilpotent 
residues at the point corresponding to $\wh{\mathfrak{m}}$, 
by the comparison theorem of log de Rham and log crystalline cohomology \cite[Proposition 2.20]{hk}.

Let $F: \wh{S}_W \lra \wh{S}_W$ be a lift of Frobenius on 
$S_0$ compatible with the morphism $\Af^{r,s}_W \lra \Af^{r,s}_W$ 
over the Frobenius on $W$ which 
sends the coordinates of $\Af^{r,s}_W$ 
to their $p$-th powers. 
By the condition (1) on $f$, 
we see that $f_0$ is a proper log smooth integral morphism of Cartier 
type \cite[Proposition II.2.14]{tsuji19}. So, by \cite{hk}, 
the endomorphism $F^*$ on 
$E$ induced by $F$ is an 
isomorphism.  To prove the freeness at $\wh{\mathfrak{m}}$, 
it suffices to prove that 
$E/\wh{\mathfrak{m}}^nE$ is a free $\wh{B}_{\Q}/\wh{\mathfrak{m}}^n$-module 
for any $n$. Take an surjective homomorphism 
$$ (\wh{B}_{\Q}/\wh{\mathfrak{m}}^n)^{\oplus r} \lra E/\wh{\mathfrak{m}}^nE $$
which is isomorphic modulo $\wh{\mathfrak{m}}$, and denote the kernel 
of it by $K$. Then it induces the isomorphism 
\begin{align*}
((\wh{B}_{\Q}/\wh{\mathfrak{m}}^n)^{\oplus r}/K) \otimes_{\wh{B}_{\Q}/\wh{\mathfrak{m}}^n, (F^*)^n} 
(\wh{B}_{\Q}/\wh{\mathfrak{m}}^n) & \os{\cong}{\lra}   
(E/\wh{\mathfrak{m}}^nE) \otimes_{\wh{B}_{\Q}/\wh{\mathfrak{m}}^n, (F^*)^n}  
(\wh{B}_{\Q}/\wh{\mathfrak{m}}^n) \\ & \cong E/\wh{\mathfrak{m}}^nE, 
\end{align*}
where the last isomorphism is induced by the isomorphism $F^*$ on 
$E$. Because $K \subseteq (\wh{\mathfrak{m}}/\wh{\mathfrak{m}}^n)^{\oplus r}$ and 
$(F^*)^n(\wh{\mathfrak{m}}) \subseteq \wh{\mathfrak{m}}^{p^n} \subseteq \wh{\mathfrak{m}}^n$,  
the source of this map is isomorphic to 
$(\wh{B}_{\Q}/\wh{\mathfrak{m}}^n)^{\oplus r}$. Hence 
$E/\wh{\mathfrak{m}}^nE$ is a free $\wh{B}_{\Q}/\wh{\mathfrak{m}}^n$-module 
for any $n$, as required. 

Now we prove the nilpotence of residues 
of the connection $\nabla$ defined on $E$ at $\wh{\mathfrak{m}}$. 
Because of the functoriality of 
crystalline cohomologies, the morphism $F: \wh{S}_W \lra \wh{S}_W$ induces an endomorphism on 
$(E,\nabla)$. Thus, any residue map 
$N: E/\wh{\mathfrak{m}} E \lra E/\wh{\mathfrak{m}} E$ 
at $\wh{\mathfrak{m}}$ satisfies the equality 
$N \circ F^* = p F^* \circ N$, and this implies the nilpotence of $N$. 
Hence the proof of the proposition is finished. 
\end{proof}

\begin{rem}\label{rem:gmcoin}
In this remark, we give several sufficient conditions for the 
coincidence of the definition of Gauss--Manin connection 
in this paper and the usual definition of 
Gauss--Manin connection via Katz--Oda spectral sequence. 
(We will use the settings in this remark later in this paper, but 
we will not use the coincidence results.) 

(1) \, 
Let the situation be as in Proposition \ref{prop:abcde-rs} and 
assume that we can always take $s \leq r$ in the condition (2) there. 
(This implies that $S$ is log smooth over $k$.) Let $(E,\nabla)$ be an object in $\MIC(X/k)$. 
Also, assume the following condition $(*)$: \\ 
\quad \\ 
$(*)$ \,\,\, 
Etale locally on $S$, there exist diagrams 
\eqref{eq:ttdiag1}, \eqref{eq:ttdiag2}, \eqref{eq:ttdiag3}, \eqref{eq:ttdiag4}, \eqref{eq:ttdiag6} (with $s \leq r$) 
such that the restriction of $(E,\nabla)$ to 
$\MIC(X_{k_W}/k_W)$ by $\epsilon: \Spec k_W \lra \Spec k$ comes from an object 
$(E_W,\nabla_W)$ in $\MIC(X_W/W)$ whose mod $p^n$ reductions ($n \in \N$)
are quasi-nilpotent in the sense of \cite{berthelotogus}. \\
\quad \\ 
Then the definition of Gauss--Manin connection on 
$R^if_{\dR*}(E,\nabla)$ in this paper coincides with that via 
Katz--Oda spectral sequence. 

We prove this claim. If we denote the mod $p^n$ reduction of 
$f_W$ and $(E,\nabla)$ by $f_n:X_n \lra S_n$, $(E_n,\nabla_n)$ respectively, 
it suffices to prove the coincidence of two definitions of 
Gauss--Manin connection on $R^if_{n, \dR*}(E_n,\nabla_n)$. 
The condition $(*)$ implies that $(E_W,\nabla_W)$ comes from 
a locally free crystal $\cE$ on $(\wh{X}/\wh{S}_W)_{\crys}$ and so 
$(E_n,\nabla_n)$ comes from the mod $p^n$ reduction of $\cE$, 
which is a locally free crystal on $(X_n/S_n)_{\crys}$. Then 
the coincidence of two definitions of Gauss--Manin connection 
follows from \cite[V Proposition 3.6.4]{BE}, 
which works without any change in the case 
$S$ is log smooth over $k$. 

We expect that the argument in the proof of 
\cite[V Proposition 3.6.4]{BE} will work also in the case $s > r$, 
but we will not pursue this topic in this paper. 

(2) \, Let $f: X \lra S$ be as in Notation \ref{notation} and 
assume that we have a Cartesian diagram 
\begin{equation*}
\begin{CD}
X @>>> X' \\ 
@VfVV @V{f'}VV \\ 
S @>>> S',   
\end{CD}
\end{equation*}
with $f'$ proper log smooth integral. 
Also, let 
$(E,\nabla)$ be an object in $\MIC(X/k)$ and assume the following. \\ \quad \\ 
(a) \, $(E,\nabla)$ is the pullback of an object $(E',\nabla')$ in $\MIC(X'/k)$. \\ 
(b) \, For the object $(E',\nabla')$ in (a) and for any $i$, 
the underlying ${\mathcal{O}}_{S'}$-module of 
$R^if'_{\dR *}(E',\nabla')$ is locally free. \\ 
(c) \, The two definitions of Gauss--Manin connection on 
$R^if'_{\dR *}(E',\nabla')$ coincide. \\ \quad \\ 
Then two definitions of Gauss--Manin connection on 
$R^if_{\dR *}(E,\nabla)$ coincide, because it is the pullback of 
$R^if'_{\dR *}(E',\nabla') \in \MIC(S'/k)$ to $\MIC(S/k)$ by remark \ref{gmexseq} (1).

(3) \, 
We can apply (1) and (2) in the following setting. 
Let $\cS$ be the log scheme $(\Spec k[[t]], \cM_{\cS})$, where  
$\cM_{\cS}$ is the log structure associated to the monoid homomorphism 
$\N \lra k[[t]];\, 1 \mapsto t$. (Note that 
$\cM_{\cS} = \cO_{\cS^{\circ}} \setminus \{0\}$.) 
Let $\wt{f}: \cX \lra \cS$ be a regular semistable model 
of a proper smooth curve of genus $g \geq 2$ over $\Spec k((t))$, 
endowed with log structure associated to the special fiber. 
Let $f: X \lra S$ be the special fiber of 
$\wt{f}$ with pullback log structure. 

Let ${\cX'}^{\circ}$ be the stable model of $\cX^{\circ}$ with blow-down 
$\varphi^{\circ}: \cX^{\circ} \lra {\cX'}^{\circ}$ 
of schemes (without log structures). 
If we denote by $\ol{M}_g$ the moduli log stack of 
stable log curves of genus $g$ over $k$ and by $\ol{C}_g$ the universal stable log curve 
over $\ol{M}_g$, 
we have a morphism 
$\cS^{\circ} \lra \ol{M}_g^{\circ}$ such that $\ol{C}^{\circ}_g \times_{\ol{M}^{\circ}_g} 
\cS^{\circ} \cong {\cX'}^{\circ}$. 
By pulling back the log structures on $\ol{M}_g$ and $\ol{C}_g$, we obtain 
a morphism of log schemes $({\cX'}^{\circ}, {\cal N}_{\cX^{\circ}}) \lra 
(\cS^{\circ}, {\cal N}_{\cS^{\circ}})$. Because the generic fiber of 
${\cX'}^{\circ} \lra \cS^{\circ}$ is smooth, ${\cal N}_{\cS^{\circ}}$ is 
trivial on the generic point. Thus the structure morphism 
${\cal N}_{\cS^{\circ}} \lra \cO_{\cS^{\circ}}$ of the log structure 
${\cal N}_{\cS^{\circ}}$ factors through 
$\cO_{\cS^{\circ}} \setminus \{0\} = \cM_{\cS}$ and so 
we have a morphism 
$\cS \lra (\cS^{\circ}, {\cal N}_{\cS^{\circ}})$. 
By base change, we obtain the morphism 
$\wt{f}': \cX' := ({\cX'}^{\circ}, {\cal N}_{\cX^{\circ}}) 
\times_{(\cS^{\circ}, {\cal N}_{\cS^{\circ}})} \cS \lra \cS$. 
According to \cite[1.1]{fk}, the log scheme 
$\cX'$ has the form $(k[[t]][x,y]/(xy - t^n), \N^2 \oplus_{\Delta, \N,n} \N)$ 
(where $\Delta: \N \lra \N^2$ is the diagonal map and $n:\N \lra \N$ is the multiplication by $n$) 
etale locally around double points. 
Because $\varphi^{\circ}$ is the composition of blow-ups 
at double points, it is realized as the composition of log blow-ups 
(\cite[\S 4]{niziol}) $\varphi: \cX \lra \cX'$.   
Thus we have the following commutative diagram 
\begin{equation}\label{eq:moduli}
\begin{CD}
X @>>> \cX @>{\varphi}>> \cX' @>>> \ol{C}_g \\ 
@VfVV @V{\wt{f}}VV @VVV @VVV \\ 
S @>>> \cS @= \cS @>>> \ol{M}_g, 
\end{CD}
\end{equation}
where the left and the right squares are Cartesian. 

Let $(E,\nabla)$ be an object in $\MIC(X/k)$ which comes from 
an object in $(E', \nabla')$ in $\MIC(\ol{C}_g/k)$ which satisfies the 
condition (b) in (2). Moreover, assume that $(E',\nabla')$ satisfies the 
condition $(*)$ in (1) etale locally on $\ol{M}_g$. 
We denote the pullback of $(E',\nabla')$ to $\cX'/\cS$, $\cX/\cS$ by 
$(E_{\cX'}, \nabla_{\cX'}), (E_{\cX}, \nabla_{\cX})$, respectively. 
Then we can apply (1), (2) to see that the 
two definitions of Gauss--Manin connection on 
$R^i\wt{f}'_{\dR *}(E_{\cX'},\nabla_{\cX'})$ coincide. 
Then, using 
the quasi-isomorphism   
$$ R\varphi_*\Omega^{j}_{\cX/\cS} = \Omega^{j}_{\cX'/\cS} \quad (j \geq 0) $$
which we prove in Proposition \ref{prop:invariance} below 
and the projection formula, we obtain the isomorphism 
$$ R^i\wt{f}_*(E_{\cX} \otimes_{\cO_{\cX}} \Omega^j_{\cX/\cS}) 
\cong R^i\wt{f}'_*(E_{\cX'} \otimes_{\cO_{\cX'}} \Omega^j_{\cX/\cS}) 
\quad (i,j \geq 0) $$
and so we obtain the isomorphism 
$R^i\wt{f}_{\dR *}(E_{\cX},\nabla_{\cX}) \cong 
R^i\wt{f}'_{\dR *}(E_{\cX'},\nabla_{\cX'})$ for $i \geq 0$ by spectral sequence argument. 
Thus the two definitions of Gauss--Manin connection on 
$R^i\wt{f}_{\dR *}(E_{\cX},\nabla_{\cX})$ coincide.
Then, since $R^if_{\dR *}(E,\nabla)$ is the pullback of 
$R^i\wt{f}_{\dR *}(E_{\cX},\nabla_{\cX})$ to $S$ by the base change isomorphism, 
we conclude that the two definitions of Gauss--Manin connection on 
$R^if_{\dR *}(E,\nabla)$ coincide. We can apply this result to the case 
$(E,\nabla) = (\cO_X,d)$. 
\end{rem}

\begin{prop}\label{prop:invariance}
Suppose that we are given a diagram of fs log schemes 
\begin{equation*}
\begin{CD}
\cX @>{\varphi}>> \cX' \\ 
@VVV @VVV \\ 
\cS @= \cS 
\end{CD}
\end{equation*}
such that $\cX'$ is log regular, $\varphi$ is a log blow-up and 
$\Omega^{1}_{\cX'/\cS}$ is locally free. Then, for each $j \geq 0$, we have 
a quasi-isomorphism 
$$ R\varphi_*\Omega^{j}_{\cX/\cS} = \Omega^{j}_{\cX'/\cS}. $$
\end{prop}

\begin{proof}
Since $\cX'$ is log regular, we can apply 
\cite[Theorem 4.7]{niziol} and \cite[Theorem 11.3]{ka2} to have 
a quasi-isomorphism 
$$ R\varphi_*\cO_{\cX} = \cO_{\cX'}. $$
Because $\varphi$ is log etale, 
$\varphi^*\Omega^{j}_{\cX'/\cS} = \Omega^j_{\cX/\cS}$. 
From the fact that $\Omega^{1}_{\cX'/\cS}$ is locally free, 
by the projection formula, we have a quasi-isomorphism 
$$ R\varphi_*\Omega^{j}_{\cX/\cS} = \Omega^{j}_{\cX'/\cS} $$
for each $j$. 
\end{proof}

\begin{rem}\label{rem:2020Aug-1}
Suppose that we are in the situation of 
Proposition \ref{prop:abcde-rs} with $k=\C$ and 
suppose moreover that we have equality of categories 
$$ \MICn(X/\C) = \MICn(X/\C)'. $$ 
(See Remark \ref{rem:nilpnonclosed2} for a sufficient condition to have these 
equalities. Note also that we have the equality $\MICn(S/\C) = \MICn(S/\C)'.$) 
In this remark, we explain the topological analogue of 
the categories 
$$\MICn(S/\C), \quad \MICn(X/\C), \quad \NfMICn(X/\C)$$ 
in this situation. 

First, let $f_{\an}: X_{\an} \lra S_{\an}$ be the morphism of fs log analytic spaces 
associated to $f$. Then we can define the category 
$\MICn(S_{\an}/\C)$ (resp. $\MICn(X_{\an}/\C)$) of locally free 
$\cO_{S_{\an}}$-modules (resp. $\cO_{X_{\an}}$-modules) of finite rank with 
integrable connection on $S_{\an}/\C$ (resp. $X_{\an}/\C$) which have 
nilpotent residues at any point. This is the analytic analogue of the category 
$\MICn(S/\C) = \MICn(S/\C)'$ (resp. $\MICn(X/\C) = \MICn(X/\C)'$). 
In fact, if $S$ and $f$ are projective, we have equivalences 
$$ \MICn(S/\C) \cong \MICn(S_{\an}/\C), \quad 
\MICn(X/\C) \cong \MICn(X_{\an}/\C) $$
of GAGA. 

Next, let $S_{\an}^{\log}$ be the topological space associated to the fs log analytic space 
$S_{\an}$ defined in the following way by Kato--Nakayama \cite{kn}: 
$$ 
S_{\an}^{\log} := 
\left\{ (s,h) \,\left|\, 
\begin{aligned}
& s \in S, h \in \Hom(\cM^{\gp}_{S_{\an},s}, {\mathbb{S}}^1) \\ 
& \text{such that } h(f) = f/|f| \text{ for } f \in \cO^{\times}_{S_{\an},s}
\end{aligned} 
\right. \right\}. 
$$
(Here ${\mathbb{S}}^1$ is the unit circle $\{z \in \C \,|\, |z| = 1\}$.) 
We have the canonical map $\tau_S: S_{\an}^{\log} \lra S_{\an}$ defined by 
$(s,h) \mapsto s$ and the fiber $\tau_S^{-1}(s)$ of a point $s \in S_{\an}$ is 
homeomorphic to $({\mathbb{S}}^1)^{r(s)}$ with $r(s) := {\rm rank}\,\ol{\cM}^{\gp}_{S_{\an},s}$. 
We can define $X^{\log}_{\an}$ and $\tau_X: X^{\log}_{\an} \lra X_{\an}$ in the same way, 
and we have the canonical commutative diagram 
$$ 
\xymatrix{
X^{\log}_{\an} \ar[r]^{f^{\log}_{\an}} \ar[d]^{\tau_X} & 
S^{\log}_{\an} \ar[d]^{\tau_S} \\ 
X_{\an} \ar[r]^{f_{\an}} & S_{\an}. 
}
$$

Let ${\rm LS}^{\rm u}(S^{\log}_{\an})$ be the category of local systems $V$ 
of finite dimensional $\C$-vector spaces on $S^{\log}_{\an}$ such that
the monodromy action of $\pi_1^{\rm top}(\tau_S^{-1}(s)) \cong \Z^{r(s)}$ 
on $V_{s'}$ ($s' \in \tau_S^{-1}(s)$) is unipotent for any $s \in S_{\an}$, 
and define the category ${\rm LS}^{\rm u}(X^{\log}_{\an})$ in the same way. 
Then, by the log Riemann--Hilbert correspondence due to 
Kato--Nakayama \cite{kn} and Ogus \cite{ogus-rh}, we have equivalences of 
categories 
$$ \MICn(S_{\an}/\C) \cong {\rm LS}^{\rm u}(S^{\log}_{\an}), \quad 
\MICn(X_{\an}/\C) \cong {\rm LS}^{\rm u}(X^{\log}_{\an}). $$
Thus the categories ${\rm LS}^{\rm u}(S^{\log}_{\an})$, ${\rm LS}^{\rm u}(X^{\log}_{\an})$ 
are the topological analogues of the categories 
$\MICn(S/\C), \MICn(X/\C)$, respectively. We can also define 
the category ${\rm N}_{f^{\log}_{\an}}{\rm LS}^{\rm u}(X^{\log}_{\an})$, which is 
the topological analogue of the category $\NfMICn(X/\C)$. 

Put it in another way, we can say that our purely algebraic study of 
certain categories of modules with integrable connections on $S$ and on $X$ 
is the algebraic version of the study of certain category of local systems 
on $S^{\log}_{\an}$ and on $X^{\log}_{\an}$. 
\end{rem}

\section{Relative Tannakian theory}

In this section, we recall the definition of relatively unipotent 
log de Rham fundamental group of Lazda \cite{laz}, 
which is based on the theory of Tannakian categories 
due to Deligne \cite{Del89}, \cite{Del90}. 

First we remark that, as observed in \cite[\S 5]{Del89} and \cite[\S 7.5]{Del90}, one can imitate in a Tannakian category $\mathcal{T}$ several basic constructions of algebra and algebraic geometry. We recall here some of them. 
Let $\mathrm{Ind}(\mathcal{T})$ be the category on ind-objects of $\mathcal{T}$ (for definition, see \cite[\S 4]{Del89} and \cite[\S 7.5]{Del90}). A commutative ring with unit of $\mathrm{Ind}(\mathcal{T})$ is an object $A$ of 
$\mathrm{Ind}(\mathcal{T})$ equipped with an associative and commutative product $ A \otimes A \longrightarrow A$ and an unit $1_{\mathcal{T}} \longrightarrow A$ which satisfy the usual axioms. 
We define the category of affine schemes in $\mathcal{T}$ as the opposite category of the category of commutative rings (with unit) in $\mathrm{Ind}(\mathcal{T})$. We denote by $\Spec A$ the affine scheme in $\mathcal{T}$  associated to a commutative ring (with unit) $A$ of $\mathrm{Ind}(\mathcal{T})$. We call a group object in the category of affine schemes in $\mathcal{T}$ an affine group scheme in $\mathcal{T}$. The category of affine group schemes in 
${\mathcal{T}}$ can be regarded as the opposite category of the category of commutative Hopf algebras with antipode in $\mathrm{Ind}(\mathcal{T})$, more precisely, the category of objects $A$ in $\mathrm{Ind}(\mathcal{T})$ endowed with an associative, commutative product $A \otimes A \longrightarrow A$, a unit $1_{\mathcal{T}} \longrightarrow A$, an associative coproduct $A \longrightarrow A \otimes A$, a counit $A \longrightarrow 1_{\mathcal{T}}$ and 
an antipode $A \longrightarrow A$ with the usual axioms. For an affine group scheme $G= \Spec A$ in $\mathcal{T}$, 
a representation of $G$ in $\mathcal{T}$ is an object $V$ in $\mathcal{T}$ endowed with a comodule structure (a coaction $V \rightarrow A \otimes V$ satisfying the usual axioms). 

   

Let $\cC, \cD$ be Tannakian categories endowed with 
exact faithful $k$-linear tensor functors 
$t:\cC \lra \cD, \omega:\cD \lra \cC$ with $\omega \circ t = \id$. 
(We regard $\cD$ as a Tannakian category over $\cC$ (via $t$)
endowed with a fiber functor $\omega$ to $\cC$.) Denote $\pi(\cC) := \Spec \cO_{\pi(\cC)}, \pi(\cD) := \Spec \cO_{\pi(\cD)}$ 
be the fundamental group of $\cC, \cD$, 
which is an affine group scheme in $\cC, \cD$ respectively. 
$\cO_{\pi(\cC)}$ is defined as  
$$ 
\cO_{\pi(\cC)} := \Coker\left(\bigoplus_{\alpha: V \ra W \in \cC}V \otimes W^{\vee} 
\lra \bigoplus_{V \in \cC} V \otimes V^{\vee}\right), 
$$ 
where the arrow is the difference of the map induced by 
$V \otimes W^{\vee} \os{\alpha \otimes \id}{\lra} W \otimes W^{\vee}$ 
$(\alpha: V \lra W \in \cC)$ and the map induced by 
$V \otimes W^{\vee} \os{\id \otimes \alpha^{\vee}}{\lra} V \otimes V^{\vee}$ 
$(\alpha: V \lra W \in \cC)$,  
$\cO_{\pi(\cD)}$ is defined in the same way (see \cite[8.13]{Del90}). 
Then the map of  affine group schemes in $\cD$ 
\begin{equation}\label{del1}
t^*: \pi(\cD) \lra t(\pi(\cC))
\end{equation}
is defined by the map
$t(\cO_{\pi(\cC)}) \lra \cO_{\pi(\cD)}$, which is induced by the map 
$\bigoplus_{V \in \cC} t(V) \otimes t(V)^{\vee} \lra 
\bigoplus_{W \in \cD} W \otimes W^{\vee}$ 
of `the inclusion into $t(V)$-component ($V \in \cC$)'. 
By applying $\omega$, we obtain the map 
\begin{equation}\label{del2}
\omega(t^*): \omega(\pi(\cD)) \lra \pi(\cC), 
\end{equation}
defined by 
$\cO_{\pi(\cC)} \lra \omega(\cO_{\pi(\cD)})$, which is induced by the map 
$\bigoplus_{V \in \cC} V \otimes V^{\vee} \lra 
\bigoplus_{W \in \cD} \omega(W) \otimes \omega(W)^{\vee}$
of `the inclusion into $t(V)$-component ($V \in \cC$)'. 
We define the group scheme $G(\cD,\omega) := \Spec \cO_{G(\cD,\omega)}$ in 
$\cC$ by $G(\cD,\omega) := \Ker\, \omega(t^*)$. By definition,  
$\cO_{G(\cD,\omega)}$ is the pushout of the diagram 
$$ 1_{\cC} \lla \cO_{\pi(\cC)} \lra 
\omega(\cO_{\pi(\cD)}) $$
in the category of Hopf algebras in ${\mathrm{Ind}}(\cC)$. 

For a group scheme $G$ in $\cC$, we consider the category 
$\Rep_{\cC} G$ of representations 
of $G$ in $\cC$. For $V \in \cD$, the 
map $\omega(V) \lra \omega(V) \otimes \omega(\cO_{\pi(\cD)})$ 
induced by the map 
$V \otimes V^{\vee} \lra \bigoplus_{W \in \cD} W \otimes W^{\vee}$ 
of `the inclusion into $V$-component' gives a coaction of 
$\omega(\cO_{\pi(\cD)})$ on $\omega(V)$, hence a representation of 
$\omega(\pi(\cD))$ on $\omega(V)$. This induces a representation of 
$G(\cD,\omega)$ on $\omega(V)$ and so the functor 
$\omega$ induces the functor 
\begin{equation}\label{dd}
\cD \lra \Rep_{\cC} G(\cD,\omega). 
\end{equation}
On the other hand, for 
a group scheme $G = \Spec \cO_G$ in $\cC$, we have the functors 
$$ t: \cC \lra \Rep_{\cC}G, \quad \omega: \Rep_{\cC}G \lra \cC, $$
where $t$ is the functor that embeds $\cC$ in $ \Rep_{\cC}G$ as  trivial representations and 
$\omega$ is the forgetful functor. Using these, we can define 
the group scheme $G(\Rep_{\cC}G, \omega)$. 
The coactions $\omega(V) \lra \omega(V) \otimes \cO_G$
 ($V \in \Rep_{\cC}G$) induce the map
$\omega(\cO_{\pi(\Rep_{\cC}G})) \lra \cO_G$. 
Then we can see that it induces the map 
$\cO_{G(\Rep_{\cC}G, \omega)} \lra \cO_G$, thus the map 
\begin{equation}\label{gg}
G \lra G(\Rep_{\cC}G, \omega). 
\end{equation}

In \cite[Prop. 2.3]{laz}, Lazda proved the following. 

\begin{prop}\label{reltannaka}
The functor \eqref{dd} is an equivalence and the morphism 
\eqref{gg} is an isomorphism. 
\end{prop}

Assume that we are given group schemes $G,G'$ in $\cC$ and a 
$k$-linear exact tensor functor 
$$\alpha: \Rep_{\cC}G \lra \Rep_{\cC}G'$$ 
compatible with the 
forgetful functors $\omega_G, \omega_{G'}$ to $\cC$. Then we have 
a morphism $\alpha^*: \pi(\Rep_{\cC}G') \lra \alpha(\pi(\Rep_{\cC}G))$, 
and it is easy to see that this induces the morphism 
$$G' = G(\Rep_{\cC}G', \omega_{G'}) \lra G(\Rep_{\cC}G, \omega_G) = G, $$
where the first and the last equality are due to Proposition \ref{reltannaka}. 

Let $\cC$ be a Tannakian category and let $\eta:\cC \lra \Vector_k$ be a fiber functor. 
Then it is well-known that, if we put $G := \eta(\pi(\cC))$, 
we have the equivalence $\cC \os{\cong}{\lra} \Rep_k(G)$ by Tannaka duality. 
For $V \in \cC$, the $G$-action on $\eta(V)$ is given in the following way: 
Let $R$ be a $k$-algebra and take $g \in G(R)$ which is 
a ring homomorphism $g: \cO_G = \eta(\cO_{\pi(\cC)}) \lra R$. 
Then, by definition of $\cO_{\pi(\cC)}$, we have a natural map \cite[8.4]{Del90}: 
$$ \eta(V) \otimes \eta(V)^{\vee} \lra \eta(\cO_{\pi(\cC)}) \os{g}{\lra} R$$ 
and this induces the map 
$$g_V: \eta(V) \otimes R \lra \eta(V) \otimes R, $$
which is defined to be the action of $g \in G(R)$ on $\eta(V) \otimes R$. 

If we consider the previous definition of 
$\cO_{G(\cD,\omega)}$, which considers it as a  quotient of 
$\bigoplus_{V \in \cD} \omega(V) \otimes \omega(V)^{\vee}$, 
we see the following: for $g \in G(R)$ as above, the action of $g$ on 
$\eta(\cO_{G(\cD,\omega)}) \otimes R$ is defined to be the morphism induced by 
\begin{equation}\label{actg}
(\eta\omega(V) \otimes \eta\omega(V)^{\vee}) \otimes R 
\lra 
(\eta\omega(V) \otimes \eta\omega(V)^{\vee}) \otimes R \quad (V \in \cD), 
\end{equation}
which is the action  of $g$  on $\omega(V) \otimes \omega(V)^{\vee}$  where $g_{\omega(V) \otimes \omega(V)^{\vee}} 
= g_{\omega(V)} \otimes g_{\omega(V)}^{\vee}$. 
If we identify $(\eta\omega(V) \otimes \eta\omega(V)^{\vee}) \otimes R$ with 
$\End(\eta\omega(V) \otimes R)$, the action \eqref{actg} is just the 
conjugate action $g_{\omega(V)} \circ - \circ g_{\omega(V)}^{-1}$. 

Let us consider the following split exact sequence of group schemes 
$$ 
\xymatrix {
1 \ar[r] & \eta(G(\cD,\omega)) \ar[r] & 
\eta\omega(\pi(\cD)) \ar[r] &  \eta(\pi(\cC)) = G \ar[r] 
\ar @/^6mm/[0,-1]^{\eta(\omega^*)} & 1.
} 
$$
Then the right action of $g \in G(R)$ on $\eta(G(\cD,\omega))(R)$ 
induced by the action on $\eta(\cO_{G(\cD,\omega)}) \otimes R$ above 
is equal to the conjugate action $\eta(\omega^*)(R)(g)^{-1} \circ - \circ \eta(\omega^*)(R)(g)$. 

\bigskip 

Now, let the notation be as in Notation \ref{notation}. 
Then we have the functors between Tannakian categories 
$$ f_{\dR}^*: \MICn(S/k) \lra \NfMICn(X/k), \quad 
\iota_{\dR}^*: \NfMICn(X/k) \lra \MICn(S/k). $$
If we apply the construction above to the above functors, we obtain 
the first definition of relatively unipotent de Rham fundamental group, 
which is due to Lazda \cite[Remark 2.4]{laz}: 

\begin{defn}[\bf{First definition of relatively unipotent $\pi_1$}]\label{def1}
Let the notations be as above. We define the relatively unipotent de Rham fundamental 
group $\pi^{\dR}_1(X/S,\iota)$ by 
$$\pi^{\dR}_1(X/S,\iota) := G(\NfMICn(X/k), \iota_{\dR}^*).$$ 
This is an affine group scheme in $\MICn(S/k)$. 
\end{defn}

We define 
$\pi^{\dR}_1(S,s)$, $\pi^{\dR}_1(X,x)$, $\pi^{\dR}_1(X_s,x)$ 
as the Tannaka dual of $(\MICn(S/k), s_{\dR}^*)$, 
$(\NfMICn(X/k),x_{\dR}^*)$, $(\NfsMIC(X_s/s), x_{\dR}^*)$, 
respectively. Then, applying the previous construction, we have the split exact sequence 
$$ 
\xymatrix {
1 \ar[r] & s_{\dR}^*\pi_1^{\dR}(X/S,\iota) \ar[r] & 
\pi^{\dR}_1(X,x) \ar[r]^{f_*} & \pi^{\dR}_1(S,s) \ar[r] 
\ar @/^6mm/[0,-1]^{\iota_*} & 1,
} 
$$
where $f_*, \iota_*$ are the morphisms induced by $f, \iota$ respectively. 
Moreover, the morphism  $\pi^{\dR}_1(X_s,x) \lra \pi_1^{\dR}(X,x)$ induced by 
the exact closed immersion $X_s \hra X$ induces the morphism 
\begin{equation}\label{fiber}
\pi_1^{\dR}(X_s, x) \lra s_{\dR}^*\pi_1^{\dR}(X/S,\iota). 
\end{equation}
Then we have the following: 

\begin{prop}\label{fiberprop}
The morphism \eqref{fiber} is an isomorphism. So we have the 
split exact sequence 
\begin{equation}\label{eq:2020Aug-0}
\xymatrix {
1 \ar[r] & \pi_1^{\dR}(X_s,x) \ar[r] & 
\pi^{\dR}_1(X,x) \ar[r]^{f_*} & \pi^{\dR}_1(S,s) \ar[r] 
\ar @/^6mm/[0,-1]^{\iota_*} & 1
} 
\end{equation}
\end{prop}

\begin{proof}
This is due to Lazda (\cite[Corollary 1.20]{laz}) 
when the log structures on $X,S$ are trivial and $S$ is smooth over $k$, 
and the same proof works also in our case. 
We provide a proof for the convenience of the reader. 

Denote the functors of restriction 
$\NfMICn(X/k) \lra \NfsMIC(X_s/s), \MICn(S/k) \allowbreak \lra \MIC(s/s)$ by 
$(E,\nabla_E) \mapsto (E,\nabla_E)|_{X_s}, 
(E,\nabla_E) \mapsto (E,\nabla_E)|_s$, respectively. 
By \cite[I Proposition 1.4]{w} and 
\cite[Appendix A]{ehs}, it suffices to prove the following three claims. 
\begin{enumerate}
\item Any object in $\NfsMIC(X_s/s)$ is a quotient of an object of the form 
$(E,\nabla_E)|_{X_s}$ \, $((E,\nabla_E) \in \NfMICn(X/k))$.  
\item For any object $(E,\nabla_E)$ in $\NfMICn(X/k)$ with $(E,\nabla_E)|_{X_s}$ trivial, 
there exists an object $(V,\nabla_V)$ in $\MICn(S/k)$ with $(E,\nabla_E) = 
f_{\dR}^*(V,\nabla_V)$. 
\item Let $(E,\nabla_E)$ be an object in $\NfMICn(X/k)$ and let 
$(F_0,\nabla_{F_0}) \in \NfsMIC(X_s/s)$ be 
the largest trivial subobject of $(E,\nabla_E)|_{X_s}$. Then there exists 
a subobject $(E_0,\nabla_{E_0})$ of $(E,\nabla_E)$ with 
$(F_0,\nabla_{F_0}) = (E_0,\nabla_{E_0})|_{X_s}$. 
\end{enumerate}
(In fact, the claim (1) implies the injectivity of 
the map \eqref{fiber} and the claims (2), (3) imply the surjectivity of the map 
\eqref{fiber}.) 

Note that we have seen before Definition \ref{def:1.11}
that, for an object $(E,\nabla_E)$ in $\NfMICn(X/k)$, the map 
\begin{equation}\label{eq:lazadj}
f_{\dR}^*f_{\dR *}(E,\nabla_E) \lra (E,\nabla_E)
\end{equation}
gives an injection onto the maximal subobject of $(E,\nabla_E)$ which 
belongs to the category $f_{\dR}^*\MICn(S/k)$. Using this, 
we prove the claim (2). Let $(E,\nabla_E)$ be an object in $\NfMICn(X/k)$ with 
$(E,\nabla_E)|_{X_s}$ trivial. By restricting \eqref{eq:lazadj}  to $X_s$ and using 
the base change property (see remark \ref{gmexseq} and \ref{locally_free_dR}), we obtain the morphism 
\begin{align*}
f_{s,\dR}^*f_{s,\dR *}((E, \nabla_E)|_{X_s}) 
\cong f_{s,\dR}^*((f_{\dR *}(E,\nabla_E))|_s) 
& = (f_{\dR}^*f_{\dR *}(E,\nabla_E))|_{X_s} \\ & \lra (E,\nabla_E)|_{X_s}, 
\end{align*}
which is an isomorphism due to the triviality of 
$(E,\nabla_E)|_{X_s}$. Thus the map \eqref{eq:lazadj} is an isomorphism  (the two objects have the same rank)
and so the claim holds if we put $(V,\nabla_V) := f_{\dR *}(E,\nabla_E)$. 

Next we prove the claim (3). Let $(E,\nabla_E)$ and $(F_0,\nabla_{F_0}) \subseteq 
(E,\nabla_E)|_{X_s}$ be as in the statement of the claim. 
Then we have $(F_0,\nabla_{F_0}) = f_{s,\dR}^*f_{s, \dR *}((E,\nabla_E)|_{X_s})$. 
So, by the base change property, we obtain the isomorphism 
\begin{align*}
(F_0,\nabla_{F_0}) = f_{s, \dR}^*f_{s, \dR *}((E,\nabla_E)|_{X_s}) 
\cong f_{s,\dR}^*((f_{\dR *}(E,\nabla_E))|_s)
= (f_{\dR}^*f_{\dR *}(E,\nabla_E))|_{X_s}. 
\end{align*}
So the claim holds if we put $(E_0,\nabla_{E_0}) = f_{\dR}^*f_{\dR *}(E,\nabla_E)$. 

Finally we explain the proof of the claim (1). 
We will construct a projective system of objects 
$\{W_n\}_n$ in $\NfMICn(X/k)$ in the next section (Remark \ref{rem:sect2}) such that, 
for any object $(F,\nabla_F)$ in $\NfsMIC(X_s/s)$, there exist 
$n, N \in \N$ and a surjection 
$W_n^{\oplus N}|_{X_s} \lra (F,\nabla_F)$. 
This implies (1). Hence the proof of proposition will be  finished modulo the construction of 
the projective system $\{W_n\}_n$. 
\end{proof}
\begin{rem} The above proposition can be seen as a log de Rham analogue of the homotopy exact sequence for the \'etale fundamental group of a fibration with a section (\cite[Proposition 4.3, Exposé XIII]{SGA1}).
\end{rem}
Since $\pi_1^{\dR}(X/S,\iota)$ is an affine group scheme in 
$\MICn(S/k)$, $s_{\dR}^*\pi_1^{\dR}(X/S,\iota)$ is endowed with 
the action of $\pi_1^{\dR}(S,s)$ (called the monodromy action) by Tannaka duality. 
By Proposition \ref{fiberprop} and by the results of this section we have the following corollary: 

\begin{cor}\label{fiberpropcor}
Let the notations be as above. Then the monodromy action of 
$\pi_1^{\dR}(S,s)$ on $s_{\dR}^*\pi_1^{\dR}(X/S,\iota)$ 
is equal to the action of $\pi_1^{\dR}(S,s)$ on $\pi_1^{\dR}(X_s,s)$
defined by the conjugate action $g \mapsto \iota_*(g)^{-1} \circ - \circ \iota_*(g) 
\, (g \in \pi_1^{\dR}(S,s))$ via the isomorphism 
$\pi_1^{\dR}(X_s,s) \cong s_{\dR}^*\pi_1^{\dR}(X/S,\iota)$ of 
Proposition \ref{fiberprop}. 
\end{cor} 

\begin{rem}\label{rem:dps} 
In \cite{dps}, 
the second and the third author
proved the homotopy exact sequence for maximal geometrically protrigonalizable 
quotients of de Rham fundamental groups, in the case where 
$S$ is the standard log point $\Af_k^{0,1}$ and 
$f: X \lra S$ is a quasi-projective normal crossing log scheme.
 
This implies the split exact sequence 
$$ 
\xymatrix {
1 \ar[r] & \pi_1^{\dR}(X_s,s) \ar[r] & 
\pi^{\dR}_1(X,x) \ar[r]^(0.4){f_*} & \pi^{\dR}_1(S,s) = \G_{a,k} \ar[r] 
\ar @/^6mm/[0,-1]^{\iota_*} & 1. 
} 
$$
In the introduction of \cite{dps}, 
they defined the monodromy action of $\pi^{\dR}_1(S,s) = \G_{a,k}$ 
on $\pi_1^{\dR}(X_s,s)$ by the conjugate 
$g \mapsto \iota_*(g)^{-1} \circ - \circ \iota_*(g) 
\, (g \in \pi_1(S,s))$. 
Hence, by Corollary \ref{fiberpropcor}, their definition of monodromy action 
is the same as the one given in this section. 
\end{rem}

%
%

\begin{rem}\label{rem:2020Aug-2} 
Let $f: X \lra S, f^{\log}_{\an}: X^{\log}_{\an} \lra S^{\log}_{\an}$ be as in 
Remark \ref{rem:2020Aug-1}. Then, by \cite[Theorem 5.1]{no}, 
$f^{\log}_{\an}$ is a topological fiber bundle. Moreover, it admits a section 
$\iota^{\log}_{\an}$ induced by $\iota$. Let $s$ be a point of $S^{\log}_{\an}$ 
and put $x := \iota^{\log}_{\an}(s)$, $X^{\log}_{\an, s} := (f^{\log}_{\an})^{-1}(s)$. 
Then we have the split homotopy exact sequence 
\begin{equation}\label{eq:2020Aug-2-1}
\xymatrix {
1 \ar[r] & \pi_1^{\rm top}(X^{\log}_{\an, s} ,s) \ar[r] & 
\pi^{\rm top}_1(X^{\log}_{\an} ,x) \ar[r]^-{f^{\log}_{\an *}} & \pi^{\rm top}_1(S^{\log}_{\an},s) \ar[r] 
\ar @/^6mm/[0,-1]^{\iota^{\log}_{\an *}} & 1. 
}
\end{equation}
of topological fundamental groups. This is the topological analogue of 
the split exact sequence \eqref{eq:2020Aug-0}. Using the categories 
${\rm LS}^{\rm u}(S^{\log}_{\an}), {\rm N}_{f^{\log}_{\an}}{\rm LS}^{\rm u}(X^{\log}_{\an})$ 
in Remark \ref{rem:2020Aug-1}, we can form a relatively unipotent algebraic group version 
of the diagram \eqref{eq:2020Aug-2-1}, which would be a close topological analogue of 
\eqref{eq:2020Aug-0}. 
\end{rem}

\section{Construction of Hadian, Andreatta--Iovita--Kim and Lazda}

Let the notations be as in Notation \ref{notation}. 
In this section, we give the second definition of 
relatively unipotent de Rham fundamental group along the lines of  
the construction of Hadian \cite{ha}, Andreatta--Iovita--Kim \cite{aik} 
and Lazda \cite{laz}, 
moreover we will prove  the coincidence of this new definition  with the one given in the 
previous section. 

The following theorem (construction) is the key ingredient of the definition.  The notations are those of the previous sections.

\begin{thm}\label{thm:wn}
There exists 
a projective system 
$(W,e) := \{(W_n, e_n)\}_{n \geq 1}$ consisting of 
\begin{itemize}
\item 
$W_n := (\ol{W}_{\mkern-4 mu n}, \{\epsilon_n^m\}_m) \in 
\NfStrCrysn(X/k) \cong \NfMICn(X/k)$, and 
\item 
A morphism $e_n: (\cO_S,d) \lra \iota_{\dR}^*W_n$ in $\MICn(S/k)$
\end{itemize}
which satisfies the following conditions: 
\begin{enumerate}
\item[{\rm (W0)}] For any $n \geq 2$, the morphism 
$f_{\dR *}(\ol{W}_{n-1}^{\vee}) \lra f_{\dR *}(\ol{W}_{\mkern-4 mu n}^{\vee})$ induced by 
the transition map $\ol{W}_{\mkern-4 mu n} \lra \ol{W}_{n-1}$ is an isomorphism. 
\item[{\rm (W1)}] For any $E \in \NfMIC(X/S)$ of index of unipotence $\leq n$ and 
a morphism $v: \cO_S \lra \iota_{\dR}^* E$ in $\MIC(S/S)$, there exists a unique 
morphism $\varphi: \ol{W}_{\mkern-4 mu n} \lra E$ in $\NfMIC(X/S)$ with 
$\iota_{\dR}^*(\varphi) \circ \ol{e}_n = v$, where 
$\ol{e}_n: \cO_S \lra \iota_{\dR}^*\ol{W}_{\mkern-4 mu n}$ is the underlying morphism 
of $e_n$ in $\NfMIC(S/S)$. Moreover, the same universality holds after base change 
by any morphism $S' \lra S$ of finite type. 
\item[{\rm (W2)}]
For any $n$, there exists an exact sequence 
\begin{equation}\label{wn-ext0}
0 \lra f_{\dR}^*R^1f_{\dR *} (W_n^{\vee})^{\vee} \lra W_{n+1} \lra 
W_n \lra 0 
\end{equation}
in $\NfMICn(X/k)$, where on $f_{\dR}^*R^1f_{\dR *} (W_n^{\vee})^{\vee}$ 
we put the log connection induced by the Gauss--Manin connection. 
\end{enumerate}
Here, we say that an object in $\NfMIC(X/S)$ is of index of unipotence $\leq n$ 
if it can be written as an iterated extension of at most $n$ objects in 
$f^*_{\dR}\MIC(S/S)$. 
\end{thm}

\begin{proof}
For $n=1$, we put $W_1 := (\cO_X,d)$ and 
define $e_1$ to be the canonical isomorphism 
$(\cO_S,d) \os{=}{\lra} \iota_{\dR}^*(\cO_X,d)$. 
Then we can check the condition (W1) for $n=1$: 
Indeed, if we are given an object 
$E = f_{\dR}^*V \in f_{\dR}^*\MIC(S/S)$ and a morphism 
$v: \cO_S \lra \iota_{\dR}^*E = V$ in $\MIC(S/S)$, 
the map $\varphi: \cO_X \lra E$ in $\NfMIC(X/S)$ satisfies 
$\iota_{\dR}^*(\varphi) = v$ if $\varphi = f_{\dR}^*v$, and 
if we have a morphism $\varphi$ as above with 
$\iota_{\dR}^*(\varphi) = v$, we necessarily have 
$$ \varphi = f_{\dR}^*f_{\dR *}\varphi = 
f_{\dR}^*\iota_{\dR}^*f_{\dR}^* f_{\dR *} \varphi = 
f_{\dR}^*\iota_{\dR}^* \varphi = f_{\dR}^*v. $$

Now we construct $(W_{n+1}, e_{n+1})$ from $(W_i,e_i) \, (i \leq n)$. 
Since $W_n$ belongs to $\NfMICn(X/k)$, 
$R^1f_{\dR *}(W_n^{\vee})^{\vee}$ belongs to 
$\MICn(S/k)$. In particular, 
$R^1f_{\dR *}(\ol{W}_{\mkern-4 mu n}^{\vee})^{\vee}$ belongs to $\MIC(S/S)$. 
When $S$ is affine, we define 
$\ol{W}_{n+1} \in \NfMIC(X/S)$ as the extension 
\begin{equation}\label{wn-ext}
0 \lra f_{\dR}^*R^1f_{\dR *} (\ol{W}_{\mkern-4 mu n}^{\vee})^{\vee} \lra \ol{W}_{n+1} \xrightarrow{\pi} 
\ol{W}_{\mkern-4 mu n} \lra 0 
\end{equation}
whose extension class, which is regarded as an element in 
$$\Gamma(S,  R^1f_{\dR *} (\ol{W}_{\mkern-4 mu n}^{\vee}) \otimes R^1f_{\dR *} (\ol{W}_{\mkern-4 mu n}^{\vee})^{\vee})  = 
\End (R^1f_{\dR *} (\ol{W}_{\mkern-4 mu n}^{\vee})),$$ is the identity. 
Also, let $\ol{e}_{n+1}: \cO_S \lra \iota_{\dR}^*\ol{W}_{n+1}$ be a morphism 
which lifts the map $\ol{e}_n: \cO_S \lra \iota_{\dR}^*\ol{W}_{\mkern-4 mu n}$. 

We check the properties (W0) and (W1) 
for $(\ol{W}_{n+1},\ol{e}_{n+1})$ (assuming $S$ and $S'$ affine). 
The method is the same as \cite[Proposition 2.6]{ha}, \cite[Proposition 3.3]{aik}
and \cite[Proposition 1.17]{laz}, but we write down the argument 
for the convenience of the reader. Since the proof on $X \times_S S'/S'$ 
is the same as that on $X/S$ (by base change property (W1)), 
we prove the property only on $X/S$. In the proof, 
we denote the map $\ol{W}_{n+1} \lra \ol{W}_{\mkern-4 mu n}$ in \eqref{wn-ext} by $\pi$. 
One has the long exact sequence 
\begin{align*}
0 & \lra f_{\dR *}(\ol{W}_{\mkern-4 mu n}^{\vee}) \os{\gamma}{\lra} f_{\dR *}(\ol{W}_{n+1}^{\vee}) 
\lra R^1f_{\dR *}(\ol{W}_{\mkern-4 mu n}^{\vee}) \\ 
& \os{\delta}{\lra} R^1f_{\dR *}(\ol{W}_{\mkern-4 mu n}^{\vee}) \os{\epsilon}{\lra} 
R^1f_{\dR *}(\ol{W}_{n+1}^{\vee}) \lra \cdots  
\end{align*}
associated to the dual of the exact sequence \eqref{wn-ext}. 
By definition of $\ol{W}_{n+1}$ and the standard argument in homological algebra, 
we see that the map $\delta$ is the identity 
(see \cite[Proposition 1.15]{laz}). Thus $\gamma$ is an isomorphism  
and $\epsilon$ is zero. 
In particular, we have checked the property (W0). Also, 
any extension in $\NfMIC(X/S)$ of the form 
$$ 0 \lra f_{\dR}^*V \lra E \lra \ol{W}_{\mkern-4 mu n} \lra 0 \quad (V \in \MIC(S/S)) $$
splits if we pull it back by $\pi: \ol{W}_{n+1} \lra \ol{W}_{\mkern-4 mu n}$. 

We take an object $E \in \NfMIC(X/S)$ of index of unipotence $\leq n+1$ 
and a morphism $v: \cO_S \lra \iota_{\dR}^* E$. 
Then there exists an exact sequence of the form 
\begin{equation}\label{eq:univ-exact}
0 \lra f_{\dR}^*V \os{\alpha}{\lra} E \os{\beta}{\lra} E' \lra 0
\end{equation}
in $\NfMIC(X/S)$, 
where $E'$ is of index of unipotence $\leq n$. By induction hypothesis, 
there exists a unique morphism $\psi: \ol{W}_{n} \lra E'$ with 
$\iota_{\dR}^*(\psi) \circ \ol{e}_{n} = \iota_{\dR}^*(\beta) \circ v$. 
By considering the pullback of the above exact sequence by 
$\psi \circ \pi$, we obtain the following diagram 
$$ 
\xymatrix {
0 \ar[r] & f_{\dR}^* V \ar[r] \ar@{=}[d] & 
\wt{E} \ar[r] \ar[d] & \ol{W}_{n+1} \ar[d]^-{\psi \circ \pi} \ar[r] \ar @/_5mm/[0,-1]_{\varphi'_1} 
& 0 \\
0 \ar[r] & f_{\dR}^*V \ar[r]^-{\alpha} & E \ar[r]^-{\beta} & 
E' \ar[r] & 0, 
} 
$$
where $\varphi'_1$ is a section. If we denote the composite
 $\ol{W}_{n+1} \os{\varphi'_1}{\lra} \wt{E} \lra E$ by $\varphi_1$, 
we have the equality 
\begin{align*}
\iota_{\dR}^*(\beta) \circ (\iota_{\dR}^*(\varphi_1) \circ \ol{e}_{n+1} - v) 
& = 
\iota_{\dR}^*(\psi) \circ \iota_{\dR}^*(\pi) \circ \ol{e}_{n+1} - 
\iota_{\dR}^*(\beta) \circ v \\ 
& = 
\iota_{\dR}^*(\psi) \circ \ol{e}_{n} - 
\iota_{\dR}^*(\beta) \circ v = 0. 
\end{align*}
Hence there exists a map $w: \cO_S \lra \iota_{\dR}^*f_{\dR}^* V = V$ 
such that $\iota_{\dR}^*(\alpha) \circ w = \iota_{\dR}^*(\varphi_1) \circ \ol{e}_{n+1} - v$. 
On the other hand, there exists a map $\varphi_2: \ol{W}_{n+1} \lra f_{\dR}^*(V)$ 
with $\iota_{\dR}^*(\varphi_2) \circ \ol{e}_{n+1} = w$ by the property (W1) 
in the case $n=1$ (and the compatibility of $(\ol{W}_{n+1}, \ol{e}_{n+1})$ with $(\ol{W}_1,\ol{e}_1)$). 
Then, if we put $\varphi := \varphi_1 - \alpha \circ \varphi_2$, we obtain that 
$$ 
\iota_{\dR}^*(\varphi)\circ \ol{e}_{n+1}= \iota_{\dR}^*(\varphi_1 - \alpha \circ \varphi_2)\circ \ol{e}_{n+1} =
\iota_{\dR}^*(\varphi_1) \circ \ol{e}_{n+1} - \iota_{\dR}^*(\alpha) \circ w = v, 
$$
as required. 

We prove the uniqueness of $\varphi$ by induction on 
the index of unipotence of $E$. When the index of unipotence is $\leq 1$, 
$E$ has the form $f_{\dR}^*V$ and to give a map 
$\varphi: \ol{W}_{n+1} \lra f_{\dR}^*V$ is equivalent to 
give a map $V^{\vee} \lra f_{\dR *}(\ol{W}_{n+1}^{\vee})$, which we denote by 
$\psi$. Because the map $f_{\dR *}(\ol{W}_{1}^{\vee}) \lra  f_{\dR *}(\ol{W}_{n+1}^{\vee})$ 
induced by the transition map $\ol{W}_{n+1} \lra \ol{W}_1$ is an isomorphism 
by the property (W0) for $\ol{W}_i \,(2 \leq i \leq n+1)$ which we have already shown, 
$\psi$ factors through $f_{\dR *}(\ol{W}_{1}^{\vee})$, namely, 
$\varphi$ factors through $\ol{W}_{1}$. Then the uniqueness of 
$\varphi$ in this case follows from the property (W1) for $(\ol{W}_1,\ol{e}_1)$. 
In general case, we take an exact sequence \eqref{eq:univ-exact} 
and assume the existence of another map 
$\varphi': \ol{W}_{n+1} \lra E$ with 
$\iota_{\dR}^*(\varphi') \circ \ol{e}_{n+1} = v$. By the induction hypothesis applied to 
$E'$, we have 
$\beta \circ \varphi = \beta \circ \varphi'$ and so  
there exists a map 
$\psi: \ol{W}_{n+1} \lra f_{\dR}^*V$ with $\varphi - \varphi' = \alpha \circ \psi$. 
Also, we have equalities 
$$ 
0 = \iota_{\dR}^*(\varphi - \varphi') \circ \ol{e}_{n+1} = 
\iota_{\dR}^*(\alpha) \circ \iota_{\dR}^*(\psi) \circ \ol{e}_{n+1}. $$
Thus we see that $\iota_{\dR}^*(\psi) \circ \ol{e}_{n+1} = 0$, and this implies 
(again by induction hypothesis) that $\psi = 0$. 
So $\varphi = \varphi'$ and the uniquess is proved. 

So far, we defined $(\ol{W}_{n+1},\ol{e}_{n+1})$ when $S$ is affine and 
checked the properties (W0), (W1) for it when $S$ and $S'$ are affine. 
When $S$ is not necessarily affine, we take 
an affine open covering $\{S_{\alpha}\}_{\alpha}$ of $S$ and 
define $(\ol{W}_{n+1},\ol{e}_{n+1})$ on each $S_{\alpha}$, which 
we denote by $(\ol{W}_{n+1},\ol{e}_{n+1})_{\alpha}$. Then, on 
each $S_{\alpha} \cap S_{\beta}$, the universality in (W1) implies that 
there exists a unique isomorphism 
$(\ol{W}_{n+1},\ol{e}_{n+1})_{\alpha}|_{S_{\alpha} \cap S_{\beta}} \cong (\ol{W}_{n+1},\ol{e}_{n+1})_{\beta}|_{S_{\alpha} \cap S_{\beta}}.$ 
Hence $(\ol{W}_{n+1},\ol{e}_{n+1})_{\alpha}$'s glue and give the pair 
$(\ol{W}_{n+1}, \ol{e}_{n+1})$ on $S$.  One can check the properties (W0), (W1) of this pair 
for general $S$ and $S'$ by reducing to the affine case. 

In the next step we are  going to enhance the relative connection structure of $\ol{W}_{n+1}$ to an absolute one, defining in this way ${W}_{n+1}$. Let $f_j^m: X_j^m \lra S^m(1)$, 
$p_j^m: S^m(1) \lra S$, 
$q_j^m: X_j^m \lra X$, 
$\wh{f}^m: \wh{X}^m \lra S^m(1)$ and 
$\wh{q}_j^m: \wh{X}^m \lra X$ 
be as in Section 1, (before Definition \ref{strcrys} and Proposition \ref{basechange}), and let $\iota_j^m:S^m(1) \lra X_j^m$ be 
the base change of $\iota$, which is a section of 
$f_j^m$. Then, by the property (W1), there exists  a
unique isomorphism 
$\epsilon^m_{n+1}: ({q_2^m}_{\dR}^*\ol{W}_{n+1}, {p_2^m}_{\dR}^*\ol{e}_{n+1}) \os{\cong}{\lra} 
({q_1^m}_{\dR}^*\ol{W}_{n+1}, {p_1^m}_{\dR}^*\ol{e}_{n+1}) \, (m \in \N)$ 
in the categories (see Proposition \ref{prop:trinity}) 
$$ \MIC(X^m_1/S^m(1)) \os{\cong}{\lla} \Crys((X/S^m(1))_{\inf}) 
\os{\cong}{\lra} \MIC(X^m_2/S^m(1)) $$
which satisfies the 
cocycle condition. (To prove the cocycle condition, we need to work on 
pullbacks of $f:X \lra S$ to $S^m(2)$. We leave the reader to write the detailed argument.) 
So the object $W_{n+1} := (\ol{W}_{n+1}, \{\epsilon_{n+1}^m\}_m)$ in 
$\StrCrys(X/k) \cong \MIC(X/k)$ is defined, and 
the isomorphisms ${p_2^m}_{\dR}^*\ol{e}_{n+1} \cong {p_1^m}_{\dR}^*\ol{e}_{n+1} \, (m \in \N)$ 
induce the map $e_{n+1}: (\cO_S,d) \lra \iota_{\dR}^*W_{n+1}$. 

Let us prove that $W_{n+1}$ belongs to 
$\NfMICn(X/k)$ and it satisfies the property (W2). 
By the commutative diagram 
\[ 
\xymatrix{
\MIC(X^m_1/S^m(1)) 
\ar[rd] & 
\Crys((X/S^m(1))_{\inf}) \ar[l]_{\cong} \ar[r]^{\cong} \ar[d] & 
\MIC(X^m_2/S^m(1)) 
\ar[ld]
\\ & 
\MIC(\wh{X}^m/S^m(1)), 
}
\]
the isomorphisms 
$\epsilon^m_n: ({q_2^m}_{\dR}^*\ol{W}_{\mkern-4 mu n}, {p_2^m}_{\dR}^*\ol{e}_n) \os{\cong}{\lra} 
({q_1^m}_{\dR}^*\ol{W}_{\mkern-4 mu n}, {p_1^m}_{\dR}^*\ol{e}_n) \, (m \in \N)$
induce the isomorphisms 
$(\wh{q}_2^m)_{\dR}^*\ol{W}_{n}^{\vee} \os{\cong}{\lra} (\wh{q}_1^m)_{\dR}^*\ol{W}_{n}^{\vee} \, (m \in \N)$
in $\MIC(\wh{X}^m/S^m(1))$, and by the argument explained after 
Proposition \ref{basechange}, we obtain 
the isomorphisms  
$$ \delta^m_n: {p_2^m}_{\dR}^* R^1f_{\dR *}(\ol{W}_{\mkern-4 mu n}^{\vee}) \os{\cong}{\lra} 
 {p_1^m}_{\dR}^* R^1f_{\dR *}(\ol{W}_{\mkern-4 mu n}^{\vee}) \quad (m \in \N) $$
 in $\MIC(S^m(1)/S^m(1)) \, (m \in \N)$.  
The object $(R^1f_{\dR *}(\ol{W}_{\mkern-4 mu n}^{\vee}), \{ \delta^m_n \})$, 
regarded as an object in $\Strat(S/k) \cong \MIC(S/k)$, is nothing 
but the object endowed with the Gauss--Manin connection (which we have denoted by 
$R^1f_{\dR *}(W_n^{\vee})$).  If we pull them back to 
the categories 
$$ 
\MIC(X^m_1/S^m(1)) \os{\cong}{\lla} \Crys((X/S^m(1))_{\inf}) 
\os{\cong}{\lra} \MIC(X^m_2/S^m(1)) \quad (m \in \N),   
$$
we obtain the isomorphisms 
\begin{align*}
\wt{\delta}_n^m: {q_2^m}_{\dR}^* f_{\dR}^*R^1f_{\dR *}(\ol{W}_{\mkern-4 mu n}^{\vee}) 
& =  {f_2^m}_{\dR}^* {p_2^m}_{\dR}^*R^1f_{\dR *}(\ol{W}_{\mkern-4 mu n}^{\vee}) \\ 
& \hspace{-2cm} \os{\cong}{\lra} 
{f_1^m}_{\dR}^* {p_1^m}_{\dR}^*R^1f_{\dR *}(\ol{W}_{\mkern-4 mu n}^{\vee}) 
=  {q_1^m}_{\dR}^* f_{\dR}^*R^1f_{\dR *}(\ol{W}_{\mkern-4 mu n}^{\vee}) 
\quad (m \in \N).  
\end{align*}
Then the object $ (f_{\dR}^*R^1f_{\dR *}(W_n^{\vee}),\{\wt{\delta}_n^m\})$, 
regarded as an object in $\Strat\Crys(X/k) \cong \MIC(X/k)$, is 
$f_{\dR}^*R^1f_{\dR *}(W_n^{\vee})$. 

When $S$ is affine, 
the extension group for the exact sequence of the form 
${q_j^m}_{\dR}^*$\eqref{wn-ext} is calculated as 
the group of global sections of 
\begin{align*}
& R^1{f_j^m}_{\dR *}({q_j^m}_{\dR}^*(\ol{W}_{\mkern-4 mu n}^{\vee}) \otimes 
{q_j^m}_{\dR}^*f_{\dR}^*R^1f_{\dR *}(\ol{W}^{\vee}_n)^{\vee}) \\
= \,\, & 
R^1{f_j^m}_{\dR *}({q_j^m}_{\dR}^*(\ol{W}_{\mkern-4 mu n}^{\vee}) \otimes 
{f_j^m}_{\dR}^*{p_j^m}_{\dR}^*R^1f_{\dR *}(\ol{W}^{\vee}_n)^{\vee}) \\
\cong \,\, & 
{p_j^m}_{\dR}^*R^1f_{\dR *}(\ol{W}^{\vee}_n) \otimes 
{p_j^m}_{\dR}^*R^1f_{\dR *}(\ol{W}^{\vee}_n)^{\vee}, 
\end{align*}
where the last isomorphism follows from the projection formula and 
the base change isomorphism. If we identify 
${q_j^m}_{\dR}^*\ol{W}_{\mkern-4 mu n} \, (j=1,2)$ via $\epsilon_n^m$ and 
${q_j^m}_{\dR}^*f_{\dR}^*R^1f_{\dR *} (\ol{W}_{\mkern-4 mu n}^{\vee})^{\vee} \,(j=1,2)$ 
via the dual of $\wt{\delta}_n^m$, the sheaves 
${p_j^m}_{\dR}^*R^1f_{\dR *}(\ol{W}^{\vee}) \otimes 
{p_j^m}_{\dR}^*R^1f_{\dR *}(\ol{W}^{\vee})^{\vee} \, (j=1,2)$
are identified via $\delta_n^m \otimes {\delta_n^m}^{\vee}$. 
(Note that $\delta_n^m$ is the isomorphism induced by $\epsilon_n^m$, and 
$\wt{\delta}_n^m$ is just the pullback of $\delta_n^m$.) 
Therefore, the identification $\delta_n^m \otimes {\delta_n^m}^{\vee}$
induces the isomorphism of extension groups 
\begin{align*}
& \End({p_2^m}_{\dR}^*R^1f_{\dR *}(\ol{W}^{\vee})) = 
\Gamma(S, {p_2^m}_{\dR}^*R^1f_{\dR *}(\ol{W}^{\vee}) \otimes 
{p_2^m}_{\dR}^*R^1f_{\dR *}(\ol{W}^{\vee})^{\vee}) \\ 
\isom \,\, & 
\Gamma(S, {p_1^m}_{\dR}^*R^1f_{\dR *}(\ol{W}^{\vee}) \otimes 
{p_1^m}_{\dR}^*R^1f_{\dR *}(\ol{W}^{\vee})^{\vee})  
= 
\End({p_1^m}_{\dR}^*R^1f_{\dR *}(\ol{W}^{\vee})),  
\end{align*}
which sends the identity map in 
$\End({p_2^m}_{\dR}^*R^1f_{\dR *}(\ol{W}^{\vee}))$
to the 
identity map in $\End({p_1^m}_{\dR}^*R^1f_{\dR *}(\ol{W}^{\vee}))$. 
Hence the extension class of 
the exact sequence ${q_2^m}_{\dR}^*$\eqref{wn-ext} is identified with 
that of ${q_1^m}_{\dR}^*$\eqref{wn-ext} and so there exists 
an isomorphism $\phi_{n+1}^m:  {q_2^m}_{\dR}^*\ol{W}_{n+1} \lra 
 {q_1^m}_{\dR}^*\ol{W}_{n+1}$ which makes the following diagram commutative: 
\begin{equation}\label{wn-ext2}
\begin{CD}
0 @>>> {q_2^m}_{\dR}^*f_{\dR}^*R^1f_{\dR *} (\ol{W}_{\mkern-4 mu n}^{\vee})^{\vee} @>>> {q_2^m}_{\dR}^*\ol{W}_{n+1} @>>>  
{q_2^m}_{\dR}^*\ol{W}_{\mkern-4 mu n} @>>> 0 \\ 
@. @V{(\wt{\delta}_n^m)^{\vee}}VV @V{\phi_{n+1}^m}VV @V{\epsilon_n^m}VV \\ 
0 @>>> {q_1^m}_{\dR}^*f_{\dR}^*R^1f_{\dR *} (\ol{W}_{\mkern-4 mu n}^{\vee})^{\vee} @>>> {q_1^m}_{\dR}^*\ol{W}_{n+1} @>>>  
{q_1^m}_{\dR}^*\ol{W}_{\mkern-4 mu n} @>>> 0.  
\end{CD} 
\end{equation}
Moreover, by using the property (W1) for $\ol{W}_{\mkern-4 mu n}$, we can modify the isomorphism 
$\phi_{n+1}^m$ (by some morphism ${q_2^m}_{\dR}^*\ol{W}_{\mkern-4 mu n} \lra 
{q_1^m}_{\dR}^*f_{\dR}^*R^1f_{\dR *} (\ol{W}_{\mkern-4 mu n}^{\vee})^{\vee}$)
to the unique isomorphism 
such that ${\iota_2^m}_{\dR}^*(\phi_{n+1}^m) \circ {q_2^m}^*(\ol{e}_{n+1}) = 
 {q_1^m}^*(\ol{e}_{n+1})$. The uniqueness allows us to define 
the isomorphism $\phi_{n+1}^m$ globally in the non affine case too, and 
by the property (W1) for $\ol{W}_{n+1}$, we see that 
the map $\phi_{n+1}^m$ is equal to $\epsilon_{n+1}^m$. Hence the diagram  
\eqref{wn-ext2} tells us that $W_{n+1}$ fits in the exact sequence 
as in \eqref{wn-ext0}. 
Hence $W_{n+1}$ belongs to $\NfMICn(X/k)$ and we have the 
property (W2) for $W_{n+1}$, as required. 
\end{proof}

\begin{rem}\label{rem:sect2}
By  the base change property in (W1), for any object $E$ in $\NfsMIC(X_s/s)$ and 
any morphism $v: k \lra x_{\dR}^* E$ in $\MIC(s/s) = \Vector_k$, there exists a unique 
morphism $\varphi_{v}: W_n|_{X_s} \lra E$ in $\NfsMIC(X_s/s)$ with 
$x_{\dR}^*(\varphi_{v}) \circ (e_n|_{s}) = v$ for some $n$. 
By considering maps $v_1, ..., v_N: k \lra x_{\dR}^* E$ 
whose direct sum $k^{\oplus N} \lra x_{\dR}^*E$ is surjective, 
we obtain a surjective map 
$\oplus_{i=1}^N \varphi_{v_i}: W_n^{\oplus N}|_{X_s} \lra E$ for some $n$. 
This is the property we used in the proof of Proposition \ref{fiberprop} and so 
its proof is now finished. 
\end{rem}

\begin{rem}\label{rem:wn-compati}
The definition of $\{(W_n, e_n)\}_{n \geq 1}$ is functorial 
in the following sense: 
If we are given a commutative diagram 
\begin{equation}\label{eq:wn-compati}
\begin{CD}
S @>{\iota}>> X @>f>> S @>g>> \Spec k \\ 
@A{\varphi_S}AA @A{\varphi_X}AA @A{\varphi_S}AA @A{\varphi_k}AA \\ 
S' @>{\iota'}>> X' @>{f'}>> S' @>{g'}>> \Spec k' 
\end{CD}
\end{equation}
such that $f', g'$ and $\iota'$ also satisfy the conditions in 
Notation \ref{notation} and if we denote the objects 
$\{(W_n, e_n)\}_{n \geq 1}$ for $f',g',\iota'$ by 
$\{(W'_n, e'_n)\}_{n \geq 1}$, we have the canonical morphism  in $\NfMICn(X'/k)$
\begin{equation}\label{eq:wn-compati22}
\{(W'_n, e'_n)\}_{n \geq 1} \lra \{(\varphi_{X, \dR}^*W_n, \varphi_{S, \dR}^*e_n)\}_{n \geq 1}. 
\end{equation}
In fact, the underlying morphism 
\begin{equation}\label{eq:wn-compati3}
\{(\ol{W}'_n, \ol{e}'_n)\}_{n \geq 1} \lra 
\{(\varphi_{X, \dR}^*\ol{W}_{\mkern-4 mu n}, \varphi_{S, \dR}^*\ol{e}_n)\}_{n \geq 1}. 
\end{equation}
of 
\eqref{eq:wn-compati22} in $\NfMIC(X/S)$ is defined by the 
property (W1) and it is upgraded to a morphism 
\eqref{eq:wn-compati22} in $\NfMICn(X/k)$,
because the construction of $\{(W'_n, e'_n)\}_{n \geq 1}$ as 
a pointed projective system in $\NfMICn(X/k)$ is done again via  
the universal property (W1). 

Also, the definition of $\{(W_n, e_n)\}_{n \geq 1}$ is 
compatible with base change  
in the following sense: if we are given the commutative diagram 
\eqref{eq:wn-compati} as above and if the middle square of 
\eqref{eq:wn-compati} is Cartesian, then the morphism 
\eqref{eq:wn-compati22} is an isomorphism. 
In fact, in our construction  by induction of $W_{n+1}, W'_{n+1}$ (using the base change property  first and (W1)) we get  
 the isomorphisms 
$\varphi_{S, \dR}^*R^1f_{\dR *}(\ol{W}_{\mkern-4 mu n}^{\vee}) \cong 
R^1f'_{\dR *}(\varphi_{X, \dR}^*\ol{W}_{\mkern-4 mu n}^{\vee}) \cong 
R^1f'_{\dR *}({\ol{W}'}_n^{\vee})
$; moreover   the local definition of $\ol{W}_{n+1}, \ol{W}'_{n+1}$ as extension classes 
is compatible via $\varphi_{S, \dR}^*$. 
So we can define the isomorphism 
\eqref{eq:wn-compati3} locally, and this local definition glues by 
the property (W1).  Then the isomorphism 
\eqref{eq:wn-compati3}  can be upgraded to the isomorphism 
\eqref{eq:wn-compati22}  in  $\NfMICn(X/k)$ as before. 

Also, 
if we are given the commutative diagram 
\eqref{eq:wn-compati} as above 
(with the middle square not necessarily Cartesian)
with $\varphi_S = \id_S$ such that the induced morphism 
$R^if_{\dR *}(\cO_X, d) \lra R^if'_{\dR *}(\cO_{X'},d)$
is an isomorphism for $i=0,1$ and injective for $i=2$, 
 then the morphism \eqref{eq:wn-compati22} is an isomorphism. 
In fact, we see that 
$R^1f_{\dR *}(\ol{W}_{\mkern-4 mu n}^{\vee}) \cong  
R^1f'_{\dR *}\varphi^*_{X \dR}(\ol{W}_{\mkern-4 mu n}^{\vee}) \cong 
R^1f'_{\dR *}({\ol{W}'}_n^{\vee})$ in this case, 
and the local definitions of $\ol{W}_{n+1}, \ol{W}'_{n+1}$ as extension classes  
are compatible. 
So we can define the isomorphism 
\eqref{eq:wn-compati3} locally, which glues and which is upgraded 
to the isomorphism 
\eqref{eq:wn-compati22} as before. 
\end{rem}

\begin{rem}\label{rem:wn-strongw1}
The property (W1) implies that, for any $\ol{E} \in \NfMIC(X/S)$, 
the morphism 
\begin{equation}\label{eq:strongw1}
f_{\dR *}{\cal H}om(\ol{W},\ol{E}) \lra \iota_{\dR}^*\ol{E}; \quad 
\varphi \mapsto \iota_{\dR}^*(\varphi)(\ol{e})
\end{equation}
(where $\ol{e} := \{\ol{e}_n\}_n$) 
is an isomorphism in $\MIC(S/S)$, where ${\cal H}om(-,-)$ denotes the internal hom object 
in the ind-category. This is true also after a base change by any morphism 
$S' \lra S$. Hence we obtain by considering stratification that, for any 
$E \in \NfMICn(X/k)$, 
the morphism 
\begin{equation}\label{eq:strongw1-2}
f_{\dR *}{\cal H}om(W,E) \lra \iota_{\dR}^*E; \quad 
\varphi \mapsto \iota_{\dR}^*(\varphi)(e)
\end{equation}
is an isomorphism in $\MICn(S/k)$. 
This property characterizes $(W,e)$. 
\end{rem}

\begin{rem}\label{rem:wildeshaus}
Assume that the log structures on $X$ and $S$ are trivial and $S$ is smooth 
over $k$. In \cite[I Theorem 3.5]{w}, Wildeshaus constructed 
the variation of Hodge structure on $X$ called the generic pro-sheaf 
which satisfies the property analogous to the isomorphism 
\eqref{eq:strongw1-2}. Also, he remarked in \cite[I Remark in p.61]{w} that 
its underlying module with integrable connection
satisfies the isomorphism \eqref{eq:strongw1-2}. 
So the underlying module with integrable connection
of his generic pro-sheaf is the same as $W$ in this paper. 
Note that he starts the argument in the case $k=\C$ and uses the 
homotopy exact sequence 
$$ 1 \lra \pi_1^{\rm top}(X_{\an, s_{\an}}, x_{\an}) \lra \pi_1^{\rm top}(X_{\an},x_{\an}) \lra 
\pi_1^{\rm top}(S_{\an},s_{\an}) \lra 1 $$
of topological fundamental groups in the construction of 
generic pro-sheaf (see \cite[p.58]{w}). So his 
construction is not purely algebraic. 
\end{rem}

\begin{rem}\label{rem:e'}
By the property (W0), we have the isomorphism 
$$ \cO_S = f_{\dR *}(W_1^{\vee}) \os{\cong}{\lra} \cdots \os{\cong}{\lra} f_{\dR *}(W_{n}^{\vee}) $$
in $\MICn(S/k)$, which we denote by $e'_n$. 
We see that the composite 
\begin{equation}\label{ene'n}
\cO_S \os{e'_n}{\lra} 
f_{\dR *}(W_n^{\vee}) = \iota_{\dR}^* f_{\dR}^*f_{\dR *} (W_n^{\vee})  
\lra \iota_{\dR}^*(W_n^{\vee}) \os{e_n^{\vee}}{\lra} \cO_S
\end{equation}
is an isomorphism. 
\end{rem}

\begin{rem}\label{rem:wn-lazda}
Assume that the log structures on $X$ and $S$ are trivial and $S$ is smooth 
over $k$. In \cite{laz}, Lazda gives a different, purely algebraic construction of 
the projective system $\{(W_n, e_n)\}_{n \geq 1}$ in $\NfMICn(X/k)$, which 
we explain here. 

We denote by $f_{\dR *}^D(-), R^1f_{\dR *}^D(-)$ the relative 
zeroth and first de Rham cohomology endowed with the Gauss--Manin connection 
induced by Katz--Oda filtration. (It is the same as the higher 
direct image in the sense of $D$-modules \cite[Proposition 1.4]{dimca} 
up to shift of index.) 
Lazda constructs the projective system of triples $\{(W_n, e_n, e'_n)\}_{n \geq 1}$ 
with $e'_n$ is a map $\cO_S \lra f_{\dR *}^D(W_{n}^{\vee})$ in $\MICn(S/k)$  and $e_n$ is  as in  Theorem \ref{thm:wn} , such that the  
composite \eqref{ene'n} (with $f_{\dR *}$ replaced by $f_{\dR *}^D$) is an isomorphism. 
When $n=1$, $W_1 := (\cO_X,d)$, $e_1$ is the identity  
$\cO_S \lra \iota_{\dR}^*\cO_X = \cO_S$ and $e'_1$ is also the 
identity $\cO_S \lra f_{\dR *}^D\cO_X = f_{\dR *}\cO_X = \cO_S$. 

He constructs $(W_{n+1}, e_{n+1}, e'_{n+1})$ from 
$(W_n, e_n, e'_n)$ in the following way: he considers
the Leray spectral sequence for $W_n^{\vee} \otimes f_{\dR}^*R^1f_{\dR *}^D
(W_n^{\vee})^{\vee}$  
\begin{align}
0 & \lra 
H^1_{\dR}(S, R^1f_{\dR *}^D (W_n^{\vee})^{\vee}) 
\os{\alpha_1}{\lra} H^1_{\dR}(X, W_n^{\vee} \otimes f_{\dR}^*R^1f_{\dR *}^D(W_n^{\vee})^{\vee}) 
\label{leray} \\ 
& \os{\beta}{\lra} \End(R^1f_{\dR *}^D(W_n^{\vee})) \nonumber \\ 
& \lra H^2_{\dR}(S,R^1f_{\dR *}^D(W_n^{\vee})^{\vee}) 
\os{\alpha_2}{\lra} H^2_{\dR}(X, W_n^{\vee} \otimes f_{\dR}^*R^1f_{\dR *}^D(W_n^{\vee})^{\vee}).  
\nonumber  
\end{align}
For $i=1,2$, we have a section 
\begin{align*}
\gamma_i: 
H^i_{\dR}(X, W_n^{\vee} \otimes f_{\dR}^*R^1f_{\dR *}^D(W_n^{\vee})^{\vee})  
& \lra 
H^i_{\dR}(S, \iota_{\dR}^*W_n^{\vee} \otimes R^1f_{\dR *}^D(W_n^{\vee})^{\vee}) \\ 
& \lra 
H^i_{\dR}(S, R^1f_{\dR *}^D(W_n^{\vee})^{\vee})
\end{align*}
of $\alpha_i$ induced by $e_n^{\vee}$. Thus we see that there exists a unique element $\e$ in 
$H^1_{\dR}(X, W_n^{\vee} \otimes f_{\dR}^*R^1f_{\dR *}^D(W_n^{\vee})^{\vee})$ 
with $\beta(\e) = \id \in \End(R^1f_{\dR *}^D(W_n^{\vee}))$ 
and $\gamma_1(\e) = 0$. Then $W_{n+1}$ is defined to be the object 
in the extension 
\begin{equation}\label{wn-extD}
0 \lra f_{\dR}^*R^1f_{\dR *}^D (W_n^{\vee})^{\vee} \lra W_{n+1} \lra 
W_n \lra 0 
\end{equation}
whose extension class is $\e$. The condition $\beta(\e) = \id$ implies that 
the dual of the exact sequence \eqref{wn-extD} induces the 
isomorphism $f_{\dR *}^D(W_n^{\vee}) \os{\cong}{\lra} f_{\dR *}^D(W_{n+1}^{\vee})$ as we proved in the proof of Theorem \ref{thm:wn}, thus we define $ e'_{n+1}$ as the composition (isomorphism) 
$$ e'_{n+1}: \cO_S \os{e'_n}{\lra} f_{\dR *}^D(W_n^{\vee}) \os{\cong}{\lra} f_{\dR *}^D(W_{n+1}^{\vee}). $$
The condition $\gamma_1(\e) = 0$ implies that the dual of $\iota_{\dR}^*$\eqref{wn-extD} splits 
after we push it out by $e_n^{\vee}$. Thus we have a morphism 
$\iota_{\dR}^*W_{n+1}^{\vee} \lra \cO_S$ which extends $e_n^{\vee}$. 
We define $e_{n+1}$ to be the dual of this map. By construction, 
the property (W0) is satisfied.  

The projective system $\{(W_n, e_n)\}_{n \geq 1}$ Lazda constructed satisfies 
the property (W1) (the proof is the same as that in our case), and satisfies 
the exact sequence \eqref{wn-extD}. But it does not immediately imply the 
isomorphism \eqref{eq:strongw1-2} because the structure of 
$W_n$ as an object in $\NfMICn(X/k) = \NfStrCrysn(X/k)$ is not induced by 
stratification. 

Now we compare our definition of the projective system $\{(W_n, e_n)\}_{n \geq 1}$
with that by Ladza (which we denote by $\{(W_n^{\rm L}, e_n^{\rm L})\}_{n \geq 1}$ 
in the sequel). First, by  \eqref{eq:strongw1-2}, we have a unique projective system 
of morphisms
$\{\varphi_n: (W_n, e_n) \lra (W_n^{\rm L}, e_n^{\rm L})\}_{n \geq 1}$ in 
$\NfMICn(X/k)$. Because both $(W_n, e_n)|_{X_s}$, $(W_n^{\rm L}, e_n^{\rm L})|_{X_s}$ 
satisfies the property (W1) (after the pullback by $s \lra S$), the map 
$\varphi_n|_{X_s}$ is an isomorphism. Hence $\varphi_n$ is an isomorphism and so 
we have shown that our definition and Lazda's definition are the same. 

Moreover, we are going to prove that 
the isomorphism 
 $\ol{\varphi}'_n: R^1f_{\dR *} (\ol{W}_{\mkern-4 mu n}^{\vee})^{\vee} \os{\cong}{\lra} 
R^1f_{\dR *} ({\ol{W}_{\mkern-4 mu n}^{\rm L}}^{\vee})^{\vee}$ in $\MIC(S/S)$  
induced by the underlying isomorphism $\ol{\varphi}_n: \ol{W}_{\mkern-4 mu n} \lra \ol{W}^{\rm L}_n$
of $\varphi_n$ in $\NfMIC(X/S)$ can be enriched to 
an isomorphism 
$R^1f_{\dR *}(W_n^{\vee})^{\vee} \os{\cong}{\lra} R^1f_{\dR *}^D({W_n^{\rm L}}^{\vee})^{\vee}$ in 
$\MICn(S/k)$. In fact, consider the following diagram in $\MIC(X/S)$ 
\begin{equation}\label{wn-ext2D}
\begin{CD}
0 @>>> f_{\dR}^*R^1f_{\dR *} (\ol{W}_{\mkern-4 mu n}^{\vee})^{\vee} @>>> \ol{W}_{n+1} @>>>  
\ol{W}_{\mkern-4 mu n} @>>> 0 \\ 
@. @V{f_{\dR}^*(\ol{\varphi}'_n)}VV @. @V{\ol{\varphi}_n}VV \\ 
0 @>>> f_{\dR}^*R^1f_{\dR *} ({\ol{W}_{\mkern-4 mu n}^{\rm L}}^{\vee})^{\vee} @>>> \ol{W}_{n+1}^{\rm L} @>>>  
\ol{W}_{\mkern-4 mu n}^{\rm L} @>>> 0, 
\end{CD} 
\end{equation}
where the vertical arrows are isomorphisms. When $S$ is affine, 
the extension class defined by $\ol{W}_{n+1}$ is equal to 
that by $\ol{W}_{n+1}^{\rm L}$ via the above diagram by construction. 
Thus there exists a map $\phi_{n+1}: \ol{W}_{n+1} \lra \ol{W}^{\rm L}_{n+1}$ which fits into 
the diagram \eqref{wn-ext2D}, and we can adjust it uniquely using 
the property (W1) for $\ol{W}_{\mkern-4 mu n}$
so that $\iota_{\dR}^*(\phi_{n+1}) \circ \ol{e}_{n+1} = 
\ol{e}^{\rm L}_{n+1}$ (as in the last  lines of the proof of Theorem \ref{thm:wn}). So the isomorphism $\phi_{n+1}$ is defined globally and 
by the property (W1), we see that 
$\ol{\varphi}_{n+1} = \phi_{n+1}$. 

The isomorphisms $\varphi_n$, $\varphi_{n+1}$ and the exact sequences \eqref{wn-ext0}, \eqref{wn-extD} induce the isomorphism $R^1f_{\dR *}(W_n^{\vee})^{\vee} \os{\cong}{\lra} R^1f_{\dR *}^D({W_n^{\rm L}}^{\vee})^{\vee}$ 
in $\MICn(S/k)$. 
When we consider its restriction $R^1f_{\dR *} (\ol{W}_{\mkern-4 mu n}^{\vee})^{\vee} \os{\cong}{\lra} 
R^1f_{\dR *} ({\ol{W}_{\mkern-4 mu n}^{\rm L}}^{\vee})^{\vee}$ to $\MIC(S/S)$, 
we see from the diagram \eqref{wn-ext2D} together with the morphism 
$\phi_{n+1}$ and the equality $\ol{\varphi}_{n+1} = \phi_{n+1}$ that 
it coincides with $\ol{\varphi}'_n$. 
Hence we have the enriched structure as we wanted.
%

The claim we proved above can be rephrased to the fact that 
the two definitions of Gauss--Manin connections on 
$R^1f_{\dR *}(W_n^{\vee})$ are the same. 
\end{rem}

\begin{rem}\label{rem:coin2}
In Remark \ref{rem:wn-lazda}, we proved the coincidence of  the
two definitions of Gauss--Manin connections on 
$R^1f_{\dR *}(W_n^{\vee})$ when the log structures on $X$ and $S$ are 
trivial and $S$ is smooth over $k$. In this remark, we extend 
this result to some other cases. 

(1) \, 
Let the situation be as in Proposition \ref{prop:abcde-rs} and 
assume that we can always take $s \leq r$ in the condition (2) there. 
Assume moreover that $f^{-1}(S_{\rm triv}) = X_{\rm triv}$, where 
$S_{\triv}, X_{\triv}$ is the locus in $S, X$ on which the log structure is 
trivial. In this case, the coincidence holds
on $S_{\rm triv}$ by 
Remark \ref{rem:wn-lazda} and this implies 
the coincidence on $S$, because the canonical map $E \lra j_*j^*E$ is injective 
if $E$ is a locally free sheaf on $S_{\rm triv}$ and $j$ is the open immersion 
$S_{\rm triv} \hra S$. 

(2) \, Let the situation be as in Remark \ref{rem:gmcoin}(3). 
In this case, the coincidence holds on $\ol{M}_g$ by (1).  
This implies the coincidence on $\cS$ and on $S$ by 
Remark \ref{rem:wn-compati}, the base change property and 
the log blow-up invariance of Katz--Oda spectral sequence 
on the diagram in \eqref{eq:moduli}. 

\end{rem}

\begin{rem}\label{AIK1}
When we are given a projective system $\{(\ol{W}_{\mkern-4 mu n},\ol{e}_n)\}_{n}$ 
of objects $\ol{W}_{\mkern-4 mu n} \, (n \geq 1)$ in $\NfMIC(X/S)$ and morphisms 
$\ol{e}_n: \cO_S \lra \iota_{\dR}^*\ol{W}_{\mkern-4 mu n} \,(n \geq 1)$ in $\MIC(S/S)$ 
with the property (W1) and the exact sequences \eqref{wn-ext} (for $n \geq 1$), 
there is at most one way to upgrade it to a projective system 
$\{(W_n,e_n)\}_{n}$ of objects $W_n \, (n \geq 1)$ in $\NfMICn(X/k)$ and morphisms 
$e_n: \cO_S \lra \iota_{\dR}^*W_n \,(n \geq 1)$ in $\MICn(S/k)$ 
which satisfies one of the following conditions: 
\begin{enumerate}
\item For any $n \geq 1$, 
the exact sequence \eqref{wn-ext} in $\MIC(X/S)$ can be 
upgraded to an exact sequence \eqref{wn-ext0} 
in $\NfMICn(X/k)$, 
where we endow $R^1f_{\dR *}(W_n^{\vee})^{\vee}$ with 
the Gauss--Manin connection (in our sense). 
\item The same as (1) holds when 
we endow $R^1f_{\dR *}(W_n^{\vee})^{\vee}$ with 
the Gauss--Manin connection induced by Katz--Oda filtration. 
\end{enumerate}

The proof is essentially the same as \cite[Prop 6.1(2)]{aik}, but we give 
a detailed argument here. The uniqueness of $e_n \, (n \in \N)$ is clear, because 
$e_n$ is uniquely determined by $\ol{e}_n$. 
Assume we have the required uniqueness up to $W_n$ 
and denote by $\nabla_1, \nabla_2: \ol{W}_{n+1} \lra 
\ol{W}_{n+1} \otimes \Omega^1_{X/k}$ two integrable connections 
which upgrade the one on $\ol{W}_{n+1}$. 
Then their difference $\nabla_1 - \nabla_2$ induces a morphism 
$\phi: \ol{W}_{n} \lra f_{\dR}^*R^1f_{\dR *}(\ol{W}_{\mkern-4 mu n}^{\vee})^{\vee} \otimes_{\cO_S} 
\Omega^1_{S/k}$ as $\cO_S$-modules. 

We check that $\phi$ is a morphism in $\MIC(X/S)$ by local computation. 
Let $\dlog x_1, \allowbreak \dots, \dlog x_{r} \, (x_i \in \cM_S)$ be a basis of 
$\Omega^1_{S/k}$ and let $\dlog x_1, \dots, \dlog x_{s} \, (x_i \in \cM_X)$ be 
a basis of $\Omega^1_{X/k}$ extending it. We denote the image of 
$\dlog x_{r+1}, \dots, \dlog x_s$ in $\Omega^1_{X/S}$ by 
$\ol{\dlog} x_{r+1}, \dots, \ol{\dlog} x_s$, respectively. 
We define $\nabla_1^{(i)}, \nabla_2^{(i)}$ by 
$$ \nabla_l(m) = \sum_{i=1}^s\nabla_l^{(i)}(m) \dlog x_i \quad (l=1,2, \,\, m \in \ol{W}_{n+1}). $$
Note that $\nabla_1^{(i)} = \nabla_2^{(i)}$ for $r+1 \leq i \leq s$ because both 
$\nabla_1, \nabla_2$ upgrades the connection on $\ol{W}_{n+1}$.
Let $\nabla', \nabla''$ be the connection of $\ol{W}_{n}, 
f_{\dR}^*R^1f_{\dR *}(\ol{W}_{\mkern-4 mu n}^{\vee})^{\vee} \otimes_{\cO_S} 
\Omega^1_{S/k}$, respectively. 
For an element $m \in \ol{W}_{n}$, if we denote its local lift to $\ol{W}_{n+1}$ by $\wt{m}$, 
$\phi(m)$ is equal to $\sum_{i=1}^r (\nabla_1^{(i)} - \nabla_2^{(i)})(\wt{m}) \dlog x_i$ 
and so 
$$ (\nabla'' \circ \phi)(m) = \sum_{j=r+1}^s 
\sum_{i=1}^r \nabla_1^{(j)} \circ (\nabla_1^{(i)} - \nabla_2^{(i)})(\wt{m}) \dlog x_i \otimes 
\ol{\dlog} x_j. $$
On the other hand, $\sum_{j=r+1}^s \nabla_1^{(j)}(\wt{m}) \ol{\dlog} x_j$ 
is a lift of $\nabla'(m)$ to $\ol{W}_{n+1} \otimes \Omega^1_{X/S}$ and so 
$$ (\phi \circ \nabla')(m) = \sum_{i=1}^r \sum_{j=r+1}^s 
 (\nabla_1^{(i)} - \nabla_2^{(i)}) \circ \nabla_1^{(j)}(\wt{m}) \dlog x_i \otimes \ol{\dlog} x_j. $$
Because 
\begin{align*}
& \nabla_1^{(j)} \circ (\nabla_1^{(i)} - \nabla_2^{(i)}) 
=  \nabla_1^{(j)} \circ \nabla_1^{(i)} - \nabla_1^{(j)} \circ \nabla_2^{(i)} 
= \nabla_1^{(j)} \circ \nabla_1^{(i)} - \nabla_2^{(j)} \circ \nabla_2^{(i)} \\ 
= \,\, & \nabla_1^{(i)} \circ \nabla_1^{(j)} - \nabla_2^{(i)} \circ \nabla_2^{(j)}  
= \nabla_1^{(i)} \circ \nabla_1^{(j)} - \nabla_2^{(i)} \circ \nabla_1^{(j)}
= (\nabla_1^{(i)} - \nabla_2^{(i)}) \circ \nabla_1^{(j)}
\end{align*}
by the integrability of $\nabla_1, \nabla_2$ and the equality 
$\nabla_1^{(j)} = \nabla_2^{(j)} \, (r+1 \leq j \leq s)$, we see the equality 
$\nabla'' \circ \phi = \phi \circ \nabla'$. So $\phi$ is a morphism in 
$\MIC(X/S)$, as required. 

Since $\phi$ is a morphism in $\MIC(X/S)$, 
and by the property (W1), 
it is determined by the morphism 
$\iota_{\dR}^*(\nabla_1 - \nabla_2) \circ \ol{e}_{n}$. 
It is zero because it is equal to 
$\iota_{\dR}^*(\nabla_1) \circ \ol{e}_{n+1} - 
\iota_{\dR}^*(\nabla_2) \circ \ol{e}_{n+1}$ and 
the section $\ol{e}_{n+1}(1) = e_{n+1}(1)$ is horizontal 
with respect to both $\iota_{\dR}^*(\nabla_1)$ and $\iota_{\dR}^*(\nabla_2)$. 
Hence $\nabla_1 = \nabla_2$, as required. 
\end{rem}

\begin{rem}\label{AIK2}
In this remark, we compare our definition of 
the projective systems 
$\{(\ol{W}_{\mkern-4 mu n},\ol{e}_n)\}_n$, $\{(W_n,e_n)\}_n$ and that 
of Andreatta--Iovita--Kim \cite{aik}. 

Let the situation be as in Remark \ref{rem:gmcoin}(3) 
and we work on $X/S$ or on $\cX/\cS$. 
(To simplify the description, we  only deal  with
$X/S$: the same argument works also for $\cX/\cS$.)
We put $H^i := R^if_{\dR *}(\cO_X,d) \in \MIC(S/S)$ and 
define 
$$R^n, \quad \gamma_n: R^n \otimes H^1 \lra R^{n-1} \otimes H^2 \,(n \geq 1), 
\quad i_n: R^n \hra R^{n-1} \otimes H^1 \,(n \geq 2)$$ inductively in 
the following way: First, put 
$R^0 := \cO_S$, 
$R^1 := H^1$ and let $\gamma_1: R^1 \otimes H^1 = 
H^1 \otimes H^1 \os{\cup}{\lra} H^2$ be the 
cup product. If we have defined $R^n$ and $\gamma_n$, let 
$R^{n+1} := \Ker \gamma_n$, let $i_{n+1}$ be the 
canonical inclusion $R^{n+1} \hra R^n \otimes H^1$ and let 
$\gamma_{n+1}$ be the composite 
$$ R^{n+1} \otimes H^1 \os{i_{n+1} \otimes \id}{\lra} 
R^n \otimes H^1 \otimes H^1 \os{\id \otimes \cup}{\lra} R^n \otimes H^2. $$

Andreatta--Iovita--Kim gave a definition of the projective system 
$\{(\ol{W}_{\mkern-4 mu n}, \ol{e}_n)\}_n$ with exact sequences 
$$ 
(*)_n \quad 
0 \lra \ol{W}_{\mkern-4 mu n}^{\vee} \lra \ol{W}_{n+1}^{\vee} \lra f_{\dR}^*R^n \lra 0 $$
 for $n \geq 1$ (where the map $\ol{W}_{\mkern-4 mu n}^{\vee} \lra \ol{W}_{n+1}^{\vee}$ is the dual of 
the transition map $\ol{W}_{n+1} \lra \ol{W}_{\mkern-4 mu n}$) such that 
the connecting map $a_n: R^n \lra R^1f_{\dR *}(\ol{W}_{\mkern-4 mu n}^{\vee})$ associated 
to $(*)_n$ is an isomorphism and that the map 
$b_n: R^1f_{\dR *}(\ol{W}_{n+1}^{\vee}) \lra R^n \otimes H^1$ 
on the first cohomology groups associated 
to $(*)_n$ satisfies the equality $b_n \circ a_{n+1} = i_{n+1}$. 
Their method is the following: 
first, they put $(\ol{W}_1, \ol{e}_1) := (\cO_X, \cO_S \os{=}{\lra} \iota_{\dR}^*\cO_X)$
and define $a_1$ to be the identity map on $R^1 =  R^1f_{\dR *}\cO_X$. 
Next, once defined $(\ol{W}_{i}, \ol{e}_{i})$ and the isomorphisms $a_i$ 
for $1 \leq i \leq n$, they introduce $(\ol{W}_{n+1}, \ol{e}_{n+1})$ 
fitting in the exact sequence $(*)_n$ by identifying $(*)_n$ with 
the dual of \eqref{wn-ext} via $a_n$ and arguing 
in the same way as our definition, meaning that the extension corresponds to the identity as in \eqref{wn-ext}. Then $a_n$ becomes the connecting map 
of $(*)_n$. 
Then they consider the diagram 
$$ 
\xymatrix{ 
R^1f_{\dR *}(\ol{W}_{n+1}^{\vee}) \ar[r]^-{b_n}
& 
R^n \otimes H^1 \ar[r] \ar[d]^-{a_n \otimes \id}
& 
R^2f_{\dR *}(\ol{W}_{\mkern-4 mu n}^{\vee}) \ar[dd]
\\ 
& 
R^1f_{\dR *}(\ol{W}_{\mkern-4 mu n}^{\vee}) \otimes H^1 
\ar[d]^-{b_{n-1} \otimes \id} \\
& 
R^{n-1} \otimes H^1 \otimes H^1 
\ar[r]^-{\cup} & 
R^{n-1} \otimes H^2, 
}
$$
where the top horizontal line is a part of the long exact sequence 
associated to $(*)_n$,  
the long vertical arrow comes from $(*)_{n-1}$ and 
the lower horizontal arrow is the cup product. 
Then they prove that the rectangle is anti-commutative. 
Since we have $b_{n-1} \circ a_{n} = i_{n}$, then $ \cup\circ(b_{n-1}\otimes \mathrm{id})\circ (a_{n}\otimes \mathrm{id})= \gamma_n$; moreover the long vertical arrow of the above diagram composed with the top horizontal arrow is the zero map, because they come form a cohomological long exact sequence. Hence the image of $b_n$ should be in the kernel of $\gamma_n$ which is $R^{n+1}$; thus we obtain a map $R^1f_{\dR *}(\ol{W}_{n+1}^{\vee}) \lra R^{n+1}$. They prove that this map is an isomorphism and denote its inverse by $a_{n+1}$, and so the induction works.

So their construction is the same as ours except that 
they build isomorphisms $a_n \, (n \geq 1)$ in the   process of construction.
We expect the existence of the isomorphisms $a_n \, (n \geq 1)$
in more general case than in the framework of Remark \ref{rem:gmcoin}(3),  but we will not pursue this topic in this paper. 

After giving the definition of the projective system 
$\{(\ol{W}_{\mkern-4 mu n}, \ol{e}_n)\}_n$, Andreatta--Iovita--Kim upgrade 
it to the projective system $\{(W_n, e_n)\}_n$ satisfying 
the condition (2) in Remark \ref{AIK1}, by using 
$p$-adic Hodge theory. Because our projective system 
$\{(W_n, e_n)\}_n$ satisfies the condition (1) in Remark \ref{AIK1} 
and the conditions (1) and (2) are equivalent in the hypothesis of 
this remark by Remark \ref{rem:coin2}(2), we see that 
the projective system 
$\{(W_n, e_n)\}_n$ of Andreatta--Iovita--Kim is the same as ours. 
\end{rem}

We give the second definition of relatively unipotent de Rham 
fundamental group $\pi_1(X/S,\iota)$ by using $W = \{(W_n,e_n)\}_n$, 
following the method of \cite{ha} and \cite{aik}. 
By \eqref{eq:strongw1-2}, we have the isomorphism 
\begin{equation}\label{eq:strongw1-3}
f_{\dR *}{\cal H}om(W,W) \os{\cong}{\lra} \iota_{\dR}^*W, 
\end{equation}
and so $\iota_{\dR}^*W$ has a structure of a ring. 
Moreover, if we denote the projective system 
$\{W_m \otimes W_n\}_{m,n}$ by 
$W \wh{\otimes} W$, we have the isomorphism 
\begin{equation}\label{eq:strongw1-4}
f_{\dR *}{\cal H}om(W,W \wh{\otimes} W) \os{\cong}{\lra} 
\iota_{\dR}^*(W \wh{\otimes} W) 
\end{equation}
again by \eqref{eq:strongw1-2}, and the morphisms 
$$ e_m \otimes e_n: \cO_S \lra \iota_{\dR}^*(W_m \otimes W_n) $$ 
in $\MICn(S/k)$ induce the morphism 
$W \lra W \wh{\otimes} W$, thus the coproduct structure 
$\iota_{\dR}^* W \lra \iota_{\dR}^* W \wh{\otimes} \iota_{\dR}^* W$ 
on $\iota_{\dR}^*W$. 
We see that, with these structures, 
$(\iota_{\dR}^*W)^{\vee} := \varinjlim_{n} (\iota_{\dR}^*W_n)^{\vee}$ forms 
a commutative Hopf algebra object in the ind-category of $\MICn(S/k)$. 
So we can give the following definition: 

\begin{defn}[{\bf Second definition of relatively unipotent $\pi_1$}]\label{def2}
Let the notations be as above. We define the relatively unipotent de Rham fundamental 
group $\pi^{\dR}_1(X/S,\iota)$ by 
$$\pi^{\dR}_1(X/S,\iota) := \Spec (\iota_{\dR}^*W)^{\vee}, $$ 
which is an affine group scheme in $\MICn(S/k)$. 
\end{defn}

Then we have the following comparison theorem on 
two definitions of relatively unipotent de Rham fundamental 
groups: 

\begin{thm}\label{thm:1and2}
The first and the second definitions of $\pi_1^{\dR}(X/S,\iota)$ are 
canonically isomorphic. 
\end{thm}

\begin{proof}
In this proof, we denote  
 the second definition of $\pi_1^{\dR}(X/S,\iota)$ by 
$\pi_1^{\dR}(X/S,\iota)'$. 
For an object $E$ in $\NfMICn(X/k)$, we have the canonical map 
$$ f_{\dR *}{\cal H}om(W,W) \otimes f_{\dR *}{\cal H}om(W,E^{\vee}) 
\lra  f_{\dR *}{\cal H}om(W,E^{\vee}) $$ 
and this corresponds to the map 
$$ \iota_{\dR}^*W \otimes \iota_{\dR}^*E^{\vee} \lra \iota_{\dR}^*E^{\vee} $$
by \eqref{eq:strongw1-2}. This map defines the representation of 
$\pi_1^{\dR}(X/S, \iota)'$ on $\iota_{\dR}^*E$ in $\MICn(S/k)$. 
Hence we have defined the functor 
$$ \NfMICn(X/k) \lra 
\Rep_{\MICn(S/k)}(\pi_1^{\dR}(X/S, \iota)'). $$
Since we have $\NfMICn(X/k) = \Rep_{\MICn(S/k)}(\pi_1^{\dR}(X/S, \iota))$ 
by Proposition \ref{reltannaka}, we obtain the morphism 
$$ \pi_1^{\dR}(X/S, \iota)' \lra \pi_1^{\dR}(X/S, \iota) $$
of affine group schemes in $\MICn(S/k)$ by the argument after Proposition 
\ref{reltannaka}. To prove that this is an isomorphism, it suffices to 
prove that the restriction of this morphism to $\MIC(s/s)$, 
which is the morphism 
\begin{equation}\label{eq:def1def2-1}
\pi_1^{\dR}(X_s/s, x)' \lra \pi_1^{\dR}(X_s, x), 
\end{equation}
is an isomorphism, because we can check the claim after applying the functor 
$\MICn(S/k) \lra \MIC(s/s)$ which is exact and faithful. 

The proof of the isomorphism \eqref{eq:def1def2-1} is the same as 
that in \cite[Theorem 2.9]{ha}, which we explain here for the convenience of the reader. 
To simplify the notation, 
we denote the category $\NfsMIC(X_s/s)$ by $\cC$, 
the fundamental group $\pi_1^{\dR}(X_s, x)'$ by $G$, 
the base change of $W = \{W_n\}$ to $X_s/s$ by $W^{\cC} = \{W^{\cC}_n\}$ 
and the fiber functor $x_{\dR}^*: \cC \lra \Vector_k$ by 
$E \mapsto E|_x$. Then $G$ is equal to $\Spec (W^{\cC}|_x)^{\vee}$
and the morphism \eqref{eq:def1def2-1} is induced by the functor 
\begin{equation}\label{eq:def1def2-2}
\cC \lra \Rep_k(G)
\end{equation}
defined by $E \mapsto E|_x$ with $E|_x$ endowed with 
a $G$-action. It suffices to prove that 
the functor \eqref{eq:def1def2-2} is an equivalence. 

First we prove the full-faithfulness by checking that the map 
$$ \Hom_{\cC}(E,F) \lra \Hom_G(E|_x, F|_x) $$
induced by \eqref{eq:def1def2-2} is an isomorphism. 
By Remark \ref{rem:sect2}, 
any $E \in \cC$ admits an exact sequence of the form 
	${W_m^{\cC}}^{\oplus s} \lra {W^{\cC}_n}^{\oplus r} \lra E \lra 0 $
for some $n,m,r,s \in \N$, and we can take $n,m$ arbitrarily large. 
Thus it suffices to prove that the map 
$$ \Hom_{\cC}(W^{\cC}_n ,F) \lra \Hom_G(W^{\cC}_n|_x, F|_x) $$
induced by \eqref{eq:def1def2-2} is an isomorphism for 
sufficiently large $n$. This is true because 
$\Hom_{\cC}(W^{\cC}_n ,F)$ is equal to $F|_x$ 
by Remark \ref{rem:sect2} and 
$\Hom_G(W^{\cC}_n|_x, F|_x) = \Hom_G(W^{\cC}|_x, F|_x)$ is 
also equal to $F|_x$. (Note that, since 
$G$ is equal to $\Spec (W^{\cC}|_x)^{\vee}$, giving a 
$G$-equivariant map $W^{\cC}|_x \lra F|_x$ is equivalent to give
an image of the unit element $1 \in W^{\cC}|_x$.) 

Next we prove the essential surjectivity. 
For any $V \in \Rep_k(G)$ and $v \in V$, there exists a unique 
$G$-equivariant map $W^{\cC}|_x \lra V$ with $1 \mapsto v$, and 
this map factors through some $W_n^{\cC}|_x$. Thus we see that 
there exists a $G$-equivariant surjection 
${W^{\cC}}^{\oplus r}_n|_x \lra V$ for some $n,r \in \N$. Repeating this argument, 
we see that there exists a $G$-equivariant exact sequence 
${W^{\cC}}^{\oplus s}_m|_x \os{\alpha}{\lra} {W^{\cC}}_n^{\oplus r}|_x \lra V \lra 0 $
for some $n,m,r,s \in \N$. 
Since the functor \eqref{eq:def1def2-2} is fully faithful, there 
exists a morphism $\wt{\alpha} : {W^{\cC}}^{\oplus s}_m \lra {W^{\cC}}^{\oplus r}_n$ 
which is sent to $\alpha$ by the functor \eqref{eq:def1def2-2}. Then 
$\Coker \,\wt{\alpha}$ is sent to $V$ and so the essential 
surjectivity is proved. 
Hence the functor \eqref{eq:def1def2-2} is an equivalence and the 
proof is finished. 
\end{proof}

\begin{rem} \label{ref.abs}
In this remark  we outline an abstract approach to Definition 
\ref{def2} and Theorem \ref{thm:1and2}: the approach   has been  suggested  to us by 
the referee.  

Let $t: \cC \lra \cD$, $\omega: \cD \lra \cC$ be as in the beginning of 
Section 2, and put $G := G(\cD,\omega)$ as defined in Section 2, so that 
$\cD \cong \Rep_{\cC}G$ by Proposition \ref{reltannaka}. 
We define a pro-object $\widetilde{W}$ in $\cD \cong \Rep_{\cC}G$ and 
a morphism $\widetilde{e}: 1_{\cC} \lra \omega(\widetilde{W})$ of pro-objects 
in $\cC$ in the following way: 
\medskip 

\noindent
(1) \, $\widetilde{W} \in \cD \cong \Rep_{\cC}G$ is defined as the pro-object 
$\cO_G^{\vee}$ in $\cC$ endowed with the $G$-action induced by 
the comultiplication $\cO_G \lra \cO_G \otimes \cO_G$. \\
(2) \, $\widetilde{e}$ is defined as the dual of the counit $\cO_G \lra 1_{\cC}$. 
\medskip 

Let $t_*: \cD \cong \Rep_{\cC}G \lra \cC$ be the functor of taking $G$-invariants, 
which is the right adjoint of the functor $t$. 
Then we can show by direct computation that, for any object $E$ in $\cD$, 
the morphism 
\begin{equation}\label{eq:20200823-1}
t_*{\cal H}om(\widetilde{W}, E) \lra \omega(E); \quad \varphi \mapsto \omega(\varphi)(e) 
\end{equation}
is an isomorphism. 

When we specialize the above abstract setting to the case $\cC = \MICn(S/k), \allowbreak \cD = \NfMICn(X/k), \allowbreak t = f^*_{\dR}$ and $\omega = \iota^*_{\dR}$, 
we have $t_* = f_{\dR *}$   then we see that the    isomorphism   we have found  is the same as  those one 
given in \eqref{eq:strongw1-2} in Remark \ref{rem:wn-strongw1}.  Therefore, 
the pair $(\widetilde{W}, \widetilde{e})$ we defined here, in the  abstract setting, gives the 
pair $(W,e)$ in   the case of  Theorem \ref{thm:wn}. 

If we define the Hopf algebra structure on $\omega(\widetilde{W})^{\vee} = \omega(\cO_G)$ in 
the way described before Definition \ref{def2}, we see again by direct computation that 
the resulting Hopf algebra structure is the same as the original Hopf algebra 
strcuture on $\cO_G$. This shows Theorem \ref{thm:1and2} in  the abstract setting, and 
specializing as before $\cC, \cD$ (and so on,  as in the previous part of this remark), we obtain 
Theorem \ref{thm:1and2}. 
\end{rem}

\section{Relative minimal model of Navarro Aznar}

In this section  we introduce  a third definition of 
relatively unipotent de Rham fundamental group 
by using the theory of relative minimal model 
developed by Navarro Aznar \cite{nagm}. Because our definition is different from his original one, 
we will  give the steps in detail. 
We will state the compatibility with the previous definitions, 
but the proof will be postponed to Section 5 (by introducing a fourth definition of the relatively unipotent de Rham fundamental group).
We  review now the theory  of  commutative differential graded algebras (cdga's), 
following \cite{nagm}. 
Let $R$ be a commutative algebra over a field $k$ of characteristic zero. 

A cdga is a graded algebra $A = \bigoplus_{i=0}^{\infty} A^i$ endowed with 
a differential $d: A^i \lra A^{i+1} \,(i \geq 0)$ with $d \circ d = 0$ 
satisfying 
\begin{align*}
& xy = (-1)^{pq} yx \quad (x \in A^p, y \in A^q), \\ 
& d(xy) = (dx) y + (-1)^p x(dy) \quad (x \in A^p, y \in A^q). 
\end{align*}
The  $ k$-algebra $R$ is regarded as an cdga with $R^0 = R, R^i = 0 \,(i>0), d=0$. 
A morphism of cdga's is a morphism of algebras compatible with 
grading and differential. For a cdga $A$, we denote its $n$-th cohomology group by $H^n(A)$ and for a morphism $\phi: A \longrightarrow B$ of cdga's, we denote by $H^n(A,B)$ the $n$-th cohomology group of the complex $M_{\phi}$ defined by $M^n_{\phi}=A^n\oplus B^{n-1}$ with differential $d: M_{\phi}^n \longrightarrow M_{\phi}^{n+1}$ given by  $d(a,b)= (d_A(a), \phi(a)-d_B(b))$, where $d_A$ (resp. $d_B$) is the differential of $A$ (resp. $B$).


An $R$-cdga is a cdga $A$ endowed with a morphism 
$R \lra A$ of cdga's. An augmented $R$-cdga is an $R$-cdga 
$(A,i:R \lra A)$ endowed further with another morphism 
$e:A \lra R$ of cdga's (called an augmentation) such that 
$e \circ i = \id$. A(n augmented) morphism of (augmented) 
$R$-cdga's is a morphism of $R$-algebras compatible with 
grading, differential (and augmentation).
(In \cite{nagm}, an augmented morphism is called a pointed morphism, and 
the adjective `pointed' is often used for several notions with augmentation. 
In this paper, we will always use the adjective `augmented' 
in place of `pointed'.) 

A Hirsch extension of an (augmented) $R$-cdga is an 
(augmented) inclusion $A \hra A \otimes \bigwedge (E)$ of 
(augmented) $R$-cdga's, where $E$ is a free $R$-module, $\bigwedge (E)$ is the free commutative graded algebra generated by $E$ in degree $1$ and 
the differential of $A \otimes \bigwedge (E)$ is 
induced by a map which sends $E$ to $A^1$. 
(It is called a Hirsch extension of degree $1$ in 
\cite[(2.1)]{nagm}, but we will omit the term `of degree $1$' because 
we will not treat the case of higher degree.) 

An (augmented) $R$-cdga $A$ is called $(1,q)$-minimal 
if there exists a sequence of Hirsch extensions 
\begin{equation}\label{eq:minfil}
R = A(0) \subseteq A(1) \subseteq \cdots \subseteq A(q) = A. 
\end{equation}
An (augmented) $R$-cdga $A$ is called $1$-minimal if 
$A$ is the union of $(1,q)$-minimal (augmented) $R$-cdga's $(q \in \N)$  
with respect to Hirsch extensions as above. 
(As a matter of fact it is indicated as  $(2,0)$-minimal in 
\cite[(2.1)]{nagm}, but we call it $1$-minimal, 
following \cite{gm}.) 

A $(1,q)$-minimal model of a(n augmented) $R$-cdga $A$ is a(n augmented) 
morphism 
$$ \rho_q: M(q) \lra A $$
from a $(1,q)$-minimal (augmented) $R$-cdga $M(q)$, endowed with 
filtration $\{M(q')\}_{q'=0}^q$ as in \eqref{eq:minfil}, such that  
the induced maps 
$H^i(M(q')) \lra H^i(A) \, (i=0,1)$ are isomorphisms 
and the induced map $H^2(M(q'-1),A) \lra H^2(M(q'),A)$ is zero for 
$1 \leq q' \leq q$. 
A $1$-minimal model of an (augmented) $R$-cdga $A$ is a(n augmented) 
morphism 
$$ \rho: M \lra A $$
which is the union of $(1,q)$-minimal models ($q \in \N$) as above 
with respect to Hirsch extensions $M(q-1) \subseteq M(q)$. 
By definition, we see that $H^i(\rho): H^i(M) \lra H^i(A)$ is an isomorphism 
for $i=0,1$ and injective for $i=2$. Note that the composition of 
a $(1,q)$-minimal model $M(q) \lra A$ and a quasi-isomorphism of 
(augmented) $R$-cdga $A \lra B$ is again a $(1,q)$-minimal model and 
the same is true also for $1$-minimal models. 

For two morphisms $f, g: A \lra B$ of $R$-cdga's, 
a {\it Sullivan homotopy} \cite[(2.3)]{nagm}
between them is a morphism of $R$-cdga's $h: A \lra R(t,dt) \otimes B$ with 
$p_0 \circ h = f, p_1 \circ h = g$, where 
$R(t,dt)$ is the $R$-cdga $R[t] \oplus R[t]dt$ and 
$p_i \, (i=0,1)$ is the morphism $R(t,dt) \otimes B \lra B$ defined by 
$t \mapsto i, dt \mapsto 0$.  
For two morphisms $f, g: A \lra B$ of augmented $R$-cdga's, 
a Sullivan homotopy between them is a morphism 
$h: A \lra R(t,dt) \otimes B$ as above 
such that the diagram 
\begin{equation*}
\xymatrix{
A\ar[d]\ar[r]^(.30)h&R(t,dt)\otimes B\ar[d]\\
R\ar[r]&R(t,dt)
}
\end{equation*}
is commutative, where vertical arrows are augmentations and 
the lower horizontal map is the canonical inclusion. 
We write a Sullivan homotopy $h$ between $f$ and $g$ by 
$h: f \simeq_{\rm Su} g$, and 
we say that two morphisms $f, g: A \lra B$ of (augmented) $R$-cdga's 
are Sullivan homotopic if there exists a Sullivan homotopy between them. 
It is easy to see from the definition that, 
if two morphisms $f, g: A \lra B$ of (augmented) $R$-cdga's are 
Sullivan homotopic, $f \circ l$ and $g \circ l$ (resp. $l \circ f$ and $l \circ g$) are 
Sullivan homotopic for any morphism of (augmented) $R$-cdga's 
$l: C \longrightarrow A$ (resp. $l : B \longrightarrow C$). 
Moreover, following \cite[Lem. 2.7]{nagm} and \cite[10.7]{gm}, we have: 

\begin{lem}\label{lem:eqrel}
If $A$ is $($augmented and$)$ $(1,q)$-minimal, 
the relation on the set of $($augmented$)$ morphisms $A \lra B$ 
defined by Sullivan homotopy is an equivalence relation. 
\end{lem} 

\begin{proof}
We only recall the proof of transitivity given in \cite[Prop. 12.7]{fht} which we use later. 
Assume that we are given (augmented) morphisms $f_0, f_1, f_2: A \lra B$, a 
Sullivan homotopy $h_1: A \lra R(t_1,dt_1) \otimes B$ between 
$f_0$ and $f_1$, and a 
Sullivan homotopy $h_2: A \lra R(t_2,dt_2) \otimes B$ between 
$f_1$ and $f_2$. Then $h_1$ and $h_2$ naturally induce the morphism 
$$ h_1 \times h_2: A \lra (R(t_1, dt_1) \times_R R(t_2, dt_2)) \otimes_R B, $$
where $R(t_1, dt_1) \times_R R(t_2, dt_2)$ denotes the 
fiber product with respect to the map $t_1 \mapsto 1, dt_1 \mapsto 0$ and 
the map $t_2 \mapsto 0, dt_2 \mapsto 0$. Define the $R$-cdga $C$ by 
$$ C := 
(R(t_0,dt_0) \otimes_R R(t_1,dt_1) \otimes_R R(t_2,dt_2))/ 
(\sum_{i=0}^2t_i - 1, \sum_{i=0}^2dt_i) $$
and consider the morphism 
$$ p: C \lra R(t_1, dt_1) \times_R R(t_2, dt_2)$$ 
defined by 
$t_0 \mapsto (1-t_1, 1-t_2), t_1 \mapsto (t_1,0), t_2 \mapsto (0,t_2). $ 
Then we have the diagram 
\begin{equation*}
\begin{CD}
@. C \otimes_R B \\ 
@. @V{p \otimes \id}VV \\ 
A @>{h_1 \times h_2}>> (R(t_1, dt_1) \times_R R(t_2, dt_2)) \otimes_R B, 
\end{CD}
\end{equation*} 
and one sees that the map $p \otimes \id$ is a surjective quasi-isomorphism. 
Using this fact, one proves (\cite[10.4]{gm}, \cite[Lem. 2.5]{nagm}, \cite[Lem. 12.4]{fht}) that 
there exists a morphism $q: A \lra C \otimes_R B$ which makes the above 
diagram commutative. Then the composition of $q$ with the morphism 
$C \otimes_R B \lra R(t,dt) \otimes_R B$ given by $t_0 \mapsto 1-t, t_1 \mapsto 0, 
t_2 \mapsto t$ gives a Sullivan homotopy between $f_0$ and $f_2$. 
\end{proof}

The following proposition (\cite[Prop. 2.9]{nagm}) assures the uniqueness of 
$(1,q)$-minimal model. 

\begin{prop}\label{prop:1q}
Let $A, A'$ be augmented $R$-cdga's and let us assume given the following 
augmented diagram of augmented $R$-cdga's 
\begin{equation}\label{eq:20190203-1}
\begin{CD}
A @>f>> A' \\ 
@A{\rho}AA @A{\rho'}AA \\ 
M(q) @. M'(q),
\end{CD}
\end{equation}
where $\rho, \rho'$ are $(1,q)$-minimal models. Then there exists 
a unique morphism $\varphi: M(q) \lra M'(q)$ which admits 
a Sullivan homotopy 
$h: f \circ \rho \simeq_{\rm Su} \rho' \circ \varphi$, and 
such a Sullivan homotopy $h$ is unique. 
Moreover, we have the following: \\ 
$(1)$ If $f$ is a quasi-isomorphism, $\varphi$ is an isomorphism. \\
$(2)$ If $f_1: A \longrightarrow A'$ is another morphism of augmented $R$-cdga's 
which is Sullivan homotopic to $f$ and $\varphi_1: M(q) \longrightarrow M'(q)$ 
is the morphism satisfying the conclusion of the proposition with $f$ replaced by $f_1$, 
then $\varphi_1$ is equal to $\varphi$.
\end{prop}

\begin{proof}
The proposition, except for the last two claims, is proven in 
\cite[Prop. 2.9]{nagm}. The claim $(1)$ is proven in \cite[Prop. 2.9]{nagm} 
in the case $A = A'$ and $f = \id$, 
and we can reduce the claim $(1)$ in general case to this case by considering the diagram 
\eqref{eq:20190203-1} as the diagram 
\begin{equation*}
\begin{CD}
A' @>{\id}>> A' \\ 
@A{f \circ \rho}AA @A{\rho'}AA \\ 
M(q) @. M'(q). 
\end{CD}
\end{equation*}

Next we prove the claim (2). If $f \simeq_{\rm Su} f_1$, we have a chain of Sullivan homotopies 
$f_1 \circ \rho  \simeq_{\rm Su}  f  \circ \rho  \simeq_{\rm Su}  \rho^{\prime}\circ  \varphi$. 
So $\varphi$ satisfies the condition required for $\varphi_1$. Thus 
$\varphi = \varphi_1$ by the uniqueness assertion for $\varphi_1$. 
\end{proof}

Next we recall some properties of sheaves of cdga's \cite[(2.14)]{nagm}.  
Let $S$ be a scheme over a field $k$ of characteristic zero. 
Then one can define the notion of a sheaf of (augmented) $\cO_S$-cdga's 
and that of a(n augmented) morphism between sheaves of $\cO_S$-cdga's in usual way. 

We will need a notion of ho-morphism. 
A covering sieve of opens of $S$ is an open covering 
$\cU = \{U_{\alpha}\}_{\alpha}$ of $S$ such that, for 
any $U \in \cU$ and any open $U' \subseteq U$, 
$U'$ belongs to $\cU$. Then, 
for sheaves of (augmented) $\cO_S$-cdga's $\cA, \cB$, 
a(n augmented) ho-morphism $f: \cA \lra \cB$ 
is defined to be a family of (augmented) morphisms 
$$ f(U): \cA(U) \lra \cB(U) $$
of $\cO_S(U)$-cdga's for every affine $U \in \cU$ 
such that, for any two affines $U' \subseteq U$ in $\cU$, the map
$\cB(U' \hra U) \circ f(U): \cA(U) \lra \cB(U')$ admits a Sullivan homotopy with  
$f(U') \circ \cA(U' \hra U)$. 
(Here $\cA(U' \hra U): \cA(U) \lra \cA(U'), \cB(U' \hra U): \cB(U) \lra \cB(U')$ 
are restriction morphisms.) 
Note that our definition of ho-morphism is 
different from that in \cite{nagm} in the sense that, in our definition,  
the morphisms $f(U)$ are defined only on affine $U$'s in $\cU$. 
A morphism is a ho-morphism, and a ho-morphism naturally induces 
maps $H^i(\cA) \lra H^i(\cB) \, (i \in \N)$ 
between cohomology sheaves. Hence the notion of 
quasi-isomorphism makes sense for ho-morphisms.

A Hirsch extension of sheaves of (augmented) $\cO_S$-cdga's
is an injective (augmented) morphism $\cA \hra \cB$ for which 
there exists a covering sieve of opens $\cU$ of $S$ such that, for each $U \in \cU$, the map $\cA(U) \hra \cB(U)$ is  
a Hirsch extension in the previous sense. 
A sheaf of (augmented) $\cO_S$-cdga's $\cA$ is called $(1,q)$-minimal 
if there exists a sequence of Hirsch extensions 
\begin{equation}\label{eq:minfil-sheaf}
\cO_S = \cA(0) \subseteq \cA(1) \subseteq \cdots \subseteq \cA(q) = \cA. 
\end{equation}
Also, a sheaf of (augmented) $\cO_S$-cdga's $\cA$ is called 
$1$-minimal if 
$\cA$ is the union of $(1,q)$-minimal 
sheaves of (augmented) $\cO_S$-cdga's 
$(q \in \N)$  
with respect to Hirsch extensions as above. 

A $(1,q)$-minimal model of a(n augmented) sheaf of 
$\cO_S$-cdga's $\cA$ is a(n augmented) 
ho-morphism 
$$ \rho_q: \cM(q) \lra \cA $$
from a $(1,q)$-minimal sheaf of (augmented) $\cO_S$-cdga's 
$\cM(q)$ such that there exists a covering sieve of opens 
$\cU$ of $S$ with any $\rho_q(U): \cM(q)(U) \hra \cA(U)$ with affine 
$U \in \cU$ a $(1,q)$-minimal model in the previous sense. 
A $1$-minimal model of a(n augmented) sheaf of 
$\cO_S$-cdga's $\cA$ is a $1$-minimal (augmented) 
sheaf of $\cO_S$-cdga's $\cM = \bigcup_q \cM(q)$ 
endowed, for each $q \in \N$, with a ho-morphism 
$$ \rho_q: \cM(q) \lra \ \cA $$
with respect to a covering sieve of opens $\cU_q$ of $S$ 
such that $\cU_q \subseteq \cU_{q-1}$ and that 
$\rho_q(U)|_{\cM(q-1)(U)} = \rho_{q-1}(U)$ for 
any $U \in \cU_q$ affine. 

For two ho-morphisms $f, g: \cA \lra \cB$ of sheaves of 
(augmented) $\cO_S$-cdga's, a Sullivan homotopy between them is 
a family of Sullivan homotopies $f(U) \simeq_{\rm Su} g(U)$ \, 
($U \in \cU$ affine) for some covering sieve of opens $\cU$ of $S$. 
We denote a (sheaf version of) Sullivan homotopy $h$ between $f$ and $g$ 
also by $h: f \simeq_{\rm Su} g$. 

Then Lemma \ref{lem:eqrel} implies that, when 
$\cA$ is (augmented and) $(1,q)$-minimal, 
the relation on the set of (augmented) ho-morphisms $\cA \lra \cB$ 
defined by Sullivan homotopy is an equivalence relation. 

To prove a sheaf version of Proposition \ref{prop:1q}, we need to impose 
the quasi-coherence condition on sheaves of $\cO_S$-cdga's. 


\begin{prop}\label{prop:1qs}
Let $\cA, \cA'$ be quasi-coherent sheaves of augmented $\cO_S$-cdga's 
and let us assume given the following 
augmented diagram of ho-morphisms 
\begin{equation*}
\begin{CD}
\cA @>f>> \cA' \\ 
@A{\rho}AA @A{\rho'}AA \\ 
\cM(q) @. \cM'(q),
\end{CD}
\end{equation*}
where $\rho, \rho'$ are $(1,q)$-minimal models. 
We assume moreover that 
$\cM(q), \cM'(q)$ are quasi-coherent, as well as 
all the sheaves of $\cO_S$-cdga's appearing in 
their filtration. Then there exists 
a unique morphism $\varphi: \cM(q) \lra \cM'(q)$ 
which admits 
a Sullivan homotopy 
$h: f \circ \rho \simeq_{\rm Su} \rho' \circ \varphi$, and 
such a Sullivan homotopy $h$ is unique. 
Moreover, we have the following: \\ 
$(1)$ If $f$ is a quasi-isomorphism, $\varphi$ is an isomorphism. \\
$(2)$ If $f_1: \cA \longrightarrow \cA'$ is another ho-morphism of sheaves of augmented $\cO_S$-cdga's 
which is Sullivan homotopic to $f$ and $\varphi_1: \cM(q) \longrightarrow \cM'(q)$ 
is the morphism satisfying the conclusion of the proposition with $f$ replaced by $f_1$, 
$\varphi_1$ is equal to $\varphi$.
\end{prop}

\begin{proof}
We can suppose that all the ho-morphisms in the diagram are 
defined with respect to the same sieve of opens $\cU$. 
Take an affine $U \in \cU$. By the quasi-coherence, 
$H^i(\cA(U)) = H^i(\cA)(U)$ and the same holds for 
$\cM(q)$. Thus 
the morphisms $\rho(U): \cM(q)(U) \lra \cA(U), 
\rho'(U): \cM'(q)(U) \lra \cA'(U)$ are $(1,q)$-minimal models. 
Hence we are in the situation of 
Proposition \ref{prop:1q} and so we have the unique 
morphism $\varphi(U): \cM(q)(U) \lra \cM'(q)(U)$ satisfying 
the conclusion of that proposition. 

To finish the proof of the proposition, it suffices to prove that 
the morphisms $\varphi(U)$'s glue to give a morphism of 
sheaves $\varphi: \cM(q) \lra \cM'(q)$. To do so, it suffices to prove that, 
for an inclusion of affine opens $U' \subset U$ in $\cU$, the composite  
$\cM(q)(U) \os{\varphi(U)}{\lra} \cM'(q)(U) \longrightarrow \cM'(q)(U')$ 
is equal to the composite 
$\cM(q)(U) \lra \cM(q)(U') \os{\varphi(U')}{\longrightarrow} \cM'(q)(U')$.

In the following, we call that a square  
$$
\xymatrix{ 
A \ar[r]^{\alpha} & A' \\ 
B \ar[r]^{\beta} \ar[u]^{\gamma} & B' \ar[u]^{\gamma'}
}
$$
of augmented $R$-cdga's is Sullivan homotopically commutative if 
$\alpha \circ \gamma$ is Sullivan homotopic to $\gamma' \circ \beta$. 
We consider the following diagrams:  
$$
\xymatrix{ 
  \cA(U) \ar[r]^-{f(U)}
&  \cA'(U) \ar[r]  &  \cA'(U') \\
  \cM(q)(U) \ar[r]^-{\varphi(U)} \ar[u]^-{\rho(U)} _{ \quad \quad\mathrm (A)} & 
\cM'(q)(U) \ar[r] \ar[u]^-{\rho'(U)}
_{ \quad \quad\mathrm (B)}
&\cM'(q)(U'), \ar[u]^-{\rho'(U')}}
$$
$$
\xymatrix{ 
  \cA(U) \ar[r] &  \cA(U') \ar[r]^-{f(U')}  &  \cA'(U') \\
  \cM(q)(U) \ar[r] \ar[u]^-{\rho(U)} _{ \quad \quad\mathrm (C)} & 
  \cM(q)(U') \ar[r]^-{\varphi(U')} \ar[u]^-{\rho(U')} _{ \quad \quad\mathrm (D)} & 
  \cM'(q)(U'). \ar[u]^-{\rho'(U')} }
$$
The squares $(A)$ and $(D)$ are Sullivan homotopically commutative by Proposition \ref{prop:1q}, and 
the fact that $\rho$ and $\rho'$ are ho-morphisms implies that the squares $(B)$ and $(C)$ are also 
Sullivan homotopically commutative. Since Sullivan homotopies behave well with respect to composition, 
we see that the big rectangles in the above two diagrams are Sullivan homotopically commutative. 
On the other hand, since $f$ is a ho-morphism, the compositions of the upper horizontal maps of the two diagrams 
are Sullivan homotopic.  Then Proposition \ref{prop:1q}(2) implies that the compositions 
of the bottom horizontal maps of the two diagrams are the same. So the proof of the proposition is finished. 
\end{proof}

We prove a theorem on the existence of $(1,q)$-minimal model and 
$1$-minimal model, the idea of the proof is based on the methods of   
\cite[Thm. 5.6]{nagm}. 
Let $g: S \lra \Spec k$ be as in Notation \ref{notation}, 
and let $p_j^m: S^m(1) \lra S \, (j = 1,2, m \in \N), 
p_{j,j'}^m: S^m(2) \lra S^m(1) \, (1 \leq j < j' \leq 3, m \in \N)$ 
be as in Section 1. 
Let $\cA$ be a quasi-coherent sheaf of augmented $\cO_S$-cdga's and 
for $r=1,2$, let $\{\cA^m(r)\}_m$ be a projective system of 
quasi-coherent sheaves of augmented $\cO_{S^m(r)}$-cdga's. 
Assume that $H^0(\cA) = \cO_S, H^0(\cA^m(r)) = \cO_{S^m(r)}$. 
Also, assume that we are given 
a ho-morphism 
$$ \eta: \cA^0(1) \lra \cA $$
and 
a projective system of ho-morphisms
\begin{align*}
& \{\eta_j^m: (p_j^m)^*\cA \lra \cA^m(1)\}_m \quad (j =1,2), \\
& \{\eta_{j,j'}^m: (p_{j,j'}^m)^*\cA^m(1) \lra \cA^m(2)\}_m \quad (1 \leq j < j' \leq 3) 
\end{align*}
which are all quasi-isomorphisms such that
$\eta \circ \eta_1^0 \simeq_{\rm Su}\id, \eta \circ \eta_2^0 \simeq_{\rm Su}\id$   
(note that $p_1^0 = p_2^0 = \id$) and that, for any $m$, the diagram 
\begin{equation}\label{eq:big}
\xymatrix{
& (p_{2,3}^m)^*(p_2^m)^*\cA \ar[r]^{(p_{2,3}^m)^*\eta_2^m} \ar@{=}[ld] & (p_{2,3}^m)^*\cA^m(1) \ar[d]^{\eta_{2,3}^m}
& (p_{2,3}^m)^*(p_1^m)^*\cA \ar[l]_{(p_{2,3}^m)^*\eta_1^m} \ar@{=}[rd] \\ 
(p_{1,3}^m)^*(p_2^m)^*\cA \ar[rd]^{(p_{1,3}^m)^*\eta_2^m}  & & \cA^m(2) & & (p_{1,2}^m)^*(p_2^m)^*\cA \ar[ld]_{(p_{1,2}^m)^*\eta_2^m}  \\ 
& (p_{1,3}^m)^*\cA^m(1) \ar[ru]^{\eta_{1,3}^m} & & (p_{1,2}^m)^*\cA^m(1) 
\ar[lu]_{\eta_{1,2}^m} \\ 
& (p_{1,3}^m)^*(p_1^m)^*\cA \ar[u]^{(p_{1,3}^m)^*\eta_1^m}  \ar@{=}[rr] & & 
 (p_{1,2}^m)^*(p_1^m)^*\cA \ar[u]^{(p_{1,2}^m)^*\eta_1^m}  
}\end{equation}
is commutative up to Sullivan homotopy. This commutativity up to Sullivan homotopy  
implies that, for any $i$, 
the composites
$$ H^i(\eta_1^m)^{-1} \circ H^i(\eta_2^m): 
(p_2^m)^*H^i(\cA) \os{\cong}{\lra} H^i(\cA^m(1)) 
\os{\cong}{\lra}  (p_1^m)^*H^i(\cA) \quad (m \in \N) $$
define a stratification on $H^i(\cA)$. 
(Remember that, thanks to the hypothesis (B) and the proof of 
Proposition \ref{prop:trinity}, the maps $p_j^m$ are flat for $j=1,2$.) 
We assume moreover that, 
with this stratification, $H^i(\cA)$ belongs to $\MICn(S/k)$. 

\begin{thm}\label{thm:minmod}
Let the situation be as above. Then there exist unique 
$(1,q)$-minimal models 
$$ \rho: \cM(q) \lra \cA, \quad 
\rho^m(r): \cM^m(r)(q) \lra \cA^m(r) \, (r=1,2, m \in \N)$$
and isomorphisms 
\begin{align*}
& \zeta_j^m: (p_j^m)^*\cM(q) \os{\cong}{\lra} \cM^m(1)(q) \quad (j = 1,2, m \in \N), \\ 
& \zeta_{j,j'}^m: (p_{j,j'}^m)^*\cM^m(1)(q) \os{\cong}{\lra} \cM^m(2)(q) \quad (1 \leq j< j'\leq 3, 
m \in \N) 
\end{align*}
which satisfy the following conditions$:$ \\
$(1)$ \, For each $m' \leq m$, there exists a unique 
morphism $\cM^m(r)(q) \lra \cM^{m'}(r)(q)$ which is 
compatible with 
$\rho^m(r)$, $\rho^{m'}(r)$ 
up to unique Sullivan homotopy and compatible with 
$\zeta_j^m$, $\zeta_j^{m'}$ 
and $\zeta_{j,j'}^m$, $\zeta_{j,j'}^{m'}$ as morphisms. \\ 
$(2)$ \, For $m \in \N$, the augmented diagrams 
\begin{equation*}
\begin{CD}
(p_j^m)^*\cA @>{\eta_j^m}>> \cA^m(1) \\ 
@A{(p_j^m)^*\rho}AA @A{\rho^m(1)}AA \\ 
(p_j^m)^*\cM(q) @>{\zeta_j^m}>> \cM^m(1)(q)
\end{CD}
\qquad 
\begin{CD}
(p_{j,j}^m)^*\cA^m(1) @>{\eta_{j,j}^m}>> \cA^m(2) \\ 
@A{(p_{j,j}^m)^*\rho^m(1)}AA @A{\rho^m(2)}AA \\ 
(p_{j,j'}^m)^*\cM^m(1)(q) @>{\zeta_{j,j'}^m}>> \cM^m(2)(q)
\end{CD}
\end{equation*}
are commutative up to unique Sullivan homotopy. \\
$(3)$ \, $\zeta_1^0 = \zeta_2^0$ 
via the canonical identification $p_1^0 = p_2^0 = \id$, and 
there exists a commutative diagram as \eqref{eq:big} 
with $\cA, \cA^m(r), \eta_j^m, \eta_{j,j'}^m$ replaced by 
$\cM(q), \cM^m(r)(q), \zeta_j^m, \zeta_{j,j'}^m$ respectively. 
$($Hence, for each $i \in \N$, the isomorphisms 
$\{(\zeta_1^m)^{-1} \circ \zeta_2^m\}_m$ induce a 
structure of stratification on 
the degree $i$ part $\cM(q)^i$ of $\cM(q).)$ \\
$(4)$ \, With respect to the structure of stratification defined in $(3)$, 
$\cM(q)^i$ belongs to $\MICn(S/k)$. 
\end{thm}

\begin{proof}
Since $H^1(\cA)$ belongs to $\MICn(S/k)$, we 
can define the covering sieve of opens $\cU_1$ in $S$ by 
$$\cU_1 := \left\{U \subseteq S \,\left|\, 
\, \text{$H^1(\cA)$ is free on $U$}
\right. \right\}. $$
Take an affine $U \in \cU_1$ and put $A_U := \cA(U)$, $R_U := \cO_S(U)$. 
By the quasi-coherence of $\cA$, $H^1(A_U)$ is a free $R_U$-module. Let 
$\sigma: E := H^1(A_U) \lra Z^1(A_U)$ be a 
section of the canonical projection $Z^1(A_U) \lra H^1(A_U)$, and 
define $M(1)_U$ by $M(1)_U := \bigwedge (E)$ with zero differential. 
Also, define the morphism $M(1)_U \lra A_U$ as the one induced by 
$E \os{\sigma}{\lra} Z^1(A_U) \hra A_U$. Then $M(1)_U$ is a $(1,1)$-minimal 
model of $A_U$. 

Suppose now $U' \subseteq U$ be affine opens in $\cU_1$ and 
consider the following diagram: 
\begin{equation*}
\begin{CD}
A_U \otimes_{R_U} R_{U'} @>>> A_{U'} \\ 
@AAA @AAA \\ 
M(1)_U \otimes_{R_U} R_{U'} @. M(1)_{U'}.
\end{CD}
\end{equation*}
Because $R_{U'}$ is flat over $R_U$ and $\cA$ is quasi-coherent, 
the horizontal map is a quasi-isomorphism, and the left 
vertical map is a $(1,1)$-minimal model. Hence, by 
Proposition \ref{prop:1q}, there is a unique isomorphism 
$M(1)_U \otimes_{R_U} R_{U'} \lra M(1)_{U'}$ which makes 
the diagram commutative up to Sullivan homotopy. 
By unicity, this isomorphism is compatible if we consider 
affine $U'' \subseteq U' \subseteq U$ in $\cU_1$. 
To this data, we can associate a sheaf $\cM(1)$ 
of $\cO_S$-cdga's which is locally free of finite rank on 
each degree (hence quasi-coherent) and we have naturally a ho-morphism of 
$(1,1)$-minimal model
$\rho: \cM(1) \lra \cA$ with respect to $\cU_1$.
We can define the $(1,1)$-minimal models 
$\rho(r)^m: \cM^m(r)(1) \lra \cA^m(r) \, (r = 1,2, m \in \N)$
in the same way. We can define the  
isomorphisms 
\begin{align*}
& \zeta_j^m: (p_j^m)^*\cM(1) \os{\cong}{\lra} \cM^m(1)(1) \quad (j = 1,2, m \in \N), \\ 
& \zeta_{j,j'}^m: (p_{j,j'}^m)^*\cM(1)(1) \os{\cong}{\lra} \cM^m(2)(1) \quad (1 \leq j < j'\leq 3, 
m \in \N) 
\end{align*}
by Proposition \ref{prop:1qs}, and 
the conditions (1), (2), (3) follow also from Proposition \ref{prop:1qs}. 
(We have the equality $\zeta_1^0 = \zeta_2^0$ because we also have 
the isomorphism $\zeta: \cM^0(1)(1) \lra \cM(1)$ compatible with 
$\eta$ up to unique Sullivan homotopy and we see that 
$\zeta \circ \zeta_1^0 = 
\zeta \circ \zeta_2^0 = \id$ by Proposition \ref{prop:1qs}.) 
The condition (4) for $i=1$ is immediate because 
$\cM(1)^1$ (endowed with the stratification $\{(\zeta_1^m)^{-1} \circ \zeta_2^m\}_m$) 
is isomorphic to $H^1(\cA)$ via $\rho$, and this implies the condition (4) for general $i$ 
because 
$\cM(1)^i$'s $(i \in \N)$ are wedge products of $\cM(1)^1$. 
So the theorem holds for $q=1$. 

Next, we prove the theorem for general $q$, assuming the existence of 
$(1,q-1)$-minimal models 
$\cM(q-1) \lra \cA$, $\cM^m(r)(q-1) \lra \cA^m(r) \, (r=1,2, m \in \N)$
and isomorphisms 
\begin{align*}
& (p_j^m)^*\cM(q-1) \os{\cong}{\lra} \cM^m(1)(q-1) \quad (j = 1,2, m \in \N), \\ 
& (p_{j,j'}^m)^*\cM(1)(q-1) \os{\cong}{\lra} \cM^m(2)(q-1) \quad (1 \leq j < j'\leq 3, 
m \in \N) 
\end{align*}
as in the statement of the theorem. 
(Note that $\cM(q-1)$ is locally free of finite rank on each degree 
and hence quasi-coherent.)  Moreover, by minimality, for  the second cohomology group of the mapping fiber we get the equality 
$H^2(\cM(q-1), \cA) = \Ker(H^2(\cM(q-1)) \lra H^2(\cA))$ and it 
is locally free by induction, because both 
$H^2(\cM(q-1)), H^2(\cA)$ belong to $\MICn(S/k)$. 
So we can define the covering sieve of opens $\cU_q$ in $S$ by
$$\cU_q := \left\{U \in \cU_{q-1} \,\left|\, 
\text{$H^2(\cM(q-1), \cA)$ is free on $U$}
\right. \right\}. $$
Take $U \in \cU_q$ affine and 
put $A := \cA(U)$, $M(q-1) := \cM(q-1)(U)$ and $R := \cO_S(U)$.
By the quasi-coherence of $\cA$ and $\cM(q-1)$, 
$H^2(M(q-1), A)$ is free. Let 
\begin{align*}
\sigma: E := H^2(M(q-1), A) & \lra Z^2(M(q-1),A) \\ & := 
\{(m,a) \in M(q-1)^2 \oplus A^1 \,|\, dm = 0, \rho'(m) = da\} 
\end{align*}
(where $d$ denotes the differential of $M(q-1)$ and $A$ and 
$\rho'$ denotes the map $M(q-1) \lra A$) 
be a section of the canonical projection $Z^2(M(q-1),A) \lra H^2(M(q-1),A)$, 
and define $M(q)$ by $M(q) := M(q-1) \otimes \bigwedge (E)$, with 
differential induced by that on $M(q-1)$ and the map 
$$E \os{\sigma}{\lra} Z^2(M(q-1),A) \hra M(q-1)^2 \oplus A^1 
\os{\rm proj.}{\lra} M(q-1)^2. $$
Also, define the morphism $M(q) \lra A$ extending 
$M(q-1) \lra A$ as the one induced by the map 
$$E \os{\sigma}{\lra} Z^2(M(q-1),A) \hra M(q-1)^2 \oplus A^1 
\os{\rm proj.}{\lra} A^1. $$
Then $M(q)$ is a $(1,q)$-minimal 
model of $A$. 
By sheafifying as before using Propositions \ref{prop:1q}, 
we see that 
$M(q)$'s for every affine $U \in \cU_q$ 
form a sheaf $\cM(q)$ and  
that the morphisms $M(q) \lra A$ for every affine $U \in \cU_q$   
form a ho-morphism of $(1,q)$-minimal model
$\rho: \cM(q) \lra \cA$ with respect to $\cU_q$. 
The $(1,q)$-minimal models 
$\rho^m(r): \cM^m(r)(q) \lra \cA^m(r) \, (r=1,2, m \in \N)$
are defined in the same way, 
and the isomorphisms 
\begin{align*}
& \zeta_j^m: (p_j^m)^*\cM(q) \os{\cong}{\lra} \cM^m(1)(q) \quad (j = 1,2, m \in \N), \\ 
& \zeta_{j,j'}^m: (p_{j,j'}^m)^*\cM(1)(q) \os{\cong}{\lra} \cM^m(2)(q) \quad (1 \leq j, j'\leq 3, 
m \in \N) 
\end{align*}
are defined by Proposition \ref{prop:1qs}, and 
the conditions (1), (2), (3) follow also from Proposition \ref{prop:1qs}. 

Finally we check the condition (4). By construction, we have 
the commutative diagram 
\[ 
\xymatrix{
& & 0 \ar[d] \\ 
E \ar[r]^-{\sigma} \ar[d]^d 
\ar@/^20pt/[rr]^-{\cong}
& Z^2(M(q-1),A) \ar[r] & H^2(M(q-1),A) \ar[d] \\ 
Z^2(M(q-1)) \ar[rr] &  & H^2(M(q-1)) \ar[d] \\
& & H^2(A),  
}\] 
where the horizontal arrows other than $\sigma$ are 
the canonical surjections, the right vertical line is 
exact the left vertical arrow $d$ is the differential on $M(q)$. 
Since $E \cong M(q)^1/M(q-1)^1$, 
the differential on $\cM(q)$ induces the exact sequence 
$$ 0 \lra \cM(q)^1/\cM(q-1)^1 \lra H^2(\cM(q-1)) \lra H^2(\cA) $$
which is compatible with the stratification. Thus 
$\cM(q)^1/\cM(q-1)^1$ belongs to $\MICn(S/k)$. Hence so does 
$\cM(q)^1$, and so does $\cM(q)^i$ because 
$\cM(q)^i$'s $(i \in \N)$ are wedge products of $\cM(q)^1$. 
So the proof of the theorem is finished. 
\end{proof}

\begin{cor}\label{cor:minmod}
Let the situation be as above. Then there exist unique 
$1$-minimal models $\cM = \bigcup_q \cM(q)$, 
$\cM^m(r) = \bigcup_q \cM^m(r)(q)$ and isomorphisms 
\begin{align*}
& \zeta_j^m: (p_j^m)^*\cM \os{\cong}{\lra} \cM^m(1) \quad (j = 1,2, m \in \N), \\ 
& \zeta_{j,j'}^m: (p_{j,j'}^m)^*\cM^m(1) \os{\cong}{\lra} \cM^m(2) \quad (1 \leq j, j'\leq 3, 
m \in \N) 
\end{align*}
which satisfy the conditions $(1)$, $(2)$, $(3)$ in Theorem \ref{thm:minmod} and 
the following condition: \\
$(4)$ \, With respect to the structure of stratification defined in $(3)$, 
the degree $i$ part $\cM^i$ of $\cM$ 
belongs to the ind-category of $\MICn(S/k)$. 
\end{cor}

\begin{proof}
The assertion immediately follows from Theorem \ref{thm:minmod} by 
taking the union with respect to $q$. 
\end{proof}

\renewcommand{\thefootnote}{\arabic{footnote}}

Now we apply Corollary \ref{cor:minmod} to our geometric situation 
and give the third definition of relatively unipotent de Rham fundamental group. 
Let the notations be as in Notation \ref{notation} and we use the methods and definitions in \cite{naho}.
Let $X = \bigcup_{i \in I}U_i$ be an affine open covering of $X$ 
with $I$ finite such that each $U_i$ is affine over $S$.
Fix an order on $I$ and for $n \in \N$, define $X_n$ by 
$X_n := \coprod_{i_0 < \cdots < i_n, \in I} U_{i_0} \cap \cdots \cap U_{i_n}$. 
Then $X_n$'s form a semi-simplicial\footnote{In \cite{naho}, the terminology `strict simplicial' is used.} scheme $X_{\b}$ and we have an affine  morphism 
$\pi: X_{\b} \lra X$, such that $f \circ \pi$ is also affine. Let 
$\pi_{S}: S_{\b} \lra S$ be the pull-back of $\pi$ by 
$\iota$ (then $\pi_S$ is also affine). Then the section $\iota$ induces 
a surjective morphism of semi-cosimplicial sheaves of $\cO_S$-cdga's 
$$ 
\cA^{\b,*} := f_*\pi_{*}\Omega^{*}_{X_{\b}/S} 
\os{\iota^*}{\lra} \pi_{S *} \cO_{S_{\b}}.  
$$ 
By applying the Thom--Whitney functor $\sTW$
(\cite[\S 3]{naho}), we obtain 
the surjective morphism of sheaves of $\cO_S$-cdga's 
$\sTW(\iota^*): \sTW(\cA^{\b,*}) \lra \sTW(\pi_{S *} \cO_{S_{\b}})$ such that 
$H^0(\sTW(\pi_{S *} \cO_{S_{\b}})) \allowbreak = \cO_S$. 

\begin{defn}\label{def:axs}
Let the notations be as above. 
We define the sheaf of $\cO_S$-cdga's 
$\cA_{X/S} \subseteq \sTW(\cA^{\b,*})$ as the inverse image of 
$\cO_S = H^0(\sTW(\pi_{S *} \cO_{S_{\b}})) \hra \sTW(\pi_{S *} \cO_{S_{\b}})$ 
by $\sTW(\iota^*)$.\footnote{It is assumed in \cite{naho} that $\sTW(\cA^{\b,*})$ is augmented over $\cO_S$ (\cite[\S6]{nagm}). Since this is not the case in general, we need to make this construction.} 
Then, $\cA_{X/S} $ is in fact a sheaf of augmented $\cO_S$-cdga's.
\end{defn}

If we denote the functor which associates 
the single complex to a semi-cosimplicial complex by $\s$, 
we have functorial quasi-isomorphisms (\cite[Thm. 3.3]{naho}) 
\begin{align*}
& \cA_{X/S} \lra \sTW(\cA^{\b,*}) \lra \s(\cA^{\b,*}). 
\end{align*}
(Note that the morphism $\cA_{X/S} \lra \sTW(\cA^{\b,*})$ is a 
quasi-isomorphism because $\cA_{X/S}$ is the kernel of the map 
$$ \sTW(\cA^{\b,*})  \oplus \cO_S \lra \sTW(\pi_{S *} \cO_{S_{\b}}) $$
which is surjective by the exactness of $\sTW$, and the inclusion 
$\cO_S \hra \sTW(\pi_{S *} \cO_{S_{\b}})$ is a quasi-isomorphism.) 

Note that each term in $\cA^{\b,*}$ is quasi-coherent. 
Then we can conclude that $\sTW(\cA^{\b,*})$ is quasi-coherent 
in each degree: in fact, our covering $X = \bigcup_{i \in I}U_i$ is finite 
and so the end functor which is used to define $\sTW$ involves 
finitely many nonzero quasi-coherent sheaves. 
For the same reason, $\sTW(\pi_{S *} \cO_{S_{\b}})$ is quasi-coherent 
in each degree. Hence $\cA_{X/S}$ is quasi-coherent in each degree.

$H^i(\s(\cA^{\b,*}))$ is nothing but $R^if_{\dR *}(\cO_X,d)$, 
we have the functorial isomorphism 
\begin{equation}\label{eq:agm}
H^i(\cA_{X/S}) \os{\cong}{\lra} R^if_{\dR *}(\cO_X,d). 
\end{equation}

Let $p_j^m: S^m(1) \lra S, q_j^m: X_j^m \lra S^m(1), 
f_j^m: X_j^m \lra S^m(1), 
\wh{f}^m: \wh{X}^m \lra S^m(1), 
\wh{q}_j^m: \wh{X}^m \lra X$ be as in Section 1, 
before Definition \ref{strcrys} and before Proposition \ref{basechange}, 
and 
let $\wh{\iota}^m: S^m(1) \lra \wh{X}^m$ be the section of 
$\wh{f}^m$ induced by $\iota$. Also, let $(X_{\b})_j^m$ be $X_{\b} \times_X X_j^m$, 
let $\wh{X}_{\b}^m$ be the log formal tube of $X_{\b}$ in 
$(X_{\b})_1^m \times_{S^m(1)} (X_{\b})_2^m$ and 
let $S^m(1)_{\b}$ be the inverse image of $\wh{X}_{\b}^m \lra \wh{X}^m$ by 
$\wh{\iota}^m$. 
Then the section $\wh{\iota}^m$ induces the 
surjective morphism of semi-cosimplicial sheaves of 
$\cO_{S^m(1)}$-cdga's 
$$ 
f_*\pi_{*}\Omega^{*}_{\wh{X}^m_{\b}/S^m(1)} 
\lra \pi_{S *} \cO_{S^m(1)_{\b}}, 
$$ 
and we can define the augmented sheaf of $\cO_{S^m(1)}$-cdga's 
from it, which we denote by 
$\cA_{X/S^m(1)}$. 
Working on $S^m(2)$ and considering the formal log tube in triple products, 
we can define the augmented sheaf of $\cO_{S^m(2)}$-cdga's 
in a similar way, which we denote by 
$\cA_{X/S^m(2)}$. By functoriality of the construction, 
we have the canonical morphisms 
\begin{align*}
& \eta: \cA_{X/S^0(1)} \lra \cA_{X/S}, \\ 
& \{\eta_j^m: (p_j^m)^*\cA_{X/S} \lra \cA_{X/S^m(1)}\}_m \quad (j =1,2), \\
& \{\eta_{j,j'}^m: (p_{j,j'}^m)^*\cA_{X/S^m(1)} \lra \cA_{X/S^m(2)}\}_m \quad (1 \leq j < j' \leq 3) 
\end{align*}
with $\eta \circ \eta_1^0 = \eta \circ \eta_2^0 = \id$ 
which fit into the commutative diagram like \eqref{eq:big}. 
Moreover, $\eta$, $\eta_j^m$'s are quasi-isomorphisms because 
the morphisms of cohomologies 
$$ 
H^i(\eta_j^m): (p_j^m)^* H^i(\cA_{X/S}) 
\lra H^i(\cA_{X/S^m(1)}) 
$$
(note that $p_j^m$'s are flat) are identified with the isomorphisms  
$$ 
{p_j^m}^*R^if_{\dR *}(\cO_X,d) 
\os{\cong}{\lra} R^i\wh{f}^m_{\dR *}(\cO_{\wh{X}^m},d) 
$$
via \eqref{eq:agm} and the variant of it for $\cA_{X/S^m(1)}$. 
By the same reason, $\eta_{j,j'}^m$'s are also quasi-isomorphisms. 
Thus the composition 
$$ 
H^i(\eta_1^m)^{-1} \circ H^i(\eta_2^m): 
(p_2^m)^* H^i(\cA_{X/S})  
\os{\cong}{\lra} H^i(\cA_{\wh{X}^m/S^m(1)}) \os{\cong}{\lra} (p_1^m)^*H^i(\cA_{X/S}) \quad (m \in \N) 
$$ 
defines a stratification on $H^i(\cA_{X/S})$. Moreover, since it is identified with 
the composition  
\begin{align*}
{p_2^m}^*R^if_{\dR *}(\cO_X,d) 
& \os{\cong}{\lra} R^i\wh{f}^m_{\dR *}(\cO_{\wh{X}^m},d) 
\os{\cong}{\lla} {p_1^m}^*R^if_{\dR *}(\cO_X,d) \quad (m \in \N) 
\end{align*} 
giving the Gauss--Manin connection 
on $R^if_{\dR *}(\cO_X,d)$, $H^i(\cA_{X/S})$ belongs to $\MICn(S/k)$. 
Therefore, our objects fit the hypotheses in Corollary \ref{cor:minmod}, and so 
there exists a unique $1$-minimal model 
$M_{X/S}$ of $\cA_{X/S}$ endowed with the stratification 
$$ (\zeta_1^m)^{-1} \circ \zeta_2^m: (p_2^m)^*M_{X/S} \os{\cong}{\lra} 
(p_1^m)^*M_{X/S} $$
with respect to which $M_{X/S}^i \,(i \in \N)$ belongs to the ind-category of 
$\MICn(S/k)$ and the differentials $M_{X/S}^i \lra M_{X/S}^{i+1}$ are 
morphisms in the ind-category of $\MICn(S/k)$ because $\zeta_j^m \,(j=1,2)$ 
are morphisms of sheaves of cdga's.  
In particular, each $M_{X/S}^i$ is flat over $\cO_S$. 
Note that $M_{X/S}$ is the union of $(1,q)$-minimal models, 
which we denote by $M_{X/S}(q)$. (In particular  each 
$M_{X/S}(q)^i$ belongs to $\MICn(S/k)$.) 

Now put $(M^1_{X/S})^{\vee} := \varprojlim_q (M_{X/S}(q)^1)^{\vee}$ 
and denote the isomorphism 
$$ {p_2^m}^*(M^1_{X/S})^{\vee} \os{\cong}{\lra} 
{p_1^m}^*(M^1_{X/S})^{\vee} $$
induced by $(\zeta_1^m)^{-1} \circ \zeta_2^m$ 
by $\epsilon^m$. Then the pair $((M^1_{X/S})^{\vee}, \{\epsilon^m\})$ 
is an object in the pro-category of $\MICn(S/k)$. 
The minus of the differential $-d: M_{X/S}^1 \lra M_{X/S}^2 = M_{X/S}^1 \wedge M_{X/S}^1$ endows it with 
a structure of Lie algebra in the pro-category of $\MICn(S/k)$ 
which is the projective limit of Lie algebras $(M_{X/S}(q)^1)^{\vee} \, (q \in \N)$ 
in $\MICn(S/k)$. 
(The minus sign comes from the fact that $M_{X/S}^1$ is in degree $1$ in the 
complex $M_{X/S}$, while we regard $M_{X/S}^1$ as a Lie coalgebra sitting in degree $0$.) 
Since the image of the 
map $-d: M_{X/S}(q)^1 \lra M_{X/S}(q)^2 = M_{X/S}(q)^1 \wedge M_{X/S}(q)^1$ is 
contained in $M_{X/S}(q-1)^2 = M_{X/S}(q-1)^1 \wedge M_{X/S}(q-1)^1$ by 
definition of $(1,q)$-minimal model, $(M_{X/S}(q)^1)^{\vee}$ is a nilpotent 
Lie algebra in $\MICn(S/k)$ and so $(M^1_{X/S})^{\vee}$ is a pro-nilpotent 
Lie algebra in the pro-category of $\MICn(S/k)$. 

To associate to $(M^1_{X/S})^{\vee}$ a pro-unipotent group scheme in 
$\MICn(S/k)$, 
we explain the correspondence between nilpotent Lie algebras and certain unipotent 
group schemes in $\MICn(S/k)$. Following \cite[1.3]{tong}, we call a group scheme 
$G$ over $S$ unipotent if $G$ is flat over $S$ and the geometric fibers of 
$G \lra S$ are unipotent algebraic groups. We call an affine group scheme $G$ in 
$\MICn(S/k)$ a unipotent group scheme in $\MICn(S/k)$ if its underlying 
group scheme over $S$ is unipotent. Let ${\rm NL}(S/k)$ be 
the category of nilpotent Lie algebras in $\MICn(S/k)$ and let 
${\rm UG}(S/k)$ be the category of unipotent group schemes in 
$\MICn(S/k)$ satisfying the following condition $(\star)$: 
\bigskip 

\noindent 
$(\star)$ \,\, There exists an object $V$ in $\MICn(S/k)$ such that 
$G$ is realized as a closed subgroup scheme of $GL(V)$ in $\MICn(S/k)$. 
\bigskip 

Concerning a sufficient condition for an affine group scheme $G$ in $\MICn(S/k)$ 
to satisfy the condition $(\star)$, we have the following lemma. 

\begin{lem}\label{lem:unip2020sep}
Let $G$ be an affine group scheme in $\MICn(S/k)$ satisfying the following conditions: \\ 
$(1)$ \, $\cO_G$ is generated as an $\cO_S$-algebra by some subobject $\cM \subseteq \cO_G$ 
in $\MICn(S/k)$. \\ 
$(2)$ \, $\cO_G$ is locally projective as an $\cO_S$-module. \\
Then $G$ satisfies the condition $(\star)$. 
\end{lem}

\begin{proof}
Let $c(\cM)$ be the image of the map 
$$ \cM \otimes_{\cO_S} \cO_G^{\vee} \hra 
\cO_G \otimes_{\cO_S} \cO_G^{\vee} 
\os{\text{comult} \otimes \id}{\lra} \cO_G \otimes_{\cO_S} \cO_G \otimes_{\cO_S} \cO_G^{\vee} 
\os{\id \otimes \text{ev}}{\lra} \cO_G, $$
where $\text{comult}$ is the comultiplication on $\cO_G$ and $\text{ev}$ is the evaluation map. 
Then, by \cite[Expos\'e VI, Lemme 11.8, Lemme 11.8.1, Proposition 11.9]{SGA3-1}, 
we see that $c(\cM)$ is an object in $\MICn(S/k)$ containing $\cM$ and it has the 
comodule structure $c(\cM) \lra c(\cM) \otimes_{\cO_S} \cO_G$ induced by the 
comultiplication of $\cO_G$. So we have a homomorphism 
$G \lra GL(c(\cM))$ of affine group schemes in $\MICn(S/k)$, and it is a closed immersion 
by the argument of \cite[Theorem 4.9]{milne}. 
\end{proof}

The functor ${\rm Lie}: {\rm UG}(S/k) \lra {\rm NL}(S/k)$ of taking Lie algebras is 
defined in usual way, and the functor ${\mathbb{W}}: {\rm NL}(S/k) \lra {\rm UG}(S/k)$ 
is defined so that ${\mathbb{W}}(L) := \ul{\Spec}_{\cO_S} {\rm Sym}(L^{\vee})$ is the vector 
bundle associated to $L$ as a scheme over $S$ and that the group structure is 
defined by the Baker--Campbell--Hausdorff formula. 
(${\mathbb{W}}(L)$ satisfies the condition ($\star$) because ${\rm Sym}(L^{\vee})$ 
is locally projective as an $\cO_S$-module and it is generated by $L^{\vee}$.) 
Then we have the following. 

\begin{prop}\label{prop:unip2020sep}
The above functors ${\rm Lie}, {\mathbb{W}}$ are the equivalences which 
are the inverses of each other. 
\end{prop}

\begin{proof}
The following proof is inspired by \cite[1.3]{tong}. The equality ${\rm Lie} \circ {\mathbb{W}} = \id$ 
is standard, and it suffices to prove the equality ${\mathbb{W}} \circ {\rm Lie} = \id$. 
Let $G$ be an object in ${\rm UG}(S/k)$. 
By the condition $(\star)$ and \cite[Proposition 1.3, Corollary 1.4]{tong}, 
we can define the exponential map 
$$ \exp: {\mathbb{W}}({\rm Lie}(G)) \lra G$$ 
of affine schemes in $\MICn(S/k)$ which is an isomorphism (but we do not know yet 
if it is an isomorphism of affine group schemes in $\MICn(S/k)$). 

It suffices to prove that the exponential map is an isomorphism of 
affine group schemes in $\MICn(S/k)$. To do so, we compare the composites 
\begin{align*}
& {\mathbb{W}}({\rm Lie}(G)) \times_S {\mathbb{W}}({\rm Lie}(G)) \os{\text{mult}}{\lra} 
{\mathbb{W}}({\rm Lie}(G))  \os{\exp}{\lra} G, \\ 
&  {\mathbb{W}}({\rm Lie}(G)) \times_S {\mathbb{W}}({\rm Lie}(G)) 
\os{\exp \times \exp}{\lra} G \times_S G \os{\text{mult}}{\lra} G, 
\end{align*}
where $\text{mult}$ denotes the multiplication. Let $D$ be the difference kernel of them, 
which is an affine scheme in $\MICn(S/k)$. Then it suffices to prove that 
$D$ is equal to ${\mathbb{W}}({\rm Lie}(G)) \times_S {\mathbb{W}}({\rm Lie}(G))$, and 
we can check it after we apply the restriction functor 
$\MICn(S/k) \lra \MIC(s/s)$ to $D$ and 
${\mathbb{W}}({\rm Lie}(G)) \times_S {\mathbb{W}}({\rm Lie}(G))$. 
So we can reduce to the case over a field and in this case, 
the concidence of the group structure on ${\mathbb{W}}({\rm Lie}(G))$ and that on 
$G$ is well-known. Thus we have shown that the exponential map is an isomorphism 
of affine group schemes in $\MICn(S/k)$ and so we are done. 
\end{proof}

The equivalence of the categories ${\rm NL}(S/k)$, ${\rm UG}(S/k)$ extends 
to the equivalence between the category of pro-objects in ${\rm NL}(S/k)$ and 
that in ${\rm UG}(S/k)$. Applying it to $(M^1_{X/S})^{\vee}$, 
we can give the third definition of relatively unipotent 
de Rham fundamental groups in the following way:

\begin{defn}[{\bf Third definition of relatively unipotent $\pi_1$}]\label{def3}
Let the notations be as above. We define the relatively unipotent de Rham fundamental group 
$\pi^{\dR}_1(X/S,\iota)$ as the pro-unipotent group scheme in 
$\MICn(S/k)$ corresponding to the pro-nilpotent Lie algebra 
$(M^1_{X/S})^{\vee}$. 
\end{defn}

Then we have the following comparison theorem. 

\begin{thm}\label{prop4_2}
The first and the third definitions of $\pi_1^{\dR}(X/S,\iota)$ are 
canonically isomorphic. 
\end{thm}

The proof of this theorem will be given in Section 5.

\begin{rem}\label{rem:mxss}
We can apply to the morphism $X_s \lra s$ the same construction 
we used in order to define the sheaf of augmented $\cO_S$-cdga's 
$\cA_{X/S}$. So we can define 
the augmented $k$-cdga $\cA_{X_s/s}$. Then, we can define the $1$-minimal 
model of it, which we denote by $M_{X_s/s} = \bigcup_q M_{X_s/s}(q)$. 
(To define $M_{X_s/s}$, the classical theory of $1$-minimal model 
in \cite[\S 13]{gm} is enough and we do not need the relative version.) 

Because we have a family of ho-morphisms 
$\{ M_{X/S}(q) \lra \cA_{X/S} \}_q$, there exists a 
decreasing sequence of affine open subschemes $U_q \, (q \in \N)$ 
containing $s$ such that the above ho-morphisms induce 
morphisms $M_{X/S}(q')(U_q) \lra \cA_{X/S}(U_q) \, (q,q' \in \N, q' \leq q)$ of 
$(1,q')$-minimal models which is compatible with respect to $q'$. 
By taking tensor product with respect to 
the map $\cO_S(U_q) \lra k$ induced by $s$ and composing it 
with the map $\cA_{X/S} \otimes_{\cO_S,s^*} k \lra \cA_{X_s/s}$ 
(where $s^*$ denotes the map $\cO_S \lra \cO_s = k$ induced by $s$), 
we obtain the compatible family of morphisms 
$M_{X/S}(q') \otimes_{\cO_S,s^*} k \lra \cA_{X_s/s} \, (1 \leq q' \leq q)$. 
We can prove that the morphism 
$M_{X/S}(q) \otimes_{\cO_S,s^*} k \lra \cA_{X_s/s}$ we obtained 
is the $(1,q)$-minimal model 
of $\cA_{X_s/s}$.  To see it, it suffices to prove that 
$H^i(M_{X/S}(q') \otimes_{\cO_S,s^*} k) \lra H^i(\cA_{X_s/s})$ is 
an isomorphism for $i=0,1, 1 \leq q' \leq q$ and that 
$H^2(M_{X/S}(q') \otimes_{\cO_S,s^*} k,\cA_{X_s/s}) \lra 
H^2(M_{X/S}(q'+1) \otimes_{\cO_S,s^*} k,\cA_{X_s/s}) \, (1 \leq q' \leq q-1)$ is the zero map. 
Since each term of $M_{X/S}(q')$ are objects in $\MICn(S/k)$, 
they and their cohomologies are locally free. Hence we have 
$H^i(M_{X/S}(q') \otimes_{\cO_S,s^*} k) = H^i(M_{X/S}(q')) \otimes_{\cO_S,s^*} k 
= H^i(\cA_{X/S}) \otimes_{\cO_S,s^*} k$ for $i=0,1$. 
Since $H^i(\cA_{X/S})$ (resp. $H^i(\cA_{X_s/s})$) is the (relative) de Rham cohomology 
$R^if_{\dR *}(\cO_X,d)$ (resp. $R^i f_{s \dR}(\cO_{X_s},d)$), 
we have the base change isomorphism $H^i(\cA_{X/S}) \otimes_{\cO_S,s^*} k = 
H^i(\cA_{X_s/s})$ and so the first condition to be the $(1,q)$-minimal model 
is proved. Because the first condition is fulfilled, 
we have 
$$H^2(M_{X/S}(q') \otimes_{\cO_S,s^*} k,\cA_{X_s/s}) = 
\Ker(H^2(M_{X/S}(q') \otimes_{\cO_S,s^*} k) \lra H^2(\cA_{X_s/s})),$$
and this is the base change of 
$H^2(M_{X/S}(q'), \cA_{X/S}) = 
\Ker(H^2(M_{X/S}(q')) \lra H^2(\cA_{X/S}))$ by its local freeness as above. 
So the second condition to be the $(1,q)$-minimal model 
follows from that for $M_{X/S}(q')(U_q) \, (1 \leq q' \leq q)$. 

So, by Proposition \ref{prop:1q}, we have a unique isomorphism 
of $(1,q)$-minimal models 
$M_{X/S}(q) \otimes_{\cO_S,s^*} k \lra M_{X_s/s}(q)$. 
We can do the same construction with $U_q$ replaced by $U_{q+1}$, 
and we obtain another isomorphism 
$M_{X/S}(q) \otimes_{\cO_S,s^*} k \lra M_{X_s/s}(q)$, which is 
compatible with the restriction. By construction, the latter isomorphism 
is the restriction of the isomorphism $M_{X/S}(q+1) \otimes_{\cO_S,s^*} k 
\lra M_{X_s/s}(q+1)$ to $M_{X/S}(q)$. 
So we can take the union of this isomorphisms with respect to $q$ and 
obtain the isomorphism 
$M_{X/S} \otimes_{\cO_S,s^*} k \lra M_{X_s/s}$. 
In other words, $M_{X_s/s}$ is canonically isomorphic to 
the image of $M_{X/S}$ by the ind-version of the restriction 
functor $\MICn(X/S) \lra \MIC(s/s)$. 
\end{rem}

\begin{rem} \label{GM} The definition of the connection on 
$(M^1_{X/S})^{\vee}$ given above is not the same as 
the original one of Navarro Aznar in \cite[\S 5, \S 6]{nagm}.  
We expect that his definition is the same as our definition in 
Definition \ref{def3}, but we will not pursue this topic in this paper. 
\end{rem}

\medskip \noindent

\section{Relative bar construction}

In this section, we give the fourth definition of 
relatively unipotent de Rham fundamental groups 
by using the relative version of bar construction. 
Then we prove the coincidence of this definition with 
the first one given in Section 2. The proof of Theorem 
\ref{prop4_2} will be achieved by showing 
that the fourth definition is equivalent to the 
third one. This will complete the circle: 
all  four definitions are equivalent. 

First we give a brief explanation on bar construction in relative setting: 
the construction is the one given in \cite{hain}, but we work over 
an arbitrary commutative ring $R$. (In \cite{hain}, $R$ is assumed to be 
a field.) Let $M$ be an augmented 
$R$-cdga. Assume that, for any $i \in \N$, the degree $i$ part 
$M^i$ of $M$ and the $i$-th cohomology $H^i(M)$ 
are flat $R$-modules and that $H^0(M) = R$. 
(The flatness assumption here is automatic in the situation in \cite{hain}, but 
we need to impose this assumption in order that the K\"{u}nneth formula, e.g. 
the middle equality in \eqref{eq:20192013} below, holds in our setting.) 
Let $\ol{M} := \Ker(M \lra R)$ be the augmentation ideal. 

For $s, t \in \N$, let 
$B^{-s,t}(M)$ be the set of elements in $\bigotimes^s_R \ol{M}$ which are 
homogeneous of degree $t$. We can define 
the combinatorial differential $d_C: B^{-s,t}(M) \lra B^{-s+1,t}(M)$ and 
the internal differential $d_I: B^{-s,t}(M) \lra B^{-s,t+1}(M)$ 
satisfying $d_Cd_I + d_Id_C = 0$, as in \cite[p. 275]{hain}. 
Then we define $B(M)$ as the single complex associated to 
$(B^{-s,t}(M), d_C, d_I)$. 

We define the bar filtration $F^{-s}B(M)$ by 
$F^{-s}B(M) := \bigoplus_{u \leq s} B^{-u,v}(M)$. 
The spectral sequence associated to it is called 
the Eilenberg-Moore spectral sequence. 
The $E_1$-term of the Eilenberg--Moore spectral sequence is written as 
\begin{equation}\label{eq:20192013}
E_1^{-s,t} = H^t(\bigotimes^s_R \ol{M}) = 
\text{(degree $t$ part of $\bigotimes^s_R H^{\b}(\ol{M})$)}
= B^{-s,t}(H^{\b}(M)). 
\end{equation}

We have the morphisms 
$$ i: B(R) \lra B(M), \quad e: B(M) \lra B(R) $$ 
induced by the inclusion $R \hra M$ and the augmentation 
$M \lra R$, and we can define the morphisms 
$$\wedge: B(M) \otimes_R B(M) \lra B(M), 
\quad \Delta: B(M) \lra B(M) \otimes_R B(M) $$
called the shuffle product and the diagonal (\cite[pp. 275--276]{hain}). 
These morphisms induce morphisms on $0$-th cohomology 
\begin{align}
& i_H: R \lra H^0(B(M)), \quad e_H: H^0(B(M)) \lra R, \label{eq:hopf} \\  
& \wedge_H: H^0(B(M)) \otimes_R H^0(B(M)) \lra H^0(B(M)), \nonumber \\ 
& \Delta_H: H^0(B(M)) \lra H^0(B(M)) \otimes_R H^0(B(M)) \nonumber 
\end{align}
and we can check that 
$H^0(B(M)) := (H^0(B(M)), i_H, e_H, \wedge_H, \Delta_H)$ 
forms a Hopf algebra over $R$. 

Now we go back to our geometric situation.  
Assume that we are in the framework of Notation \ref{notation}. 
Let $\cA_{X/S}, \cA_{X/S^m(r)} \,(r = 1,2, m \in \N)$ be the sheaf of 
augmented cdga's defined in the previous section (Definition \ref{def:axs} and what follows), and 
let 
$M_{X/S}, M_{X/S^m(r)}$ 
be their $1$-minimal models. Recall that we also have 
isomorphisms 
\begin{align*}
& \zeta_j^m: (p_j^m)^*M_{X/S} \os{\cong}{\lra} M_{X/S^m(1)} \quad (j = 1,2, m \in \N), \\ 
& \zeta_{j,j'}^m: (p_{j,j'}^m)^*M_{X/S^m(1)} \os{\cong}{\lra} M_{X/S^m(2)} 
\quad (1 \leq j, j'\leq 3, 
m \in \N) 
\end{align*}
with $\zeta_1^0 = \zeta_2^0$ 
which fits into a commutative diagram like \eqref{eq:big}. 
In particular, the composite 
$$ (\zeta_1^m)^{-1} \circ \zeta_2^m: (p_2^m)^*M_{X/S} \os{\cong}{\lra} 
(p_1^m)^*M_{X/S} $$
defines a stratification, and we saw in the previous section 
that $M_{X/S}^i \, (i \in \N)$ 
belongs to the ind-category of $\MICn(S/k)$. 
Also, $H^i(M_{X/S}) = R^if_{\dR *}(\cO_X,d) \, (i \in \N)$ belongs to 
$\MICn(S/k)$. Hence  
the conditions required for performing the bar construction above 
are satisfied for $M_{X/S}$ and 
$M_{X/S^m(r)}$. 
Therefore, we can define the $0$-th cohomologies of bar constructions  
$H^0(B(M_{X/S}))$, $H^0(B(M_{X/S^m(r)}))$, and we have the isomorphisms 
\begin{align*}
& H^0(B(\zeta_j^m)): (p_j^m)^*H^0(B(M_{X/S})) \os{\cong}{\lra} H^0(B(M_{X/S^m(1)})) \quad (j = 1,2, m \in \N), \\ 
& H^0(B(\zeta_{j,j'}^m)): (p_{j,j'}^m)^*H^0(B(M_{X/S^m(1)})) \os{\cong}{\lra} H^0(B(M_{X/S^m(2)})) 
\quad (1 \leq j, j'\leq 3, 
m \in \N) 
\end{align*}
which fit into a commutative diagram like \eqref{eq:big}. 
Hence the composite 
$$ H^0(B(\zeta_1^m))^{-1} \circ H^0(B(\zeta_2^m)): 
(p_2^m)^*H^0(B(M_{X/S})) \os{\cong}{\lra} (p_1^m)^*H^0(B(M_{X/S})) $$
defines a stratification on $H^0(B(M_{X/S}))$. 
Moreover, the Eilenberg--Moore spectral sequences for 
$B(M_{X/S}), B(M_{X/S^m(r)})$ are compatible with 
the maps $B(\zeta_j^m), B(\zeta_{j,j'}^m)$. Since their $E_1$-terms 
are objects in $\MICn(S/k)$, 
the Eilenberg--Moore spectral sequences can be understood as  
a spectral sequence in the ind-category of $\MICn(S/k)$. 
Thus we see that $H^0(B(M_{X/S}))$ is a Hopf algebra object in the ind-category of 
$\MICn(S/k)$. Using this, we give the fourth definition of 
relatively unipotent de Rham fundamental groups. 

\begin{defn}[{\bf Fourth definition of relatively unipotent $\pi_1$}]\label{def5}
Let the notations be as above. 
Then we define the relatively unipotent de Rham fundamental group $\pi^{\dR}_1(X/S,\iota)$ by $$\pi^{\dR}_1(X/S,\iota) := \Spec H^0(B(M_{X/S})), $$ 
which is an affine group scheme in $\MICn(S/k)$. 
\end{defn}

\begin{rem}
One may be tempted to consider the bar construction of 
cdga's $\cA_{X/S}, \cA_{X/S^m(r)}$. 
Note that we cannot perform the same construction as before 
because we do not know the flatness of $\cA_{X/S}^i, \cA_{X/S^m(r)}^i$. 
\end{rem}

\begin{rem}
We can also define the bar construction $B(M_{X_s/s})$ of the minimal augmented 
$k$-cdga $M_{X_s/s}$ associated to the morphism $X_s \lra s$, and 
$H^0(B(M_{X_s/s}))$ has a structure of Hopf algebra over $k$. 
Note that each cohomology group $H^n(B(M_{X/S})) \, (n \allowbreak \in \N)$  
has a structure of an object in the ind-category of $\MICn(S/k)$
defined by the composite $H^n(B(\zeta_1^m))^{-1} \circ H^n(B(\zeta_2^m))$, 
hence it is flat over $\cO_S$. 
This fact and the flatness of each term of $B(M_{X/S})$ over $\cO_S$ 
imply that the canonical morphism 
$$ H^0(B(M_{X/S})) \otimes_{\cO_S, s^*} k \lra H^0(B(M_{X_s/s})) $$
(where $s^*:\cO_S \lra \cO_s = k$ is the morphism induced by $s$) 
is an isomorphism. In other words, $H^0(B(M_{X_s/s}))$ is the image of 
$H^0(B(M_{X/S}))$ by the ind-version of the restriction functor 
$\MICn(S/k) \lra \MIC(s/s)$. 
\end{rem}

Then we have the following comparison result. 

\begin{thm} \label{15}
The first and the fourth definitions of $\pi_1^{\dR}(X/S,\iota)$ are 
canonically isomorphic. 
\end{thm}

Before proving Theorem \ref{15}, we give a proof of 
Theorem \ref{prop4_2} assuming Theorem \ref{15}. 

\begin{proof}[Proof of Theorem \ref{prop4_2} assuming Theorem \ref{15}]
We prove that the third and the fourth definitions of 
$\pi_1^{\dR}(X/S,\iota)$ are canonically isomorphic. 

First, we prove the claim that the pro-nilpotent Lie algebra 
in the pro-category of $\MICn(S/k)$ 
associated to 
$\Spec H^0(B(M_{X/S}))$ is naturally isomorphic to 
$(M_{X/S}^1)^{\vee}$. We put  
$$ QH^0(B(M_{X/S})) := \Ker (e_H)/\wedge_H(\Ker(e_H) \otimes \Ker (e_H)), $$
which is called the module of indecomposables of $H^0(B(M_{X/S}))$. 
Then it has a structure of a Lie coalgebra 
$$  QH^0(B(M_{X/S})) \lra \bigwedge^2 QH^0(B(M_{X/S})) $$
induced by the map 
$$ \Delta_H - \tau \circ \Delta_H : 
H^0(B(M_{X/S})) \lra H^0(B(M_{X/S})) \otimes H^0(B(M_{X/S})), $$
where $\tau$ is the endomorphism on $H^0(B(M_{X/S})) \otimes H^0(B(M_{X/S}))$ 
defined by $x \otimes y \mapsto y \otimes x$. 
On the other hand, the minus of the differential $-d:M_{X/S}^1 \lra M_{X/S}^2 = \bigwedge^2 M_{X/S}^1$ 
defines a structure of Lie coalgebra on $M_{X/S}^1$. Then 
it suffices to prove that these two Lie coalgebras (in the ind-category of $\MICn(S/k)$) 
are isomorphic, because the Lie algebra associated to the group scheme 
$\Spec H^0(B(M_{X/S}))$ (in the ind-category of $\MICn(S/k)$) is 
the dual of $QH^0(B(M_{X/S}))$. 

One can check that the projection 
$B(M_{X/S}) \lra B(M_{X/S})^{-1,1} = M_{X/S}^1[1]$ 
(where $[1]$ denotes the shift of the degree by $1$) 
induces the morphism 
$$  QH^0(B(M_{X/S})) \lra M_{X/S}^1 $$
of Lie coalgebras on $S$, and since this construction is functorial, 
this is actually a morphism of Lie coalgebra objects in the ind-category of 
$\MICn(S/k)$. Therefore, to prove that it is an isomorphism, it suffices to prove that 
the restriction of this morphism to $\MIC(s/s)$, which is the morphism 
$$  QH^0(B(M_{X_s/s})) \lra M_{X_s/s}^1, $$
is isomorphic, because the functor $\MICn(S/k) \lra \MIC(s/s)$ is 
a fiber functor, hence exact and faithful. This is proven by Bloch--Kriz (\cite[Theorem 2.30]{bk}, see also 
\cite[Theorem 2.6.2]{hain}). 
So we have proven the required claim. 

Then, by Proposition \ref{prop:unip2020sep}, it suffices to prove that 
the affine group scheme $\Spec H^0(B(M_{X/S}))$ is a pro-object of 
the category ${\rm UG}(S/k)$ defined in Section 4. Because 
$H^0(B(M_{X/S}))$ is the inductive limit of $H^0(B(M_{X/S}(q))) \, (q \in \N)$, 
we are reduced to proving that the affine group scheme 
$\Spec H^0(B(M_{X/S}(q)))$ is an object of ${\rm UG}(S/k)$. 
Note that we have the equality 
$$ H^0(B(M_{X/S}(q))) = \bigcup_{s \geq 0} \Fil^{-s}H^0(B(M_{X/S}(q))), $$ 
where $\Fil^{-s}H^0(B(M_{X/S}(q))) \, (s \geq 0)$ is the filtration defined 
by the Eilenberg--Moore spectral sequence. Because 
the Eilenberg--Moore spectral sequence is a spectral sequence in 
the ind-category of $\MICn(S/k)$ and the graded quotient 
$${\rm gr}^{-s}H^0(B(M_{X/S}(q))) := \Fil^{-s}H^0(B(M_{X/S}(q)))/\Fil^{-s+1}H^0(B(M_{X/S}(q)))$$ 
is a subquotient of $E_1^{-s,s} = \bigotimes^s_{\cO_S} H^1(M_{X/S}(q))$, 
${\rm gr}^{-s}H^0(B(M_{X/S}(q))) \, (s \geq 0)$ 
are objects in $\MICn(S/k)$. Therefore, locally on $S$, 
$H^0(B(M_{X/S}(q)))$ is isomorphic to 
$\bigoplus_{s=0}^{\infty} {\rm gr}^{-s} H^0(B(M_{X/S}(q)))$ 
as an $\cO_S$-module and so it is a locally projective $\cO_S$-module. 

Moreover, we have the isomorphism 
$QH^0(B(M_{X_s/s}(q))) \cong M_{X_s/s}(q)^1$ of finite dimensional 
Lie coalgebras by \cite[Theorem 2.30]{bk} and then, by \cite[Corollary 2.31]{bk}, 
$H^0(B(M_{X_s/s}(q)))$ is generated by some finite dimensional $k$-vector space 
as $k$-algebra. Thus there exists some $t \geq 0$ such that 
$\Fil^{-t}H^0(B(M_{X_s/s}(q)))$ generates $H^0(B(M_{X_s/s}(q)))$ as $k$-algebra, namely, 
the natural map 
$$ {\rm Sym}\,\Fil^{-t}H^0(B(M_{X_s/s}(q))) \lra H^0(B(M_{X_s/s}(q)))$$ 
is surjective. Then, since the cokernel of the natural map 
\begin{equation}\label{eq:unip2020sep}
{\rm Sym}\,\Fil^{-t}H^0(B(M_{X/S}(q))) \lra H^0(B(M_{X/S}(q)))
\end{equation}
has a structure of an ind-object in $\MICn(S/k)$ and its fiber at $s$ is zero, 
we see that the map \eqref{eq:unip2020sep} is surjective, that is, 
$\Fil^{-t}H^0(B(M_{X/S}(q)))$ generates $H^0(B(M_{X/S}(q)))$ as 
$\cO_S$-algebra. So we have shown that $\Spec H^0(B(M_{X/S}(q)))$ 
satisfies the assumptions in Lemma \ref{lem:unip2020sep} and thus satisfies 
the condition $(\star)$ in Section 4. (Also, we see that 
the geometric fibers of $\Spec H^0(B(M_{X/S}(q)))$ are algebraic groups.) 

By definition of the coproduct $\Delta$ on 
$H^0(B(M_{X/S}(q)))$, we have 
\begin{align*}
& \Fil^{0}H^0(B(M_{X/S}(q))) = \cO_S, \\ 
& \Delta(\Fil^{-s}H^0(B(M_{X/S}(q)))) \subseteq 
\sum_{t=0}^s \Fil^{-s+t}H^0(B(M_{X/S}(q))) \otimes \Fil^{-t}H^0(B(M_{X/S}(q))).  
\end{align*}
So, by \cite[Theorem 14.5]{milne}, the geometric fibers of $\Spec H^0(B(M_{X/S}(q)))$ 
are unipotent algebraic groups. In conclusion, $\Spec H^0(B(M_{X/S}(q)))$ 
is a unipotent group scheme in $\MICn(S/k)$ satisfying the condition $(\star)$, that is, 
$\Spec H^0(B(M_{X/S}(q)))$ belongs to ${\rm UG}(S/k)$. 
So the proof of the theorem is finished. 
\end{proof}

\begin{rem}
The definition of Lie coalgebra on $QH^0(B(M_{X/S}))$ follows \cite[2.6]{hain} 
(hence follows \cite[(6.24)]{hain-memoir}), 
because the one given in \cite[Lemma 2.29(a)]{bk} is not correct. 
(The coproduct $\Delta_H$ itself does not induce the Lie coalgebra structure.) 
Also, our definition of 
the Lie coalgebra structure on $M_{X/S}^1$ involves the minus sign. We can 
check that the morphism $  QH^0(B(M_{X/S})) \lra M_{X/S}^1 $ in the proof 
is compatible with these corrected definitions. 
\end{rem}

To prove Theorem \ref{15}, we define the category of 
$\cA_{X/S}$-connections and prove some of its basic properties. Let $\MIC(\cA_{X/S})$ be the category of locally free modules with 
integrable $\cA_{X/S}$-connection, that is, the category of 
pairs $(E,\nabla)$ consisting of a locally free $\cA_{X/S}^0$-module 
$E$ of finite rank and 
an integrable connection $\nabla: E \lra E \otimes_{\cA_{X/S}^0} \cA_{X/S}^1$ 
(an additive map satisfying the Leibniz rule $\nabla(ae) = a \nabla(e) + e \otimes da 
\,(a \in \cA_{X/S}^0, e \in E)$ such that the composite 
$$ E \os{\nabla}{\lra} E \otimes_{\cA_{X/S}^0} \cA_{X/S}^1 
\os{\nabla_1}{\lra} E \otimes_{\cA_{X/S}^0} \cA_{X/S}^2, $$
where $\nabla_1$ is the map defined by $e \otimes \omega 
\mapsto \nabla(e) \wedge \omega + e \otimes d\omega$, 
is zero). We can also define the category  $\MIC(\cA_{X/S^m(r)})$ for $r = 1,2, m \in \N$. 
The categories $\MIC(\cA_{X/S})$, $\MIC(\cA_{X/S^m(r)})$ are 
exact tensor categories in natural way. We have `the pullback functors'  
{\allowdisplaybreaks{
\begin{align*}
& f^*: \MIC(S/S) \lra \MIC(\cA_{X/S}), \\ 
& {f^m}^*: \MIC(S^m(r)/S^m(r)) \lra \MIC(\cA_{X/S^m(r)}) \quad (r=1,2, m \in \N), \\ 
& {q_j^m}^*: \MIC(\cA_{X/S}) \lra \MIC(\cA_{X/S^m(1)}) \quad (j=1,2, m \in \N), \\  
& {q_{j,j'}^m}^*: \MIC(\cA_{X/S^m(1)}) \lra \MIC(\cA_{X/S^m(2)}) \quad (1 \leq j < j' \leq 3, m \in \N),  \\ 
&  \eta^*: \MIC(\cA_{X/S^0(1)}) \lra \MIC(\cA_{X/S}) 
\end{align*}}}%
which are induced by the morphisms $\cO_S \lra \cA_{X/S}, 
\cO_{S^m(r)} \lra \cA_{X/S^m(r)}, 
\cA_{X/S} \lra \cA_{X/S^m(1)}$, $\cA_{X/S^m(1)} \lra \cA_{X/S^m(2)}$ 
and $\cA_{X/S^0(1)} \lra \cA_{X/S}$
defined by $f, f^m, \wh{q}_j^m$, the analogue of $\wh{q}_j^m$ for the map $q_{j,j'}^m$ 
and the diagonal $X \lra \widehat{X}^0$, 
respectively. These functors are exact tensor functors. 

Let $\NfMIC(\cA_{X/S})$ be the full subcategory of 
$\MIC(\cA_{X/S})$ 
whose objects are iterated 
extensions of objects in $f^*\MIC(S/S) \subseteq \MIC(\cA_{X/S})$. 
It is also an exact tensor category in a natural way. 

Also, we define a category $\cE$ and its full subcategory $\Nf\cEn$ 
which will play an important role in this section. 

\begin{defn}
Let $\cE$ be the category of pairs $(E,\{\epsilon^m\})$ consisting of 
an object $E$ in $\MIC(\cA_{X/S})$ and a compatible 
family of isomorphisms $($called stratifications$)$ 
$\{\epsilon^m:{q_2^m}^*E \os{\cong}{\lra} {q_1^m}^*E\}$ 
in $\MIC(\cA_{X/S^m(1)})$ with $\eta_*(\epsilon^0) = \id$ 
which satisfies the cocycle condition 
${q_{1,2}^m}^*(\epsilon^m) \circ {q_{2,3}^m}^*(\epsilon^m) 
= {q_{1,3}^m}^*(\epsilon^m)$ in $\MIC(\cA_{X/S^m(2)}) \, (m \in \N)$. 
Then the pullback functors $f^*, {f^m}^*$ above induce the pullback functor 
$$ f^*: \MIC(S/k) \cong \Strat(S/k) \lra \cE. $$

Let $\Nf\cEn$ be the full subcategory of $\cE$ 
whose objects are iterated 
extensions of objects in the essential image of the functor 
$$ \MICn(S/k) \hra \MIC(S/k) \os{f^*}{\lra} \cE. $$
\end{defn}
Note that $\Nf\cEn$ is also an exact tensor category in a natural way. 

We say an object in $\NfMIC(\cA_{X/S})$ (resp. $\Nf\cEn$) is of index of unipotence $\leq n$ 
if it can be written as an iterated extension of at most $n$ objects in 
$f^*\MIC(S/S)$ (resp. $f^*\MICn(S/k)$). 

An object $(E,\nabla)$ in $\MIC(\cA_{X/S})$ naturally induces the de Rham complex 
$E \otimes_{\cA_{X/S}^0} \cA_{X/S}$. We define the $i$-th relative 
de Rham cohomology $R^if_*E$ of $(E,\nabla)$ by 
$$R^if_* E := H^i(E \otimes_{\cA_{X/S}^0} \cA_{X/S}), $$ 
and we put $f_*E := R^0f_*E$. 
For $E \in \MIC(S/S)$ and $E' \in \MIC(\cA_{X/S})$, we have the 
projection formula 
$$ R^if_*(f^*E \otimes E')= E \otimes R^if_*E'. $$
When $E' = (\cA_{X/S}^0,d)$, it is written as 
$$ R^if_*f^*E = E \otimes H^i(\cA_{X/S}) = E \otimes R^if_{\dR *}(\cO_X,d). $$ 
We define the $i$-th relative de Rham cohomology 
$R^if^m_*E$ for $E \in \MIC(\cA_{X/S^m(r)})$ in the same way, and 
the projection formula holds also in this case. 

When we are given an object $(E,\{\epsilon^m\})$ in $\cE$, 
we have functorial diagrams  
\begin{equation}\label{eq:axs1}
{p_2^m}^* R^if_*E \lra R^if^m_*{q_2^m}^*E \os{R^if^m_*\epsilon^m}{\lra} 
R^if^m_*{q_1^m}^*E \lla {p_1^m}^* R^if_*E \quad (m \in \N)
\end{equation}
in which the second arrow is an isomorphism. We do not know 
whether the first and the third arrows are isomorphisms in general, but 
they are isomorphisms when $E$ belongs to 
$f^*\MICn(S/k)$, by projection formula and 
the isomorphism \eqref{eq:gm} for $(\ol{E},\ol{\nabla}) = (\cO_X,d)$. 
Hence we see by induction that 
the arrows in \eqref{eq:axs1} are isomorphisms when 
$(E,\{\epsilon^m\})$ belongs to $\Nf\cEn$, and the diagrams \eqref{eq:axs1} 
induce the structure of an object in $\MICn(S/k) \subseteq \MIC(S/k) \cong 
\Strat(S/k)$. So we see that the functor $R^if_*$ induces the functor 
\begin{equation}\label{eq:dr-a}
R^if_*: \Nf\cEn \lra \MICn(S/k). 
\end{equation}
When $i=0$, we denote this functor by $f_*$. 
We have a morphism of functors $f^*f_* \lra \id$ from 
$\Nf\cEn$ to itself, and we can check (by using projection formula) that, 
for any $E \in \Nf\cEn$, $f^*f_*E \lra E$ gives an injection onto the maximal 
subobject of $E$ which belongs to $f^*\MICn(S/k)$. 

\begin{prop}\label{prop:tannakaen}
$\Nf\cEn$ is a neutral Tannakian category. 
\end{prop}

\begin{proof}
We prove that $\Nf\cEn$ is an abelian category. 
(Then it is easy to see that $\Nf\cEn$ is a rigid tensor abelian category 
and a fiber functor is defined by the composition   
$$ \Nf\cEn \os{\iota^*}{\lra} \MICn(S/k) \lra \MIC(s/s) = \Vector_k $$
of `the pullback functor' $\iota^*$ by the augmentations 
$\cA_{X/S} \lra \cO_{S}$, 
$\cA_{X/S^m(r)} \lra \cO_{S^m(r)}$ 
and the functor  `taking fiber at $s$'.) 

Let $\alpha: E \lra F$ be a morphism in $\Nf\cEn$ with 
$E$ (resp. $F$) of index of unipotence $\leq n$ (resp. $\leq m$). 
We prove that the kernel of $\alpha$ exists and it is of index of unipotence 
$\leq n$ and that the cokernel of $\alpha$ exists and it is of index of unipotence 
$\leq m$. 

When $n = m = 1$,  $E,F$ belong to $f^*\MICn(S/k)$ and so 
$\alpha$ is equal to $f^*f_*\alpha: f^*f_*E \lra f^*f_*F$. 
Because $f^*$ is exact on $\MICn(S/k)$ and $\MICn(S/k)$ is assumed to be 
abelian by condition (E) of Notation \ref{notation}, 
we see that $f^*\Ker (f_*\alpha)$ is the kernel of $\alpha$ and 
that $f^*\Coker (f_*\alpha)$ is the cokernel of $\alpha$. This shows the claim 
in this case. 

Next we prove by induction  the claim for  general  $n$ and $m=1$. 
We have an exact sequence in $\Nf\cEn$ 
$$
0 \longrightarrow E^{\prime} \longrightarrow E \longrightarrow E^{\prime \prime} 
\longrightarrow 0
$$
with $E'$ (resp. $E''$) of index of unipotence $\leq n-1$ (resp. $\leq 1$). 
So we have the commutative diagram  
$$
\xymatrix{ 
0 \ar[r]&  E^{\prime} \ar[d]^-{\alpha'}  \ar[r]&  E \ar[r] \ar[d]^-{\alpha} & E^{\prime \prime} \ar[r]  \ar[d] & 0 \\
0 \ar[r]&  F   \ar[r]&  F \ar[r]  & 0  \ar[r]&0}
$$
in $\Nf\cEn$ with exact rows. 
It induces the following exact sequence of $\cA_{X/S}^0$-modules: 
$$
0 \rightarrow \Ker\,\alpha' \rightarrow \Ker\,\alpha \rightarrow {E^{\prime \prime}} 
\os{\delta}{\ra} \Coker\,\alpha' \rightarrow \Coker\,\alpha \rightarrow 0. 
$$
By induction hypothesis, $\Coker\,\alpha'$ is an object in 
$\Nf\cEn$ of index of unipotence $\leq 1$. 
We see also, as above, that 
$\delta$ is enriched to a morphism in $\Nf\cEn$ by functoriality of the connecting map $\delta$: 
in fact, there are analogous diagrams and exact sequences of $\cA_{X/S^m(r)}^0$-modules
which are compatible with respect to stratifications. 
Thus $\Coker\,\alpha$ is an object in $\Nf\cEn$ of index of unipotence $\leq 1$. 
Moreover, by induction hypothesis, 
the kernel of $\delta$ is an object in $\Nf\cEn$ of index of unipotence $\leq 1$ and $\Ker\,\alpha'$ is 
an object in $\Nf\cEn$ of index of unipotence $\leq n-1$. 
Thus $\Ker\,\alpha$ is an object in 
$\Nf\cEn$ of index of unipotence $\leq n$. Hence the claim is proved in this case. 

Finally we prove the claim in general case by induction on $m$. 
We have an exact sequence in $\Nf\cEn$ 
$$
0 \longrightarrow F^{\prime} \longrightarrow F \longrightarrow F^{\prime \prime} 
\longrightarrow 0
$$
with $F'$ (resp. $F''$) of index of nilpotence $\leq 1$ (resp. $\leq m-1$). 
So we have the commutative diagram  
$$
\xymatrix{ 
0 \ar[r]&  0 \ar[d] \ar[r]&  E \ar[r] \ar[d]^{\alpha} & E \ar[r]  \ar[d]^{\alpha'}  & 0 \\
0 \ar[r]&  F' \ar[r]&  F \ar[r]  & F''  \ar[r]&0}
$$
in $\Nf\cEn$ with horizontal lines exact. 
It induces the following exact sequence of $\cA_{X/S}^0$-modules: 
$$
0 \rightarrow \Ker\,\alpha \rightarrow \Ker\,\alpha' \os{\delta}{\rightarrow} F'  
\ra \Coker\,\alpha \rightarrow \Coker\,\alpha' \rightarrow 0. 
$$
By induction hypothesis, $\Ker\,\alpha'$ is an object in 
$\Nf\cEn$ of index of unipotence $\leq n$, and we see that  
$\delta$ is enriched to a morphism in $\Nf\cEn$ by functoriality of the connecting map $\delta$. 
Thus $\Ker\,\alpha$ is an object in $\Nf\cEn$ of index of unipotence $\leq n$. 
Moreover, by induction hypothesis, 
the cokernel of $\delta$ is an object in $\Nf\cEn$ of index of unipotence $\leq 1$ 
and $\Coker\,\alpha'$ is an object in 
$\Nf\cEn$ of index of unipotence $\leq m-1$. 
Thus $\Coker\,\alpha$ is an object in 
$\Nf\cEn$ of index of unipotence $\leq m$. Hence the claim is proved, and so 
the proof of the proposition is finished. 
\end{proof}

By using the functors between Tannakian categories 
$$ f^*: \MICn(S/k) \lra \Nf\cEn, \quad 
\iota^*: \Nf\cEn \lra \MICn(S/k), $$ 
we can define the group scheme $G(\Nf\cEn, \iota^*)$ in $\MICn(S/k)$ 
by the method explained in Section 2. 

Also, recall that we defined the augmented $k$-cdga $\cA_{X_s/s}$ 
in Remark \ref{rem:mxss}. Hence we can define 
the category $\MIC(\cA_{X_s/s})$ of locally free $\cA_{X_s/s}$-modules 
of finite rank endowed with integrable $\cA_{X_s/s}$-connection 
and the pullback functor 
$f_s^*: \MIC(\cA_{X_s/s}) \allowbreak \lra \MIC(s/s)$ induced by 
the canonial morphism $k \hra \cA_{X_s/s}$. 
Let $\Nfs\MIC(\cA_{X_s/s})$ be the full subcategory whose objects are 
iterated extensions of objects in $f_s^*\MIC(s/s)$. 
An object $(E,\nabla)$ in $\MIC(A_{X_s/s})$ naturally induces the de Rham complex 
$E \otimes_{\cA_{X_s/s}^0} \cA_{X_s/s}$, and we define the $i$-th 
de Rham cohomology $R^if_{s *}E$ of $(E,\nabla)$ by 
$$R^if_{s *} E := H^i(E \otimes_{\cA_{X_s/s}^0} \cA_{X_s/s}). $$ 
We put $f_{s *}E := R^0f_{s *}E$. 
As in the previous case, 
for $E \in \MIC(s/s)$ and $E' \in \MIC(\cA_{X_s/s})$, we have the 
projection formula 
$$ R^if_{s *}(f^*_s E \otimes E')= E \otimes R^if_{s *}E' $$
and when $E' = (\cA_{X_s/s}^0,d)$, it is written as 
$$ R^if_{s *}f^*_s E = E \otimes H^i(\cA_{X_s/s}) = E \otimes R^if_{s \dR *}(\cO_{X_s},d). $$ 
We have a morphism of functors $f_s^*f_{s*} \lra \id$ from 
$\Nfs\MIC(\cA_{X_s/s})$ to itself, and we can check (by using projection formula) that, 
for any $E \in \Nfs\MIC(\cA_{X_s/s})$, $f_s^*f_{s*}E \lra E$ gives an injection onto the 
maximal trivial subobject of $E$.

We have the restriction functor 
$\Nf\cEn \os{|_{X_s}}{\lra} \Nfs\MIC(\cA_{X_s/s})$
 induced by the morphism $\cA_{X/S} \lra \cA_{X_s/s}$ defined by the closed immersion 
$X_s \hra X$. 
By using the projection formula and induction on index of unipotence, 
we see that, for 
$E \in \Nf\cEn$, there exists the base change isomorphism 
\begin{equation}\label{eq:bcp-a}
(R^if_*E)|_s \os{\cong}{\lra} R^if_{s *}(E|_{X_s}), 
\end{equation}
where we denote the restriction $\MICn(S/k) \lra \MIC(s/s)$ by $|_s$. 
By the same reason, we have the similar base change isomorphism 
with respect to any morphism $S' \lra S$ of finite type 
for $E \in \Nf\cEn$. 

We can check in the same way as the proof of Proposition 
\ref{prop:tannakaen} that $\Nfs\MIC(\cA_{X_s/s})$ is a neutral 
Tannakian category with 
the fiber functor $x^*: \Nfs\MIC(\cA_{X_s/s}) \lra \Vector_k$ 
induced by the augmentation map $\cA_{X_s/s} \lra k$. 
So we can define the Tannaka dual 
of $(\Nfs\MIC(\cA_{X_s/s}), x^*)$, which we denote by 
$G(\Nfs\MIC(\cA_{X_s/s}) , x^*)$. Moreover, 
since the composite of pullback functors 
$$ 
\MICn(S/k) \os{f^*}{\lra} \Nf\cEn \os{|_{X_s}}{\lra} \Nfs\MIC(\cA_{X_s/s})
$$ 
factors through $\MIC(s/s) = \Vector_k$, 
the functor $|_{X_s}$ induces the morphism 
\begin{equation}\label{fibera}
G(\Nfs\MIC(\cA_{X_s/s}) , x^*) \lra s_{\dR}^*G(\Nf\cEn, \iota^*). 
\end{equation}

Then we have the following: 

\begin{prop}\label{fiberpropa}
The morphism \eqref{fibera} is an isomorphism.  
\end{prop}

\begin{proof}
We imitate the proof of Proposition \ref{fiberprop}. 
It suffices to prove the following three claims. 
\begin{enumerate}
\item Any object in $\NfsMIC(\cA_{X_s/s})$ is a quotient of an object of the form 
$E|_{X_s}$ \, $(E \in \Nf\cEn)$.  
\item For any object $E$ in $\Nf\cEn$ with $E|_{X_s}$ trivial, 
there exists an object $V$ in $\MICn(S/k)$ with $E = 
f^*V$. 
\item Let $E$ be an object in $\Nf\cEn$ and let 
$F_0 \in \Nfs\MIC(\cA_{X_s/s})$ be 
the largest trivial subobject of $E|_{X_s}$. Then there exists 
a subobject $E_0$ of $E$ with 
$F_0 = E_0|_{X_s}$. 
\end{enumerate}

First we prove the claim (2). By restricting the map 
$f^*f_*E \lra E$ to $\MIC(\cA_{X_s/s})$ and using base change 
property \eqref{eq:bcp-a}, we obtain the morphism 
\begin{align*}
f_{s}^*f_{s *}(E|_{X_s}) 
\cong f_{s}^*((f_{*}E)|_s) 
& = (f^*f_{*}E)|_{X_s} \lra E|_{X_s}, 
\end{align*}
which is an isomorphism due to the triviality of 
$E|_{X_s}$. Thus the map $f^*f_*E \lra E$ is an isomorphism 
and so the claim holds if we put $V := f_{*}E$. 

Next we prove the claim (3). If $E$ and $F_0 \subseteq 
E|_{X_s}$ be as in the statement of the claim, 
we have $F_0 = f_{s}^*f_{s *}(E|_{X_s})$. 
So, by the base change property \eqref{eq:bcp-a}, we obtain the isomorphism 
\begin{align*}
F_0 = f_{s}^*f_{s *}(E|_{X_s}) 
\cong f_{s}^*((f_{*}E)|_s)
= (f^*f_{*}E)|_{X_s}. 
\end{align*}
So the claim holds if we put $E_0 = f^*f_{*}E$. 

Finally, in order to prove the claim (1), 
it suffices to construct a projective system 
of objects $\{W_n\}_n$ in $\Nf\cEn$ such that, 
for any object $F$ in $\NfsMIC(\cA_{X_s/s})$, there exist 
$n, N \in \N$ and a surjection 
$W_n^{\oplus N}|_{X_s} \lra F$. 
This will be done in Proposition \ref{prop:wn2} 
and Remark \ref{rem:wn2} 
below. 
\end{proof}

\begin{prop}\label{prop:wn2}
There exists 
a projective system 
$(W,e) := \{(W_n, e_n)\}_{n \geq 1}$ consisting of 
\begin{itemize}
\item 
$W_n := (\ol{W}_{\mkern-4 mu n}, \{\epsilon_n^m\}_m) \in \Nf\cEn$, and 
\item 
A morphism $e_n: (\cO_S,d) \lra \iota^*W_n$ in $\MICn(S/k)$
\end{itemize}
which satisfies the following conditions: 
\begin{enumerate}
\item[{\rm (W0)'}] For any $n \geq 2$, the morphism 
$f_*(\ol{W}_{n-1}^{\vee}) \lra f_*(\ol{W}_{\mkern-4 mu n}^{\vee})$ induced by 
the transition map $\ol{W}_{\mkern-4 mu n} \lra \ol{W}_{n-1}$ is an isomorphism. 
\item[{\rm (W1)'}] For any $E \in \NfMIC(\cA_{X/S})$ of index of unipotence $\leq n$ and 
a morphism $v: \cO_S \lra \iota^* E$ in $\MIC(S/S)$, there exists a unique 
morphism $\varphi: \ol{W}_{\mkern-4 mu n} \lra E$ in $\NfMIC(\cA_{X/S})$ with 
$\iota^*(\varphi) \circ \ol{e}_n = v$, where 
$\ol{e}_n: \cO_S \lra \iota^*\ol{W}_{\mkern-4 mu n}$ is the underlying morphism 
of $e_n$ in $\NfMIC(\cA_{X/S})$. Moreover, the same universality holds after base change 
by any morphism $S' \lra S$ of finite type. 
\item[{\rm (W2)'}]
For any $n$, there exists an exact sequence 
\begin{equation}\label{wn-ext0-a}
0 \lra f^*R^1f_{*} (W_n^{\vee})^{\vee} \lra W_{n+1} \lra 
W_n \lra 0 
\end{equation}
in $\Nf\cEn$, where $R^1f_{*}: \Nf\cEn \lra \MICn(S/k)$ 
is the functor defined in \eqref{eq:dr-a}
\end{enumerate}
\end{prop}

\begin{proof}
We repeat the proof of Theorem \ref{thm:wn}, with some mild 
modifications. For $n=1$, we put $W_1 := (\cA_{X/S}^0,d)$ and 
define $e_1$ to be the canonical isomorphism 
$(\cO_S,d) \os{=}{\lra} \iota^*(\cA_{X/S}^0,d)$. 
Then we can check the condition (W1)' for $n=1$ in the same way as 
the proof of Theorem \ref{thm:wn}.   

We construct $(W_{n+1}, e_{n+1})$ from $(W_i,e_i) \, (i \leq n)$. 
By induction hypothesis, we see that 
$R^1f_{*}(\ol{W}_{\mkern-4 mu n}^{\vee})$ belongs to $\MIC(S/S)$. 
So, when $S$ is affine, we can define 
$\ol{W}_{n+1} \in \NfMIC(\cA_{X/S})$ as the extension 
\begin{equation}\label{wn-ext-a}
0 \lra f^*R^1f_{*} (\ol{W}_{\mkern-4 mu n}^{\vee})^{\vee} \lra \ol{W}_{n+1} \lra 
\ol{W}_{\mkern-4 mu n} \lra 0 
\end{equation}
whose extension class, which is regarded as an element in 
$$\Gamma(S,  R^1f_{*} (\ol{W}_{\mkern-4 mu n}^{\vee}) \otimes R^1f_{*} (\ol{W}_{\mkern-4 mu n}^{\vee})^{\vee})  = 
\End (R^1f_{*} (\ol{W}_{\mkern-4 mu n}^{\vee})),$$ is the identity. 
Also, let $\ol{e}_{n+1}: \cO_S \lra \iota^*\ol{W}_{n+1}$ be a morphism 
which lifts the map $\ol{e}_n: \cO_S \lra \iota^*\ol{W}_{\mkern-4 mu n}$. 

We need to check the properties (W0)' and (W1)' 
for $(\ol{W}_{n+1},\ol{e}_{n+1})$, assuming $S$ and $S'$ affine. 
By using the base change property after \eqref{eq:bcp-a}, 
we see that the proof on $X \times_X S'/S'$ 
is the same as that on $X/S$  
and the proof of the properties on $X/S$ can be done 
in the same way as that of Theorem \ref{thm:wn}. 
Then we can define $(\ol{W}_{n+1},\ol{e}_{n+1})$ for general $S$ 
by gluing and check the properties (W0)', (W1)' for general $S$ and $S'$
in the same way as in the proof of Theorem \ref{thm:wn}. 

Next, let $f_j^m: X_j^m \lra S^m(1)$, 
$p_j^m: S^m(1) \lra S$, 
$q_j^m: X_j^m \lra X$, 
$\wh{f}^m: \wh{X}^m \lra S^m(1)$ and 
$\wh{q}_j^m: \wh{X}^m \lra X$ 
be as in Section 1, and let $\iota_j^m:S^m(1) \lra X_j^m$ be 
the base change of $\iota$, which is a section of 
$f_j^m$. Then we can define the sheaves of augmented cdga's 
$\cA_{X_j^m/S^m(1)}$ and the categories 
${\rm N}_{f_j^m}\MIC(\cA_{X_j^m/S^m(1)})$ from 
the morphism $f_j^m$. Also, we can define the category 
${\rm N}_{\wh{f}^m}\MIC(\cA_{X/S^m(1)})$ from 
the morphism $\wh{f}^m$. Then we have 
the quasi-isomorphisms 
$$ 
\cA_{X_2^m/S^m(1)} \lra 
\cA_{X/S^m(1)} \lla 
\cA_{X_1^m/S^m(1)}
$$ 
by Proposition \ref{prop:fib}, and these induce 
the equivalences of categories 
$$ 
{\rm N}_{f_2^m}\MIC(\cA_{X_2^m/S^m(1)}) 
\os{\cong}{\lra} 
{\rm N}_{\wh{f}^m}\MIC(\cA_{X/S^m(1)}) 
\os{\cong}{\lla} 
{\rm N}_{f_1^m}\MIC(\cA_{X_1^m/S^m(1)}). 
$$
Then, by the property (W1)', there exists 
uniquely the isomorphism 
$$\epsilon^m_{n+1}: ({q_2^m}^*\ol{W}_{n+1}, {p_2^m}^*\ol{e}_{n+1}) \os{\cong}{\lra} 
({q_1^m}^*\ol{W}_{n+1}, {p_1^m}^*\ol{e}_{n+1}) \, (m \in \N)$$ 
in the framework of the categories in the above diagram, 
which satisfies the cocycle condition. 
(To prove the cocycle condition, we need to work on 
pullbacks of $f:X \lra S$ to $S^m(2)$. We leave the reader to write 
out the detailed argument.) 
So the object $W_{n+1} := (\ol{W}_{n+1}, \{\epsilon_{n+1}^m\}_m)$ in 
the category $\cE$ is defined, and 
the isomorphisms ${p_2^m}^*\ol{e}_{n+1} \cong {p_1^m}^*\ol{e}_{n+1} \, (m \in \N)$ 
induces the map $e_{n+1}: (\cO_S,d) \lra \iota^*W_{n+1}$. 

Finally we need to prove that $W_{n+1}$ belongs to 
$\Nf\cEn$ and that it satisfies the property (W2)'. 
Since $W_n$ belongs to $\Nf\cEn$, 
we obtain the isomorphisms of Gauss--Manin connection 
$$ \delta_n^m: {p_2^m}^* R^1f_{*}(\ol{W}_{\mkern-4 mu n}^{\vee}) \os{\cong}{\lra} 
 {p_1^m}^* R^1f_{*}(\ol{W}_{\mkern-4 mu n}^{\vee}) \quad (m \in \N).
$$
by \eqref{eq:dr-a}. Using this, the rest of the proof  
can be done in the same way as the proof of 
Theorem \ref{thm:wn}. So the proof of the proposition is finished. 
\end{proof}

\begin{rem}\label{rem:wn2}
In this remark, which corresponds to Remark \ref{rem:sect2}, we give   the last step 
of the proof of Proposition \ref{fiberpropa}. 
By (W1)', for any object $E$ in $\NfsMIC(\cA_{X_s/s})$ and 
any morphism $v: k \lra x^* E$ in $\MIC(s/s) = \Vector_k$,  
there exists a unique 
morphism $\varphi_{v}: W_n|_{X_s} \lra E$ in $\NfsMIC(\cA_{X_s/s})$ with 
$x^*(\varphi) \circ (e_n|_{X_s}) = v$ for some $n$. 
By considering maps $v_1, ..., v_N: k \lra x^* E$ 
whose direct sum $k^{\oplus N} \lra x^*E$ is surjective, 
we obtain a surjective map 
$\oplus_{i=1}^N \varphi_{v_i}: W_n^{\oplus N}|_{X_s} \lra E$ for some $n$. 
\end{rem}

Now we compare the categories 
$\Nf\MICn(X/k)$ and $\Nf\cEn$, which is the first step of 
the proof of Theorem \ref{15}: 

\begin{prop}\label{prop:1st}
There exists an equivalence of categories 
$\Nf\MICn(X/k) \os{\cong}{\lra} \Nf\cEn$ over $\MICn(S/k)$. 
Equivalently, there exists an isomorphism 
$$ G(\Nf\cEn,\iota^*) \os{\cong}{\lra} \pi_1^{\dR}(X/S,\iota) $$ 
of affine group schemes in $\MICn(S/k)$, where $\pi_1^{\dR}(X/S,\iota)$ 
denotes the first definition of relatively unipotent de Rham fundamental group. 
\end{prop}

\begin{proof}
First we define the functor $\Nf\MICn(X/k) \lra \Nf\cEn$. 
Let $E := (E,\{\epsilon^m\})$ be an object in $\Nf\MICn(X/k) \cong 
\NfStrCrysn(X/k)$. 

Let $\pi: X_{\b} \lra X$, $\pi_{S}: S_{\b} \lra S$ be as in Section 4.
Consider the surjective morphism of semi-cosimplicial sheaves 
$$ 
\cA^{\b,*}(E) := f_*\pi_{*}((E \otimes \Omega^{*}_{X/S})|_{X_{\bullet}}) 
\os{\iota^*}{\lra} \pi_{S *} (\iota^*E|_{S_{\b}})
$$ 
induced from $\iota$, where $E \otimes \Omega^{*}_{X/S}$ denotes 
the de Rham complex associated to $E$, regarded as an object in 
$\NfMIC(X/S)$. 
By applying the Thom--Whitney functor $\sTW$
(\cite[\S 3]{naho}), we obtain 
the surjective morphism of complexes of sheaves 
$\sTW(\iota^*): \sTW(\cA^{\b,*}(E)) \lra \sTW(\pi_{S *} (\iota^*E|_{S_{\b}}))$ such that 
$H^0(\sTW(\pi_{S *} (\iota^*E|_{S_{\b}}))) \allowbreak = \iota^*E$. 
We define the complex of sheaves 
$\cA_{X/S}(E) \subseteq \sTW(\cA^{\b,*}(E))$ as the inverse image of 
$\iota^*E \hra \sTW(\pi_{S *} (\iota^*E|_{S_{\b}}))$ by $\sTW(\iota^*)$. 
By construction, $\cA_{X/S}(E)$ has a structure of 
differential graded module over $\cA_{X/S}$. 

Note first that  the functor $\NfMIC(X/S) \ni E \mapsto \cA_{X/S}(E)$ 
is exact.  In fact, if we are given an exact sequence in $\NfMIC(X/S)$
$$
0 \lra E' \lra E \lra E'' \lra 0, 
$$
we obtain the following commutative diagram with exact horizontal lines because $f \circ \pi$ and $\pi_S$ are affine, the objects in  $\NfMIC(X/S)$ are locally free and the functor  $\sTW$ is exact (\cite[\S 2]{naho}): 
\begin{equation*}
\xymatrix{ 
0 \ar[r]& \sTW(\cA^{\b,*}(E')) \ar[r]\ar[d]&  \sTW(\cA^{\b,*}(E)) \ar[r] \ar[d] &  \sTW(\cA^{\b,*}(E'')) \ar[d]\ar[r]   & 0\\
 0 \ar[r]&  \sTW(\pi_{S *} (\iota^*E'|_{S_{\b}}))   \ar[r]&  \sTW(\pi_{S *} (\iota^*E|_{S_{\b}}))  \ar[r] &  \sTW(\pi_{S *} (\iota^*E''|_{S_{\b}}))  \ar[r]   & 0\\
 0 \ar[r]& \iota^*E' \ar[r] \ar[u]& \iota^*E \ar[r] \ar[u] & \iota^*E'' \ar[u] \ar[r]   & 0.}
\end{equation*}
We can then consider the following diagram in which the vertical maps are surjective and the horizontal lines are exact:
\begin{equation*}
\xymatrix{ 
0 \ar[r]& \sTW(\cA^{\b,*}(E')) \oplus   \iota^*E' \ar[r]\ar[d]&  \sTW(\cA^{\b,*}(E))  \oplus   \iota^*E \ar[r] \ar[d] &  \sTW(\cA^{\b,*}(E''))   \oplus   \iota^*E'' \ar[d]\ar[r]   & 0\\
 0 \ar[r]&.  \sTW(\pi_{S *} (\iota^*E'|_{S_{\b}}))   \ar[r]&  \sTW(\pi_{S *} (\iota^*E|_{S_{\b}}))  \ar[r] &  \sTW(\pi_{S *} (\iota^*E''|_{S_{\b}}))  \ar[r]   & 0.}
\end{equation*}
It then follows that the sequence of the kernel is exact, but it is exactly the sequence 
$$
0 \lra \cA_{X/S}(E') \lra \cA_{X/S}(E) \lra  \cA_{X/S}(E'') \lra 0. 
$$
Thus we have proven that the functor $\cA_{X/S}(-)$ is exact.

We now prove that the pair $(\cA_{X/S}(E)^0, \cA_{X/S}(E)^0 \lra \cA_{X/S}(E)^1)$ 
defines an object in $\MIC(\cA_{X/S})$. To do so, it suffices to show  that 
$\cA_{X/S}(E)^0$ is a locally free $\cA^0_{X/S}$-module and that 
the natural maps
$\cA_{X/S}(E)^0 \otimes_{\cA^0_{X/S}} \cA^i_{X/S} \lra \cA_{X/S}(E)^i \, (i \in \N)$ 
are isomorphisms.  By the exactness of the functor 
$\cA_{X/S}(-)$,  it suffices to check the claim when $E \in \NfMIC(X/S)$ is of the form 
$f_{\dR}^*F \, (F \in \MIC(S/S))$. Also, since 
 the claim is local on $S$, 
we may assume that $F$ is free and 
the claim is trivially true in this case. So we have defined an object 
$$\cA_{X/S}(E)^0 := (\cA_{X/S}(E)^0, 
\cA_{X/S}(E)^0 \lra \cA_{X/S}(E)^0 \otimes_{\cA^0_{X/S}} \cA^1_{X/S})$$ 
in $\MIC(\cA_{X/S})$. 

Also, we can apply the same construction to 
$(\wh{q}_2^m)_{\dR}^* E \cong (\wh{q}_1^m)_{\dR}^* E \in \MIC(\wh{X}^m/S^m(1))$ 
(where the notation is as in Section 1, before Proposition \ref{basechange}) and 
obtain an object in $\MIC(\cA_{X/S^m(1)})$, which we denote by 
$$\cA_{X/S^m(1)}(E)^0 := (\cA_{X/S^m(1)}(E)^0, 
\cA_{X/S^m(1)}(E)^0 \lra \cA_{X/S^m(1)}(E)^0 \otimes_{\cA^0_{X/S^m(1)}} \cA^1_{X/S^m(1)}). $$ 
We can define the object 
$$\cA_{X/S^m(2)}(E)^0 := (\cA_{X/S^m(2)}(E)^0, 
\cA_{X/S^m(2)}(E)^0 \lra \cA_{X/S^m(2)}(E)^0 \otimes_{\cA^0_{X/S^m(2)}} \cA^1_{X/S^m(2)})$$ 
in the same way and we can check the 
compatibility of $\cA_{X/S}(E)^0$ and 
$\cA_{X/S^m(r)}(E)^0$'s $(r=1,2, m \in \N)$ by 
induction of index of unipotence of $E$. So they form an object in 
$\Nf\cEn$ and so we have defined a functor 
$\NfMICn(X/k) \lra \Nf\cEn$. 

This functor induces a 
morphism of group schemes $G(\Nf\cEn, \iota^*) \lra \pi_1(X/S,\iota)$ 
in $\MICn(S/k)$. To prove this is an isomorphism, it suffices 
to prove it after we apply the 
restriction functor $\MICn(S/k) \lra \MIC(s/s)$ which is exact and faithful. 
By Propositions \ref{fiberprop} and \ref{fiberpropa}, 
this is equivalent to the claim that the functor 
$\NfsMIC(X_s/s) \lra \Nfs\MIC(\cA_{X_s/s})$, which is defined 
from the morphism $X_s \lra s$  as above, 
is an equivalence 
of categories. Because objects on both hand sides 
are iterated extensions of trivial objects, 
it suffices to check that $\Hom$ and $\Ext$ groups of  such
objects are isomorphic. This is true because $\Hom$ and $\Ext$ 
can be understood as zeroth and first cohomology groups
and cohomology groups with trivial coefficient are isomorphic 
by construction of $\cA_{X_s/s}$ 
 (to obtain the equivalence, one actually  needs  isomorphism  
on zeroth and first cohomology groups and the injectivity on the second 
cohomology groups). 
\end{proof}

By using the $1$-minimal models $M_{X/S}, M_{X/S^m(r)}$ 
(resp. the $(1,q)$-minimal models $M_{X/S}, M_{X/S^m(r)}$)
instead of 
$\cA_{X/S}, \cA_{X/S^m(r)}$, we can define the category similar to 
$\Nf\cEn$, which we denote by $\Nf\cEn_M$ (resp. $\Nf\cEn_{M(q)}$). 
Next we compare the categories 
$\Nf\cEn$, $\Nf\cEn_M$, 
which is the second step of 
the proof of Theorem \ref{15}: 

\begin{prop}\label{prop:2nd}
There exists an equivalence of categories 
$\Nf\cEn_M \lra \Nf\cEn$ over $\MICn(S/k)$.  
Equivalently, there exists an isomorphism 
$$ G(\Nf\cEn,\iota^*) \os{\cong}{\lra} G(\Nf\cEn_M,\iota^*) $$ 
of affine group schemes in $\MICn(S/k)$. 
\end{prop} 

\begin{proof}
Recall that we have families of ho-morphisms 
$$ \rho_q: M_{X/S}(q) \lra \cA_{X/S}, \quad \rho^m(r)_q: M_{X/S^m(r)}(q) \lra \cA_{X/S^m(r)} \quad 
(r = 1,2, m \in \N, q \in \N) $$
which are compatible in a suitable sense up to unique Sullivan homotopy. 
One problem is that $\rho_q \,(q \in \N)$ are not necessarily morphisms and so 
it is not immediate that $\rho_q \, (q \in \N)$ induce the functor 
$\NfMIC(M_{X/S}) \lra \NfMIC(\cA_{X/S}).$ 
We will construct this functor by constructing it locally 
on the level of $(1, q)$-minimal case, 
patching the local constructions, and taking the union with respect to $q \in \N$. 

Fix $q \in \N$. Then 
there exists an affine open covering 
$S = \bigcup_{\alpha} U_{\alpha}$ of $S$ such that 
the above ho-morphism induces for each $U_{\alpha}$ 
the morphism 
$\rho_{q,\alpha}(U_{\alpha}): M_{X/S}(q)(U_{\alpha}) \lra \cA_{X/S}(U_{\alpha})$. 
We write by $\rho_{q,\alpha}: M_{X/S}(q)|_{U_{\alpha}} \lra 
\cA_{X/S}|_{U_{\alpha}}$ the associated morphism of sheaves. 
Then we can define the corresponding functor 
$$ \rho_{q,\alpha,*}: \NfMIC(M_{X/S}(q)|_{U_{\alpha}}) \lra 
\NfMIC(\cA_{X/S}|_{U_{\alpha}}). $$

We prove that $\rho_{q,\alpha,*}$'s are compatible with respect to $\alpha$. 
On $U_{\alpha\alpha'} := U_{\alpha} \cap U_{\alpha'}$, 
there exists a unique Sullivan homotopy 
$$h: M_{X/S}(q)|_{U_{\alpha\alpha'}} \lra \cO_S(t,dt) \otimes \cA_{X/S}|_{U_{\alpha\alpha'}} $$
between $\rho_{q,\alpha}|_{U_{\alpha\alpha'}}$ and 
$\rho_{q,\alpha'}|_{U_{\alpha\alpha'}}$. Then 
$\rho_{q,\alpha,*}|_{U_{\alpha\alpha'}}$ is written as the composite 
$$ \NfMIC(M_{X/S}(q)|_{U_{\alpha\alpha'}}) \os{h_*}{\lra}  
\Nf\MIC(\cO_S(t,dt) \otimes \cA_{X/S}|_{U_{\alpha\alpha'}})
\os{p_{0,*}}{\lra}
\NfMIC(\cA_{X/S}|_{U_{\alpha\alpha'}}) $$ 
(where $h_*$ is the functor induced by $h$ and 
$p_{0,*}$ is the functor induced by the morphism 
$p_0: \cO_S(t,dt) \otimes \cA_{X/S}|_{U_{\alpha\alpha'}} \lra 
\cA_{X/S}|_{U_{\alpha\alpha'}}$ appeared in the definition of Sullivan homotopy) 
and $\rho_{q,\alpha',*}|_{U_{\alpha\alpha'}}$ is written as a similar 
composite, with $p_0$ replaced by $p_1$. 
On the other hand, if we denote the inclusion 
$\cA_{X/S}|_{U_{\alpha\alpha'}} \hra \cO_S(t,dt) \otimes \cA_{X/S}|_{U_{\alpha\alpha'}}$ 
by $i$, it induces the functor 
$$ i_*: \NfMIC(\cA_{X/S}|_{U_{\alpha\alpha'}}) \lra 
\NfMIC(\cO_S(t,dt) \otimes \cA_{X/S}|_{U_{\alpha\alpha'}}) $$
and it is an equivalence because $i$ is a quasi-isomorphism. 
Since we have $p_0 \circ i = \id = p_1 \circ i$, we see that 
\begin{equation}\label{homot}
\rho_{q,\alpha,*}|_{U_{\alpha\alpha'}} = p_{0,*}h_* \cong i_*^{-1}h_* \cong  
p_{1,*}h_* = \rho_{q,\alpha',*}|_{U_{\alpha\alpha'}}. 
\end{equation}
We should check that the isomorphism \eqref{homot}, which we denote by 
$\gamma_{\alpha\alpha'}$, satisfies the cocycle condition. 
The isomorphism $\gamma_{\alpha\alpha'}$ is defined via the Sullivan homotopy 
$h$, and the isomorphism $\gamma_{\alpha'\alpha''}$ is defined in the same way via 
the unique Sullivan homotopy of the form 
$$ h': M_{X/S}(q)|_{U_{\alpha'\alpha''}} \lra \cO_S(t,dt) \otimes A_{X/S}|_{U_{\alpha'\alpha''}}. $$
Then, on $U_{\alpha\alpha'\alpha''} := 
U_{\alpha} \cap U_{\alpha'} \cap U_{\alpha''}$, we have the Sullivan homotopy 
$h|_{U_{\alpha\alpha'\alpha''}}$ between 
$\rho_{q,\alpha}|_{U_{\alpha\alpha'\alpha''}}$ and 
$\rho_{q,\alpha'}|_{U_{\alpha\alpha'\alpha''}}$ and the 
Sullivan homotopy $h'|_{U_{\alpha\alpha'\alpha''}}$ 
between $\rho_{q,\alpha'}|_{U_{\alpha\alpha'\alpha''}}$ and 
$\rho_{q,\alpha''}|_{U_{\alpha\alpha'\alpha''}}$. 
Then, we can construct the Sullivan homotopy $H$ between 
$\rho_{q,\alpha}|_{U_{\alpha\alpha'\alpha''}}$ and 
$\rho_{q,\alpha''}|_{U_{\alpha\alpha'\alpha''}}$
by the construction in the proof of Lemma \ref{lem:eqrel}, 
and we see 
from the construction that the isomorphism 
$\rho_{q,\alpha,*}|_{U_{\alpha\alpha'\alpha''}} \cong 
\rho_{q,\alpha'',*}|_{U_{\alpha\alpha'\alpha''}}$ we obtain from $H$ is 
the same as the composite $\gamma_{\alpha'\alpha''} \circ \gamma_{\alpha\alpha'}$. 
On the other hand, by the uniqueness of Sullivan homotopy in Propositon \ref{prop:1q}, 
we see that $H$ is equal to 
the restriction of the unique Sullivan homotopy 
$$ h'': M_{X/S}(q)|_{U_{\alpha\alpha''}} \lra \cO_S(t,dt) \otimes A_{X/S}|_{U_{\alpha\alpha''}} $$
between $\rho_{q,\alpha}|_{U_{\alpha\alpha''}}$ and 
$\rho_{q,\alpha''}|_{U_{\alpha\alpha''}}$ 
to $U_{\alpha\alpha'\alpha''}$. Thus we see the equality 
 $\gamma_{\alpha'\alpha''} \circ \gamma_{\alpha\alpha'} = \gamma_{\alpha\alpha''}$, 
namely, we proved the required cocycle condition. Therefore, the isomorphisms 
\eqref{homot} for $\alpha, \alpha'$ glue to give a functor 
$$ \rho_{q,*}: \NfMIC(M_{X/S}(q)) \lra \NfMIC(\cA_{X/S}). $$

By the same argument, we can construct the functors 
$$ \rho^m(r)_{q,*}: \NfMIC(M_{X/S^m(r)}(q)) \lra \NfMIC(\cA_{X/S^m(r)}) \quad 
(r=1,2, m \in \N) $$
and the functors $\rho_{q,*}, \rho^m(r)_{q,*}$ are compatible. 
(To prove the compatibility, we argue again  
with Sullivan homotopy because the diagram in Theorem \ref{thm:minmod}(2) is commutative only up to unique Sullivan homotopy.) Hence these functors induce the compatible functors  
$\Nf\cEn_{M(q)} \lra \Nf\cEn \, (q \in \N)$. 
(The compatibility here follows again from the argument, with Sullivan homotopy.) 
Because $\Nf\cEn_M$ is the inductive limit of 
$\Nf\cEn_{M(q)} \, (q \in \N)$, we obtain the functor 
$\Nf\cEn_{M} \lra \Nf\cEn$, thus the morphism 
\begin{equation}\label{eq:eme}
G(\Nf\cEn, \iota^*) \lra G(\Nf\cEn_M, \iota^*)
\end{equation}
of affine group schemes in $\MICn(S/k)$. To prove this morphism is an isomorphism, 
it suffices to check it after we apply the restriction functor 
$\MICn(S/k) \lra \MIC(s/s)$ which is exact and faithful. So, 
by Proposition \ref{fiberpropa} (and its analogue for $M_{X/S}$), 
it suffices to prove that the morphism 
$$ G(\NfsMIC(\cA_{X_s/s}), x^*) \lra G(\NfsMIC(M_{X_s/s}), x^*) $$
is an isomorphism. 
Thus it suffices to prove that the functor 
$$ \NfsMIC(M_{X_s/s}) \lra \NfsMIC(\cA_{X_s/s}) $$ 
is an equivalence. This is true because the map on cohomologies 
$H^i(M_{X_s/s}) \lra H^i(\cA_{X_s/s})$ is an isomorphism 
for $i=0,1$ and injective for $i=2$ by construction of $M_{X_s/s}$. 
So we finished the proof of proposition. 
\end{proof}

The third step of the proof is to compare the category 
$\Nf\cEn_M$ and the category 
$\Rep_{\MICn(S/k)}(\Spec H^0(B(M_{X/S})))$ of representations of 
$\Spec H^0(B(M_{X/S}))$ in $\MICn(S/k)$. 

\begin{prop}\label{prop:3rd}
There exists an equivalence of categories 
\begin{equation}\label{hhhh}
\Rep_{\MICn(S/k)}(\Spec H^0(B(M_{X/S}))) \os{\cong}{\lra} \Nf\cEn_M
\end{equation}
over $\MICn(S/k)$. 
Equivalently, there exists an isomorphism 
\begin{equation}\label{hhhhgroup}
G(\Nf\cEn_M, \iota^*) \os{\cong}{\lra}  \Spec H^0(B(M_{X/S})) 
\end{equation}
of affine group schemes in $\MICn(S/k)$. 
\end{prop}

\begin{proof}
The proof is inspired by the argument of  \cite[Thm 7.6]{terasoma} in the absolute case and we will reduce to it.
The composition 
$$\Delta': B(M_{X/S}) \lra B(M_{X/S}) \otimes_{\cO_S}  B(M_{X/S}) 
\os{\text{proj}}{\lra} B(M_{X/S}) \otimes_{\cO_S} M_{X/S}^1 $$
of the map $x \mapsto \Delta(x) -  x \otimes 1$ and the canonical 
projection induce an integrable $M_{X/S}$-connection 
$$ 
\nabla_H:  H^0(B(M_{X/S})) \lra  H^0(B(M_{X/S})) \otimes_{\cO_S} M^1_{X/S}. $$
Because $\nabla_H$ can be written as the inductive limit of the maps 
$$ H^0(F^{-s}B(M_{X/S})) \lra  H^0(F^{-s+1}B(M_{X/S})) \otimes_{\cO_S} M^1_{X/S}, $$
we see that $(H^0(B(M_{X/S})), \nabla_H)$ actually belongs to the ind-category of 
$\Nf\MIC(M_{X/S})$. Moreover, since this construction is also  possible  on 
$S^m(r) \, (r=1,2, m \in \N)$, $(H^0(B(M_{X/S})), \nabla_H)$ is 
naturally regarded as an object in the ind-category of $\Nf\cEn_M$. 

An object in the category 
$\Rep_{\MICn(S/k)}(\Spec H^0(B(M_{X/S})))$ induces 
a pair $(V,\Delta_V)$  of a locally free $\cO_S$-module of 
finite rank endowed with a comodule structure 
$\Delta_V: V \lra V \otimes_{\cO_S}H^0(B(M_{X/S}))$, and similar pairs 
on $S^m(r) \,(r = 1,2, m \in \N)$ which are compatible. 
We define the integrable $M_{X/S}$-connection corresponding to 
$(V,\Delta_V)$ as the kernel of the map 
\begin{equation}\label{eq:hhh}
\Delta_V \otimes \id - \id \otimes \Delta_H: 
V \otimes H^0(B(M_{X/S})) \lra V \otimes H^0(B(M_{X/S})) \otimes 
H^0(B(M_{X/S})) 
\end{equation}
endowed with a structure of connection induced from that on 
$H^0(B(M_{X/S}))$'s on the right. Since this construction is possible also 
on $S^m(r) \,(r = 1,2, m \in \N)$, we can define the above object as 
an object in the ind-category of $\Nf\cEn_M$. 
However, it actually belongs to $\Nf\cEn_M$ because the kernel of the map 
\eqref{eq:hhh} is equal to the image of 
$\Delta_V: V \hra V \otimes_{\cO_S}H^0(B(M_{X/S}))$. 
In this way we define the functor \eqref{hhhh}, 
hence the morphism \eqref{hhhhgroup}. To see that it is an isomorphism, 
it suffices to prove it after we apply the restriction functor 
$\MICn(S/k) \lra \MIC(s/s)$ which is exact and faithful. So it suffices to prove the 
functor 
$$ \Rep_{k}(\Spec H^0(B(M_{X_s/s}))) \os{\cong}{\lra} \NfsMIC(M_{X_s/s}) $$ 
defined in a similar way as above is an equivalence, and this is shown in 
the proof of \cite[Thm. 7.6]{terasoma}. So we are done. 
\end{proof}

Now we can prove Theorem \ref{15}: 

\begin{proof}[Proof of Theorem \ref{15}] 
By Propositions \ref{prop:1st}, \ref{prop:2nd} and \ref{prop:3rd}, we have 
isomorphisms 
$$ 
\pi_1^{\dR}(X/S,\iota) \os{\cong}{\lla} G(\Nf\cEn,\iota^*) 
\os{\cong}{\lra} G(\Nf\cEn_M,\iota^*) \os{\cong}{\lla} 
\Spec H^0(B(M_{X/S}))
$$
of affine group schemes in $\MICn(S/k)$. So we are done. 
\end{proof}

\begin{rem} 
Assume that the log structures on $X$ and $S$ are trivial and that 
$S$ is smooth over $k$. Then, as we explained in Remark \ref{GM}, there is 
a priori another way to define an integrable connection on $S/k$ 
on $M_{X/S}^1$, hence on $M_{X/S}$, which is due to Navarro Aznar. 
This induces the structure of an integrable connection on the bar construction 
$H^0(B(M_{X/S}))$, and so we can regard $\Spec H^0(B(M_{X/S}))$ 
as an affine group scheme in $\MICn(S/k)$. So we obtain yet another definition 
of the relatively unipotent de Rham fundamental group $\pi_1^{\dR}(X/S,\iota)$. 

We expect that this definition is compatible with our definition in Definition 
\ref{def5}, but we will not pursue this topic in this paper. 
\end{rem}

\section{Calculation of monodromy for stable log curves} 

In this section  we  calculate 
the monodromy action on  the relatively unipotent 
log de Rham fundamental group  in the case of stable log curves:  we will use all the various definitions we gave in the previous paragraphs.  
As an application, we will have  a purely algebraic proof 
of the result of Andreatta--Iovita--Kim \cite{aik}.

\subsection{Statement of main result and first reductions}

Throughout this section,  $k$ will be a field of characteristic zero and 
$S$ will  be the standard log point over $k$. First we introduce the notion of 
log curve, which is due to F. Kato \cite{fk}. 

\begin{defn}\label{def:logcurve}
A morphism $f: X \lra S$ of fs log schemes is a log curve 
if it is log smooth, integral and if its geometric fiber 
is a reduced and connected curve.
\end{defn}

Then we have the following local description by \cite[1.1]{fk}:  

\begin{prop}\label{prop:logcurve}
Let $f: X:=(X^{\circ}, \cM_X) \lra S := (S^{\circ},\cM_S)$ be a log curve and assume $k$ is algebraically closed. 
Then the underlying scheme $X^{\circ}$ of $X$ has at worst ordinary double 
points. Moreover, there exists a finite set of closed points $\{s_1, ..., s_r\}$ 
in the smooth locus $X^{\circ,{\rm sm}}$ of $X^{\circ}$ $($called marked points$)$ such that 
the log structure of $f:X \lra S$ is described in the following way: \\ 
$(1)$ \, Etale locally around a double point, the underlying morphism 
$f^{\circ}:X^{\circ} \lra S^{\circ} = \Spec k$ of $f$ factors as 
$$ X^{\circ} \lra \Spec k[x,y]/(xy) \lra \Spec k$$ 
with the first morphism etale, and the log structure of $f$ is associated to 
the chart 
\begin{align*}
& \N^2 \oplus_{\Delta,\N,m} \N \to \cO_X, \quad 
\N \to k, \quad \N \to \N^2 \oplus_{\Delta,\N,m} \N. \\ 
& \hspace{5mm} ((1,0),0) \mapsto x \qquad\,\,\,\, 1 \mapsto 0 \qquad 1 \mapsto 
((0,0),1) \\ 
&   \hspace{5mm} ((0,1),0) \mapsto y \\ 
&   \hspace{5mm} ((0,0),1) \mapsto 0,
\end{align*}
for some $m \geq 1$. Here, $\N^2 \oplus_{\Delta,\N,m} \N$ denotes the pushout of  the diagram
$$ \N^2 \os{\Delta}{\lla} \N \os{m}{\lra} \N
$$
in the category of fs monoids with $\Delta: \N \rightarrow \N^2$ the diagonal map and $m: \N \rightarrow \N$ the multiplication by $m$. \\ 
$(2)$ \, Etale locally around each $s_i$, 
$f^{\circ}$ factors as 
$$ X^{\circ} \lra \Spec k[x] \lra \Spec k$$ 
with the first morphism etale, and 
the log structure of $f$ is associated to the chart 
\begin{align*}
& \N^2 \to \cO_X, \quad 
\N \to k, \quad \N \to \N^2. \\ 
& \! (1,0) \mapsto 0 \quad \,\,\, 1 \mapsto 0 \quad 1 \mapsto 
(1,0) \\ 
& \! (0,1) \mapsto x 
\end{align*} 
$(3)$ \, Etale locally at other points, $f^{\circ}$ is smooth and 
the log structure on $f$ is associated to the chart 
\begin{align*}
& \N \to \cO_X, \quad 
\N \to k, \quad \N \to \N. \\ 
& \hspace{2mm} 1 \mapsto 0 \qquad \,\,\, 1 \mapsto 0 \quad  \,\,\, 1 \mapsto 1
\end{align*} 
Moreover, the locus $X_{\rm triv} := \{ x \in X \,|\, \cM_{S,\ol{f(x)}}/\cO^{\times}_{S,\ol{f(x)}} 
\os{\cong}{\to} \cM_{X,\ol{x}}/\cO^{\times}_{X,\ol{x}} \}$ 
is equal to $X^{\circ,{\rm sm}} \setminus \{s_1, ..., s_r\}$. 
\end{prop}

For a log curve $f:X \lra S$, the reduced divisor of marked points of its geometric fiber 
descends to a reduced divisor on $X$, which we denote by $D$. 
We say that $f:X \lra S$ is marked (resp. unmarked) if $D$ is nonempty (resp. empty). 

\begin{rem}\label{rem:20190307}
Let $f: X:=(X^{\circ}, \cM_X) \lra S := (S^{\circ},\cM_S)$ be a log curve
with $k$ not necessarily algebraically closed, and let ${\bf s}$ be a $k$-rational 
point of $X$ contained in the divisor $D$ of marked points defined above. 
Let $\ol{{\bf s}} \cong \Spec\,\ol{k}$ (where $\ol{k}$ is an algebraic closure of $k$) 
be a geometric point over ${\bf s}$ and let $\cM_{\bf s}, \cM_{\ol{{\bf s}}}$ be 
the pullback of the log structure $\cM_X$ to ${\bf s}, \ol{{\bf s}}$, respectively. 
Then Proposition \ref{prop:logcurve}(2) implies that $\cM_{\ol{{\bf s}}}$ is associated to 
the chart 
$$ \N^2 \to \cO_{\ol{{\bf s}}}; \quad (1,0) \mapsto 0, \quad (0,1) \mapsto 0. $$
Thus $\cM_{\ol{{\bf s}}}/\cO^{\times}_{\ol{{\bf s}}} \cong \N^2$. Since $\cM_{\ol{{\bf s}}}/
\cO^{\times}_{\ol{{\bf s}}}$ 
is the stalk at $\ol{{\bf s}}$ of the etale sheaf $\cM_{{\bf s}}/\cO^{\times}_{{\bf s}}$ on ${\bf s}$, 
the monoid $\cM_{\ol{{\bf s}}}/\cO^{\times}_{\ol{{\bf s}}} \cong \N^2$ is endowed with the canonical action 
of ${\rm Aut}(\ol{{\bf s}}/{\bf s}) \cong {\rm Gal}(\ol{k}/k)$. 

We prove that this action is trivial. To show it, it suffices to prove that 
the action of any element of 
${\rm Aut}(\ol{{\bf s}}/{\bf s})$ on the set $\{(1,0),(0,1)\}$ 
is trivial. This holds because the element $(1,0)$, which is the pullback of 
an element of the log structure $\cM_S$, is fixed by any element of 
${\rm Aut}(\ol{{\bf s}}/{\bf s})$ because $S$ is the standard log point over $k$. 

The triviality of the above action implies that $\cM_{{\bf s}}/\cO^{\times}_{{\bf s}}$ 
is the constant sheaf $\N^2$ on ${\bf s}$. Also, by \cite[Proposition 1.3]{HN17}, 
we have an isomorphism 
$$ \cM_{\bf s} \cong \cO_{\bf s}^{\times} \times (\cM_{\bf s}/\cO_{\bf s}^{\times}) 
\cong \cO_{\bf s}^{\times} \times \N^2. $$
Hence we conclude that $\cM_{{\bf s}}$ is associated to 
the chart 
\begin{equation}\label{eq:20190307}
\N^2 \to \cO_{{\bf s}}; \quad (1,0) \mapsto 0, \quad (0,1) \mapsto 0, 
\end{equation}
namely, the log structure $\cM_{\bf s}$ admits the chart \eqref{eq:20190307} 
globally on ${\bf s}$, without taking any etale extension of ${\bf s}$. 

A similar property holds also for $k$-rational smooth points of $X$, but 
it does not necessarily hold for $k$-rational double points of $X$. 
\end{rem}

In this section we apply the results of the previous sections to the diagram 
\begin{equation}\label{6-0}
\xymatrix {
X \ar[r]^f & S \ar[r]^(0.36){g} \ar @/^4mm/[0,-1]^{\iota} & \Spec k, 
} 
\end{equation}
where $f$ is a proper log curve, $g$ is the structure morphism and 
$\iota$ is a section of $f$. Also, let $s \hra S$ be the identity map 
and let $x$ be the composition $s \hra S \os{\iota}{\lra} X$. 
Then we are in the situation of \ref{notation}. (The conditions (A), (B) (C), (D) (E) 
follows from Proposition \ref{prop:abcde-rs}.) 
We translate in this situation various objects we have studied in the last sections.
The category $\mathrm{MIC}^{\mathrm{n}}(S/k)$ is canonically equivalent to  
the category of pairs $(V, N)$ consisting of a finite-dimensional vector space $V$ 
and a nilpotent endomorphism $N$, and so we identify them in the following. 
The fiber functor $s^{*}_{\dR}:\mathrm{MIC}^{\rm{n}}(S/k)\rightarrow \mathrm{Vec}_k$ 
is the forgetful functor $(V,N) \mapsto V$. 
The Tannaka dual $\pi_1^{\dR}(S,s)$ of $(\MICn(S/k),s^*_{\dR})$ is 
isomorphic to $\G_{a,k}$ and an object $(V,N) \in \MICn(S/k)$ corresponds to 
the representation $\G_{a,k} \lra GL(V); t \mapsto \exp(tN)$ by the 
equivalence $\MICn(S/k) \cong \Rep(\pi_1^{\dR}(S,s)) = \Rep(\G_{a,k})$ of 
Tannaka duality. We denote the Tannaka dual of 
$(\NfMICn(X/k),x_{\dR}^*)$ by $\pi_1^{\dR}(X,x)$, as in the previous 
sections. However, unlike the previous sections, we abusively denote 
the group scheme $s_{\dR}^*\pi_1^{\dR}(X/S,\iota) = \pi_1^{\dR}(X_s,x)$ over $k$, 
which is the Tannaka dual of $(\NfMIC(X_s/s), x_{\dR}^*) = (\NfMIC(X/S), x_{\dR}^*)$ 
($s$ is the identity), by $\pi_1^{\dR}(X/S,\iota)$, because we found that the notation we have used before ($\pi_1^{\dR}(X_s,x)$) could arise confusion. On the other hand, we  can regard the relatively 
unipotent de Rham fundamental group of $X/S$, 
which is an affine group scheme in $\MICn(S/k)$, as the group scheme 
$\pi_1^{\dR}(X/S,\iota)$ over $k$ endowed with the conjugate action of 
$\pi_1^{\dR}(S,s) = \G_{a,k}$ associated to the split exact sequence 
\begin{equation*}
\xymatrix {
1 \ar[r] & \pi_1^{\dR}(X/S,\iota) \ar[r] & 
\pi^{\dR}_1(X,x) \ar[r]^{f_*} & \pi^{\dR}_1(S,s) \ar[r] 
\ar @/^6mm/[0,-1]^{\iota_*} & 1
} 
\end{equation*}
(see Section 2). The action of $1 \in \G_{a,k} = \pi_1^{\dR}(S,s)$ on 
$\pi_1^{\dR}(X/S,\iota)$ (the conjugation by $\epsilon := \iota_*(1)$) 
is called the monodromy action. 

To state the main result in this section, we introduce the notion 
of stable log curve (due to F.~Kato \cite[Cor. 1.1]{fk}). 

\begin{defn}\label{def:stablelogcurve}
When $k$ is algebraically closed, 
a proper log curve $f:X \lra S$ is called a stable log curve 
if the underlying morphism of schemes $f^{\circ}:X^{\circ} \lra S^{\circ}$ 
endowed with marked points $\{s_1, ..., s_r\}$ in Proposition 
\ref{prop:logcurve} is a pointed stable curve in the sense of 
\cite{knudsen}, namely, if it is not a smooth genus $1$ curve with no marked points and, 
for any irreducible component $Y$ of 
$X^{\circ}$ birational to $\Pr^1_k$, the cardinality of double points of $X$ in $Y$
$($counted doubly when a double point is a self intersection point$)$ plus 
that of marked points in $Y$ is $\geq 3$. 
For general $k$, a proper log curve $f:X \lra S$ is called a stable log curve 
if so is its geometric fiber. 
\end{defn}

For an affine group scheme $G$ over $k$, we denote the (abstract) group 
of automorphisms of $G$ by ${\mathrm{Aut}}(G)$, its subgroup of inner 
automorphisms induced by elements of $G(k)$ by ${\mathrm{Inn}}(G)$, and 
the group ${\mathrm{Aut}}(G)/{\mathrm{Inn}}(G)$ of outer automorphisms by 
${\mathrm{Out}}(G)$. Then the main result of this section is the following: 

\begin{thm} \label{!!}
Let $f\colon X \lra S$ be a stable log curve and $\iota:S \lra X$ be its section. 
Then the monodromy action on $\pi_1^{\dR}(X/S,\iota)$ is trivial as an element 
in $\mathrm{Out}(\pi_1^{\mathrm{dR}}(X/S, \iota))$ if and only if 
the underlying morphism of schemes of $f$ is smooth. 
\end{thm}

First we prove the `if' part of Theorem \ref{!!}. To state it in a more precise form, 
we introduce the notion of good section. 

\begin{defn} 
The section $\iota: S\rightarrow X$ is called good 
if its image is in $X_{\rm triv}$. 
\end{defn}

\begin{rem}\label{good_notgood_section}
Suppose that ${\bf { s}}$ is a $k$-rational double point of $X$ and assume that 
the pullback log structure $\cM_{\bf s}$ on ${\bf {s}}$ from that on $X$ admits 
a chart of the form 
$$ \N^2 \oplus_{\Delta,\N,m} \N \to \cO_{\bf {s}}, \quad 
((1,0),0) \mapsto 0, \quad ((0,1),0) \mapsto 0, \quad ((0,0),1) \mapsto 0 $$
globally, without taking any etale extension of ${\bf s}$ 
(cf. Proposition \ref{prop:logcurve}(1)). 
Then, when $m=1$, 
there does not exist a section $\iota:S \lra X$ of $f$ whose image is $\bf {s }$. Indeed, 
 such a section would induce a map  $\widetilde{\alpha}: \cM_{\bf s} \longrightarrow \cM_S$ inducing the isomorphism  
${\cal O}_{\bf s}^{\times} \os{\cong}{\lra} {\cal O}_{S}^{\times}$, hence a map
$\alpha: \N^2 \oplus_{\Delta,\N,1} \N  = \N^2= \cM_{\bf s}/{\cal O}_{\bf s}^{\times}   \lra \N= \cM_S/ {\cal O}_{S}^{\times}$  with $\alpha^{-1}(0) = \{(0,0)\}$.  Moreover, since $\iota$ is a section of $f$, $\alpha$ would be  a section 
of the diagonal map $\Delta: \N \longrightarrow \N^2$. This is impossible because $\alpha(1,1)= \alpha (\Delta(1))=1$, while $\alpha(1,1)= \alpha(1,0)+\alpha(0,1)\geq 1+1=2$. 


On the other hand, when $m \geq 2$, 
there exists a section $\iota:S \lra X$ of $f$ whose image is $\bf {s}$. 
(For example, the map 
$$ \N^2 \oplus_{\Delta,\N,m} \N \to \N; \quad ((a,b),c) \mapsto a(m-1)+b+c $$
is a section of the map $\N \to \N^2 \oplus_{\Delta,\N,m} \N$ in Proposition \ref{prop:logcurve}(1), 
and this induces the required section $\iota$ of $f$.) 
But this section is not good by definition. 

Moreover, if ${\bf { s}}$ is a $k$-rational point in the marked divisor, 
we have a section $\iota:S \lra X$ of 
$f$ whose image is ${\bf { s}}$. Indeed, we always have a global chart 
\eqref{eq:20190307} by Remark \ref{rem:20190307} and there exists a section 
$\beta: \N^2 \to \N$ of the 
map $\N \to \N^2$ in Proposition \ref{prop:logcurve}(2) 
with $\beta^{-1}(0) = \{0\}$, which induces the required section $\iota$ of $f$. 
(Such a map $\beta$ is necessarily of the form 
$(1,0) \mapsto 1, (0,1) \mapsto b$ for some $b \geq 1$.) This section is not 
good either. 
\end{rem}

Then the `if' part of Theorem \ref{!!} follows from the following proposition. 

\begin{prop}\label{not_good_trivial}
Let $f\colon X\rightarrow S$ be a stable log curve such 
that underlying morphism of schemes of $f$ 
is smooth. Then we have the following$:$ \\
$(1)$ \, If the section $\iota$ is good, 
the monodromy action on $\pi_1^{\mathrm{dR}}(X/S, \iota)$ is trivial as an 
element of $\mathrm{Aut}(\pi_1^{\mathrm{dR}}(X/S, \iota)).$ \\ 
$(2)$ \, If  the section $\iota$ is not good, the monodromy action on
$\pi_1^{\mathrm{dR}}(X/S, \iota)$ is trivial as an 
element of $\mathrm{Out}(\pi_1^{\mathrm{dR}}(X/S, \iota))$.
\end{prop}

\begin{proof} 
Let $\epsilon := \iota_*(1) \in \pi_1^{\dR}(X,x)$ be as before Definition \ref{def:stablelogcurve}. 
Take $\gamma \in \pi_1^{\dR}(X/S,\iota)$ and for $(E,\nabla) \in \NfMICn(X/k)$, 
we calculate the action of $\epsilon^{-1} \circ \gamma \circ \epsilon$ on 
$\iota^*E$. 

(1) \, If we denote the image of $(E,\nabla)$ by $\iota_{\dR}^*: \NfMICn(X/k) \lra \MICn(S/k)$ 
by $(\iota^*E,N)$, the action of $\epsilon$ on $\iota^*E$ is given by $\exp(N)$. 
Also, since $\epsilon^{-1} \circ \gamma \circ \epsilon$ is an element in $\pi_1^{\dR}(X/S,\iota)$, 
the action of it on $\iota^*E$ depends only on the image $(E,\ol{\nabla})$ of 
$(E,\nabla)$ in $\NfMIC(X/S)$. Since the underlying morphism of $f$ is smooth,  
if we denote the log structure on $X^{\circ}$ associated to 
the marked divisor $D$ by ${\cal N}$ and denote the log scheme $(X^{\circ},{\cal N})$ by 
$X'$, we have the equivalence $\NfMIC(X/S) \cong {\rm N}_{f'}\MIC(X'/k)$. 
(Here $f':X' \lra \Spec k$ is the structure morphism.)  
Denote the object on the right hand side corresponding to $(E,\ol{\nabla})$ by 
$(E,\nabla')$. Then, to calculate the action of $\epsilon^{-1} \circ \gamma \circ \epsilon$ 
on $\iota^*E$, we may replace $(E,\nabla)$ by the image of $(E,\nabla')$ by the functor 
${\rm N}_{f'}\MIC(X'/k) \lra \NfMICn(X/k)$. If we do so, 
we see that $N=0$ by construction. Hence the action of 
$\epsilon ^{-1}\circ\gamma \circ\epsilon$ on $\iota^*E$ is the same as that of $\gamma$ 
and so $\epsilon ^{-1}\circ\gamma \circ\epsilon = \gamma$. 

(2) \, Since $X$ has no double points, $x$($:=$ the image of $S^{\circ}$ by $\iota$) 
is a marked point. So, if we endow $x$ with 
the pullback log structure from $X$, the log structure on $x$ is induced by 
the map $\N^2 \to k; (1,0) \mapsto 0, (0,1) \mapsto 0$ by Remark \ref{rem:20190307} 
and the restriction of $(E,\nabla)$ to $x$ 
is identified with the triple $(E|_x,N_1,N_2)$ consisting of a finite-dimensional $k$-vector space 
endowed with two commuting nilpotent endomorphisms. The map $\iota': S \to x$ induced by $\iota$ 
is induced by the identity map on $\Spec k$ and the monoid homomorphism 
$\beta: \N^2 \to \N$ as in Remark \ref{good_notgood_section} for some $b \geq 1$. By pulling back the triple 
$(E|_x,N_1,N_2)$ by $\iota'$, we see that the image of $(E,\nabla)$ by $\iota_{\dR}^*$ is 
$(\iota^*E,N_1+bN_2)$. Hence the action of $\epsilon$ on $\iota^*E$ is given by $\exp(N_1+bN_2)$. 
As in the proof of (1), to know 
the action of $\epsilon^{-1} \circ \gamma \circ \epsilon$ 
on $\iota^*E$, we may replace $(E,\nabla)$ by the image of $(E,\nabla')$ in $\NfMICn(X/k)$.   
If we do so, we see that $N_1=0$. Also, the functorial action $\exp(bN_2)$ on all  $\iota^*E$'s  defines 
an element $\eta$ in $\pi_1^{\dR}(X/S,\iota)$, because $N_2$ only depends on 
the image of $(E,\nabla)$ in $\NfMIC(X/S)$. Then 
the action of 
$\epsilon ^{-1}\circ\gamma \circ\epsilon$ on $\iota^*E$ is the same as that of 
$\eta^{-1} \circ \gamma \circ \eta$ 
and so $\epsilon ^{-1}\circ\gamma \circ\epsilon = \eta^{-1} \circ \gamma \circ \eta$. 
Therefore, the monodromy action on $\pi_1^{\mathrm{dR}}(X/S, \iota)$ is trivial as an element of 
$\mathrm{Out}(\pi_1^{\mathrm{dR}}(X/S, \iota))$. 
\end{proof}

We deal now with the  `only if' part of Theorem \ref{!!}: we will introduce some propositions in order to reduce the proof to  simpler cases.

\begin{prop}\label{prop:!!bc}
Let $k \subseteq k'$ be an extension of fields. Then the `only if' part of 
Theorem \ref{!!} is true if it is true for the base extension 
\begin{align*}
& f': X' := X \times_{\Spec k} \Spec k' \lra S' := S \times_{\Spec k} \Spec k' \\ 
& \iota': S' \lra X' 
\end{align*}
of $f$, $\iota$ by the morphism $\Spec k' \lra \Spec k$. 
\end{prop}

\begin{proof}
Note that we have the isomorphism 
$\pi_1^{\mathrm{dR}}(X'/S', \iota')\cong\pi_1^{\mathrm{dR}}(X/S, \iota)\times_{\Spec k} \Spec k'$ 
which is compatible with monodromy action.  
This follows from the second definition (Definition \ref{def2}) of 
$\pi_1^{\mathrm{dR}}(X/S, \iota)$ and 
the fact that $W_n$ in Section 3 is compatible with base change 
by the morphism $\Spec k' \lra \Spec k$ (Remark \ref{rem:wn-compati}). 
Also, it follows from the third definition (Definition \ref{def3}) of 
$\pi_1^{\mathrm{dR}}(X/S, \iota)$ and the fact that the construction of 
$1$-minimal model $M_{X/S}$ is compatible with base change by 
the morphism $\Spec k' \lra \Spec k$. 
Also, the underlying morphism $f^{\circ}$ of $f$ is smooth if and only if 
that ${f'}^{\circ}$ of $f'$ is smooth. Hence, if $f^{\circ}$ is not smooth, 
${f'}^{\circ}$ is not smooth either and so the monodromy action is not trivial 
as an element in ${\rm Out}(\pi_1^{\dR}(X'/S',\iota'))$. Hence it is not trivial as an 
element in ${\rm Out}(\pi_1^{\dR}(X/S,\iota))$ either. So we are done. 
\end{proof}

To prove the next lemma  we need  an explicit description of 
certain log blow-ups of a proper log curve $X \lra S$. Assume that $k$ is algebraically closed and 
let $\bf s $ be a double point of $X^{\circ}$. Then the local description of $X$ around $\bf s$ 
is given in Proposition \ref{prop:logcurve}(1). Assume that $m \geq 2$ in the local 
description there. If we consider the ideal 
$J := (\N^2 \oplus_{\Delta,\N,m} \N) \setminus \{0\} \subseteq \N^2 \oplus_{\Delta,\N,m} \N$ 
and glue the ideal (defined etale locally around $\bf s$) in $\cM_X$ generated by $J$ 
and $\cM_X|_{X \setminus \{{\bf s}\}}$, we obtain the coherent ideal $\cJ$ of $\cM_X$. 
We denote the log blow-up of $X$ with respect to $\cJ$ by $h: X' \lra X$. 

Etale locally, $X'$ is strict etale over the log blow-up of the log scheme 
$$ X_0 := (\Spec k[x,y]/(xy), (\N^2 \oplus_{\Delta,\N,m} \N)^a) $$ 
in Proposition \ref{prop:logcurve}(1) by the ideal $J$. 
By \cite[Proposition 4.3, Corollary 4.8]{niziol}, it is the closed subscheme 
$X'_0$ 
defined by $t$ of the normalization $Y'$ of the blow up of 
$Y := \Spec k[x,y,t]/(xy-t^m)$ along the ideal $(x,y,t)$. 
Because $Y'$ is covered by the schemes 
\begin{equation}\label{eq:localy}
\Spec k[x,t/x], \quad \Spec k[y,t/y], \quad \Spec k[x/t,y/t,t]/((x/t)(y/t) - t^{m-2}), 
\end{equation}
$X'_0$ is covered by the schemes 
\begin{equation}\label{eq:localy-2}
\Spec k[x,t/x]/(x(t/x)), \quad \Spec k[y,t/y]/(y(t/y)), \quad \Spec k[x/t,y/t]/((x/t)(y/t))  
\end{equation}
in the case $m \geq 3$, 
and they are endowed with log structures associated to 
$\N^2, \N^2, \N^2 \allowbreak \oplus_{\Delta,\N,m-2} \N$, respectively, i.e. the cases 
(1) with $m=1$, (1) with $m=1$ and (1) with $m$ replaced by $m-2$ in the notation introduced in Proposition \ref{prop:logcurve}. 
In the case $m=2$, the third affine scheme in \eqref{eq:localy-2} 
should be replaced by $\Spec k[x/t,y/t]/((x/t)(y/t)-1)$ endowed with the log structure associated to $\N^2 \oplus_{\Delta,\N,0} \N$. Because the basis of the first factor $\N^2$ 
are sent to $x/t$ and $y/t$ which are now invertible, the log structure is associated to the first factor $\N$. Hence we are in the case (3) of Proposition \ref{prop:logcurve}.
From this description, we see that $X' \lra S$ is again a log curve.  
Then we have the following: 

\begin{lem}\label{lem:bu}
Let $f: X \lra S$ be a proper log curve with a good section 
$\iota:S \lra X$ and let 
$h: X' \lra X$ be the log blow-up at the double point $\bf s$ as above. Also, 
let $f':X' \lra S$ be the composite $f \circ h$ and let 
$\iota': S \lra X'$ be the section induced by $\iota$. 
$($Note that $h$ is an isomorphism on $X \setminus \{{\bf s}\}$ and that   
the image of $\iota$ cannot be ${\bf s}$ because $\iota$ is a good section.$)$ 
Then we have the isomorphism 
$$ h_*: \pi_1^{\dR}(X'/S,\iota') \os{\cong}{\lra} \pi_1^{\dR}(X/S, \iota) $$ 
which is compatible with monodromy action. 
\end{lem}

\begin{proof}
First we prove the isomorphism 
$H^n_{\dR}(X/S) \os{\cong}{\lra} H^n_{\dR}(X'/S)$ 
of log de Rham cohomologies for any $n$. 
Let $X_0, X'_0, Y, Y'$ be as  in the notation before the lemma. Then we have the 
Cartesian diagram 
\[
\xymatrix{
X'_0 \ar[r]^{i'} \ar[d]^{h_0}& Y'\ar[d]^{h_Y}\\
X_0 \ar[r]^{i} &Y, 
}
\]
where $h_0, h_Y$ are log blow-ups and $i'$, $i$  are 
exact closed immersions. 
Because the log blow-up $h$ is log etale, we have the isomorphism 
$h^*\Omega^j_{X/S} \cong \Omega^j_{X'/S}$. By using this, the projection 
formula and a spectral sequence, we see that, to prove 
the isomorphism in the lemma, it suffices to prove that the morphism 
$$ \cO_X \lra Rh_*\cO_{X'} $$ 
is an isomorphism. (See the proof of Proposition \ref{prop:invariance}.) 
This is further reduced to proving the isomorphism 
$$ \cO_{X_0} \os{\cong}{\lra} Rh_{0,*}\cO_{X'_0},$$ 
which, in turn, is equivalent to the isomorphism 
$$ i_{\ast}\cO_{X_0} \os{\cong}{\lra} i_{\ast} Rh_{0,*}\cO_{X'_0}.$$ 

Thanks to \cite[Corollary 4.7]{niziol} and \cite[Theorem 11.3]{ka2}, we have 
the isomorphism 
\begin{equation}\label{eq:logblowup-y}
\cO_{Y} \os{\cong}{\lra} Rh_{Y,*}\cO_{Y'}.
\end{equation} 
Moreover, since the multiplication by $t$ is injective on the rings 
appearing in \eqref{eq:localy}, 
$\cO_{X'_0}$ and $\cO_{Y'}$ are Tor-independent over $\cO_{Y}$. 
Hence, 
we have the isomorphisms 
\begin{align*}
 i_{\ast} \cO_{X_0} & \cong ( i_{\ast} \cO_{X_0}) \otimes^L_{\cO_Y} \cO_Y \cong
( i_{\ast}\cO_{X_0} )\otimes^L_{\cO_Y} Rh_{Y,*} \cO_{Y'} \\ 
& 
\cong 
 Rh_{Y,*} (L h^*_Y i_{\ast}    \cO_{X_0})    \cong   Rh_{Y,*} (h^*_Y i_{\ast}    \cO_{X_0})\\ 
& \cong  Rh_{Y,*} ( i'_{\ast}    \cO_{X'_0})  \cong i_{\ast}  Rh_{0,*} \cO_{X'_0}, 
\end{align*}
where the third isomorphism is the projection formula and the fourth isomorphism follows from the Tor-independence  of  $\cO_{X'_0}$ and $\cO_{Y'}$  over $\cO_{Y}$ (see also
 \cite[\href{https://stacks.math.columbia.edu/tag/08Ib}{Tag 08IB}]{stacks} for these isomorphisms). 
In conclusion, we have the isomorphism $H^n_{\dR}(X/S)  {\cong}  H^n_{\dR}(X'/S)$ for any $n$. 

The aforementioned  isomorphism of cohomologies implies the equivalence
$$h_{\dR}^*: \NfMIC(X/S) \allowbreak \os{\cong}{\lra} \mathrm{N}_{f'}\mathrm{MIC}(X'/S).$$
Indeed, because objects on both sides are iterated extensions of trivial objects, it will suffice  to check 
the identification of $\Hom$ and $\Ext$ groups on such objects.  But this follows from the isomorphism of the zeroth and the first cohomology groups as well as the injection of the map between the second cohomology groups (see the proof of Proposition \ref{prop:1st}). 
We then obtain, by Tannakian duality, an isomorphism
$$h_*: \pi_1^{\dR}(X'/S,\iota') \os{\cong}{\lra} \pi_1^{\dR}(X/S, \iota).$$ 
Since this isomorphism fits into the commutative diagram 
\begin{equation*}
\xymatrix{ 
1 \ar[r]& \pi_1^{\dR}(X'/S,\iota') \ar[r]\ar[d]^{h_*}& \pi_1^{\dR}(X',\iota') \ar[r]^(0.4){f_*} \ar[d]^{h_*} &  \pi_1^{\dR}(S,s)\cong{\mathbb G}_{a,k} \ar[d]^{\rm{id}} \ar[r]   \ar @/^6mm/[0,-1]^{\iota'_*} & 1\\
 1 \ar[r]& \pi_1^{\dR}(X/S,\iota) \ar[r]& \pi_1^{\dR}(X,\iota) \ar[r]^(0.4){f_*}  &  \pi_1^{\dR}(S,s)\cong{\mathbb G}_{a,k}  \ar[r]   \ar @/^6mm/[0,-1]^{\iota_*} & 1}
\end{equation*}
of split exact sequences, we see that the vertical arrows are isomorphisms and that 
the isomorphism in the lemma is compatible with monodromy action. 
So the lemma is proved. 
\end{proof}

Let us introduce the notion of (minimal) semistable log curve. 

\begin{defn}
When $k$ is algebraically closed, 
a proper log curve $X \lra S$ is called a semistable log curve if, in  
the local description around any double point of $X$ given in Proposition 
\ref{prop:logcurve}, we can take $m=1$. A semistable log curve 
$X \lra S$ is called minimal if it is not a smooth genus $1$ curve with no marked points and, 
for any irreducible component $Y$ of 
$X^{\circ}$ birational to $\Pr^1_k$ which intersects with other components 
at $\leq 1$ point, 
the cardinality of double points of $X$ in $Y$
$($counted doubly when a double point is a self intersection point$)$ plus 
that of marked points in $Y$ is $\geq 3$. 
For general $k$, a proper log curve $f:X \lra S$ is called a $($minimal$)$ semistable log curve 
if so is its geometric fiber. 
\end{defn}

\begin{rem}
The terminology `semistable' comes from the fact that a typical example of 
a semistable log curve is the log special fiber of a proper curve over $\Spec k[[x]]$ with 
semistable reduction (in the sense of \cite{aik} for example) endowed with the log structure 
defined by the special fiber. Some authors say the 
notion of semistable reduction in this sense as regular semistable reduction, but 
we found that the terminology `regular semistable log curve' is not good because 
the underlying scheme of a semistable log curve is not regular unless it is smooth over $k$. 
We warn the reader that, in our terminology, a stable log curve is not necessarily a 
semistable log curve.  
\end{rem}

What we  are going to prove in the  next subsections (from Subsection 6.2 to Subsection 6.7)  is the following theorem: 

\begin{thm} \label{!!!}
Let $f\colon X \lra S$ be a minimal semistable log curve and let $\iota:S \lra X$ be a section. 
Then the monodromy action on $\pi_1^{\dR}(X/S,\iota)$ is nontrivial as an element 
in $\mathrm{Out}(\pi_1^{\mathrm{dR}}(X/S, \iota))$ if  
the underlying morphism of schemes of $f$ is not smooth. 
\end{thm}

In fact, the following proposition allows us to reduce the `only if' part of 
Theorem \ref{!!} to the above theorem: 

\begin{prop}\label{!!bu}
The `only if' part of Theorem \ref{!!} follows from Theorem \ref{!!!}. 
\end{prop}

\begin{proof}
We assume the validity of 
Theorem \ref{!!!} and prove the `only if' part of Theorem \ref{!!}. 
By Proposition \ref{prop:!!bc}, we may assume that $k$ is algebraically closed. 
Also, by Proposition \ref{prop:HES_t} below, we may assume that 
the section $\iota$ is good. 
Assume that we are given a stable log curve $f:X \lra S$ and its good  
section $\iota:S \lra X$ such that 
$X^{\circ}$ is not smooth over $k$. 
By the local description \eqref{eq:localy-2} of log blow-up as above, 
we see that there exists a composition of log blow-ups 
$X' \lra X$ as before Lemma \ref{lem:bu} (possibly identity) 
such that $X'$ is a semistable log curve with ${X'}^{\circ}$ not smooth 
over $k$.  Moreover, since 
$X$ is stable and the new components 
appearing in $X'$ are isomorphic to 
$\Pr^1_k$ and intersects with other components at $2$ points, 
$X'$ is minimal. Thus, by Theorem \ref{!!!}, the monodromy action is nontrivial 
as an element of ${\rm Out}(\pi_1^{\dR}(X'/S,\iota'))$ 
(where $\iota':S \lra X'$ is the section induced by $\iota$). 
Then, by Lemma \ref{lem:bu}, we conclude that 
the monodromy action is nontrivial 
as an element of ${\rm Out}(\pi_1^{\dR}(X/S,\iota))$.  
\end{proof}

Finally we prove a proposition which allows us to change the base points 
for  log de Rham fundamental groups. Let $f: X \lra S$ be a 
stable log curve or a minimal semistable log curve with section 
$\iota: S \lra X$ 
and suppose that we are given 
a diagram of exact faithful $k$-linear tensor functors 
\begin{equation}\label{eq:omegaff}
\xymatrix{
\NfMIC(X/S) \ar[rd]^{\ol{\omega}} & 
\NfMICn(X/k) \ar[l]^r \ar[d]^{\omega} \ar@/^1pc/[r]^{\wt{\omega}}& 
\MICn(S/k) \ar[l]^-{f_{\dR}^*} \ar[ld]^{s_{\dR}^*} \\ 
& \Vector_k 
}
\end{equation}
(where $r$ is the restriction functor) 
with $\wt{\omega} \circ f_{\dR}^* = \id, s_{\dR}^* \circ \wt{\omega} = \omega, 
\omega \circ f_{\dR}^* = s_{\dR}^*, \ol{\omega} \circ r = \omega$, and denote 
the Tannaka dual of $(\NfMIC(X/S), \ol{\omega}), (\NfMICn(X/k), \omega)$ 
by $\pi_1^{\dR}(X/S, \ol{\omega}), \pi_1^{\dR}(X,\omega)$, respectively. 
Then we have the following: 

\begin{prop}\label{prop:HES_t} 
Let the notation be as above. Then 
the diagram \eqref{eq:omegaff} of functors induces the split exact sequence 
$$ 
\xymatrix {
1 \ar[r] & \pi_1^{\dR}(X/S,\ol{\omega}) \ar[r] & 
\pi^{\dR}_1(X,\omega) \ar[r]^{f_*} & \pi^{\dR}_1(S,s) \ar[r] 
\ar @/^6mm/[0,-1]^{\wt{\omega}^*} & 1. 
} 
$$
Moreover, the  induced monodromy action on $\pi_1^{\dR}(X/S,\ol{\omega})$ 
 by the diagram   is nontrivial as an element in 
${\rm Out}(\pi_1^{\dR}(X/S,\ol{\omega}))$ if and only if 
the monodromy action on $\pi_1^{\dR}(X/S,\iota)$ 
is nontrivial as an element in 
${\rm Out}(\pi_1^{\dR}(X/S,\iota))$. In particular, 
the validity of `only if' part of Theorem \ref{!!} and Theorem 
\ref{!!!} is independent of the choice of fiber functors 
as in the diagram \eqref{eq:omegaff}. 
\end{prop}
\begin{proof}
First note that the functor 
$$ 
\text{($k$-algebras)} \lra \text{(Sets)}; \,\,\,\, 
R \mapsto 
\left\{ 
\begin{aligned}
& \text{tensor isomorphisms from} \\ 
& \text{$\NfMIC(X/S) \os{x_{\dR}^*}{\to} \Vector_k \to {\rm Proj}_R$ to} \\ 
& \text{$\NfMIC(X/S) \os{\ol{\omega}}{\to} \Vector_k \to {\rm Proj}_R$}
\end{aligned}
\right\} 
$$ 
(where ${\rm Proj}_R$ is the category of finitely generated projective 
$R$-modules) 
is representable by an affine $k$-scheme 
$\pi_1^{\dR}(X/S,x_{\dR}^*,\ol{\omega})$
which is faithfully flat over $k$ (\cite[Theorem 3.2]{DelMil82}). 
Since we may enlarge $k$ to prove the proposition, 
we can assume that $\pi_1^{\dR}(X/S,x_{\dR}^*,\ol{\omega})$ 
has a $k$-rational point $\sigma$, namely, that there exists 
an isomorphism $\sigma$ from the fiber functor $x_{\dR}^*:
\NfMIC(X/S) \lra \Vector_k$ to 
the fiber functor
$\ol{\omega}: \NfMIC(X/S) \lra \Vector_k$. 
Let $\ol{s}(\sigma)$ be the isomorphism 
$$ \pi_1^{\mathrm{dR}}(X/S,\iota) \os{\cong}{\lra} \pi_1^{\mathrm{dR}}(X/S,\ol{\omega});  
\quad \alpha\mapsto \sigma \circ \alpha \circ \sigma^{-1}. $$
Also, let $\sigma'$ be the isomorphism from 
the fiber functor $x_{\dR}^*:
\NfMICn(X/k) \lra \Vector_k$ to 
the fiber functor
$\omega: 
\NfMICn(X/k) \lra \Vector_k$ induced by $\sigma$ and 
let $s(\sigma)$ be the isomorphism 
$$ \pi_1^{\mathrm{dR}}(X,x) \os{\cong}{\lra} \pi_1^{\mathrm{dR}}(X, \omega);  
\quad \alpha\mapsto \sigma' \circ \alpha \circ {\sigma'}^{-1}. $$
Then we have the commutative diagram 
\begin{equation}\label{change_x_t}
\xymatrix{
1\ar[r]&\pi_1^{\mathrm{dR}}(X/S, \iota)\ar[r] \ar[d]^{\ol{s}(\sigma)}&\pi_1^{\mathrm{dR}}(X, x)\ar[r]\ar[d]^{s(\sigma)} &\pi_1^{\mathrm{dR}}(S,s)\ar[r]\ar[d]^{\mathrm{id}}&1\\
1 \ar[r] & \pi_1^{\mathrm{dR}}(X/S,  \ol{\omega})\ar[r] &\pi_1^{\mathrm{dR}}(X, \omega)\ar[r] &\pi_1^{\mathrm{dR}}(S,s) 
\ar[r] & 1, \\
}
\end{equation}
where the top horizontal line is the exact sequence in 
Proposition \ref{fiberprop}: Indeed, the left square is commutative by construction, and 
for the right square, $s(\sigma)$ induces a conjugation on $\pi_1^{\mathrm{dR}}(S,s)$ by some element, which 
is necessarily trivial because $\pi_1^{\mathrm{dR}}(S,s)\cong \mathbb{G}_{a,k}$ is  commutative. 
So the bottom horizontal line is also exact and so we obtain the required exact sequence. 

As for the monodromy action, it is easy to see that, via the isomorphisms $\ol{s}(\sigma)$, 
the conjugate action of $\iota_*(1) \in \pi_1^{\dR}(X,x)$ on $\pi_1^{\dR}(X/S,\iota)$ and 
that of $\wt{\omega}^*(1) \in \pi_1^{\dR}(X,\omega)$ on $\pi_1^{\dR}(X/S,\ol{\omega})$ 
differ by some inner action on $\pi_1^{\dR}(X/S,\ol{\omega})$. Thus we can identify them 
in ${\rm Out}(\pi_1^{\dR}(X/S,\ol{\omega}))$ and so we obtain the latter assertion of the proposition. 
\end{proof}

\begin{rem}
For an affine group scheme $G$, we can consider the fppf sheafification 
${\mathrm{Out}}(G)^{\rm fppf}$ of the group ${\mathrm{Out}}(G)$ of outer 
automorphisms of $G$. Then, the conclusion of Theorems 
\ref{!!}, \ref{!!!} and Proposition \ref{prop:HES_t} remains to be true 
if we replace ${\mathrm{Out}}(\pi_1^{\dR}(X/S,\ol{\omega}))$ by 
${\mathrm{Out}}(\pi_1^{\dR}(X/S,\ol{\omega}))^{\rm fppf}$ in the statement. 
Indeed, Proposition \ref{not_good_trivial} implies the `if' part of this modified version of 
Theorem \ref{!!}. Moreover, Theorem \ref{!!!} for the base changes of $f$ by 
arbitrary field extensions $k \subseteq k'$ implies this modified version of 
Theorem \ref{!!!}. Also, Proposition \ref{prop:HES_t} for the base changes of $f$ by 
arbitrary field extensions $k \subseteq k'$ implies Proposition \ref{prop:HES_t} 
in this modified setting. Then we see that Proposition \ref{!!bu} remains to be true 
in this modified setting and so we obtain the `only if' part of this modified version of 
Theorem \ref{!!}. 
\end{rem}

\subsection{Calculation of Lie algebra}

Assume $k$  algebraically closed 
and let $f:X \lra S$ be a semistable log curve. 
In this subsection, we give an explicit calculation of 
the pro-nilpotent Lie algebra $L(X/S)$ associated to the de Rham fundamental group 
$\pi_1^{\dR}(X/S,\iota)$. Although the results in this subsection are 
well-known through transcendental arguments, we give here 
purely algebraic proofs. 

Put $g := \dim H^1(X,\cO_X)$ and call it the genus of $f:X \lra S$. Also, let 
$r$ be the number of marked points of $f:X \lra S$. 
First recall the following result on the de Rham cohomologies of $X$ over $S$. 

\begin{prop}\label{prop:coh}
Let the notations be as above. \\ 
$(1)$ \, If $f:X \lra S$ is unmarked\, $(r=0)$, we have 
$\dim H^0_{\dR}(X/S) = 1, \dim H^1_{\dR}(X/S) = 2g, \dim H^2_{\dR}(X/S)=1, \dim H^n_{\dR}(X/S) = 0 \,(n \geq 3)$. \\ 
$(2)$ \, If $f:X \lra S$ is marked\, $(r > 0)$, we have 
$\dim H^0_{\dR}(X/S) = 1, \dim H^1_{\dR}(X/S) = 2g+r-1, \dim H^n_{\dR}(X/S) = 0 \, (n \geq 2)$. 
\end{prop}

\begin{proof}
In both cases, $\dim H^0_{\dR}(X/S) = 1$ by geometric reducedness and geometric 
connectedness of $f$. Also, $\dim H^n_{\dR}(X/S) = 0$ for $n \geq 3$ because 
the cohomological dimension for coherent cohomology of $X$ is $1$ (which 
is well-known in the case of smooth curves and we can reduce to this case by 
considering normalization). 

(1) \,   In this case $\Omega^1_{X/S}[1]$ is a dualizing sheaf by \cite[Theorem 2.21]{tsuji}: 
in fact using \cite[Corollary 2.6]{tsuji} one can find that the ideal $I_f$ in \cite[Theorem 2.21]{tsuji}
is the whole structure sheaf $\cO_X$. 
We have $\dim H^0(X,\Omega^1_{X/S}) = g$ by Grothendieck--Serre duality 
(see \cite[Theorem 1.5]{tsuji} for example), and 
by $E_1$-degeneration of Hodge--de Rham spectral sequence \cite{ka}, 
we see that $\dim H^1_{\dR}(X/S) = 2g$. By Tsuji's Poincar\'e duality 
\cite[\S 3]{tsuji}, we see that $\dim H^2_{\dR}(X/S) \allowbreak = 1$. 

(2) \, Let $D$ denote the divisor given by the marked points. In this case $\Omega^1_{X/S}(-D)[1]$ is the dualizing  sheaf by \cite[Theorem 2.21]{tsuji}: in this case the ideal $I_f$ in \cite[Theorem 2.21]{tsuji}
is $\cO_X(-D)$. By Tsuji's Poincar\'e duality, $$\dim H^2_{\dR}(X/S) \leq 
\dim H^0(X,\cO(-D)) = 0.$$ Also, if we denote the semistable 
log curve which we obtain by erasing one marked point $P$ 
of $X$ by $X'$, we have the exact sequence of de Rham complex 
$$ 0 \lra \Omega^{\bullet}_{X'/S} \lra \Omega^{\bullet}_{X/S} \lra k_P[-1] \lra 0. $$
Then the claim for $\dim H^1_{\dR}(X/S)$ follows from 
the result for $\dim H^n_{\dR}(X/S) \, (n=0,2)$, the long exact sequence 
associated to the above short exact sequence and induction on $r$. 
\end{proof}

The following proposition, which calculates 
$L(X/S)$ in the marked case, is the first main result in this subsection. 

\begin{prop}\label{propr>0}
When $r > 0$, $L(X/S)$ is a free pro-nilpotent Lie algebra of rank 
$2g+r-1$. 
\end{prop}

\begin{proof}
We construct explicitly the $1$-minimal model of the $k$-cdga $\cA_{X/S}$ as in Section 4. 
First note that $H^n(\cA_{X/S}) = H^n_{\dR}(X/S) \, (n \in \N)$. 
Put $V := H^1(\cA_{X/S})$ and take a section  $\sigma: V \lra Z^1(\cA_{X/S})$ 
of the canonical projection $Z^1(\cA_{X/S}) \lra H^1(\cA_{X/S})$. 
Then define the $k$-cdga $B_{X/S}$ as  
$$ B_{X/S}^0 := k, \quad B_{X/S}^1 := V, \quad B_{X/S}^n = 0 \,\, (n \geq 2) $$ 
with zero differential, and define the morphism $i:B_{X/S} \lra \cA_{X/S}$ by 
the maps $k \hra \cA_{X/S}, V \os{\sigma}{\lra} Z^1(\cA_{X/S}) \hra \cA_{X/S}^1$. 
Then $i$ is 
a quasi-isomorphism by Proposition \ref{prop:coh}(2). Hence, to construct 
the $1$-minimal model of $\cA_{X/S}$, it suffices to construct 
that of $B_{X/S}$. 

Let $L$ be the free Lie algebra associated to the dual $V^{\vee}$ of $V$, 
and let $\{\Fil^n L\}_n$ be the central descending filtration of $L$.
(We follow the convention $L = \Fil^1 L$.) 
For $q \in \N$, we put $L_q := L/\Fil^{q+1} L$, denote its dual by 
$L_q^{\vee}$ and put $L^{\vee} := \varinjlim_q L_q^{\vee}$. 
Define the $k$-cdga $\bigwedge L_q^{\vee}$ (where $L_q^{\vee}$ is 
sitting in degree $1$) by introducing the differential 
induced by the map $L_q^{\vee} \lra \bigwedge^2 L_q^{\vee}$ dual to 
$-[\phantom{a},\phantom{a}]: L_q \wedge L_q \lra L_q$ 
(the minus sign comes from the shift of the degree). 
Also, we define the map of $k$-cdga's $\rho_q: 
\bigwedge L_q^{\vee} \lra B_{X/S}$ as the one 
induced by the $k$-linear map 
$\rho_q^1: L_q^{\vee} \lra V$, which is defined inductively 
in the following way: 
If $q=1$, we define the map $\rho_1^1: L_1^{\vee} \lra V$
as the canonical map $L_1^{\vee} \os{\cong}{\lra} (V^{\vee})^{\vee} = V$. 
For general $q$, we define the map $\rho_q^1: L_q^{\vee} \lra V$ as 
any chosen map extending $\rho_{q-1}^1$. 

To prove the proposition, it suffices to prove that 
the map $\rho_q$ is the $(1,q)$-minimal model for any $q$  
(and so $\varinjlim_q \rho_q: \bigwedge L^{\vee} \lra B_{X/S}$ is the 
$1$-minimal model and $L(X/S) = \varprojlim_q L_q$). 
By definition of the $(1,q)$-minimal model, we have to prove that 
$\bigwedge L_q^{\vee} \hra \bigwedge L_{q+1}^{\vee}$ is a Hirsch 
extension, that $H^0(\bigwedge L_q^{\vee}) = k$, 
that $H^1(\bigwedge L_q^{\vee}) = H^1(B_{X/S}) = V$ and that the map 
$$H^2(\bigwedge L_q^{\vee}, B_{X/S}) = 
H^2(\bigwedge L_q^{\vee}) \lra H^2(\bigwedge L_{q+1}^{\vee}, B_{X/S}) 
= H^2(\bigwedge L_{q+1}^{\vee})$$ is zero. 

By using the fact that the map 
$-[\phantom{a},\phantom{a}]: L_{q+1} \wedge L_{q+1} \lra L_{q+1}$ factors through 
$L_q \wedge L_q$, we see that the inclusion 
$\bigwedge L_q^{\vee} \hra \bigwedge L_{q+1}^{\vee}$ is a Hirsch 
extension. It is easy to see that $H^0(\bigwedge L_q^{\vee}) = k$. 
It is also easy to see that the homology of  
$$ L_q \wedge L_q \os{-[\phantom{a},\phantom{a}]}{\lra} L_q \os{0}{\lra} k $$
is $L_1$ and so $H^1(\bigwedge L_q^{\vee}) = L_1^{\vee} = V$. 
Hence, to see that $\rho_q$ is the $(1,q)$-minimal model for any $q \in \N$, it 
suffices to see that the map 
$H^2(\bigwedge L_q^{\vee}) \lra H^2(\bigwedge L_{q+1}^{\vee})$ is zero. 
To prove this, we embed the complex 
$$ \bigwedge L_{q+1}^{\vee} := [k \lra L_{q+1}^{\vee} \lra \bigwedge^2 L_{q+1}^{\vee} 
\lra  \bigwedge^3 L_{q+1}^{\vee} \lra \cdots ] $$
in  the complex 
$$ 
(\bigwedge L)^* := [k \lra L^* \lra (\bigwedge^2 L)^*  
\lra (\bigwedge^3 L)^* \lra \cdots ] 
$$ 
calculating the Lie algebra cohomology $H^n_{\rm Lie}(L,k)$. 
(Here, for a possibly infinite-dimensional vector space $W$, we denoted 
its dual by $W^*$ to distinguish from the notation $L^{\vee}$, which is 
the ind-finite dimensional dual of the pro-finite dimensional vector space 
$L$.) Because $L$ is a free Lie algebra of rank $2g+r-1$, the category 
of $L$-modules is identified with the category of modules over 
a free group $G$ of rank $2g+r-1$. Because the cohomological dimension 
of $G$ is $1$, we see that 
$H^n((\bigwedge L)^*) = H^n_{\rm Lie}(L,k) = H^n(G,k) = 0$ for 
$n \geq 2$. Now let $\varphi$ be a $2$-cocycle of the complex 
$\bigwedge L_{q}^{\vee}$, regarded as the linear form 
$\varphi: \bigwedge^2 L_{q} \lra k$. Since it is a $2$-coboundary 
in the complex $(\bigwedge L)^*$, we see that there exists a 
linear form $\psi: L \lra k$ such that 
the composite $\bigwedge^2 L  \lra \bigwedge^2 L_{q} \os{\varphi}{\lra} k$ 
is written as 
$\bigwedge^2 L  \os{-[\phantom{a},\phantom{a}]}{\lra} 
L \os{\psi}{\lra} k.$
Since $\Fil^{q+1} L \otimes L$ is sent to zero by the former map, so is 
by the latter map and so $\psi$ factors through $L_{q+1}$. Thus we see that 
$\varphi$ is a $2$-coboundary in the complex $\bigwedge L_{q+1}^{\vee}$. 
So we have shown that the map  
$H^2(\bigwedge L_q^{\vee}) \lra H^2(\bigwedge L_{q+1}^{\vee})$ is zero, 
as required. So the proof is finished. 
\end{proof}

Next we calculate $L(X/S)$ in the unmarked case by comparing with 
the marked case. In the following, we denote the free pro-nilpotent 
Lie algebra of rank $n$ with generator $w_1, ..., w_n$ by 
$L(w_1, ..., w_n)$ and denote the ideal of $L(w_1, ..., w_n)$ generated by 
an element $y \in L(w_1, ..., w_n)$ by $\langle y \rangle$. 
Then we have the following proposition, which is the second main result 
of this subsection: 

\begin{prop}\label{prope} 
When $r=0$, we have the isomorphism of pro-nilpotent Lie algerbras  
$$ L(X/S) \cong L(w_1, ..., w_{2g})/\langle \sum_{i=1}^g [w_{2i-1},w_{2i}] \rangle. $$
\end{prop}

\begin{proof}  
Let $P^{\circ}$ be the image of the section $\iota: S \lra X$ 
(with trivial log structure). Let $f': X' \lra S$ be the proper log curve 
we obtain by adding a marked point at $P^{\circ}$ to $X$, 
and let $P$ be the log scheme 
$(P^{\circ}, \cM_{X'}|_{P^{\circ}})$. Also, let 
$f_P: P \lra S$ be the composite $P \hra X' \os{f'}{\lra} S$. 
Then we can define the category $\MICn(P/S)$, which 
is identified with the caegory of pairs $(V,N)$ of a finite-dimensional 
$k$-vector space $V$ and a nilpotent endomorphism $N$ on $V$. 
Thus the Lie algebra $L(P/S)$ associated to its Tannaka dual, ${\mathbb{G}}_{a,k}$, is 
the abelian Lie algebra $k$. Also, we have the commutative diagram 
\begin{equation*}
\begin{CD}
\MICn(P/S) @<<< \MIC(S/S) \\ 
@AAA @AAA \\ 
{\mathrm{N}}_{f'}\MIC(X'/S) @<<< \NfMIC(X/S) 
\end{CD}
\end{equation*}
of categories. We can associate to it the corresponding commutative diagram 
\begin{equation}\label{eq:polie}
\begin{CD}
L(P/S) @>>> 0 \\ 
@VVV @VVV \\ 
L(X'/S) @>>> L(X/S)  
\end{CD}
\end{equation}
of pro-nilpotent Lie algebras, by taking the pro-nilpotent Lie algebras associated to 
the Tannaka duals of the categories in the diagram. Since an object in 
$\NfMIC(X/S)$ is nothing but an object in 
${\mathrm{N}}_{f'}\MIC(X'/S)$ trivial in 
${\mathrm{N}}_{f_P}\MIC(P/S)$, the diagram \eqref{eq:polie} is 
a push-out diagram in the category of pro-nilpotent Lie algebras. 
Thus $L(X/S)$ is the quotient of $L(X'/S)$ by the ideal generated 
by one element, denoted with $z$, which is the image of $1$ by $k = L(P/S) \lra L(X'/S)$. 

To identify the element $z$ (up to scalar), we investigate the construction of 
$(1,q)$-minimal models of $\cA_{X/S}, \cA_{X'/S}$ for $q=1,2$. 
First we put $V_1 := H^1(\cA_{X/S}) = H^1(\cA_{X'/S})$ 
(the latter equality follows from the computation in the proof of 
Proposition \ref{prop:coh}) and take sections 
$\sigma_1: V_1 \lra Z^1(\cA_{X/S}), \sigma'_1:V_1 \lra Z^1(\cA_{X'/S})$ 
of the canonical projections 
$Z^1(\cA_{X/S}) \lra H^1(\cA_{X/S}), Z^1(\cA_{X'/S}) \lra H^1(\cA_{X'/S})$ 
in a compatible way. Using the maps $\sigma_1, \sigma'_1$, 
we obtain the diagram of $k$-cdga's 
\begin{equation*}
\begin{CD}
\bigwedge V_1 @>>> \cA_{X/S} \\ 
@| @VVV \\ 
\bigwedge V_1 @>>> \cA_{X'/S}
\end{CD}
\end{equation*}
(with zero differentials on $\bigwedge V_1$) 
in which the horizontal arrows are $(1,1)$-minimal models. 
Next we put 
\begin{align*}
& V_2 := H^2(\bigwedge V_1, \cA_{X/S}) = \Ker(\bigwedge^2 V_1 \lra H^2(\cA_{X/S})), \\ 
& V'_2 := H^2(\bigwedge V_1, \cA_{X'/S}) = \bigwedge^2 V_1 
\end{align*}
and take sections 
$\sigma_2: V_2 \lra Z^2(\bigwedge V_1, \cA_{X/S}), \sigma'_2:V'_2 \lra 
Z^2(\bigwedge V_1, \cA_{X'/S})$ 
of the canonical projections 
$Z^2(\bigwedge V_1, \cA_{X/S}) \lra H^2(\bigwedge V_1, \cA_{X/S}), 
Z^2(\bigwedge V_1, \cA_{X'/S}) \lra H^2(\bigwedge V_1, \cA_{X'/S})$ 
in a compatible way. Using the $(1,1)$-minimal models and 
the maps $\sigma_2, \sigma'_2$, we obtain 
the diagram of $k$-cdga's 
\begin{equation*}
\begin{CD}
\bigwedge (V_1 \oplus V_2) @>>> \cA_{X/S} \\ 
@VVV @VVV \\ 
\bigwedge (V_1 \oplus V'_2) @>>> \cA_{X'/S}
\end{CD}
\end{equation*}
(where the differential on 
$\bigwedge (V_1 \oplus V_2)$ 
is induced by 
$V_2 \hra \bigwedge^2 V_1$ and that on  
$\bigwedge (V_1 \oplus V'_2)$ is similarly defined) 
in which the horizontal arrows are $(1,2)$-minimal models. 
From this construction, we see that the surjection of Lie algebras  
\begin{equation}\label{eq:lie1}
L(X'/S)/\Fil^3L(X'/S) \lra L(X/S)/\Fil^3L(X/S) 
\end{equation}
is identified with the surjection 
\begin{equation}\label{eq:lie2}
(V_1 \oplus V'_2)^{\vee} \lra (V_1 \oplus V_2)^{\vee}. 
\end{equation}
Let $v_1, ..., v_{2g}$ be a symplectic basis of $V_1$ with respect to 
the map $\bigwedge^2 V_1 \lra H^2(\cA_{X/S})$ so that 
$v_{2i-1}\wedge v_{2i} \, (1 \leq i \leq g)$ is sent to a common non-zero element 
of $H^2(\cA_{X/S})$ and that $v_m \wedge v_n$ is sent to zero unless 
$\{m,n\} = \{2i-1, 2i\}$ for some $1 \leq i \leq g$, and let 
$w_1, ..., w_{2g}$ be the dual basis on $V_1^{\vee}$. 
(Note that $\bigwedge^2 V_1 \lra H^2(\cA_{X/S})$ is a surjective 
map to $1$-dimensional $k$-vector space because it is identified with 
the cup product pairing $\bigwedge^2 H^1_{\dR}(X/S) \lra H^2_{\dR}(X/S)$
which is perfect by Poincar\'e duality.) Then 
${V'_2}^{\vee}$ has a basis $w_j \wedge w_l \,(1 \leq j < l \leq 2g)$, and 
$-\sum_{i=1}^g w_{2i-1} \wedge w_{2i} \in {V'_2}^{\vee} \subseteq 
(V_1 \oplus V'_2)^{\vee} $ is a generator (as $k$-vector space) 
of the kernel of the surjection 
\eqref{eq:lie2}. Because this element is equal to 
$\sum_{i=1}^g [w_{2i-1}, w_{2i}]$ by definition of the Lie bracket on 
$(V_1 \oplus V'_2)^{\vee} = L(X'/S)/\Fil^3L(X'/S)$ and 
$L(X'/S)$ is a free pro-nilpotent Lie algebra with generator 
$w_1, ..., w_{2g}$ by Proposition \ref{propr>0}, we see 
that there exists an isomorphism 
$$ L(X'/S)/\Fil^3L(X'/S) \cong L(w_1, ..., w_{2g})/\Fil^3 L(w_1, ..., w_{2g}) $$ 
such that the image of $z$ in $L(X'/S)/\Fil^3L(X'/S)$ is equal to 
$\sum_{i=1}^g [w_{2i-1}, w_{2i}]$ up to scalar. 

To prove the proposition, it suffices to prove that there exists 
a compatible family of isomorphisms  
$$ L(X'/S)/\Fil^{q+1}L(X'/S) \cong L(w_1, ..., w_{2g})/\Fil^{q+1} L(w_1, ..., w_{2g}) 
\quad (q \geq 2) $$ 
such that the image of $z$ in $L(X'/S)/\Fil^{q+1}L(X'/S)$ is equal to 
$\sum_{i=1}^g [w_{2i-1}, w_{2i}]$ up to scalar, by induction on $q$. 
Assume the claim for $q$ and take 
a lift of $w_i \in L(X'/S)/\Fil^{q+1} L(X'/S)$ to $L(X'/S)/\Fil^{q+2} L(X'/S)$, 
which we denote by the same letter. 
Then the image of $z$ in 
$L(X'/S)/\Fil^{q+2} L(X'/S)$ is the same as the image of 
$$  \sum_{i=1}^g [w_{2i-1}, w_{2i}] + y $$
in $L(X'/S)/\Fil^{q+2} L(X'/S)$ for some $y \in \Fil^{q+1} L(X'/S)$ 
up to scalar. Since $w_1, \dots, w_{2g}$ generates 
$\Fil^1 L(X'/S)/\Fil^2 L(X'/S)$, $y$ is equal to an element of the form 
$$\sum_{i=1}^g [w_{2i-1},y_{2i}] + 
\sum_{i=1}^g [y_{2i-1},w_{2i}] $$ 
in $L(X'/S)/\Fil^{q+2} L(X'/S)$ for some $y_i \in \Fil^q L(X'/S) \,(1 \leq i \leq 2g)$. 
Hence the image of $z$ in $L(X'/S)/\Fil^{q+2} L(X'/S)$ has the form 
$$  \sum_{i=1}^g [w_{2i-1}, w_{2i}] + \sum_{i=1}^g [w_{2i-1},y_{2i}] + 
\sum_{i=1}^g [y_{2i-1},w_{2i}] $$ 
up to scalar. 
Since the above element can be  rewritten as 
$$  \sum_{i=1}^g [w_{2i-1}+y_{2i-1}, w_{2i}+y_{2i}] $$ 
in $L(X'/S)/\Fil^{q+2} L(X'/S)$ and by denoting $w_i + y_i$ again by   $w_i$, 
we obtain the isomorphism 
$$ L(X'/S)/\Fil^{q+2}L(X'/S) \cong L(w_1, ..., w_{2g})/\Fil^{q+2} L(w_1, ..., w_{2g}) $$
such that the image of $z$ in $L(X'/S)/\Fil^{q+2}L(X'/S)$ is equal to 
$\sum_{i=1}^g [w_{2i-1}, w_{2i}]$ up to scalar. So the proof of the proposition 
is finished. 
\end{proof}

\subsection{Tangential sections}\label{tangential points}

In this subsection, we assume that $k$ is algebraically closed,  and let 
$f: X \lra S$ be a semistable log curve and  $x$ be a double point of $X$ 
endowed with the pullback log structure. Then, as we saw in 
Remark \ref{good_notgood_section}, there is no section of $f$ with image in $x$. 
However, in this subsection, we define two canonical  {\it `sections with image $x$' }
in Tannaka theoretic sense, which we  will call the {\it tangential sections}. 
Tangential sections will play an important role in the calculation of monodromy. 

Let $f: X \lra S$, $x$ be as above and let $f_x$ be the composite 
$x \hra X \lra S$. Then $f_x$ is associated to a chart 
\begin{align*}
& \N^2 \to k, \quad 
\N \to k, \quad \N \to \N^2. \\ 
& \! (1,0) \mapsto 0 \quad \,\,\, 1 \mapsto 0 \quad 1 \mapsto 
(1,1) \\ 
& \! (0,1) \mapsto 0 
\end{align*} 
We fix such a chart. 
The categories $\MICn(x/k)$, $\MICn(x/S)$ are defined as before and we have the natural functors 
\begin{align*}
& r: \MICn(x/k) \lra \MICn(x/S) \quad \text{(restriction functor)}, \\ 
& x_{\dR}^*: \MICn(x/k) \lra \MIC(x/x) = \Vector_k, \\ 
& x_{\dR}^*: \MICn(x/S) \lra \MIC(x/x) = \Vector_k, \\ 
& f_{x,\dR}^*: \MICn(S/k) \lra \MICn(x/k). 
\end{align*}
Also, by the above expression of the log scheme $x$, we see that 
the category $\MICn(x/k)$ is equivalent to that of triples  
$(V,N,N')$ consisting of a finite-dimensional $k$-vector space $V$ and 
two commuting nilpotent endomorphisms $N, N'$ on $V$. 

Now, let $\omega, \omega': \MICn(x/k) \lra \MICn(S/k)$ be the functors 
defined by $(V,N,N') \mapsto (V,N)$ and  $(V,N,N') \mapsto (V',N')$, respectively. 
Then we have the following diagram of functors in which the three evident triangles are commutative
\begin{equation}\label{eq:tang}
\xymatrix{
\MICn(x/S) \ar[rd]^{x_{\dR}^*} & 
\MICn(x/k) \ar[l]^r \ar[d]^{x_{\dR}^*} \ar@/^1pc/[r]^{\omega}& 
\MICn(S/k) \ar[l]^-{f_{x,\dR}^*} \ar[ld]^{s_{\dR}^*} \\ 
& \Vector_k,  
}
\end{equation}
and a similar one with $\omega$ replaced by $\omega'$. 
If we denote the Tannaka dual of $(\MICn(x/k), x_{\dR}^*)$ and 
$(\MICn(x/S), x_{\dR}^*)$ by $\pi_1^{\dR}(x,x)$ and $ \pi_1^{\dR}(x/S,x)$, respectively, 
we see easily from the above description of the category $\MICn(x/k)$ that 
we have the split exact sequence 
$$ 
\xymatrix {
1 \ar[r] & \pi_1^{\dR}(x/S,x) \ar[r] & 
\pi^{\dR}_1(x,x) = \G_{a,k}^2 \ar[r]^{f_{x,*}} & \pi^{\dR}_1(S,s) = \G_{a,k} \ar[r] 
\ar @/^6mm/[0,-1]^{\omega^*} & 1. 
} 
$$
and a similar split exact sequence with $\omega^*$ replaced by ${\omega'}^*$. 
We see also that the map $f_{x,*}$ is the sum and the map 
$\omega^*$ (resp. ${\omega'}^*$) is the inclusion of $\G_{a,k}$ into the 
first (resp. the second) component of $\G_{a,k}^2$. Hence, if we put 
$\epsilon := \omega^*(1)$ and $\epsilon' := {\omega'}^*(1)$, 
${\epsilon'}^{-1} \circ \epsilon$ is equal to a non-zero (and non-torsion) element 
$\eta$ of $\pi_1^{\dR}(x/S,x)$. 

A concrete description of the element $\eta$ is given in the following way. 
An object in $\MICn(x/S)$ is given by a finite-dimensional vector space $V$ endowed 
with a nilpotent map $N:=(N_1,N_2):  V \lra V \otimes_{\Z} (\Z^2/\Delta(\Z))$, 
where $\Delta:\Z \lra \Z^2$ is the diagonal map. Then the action of $\eta$ on such 
object is given by $\exp(N_1-N_2)$. 

Now we compose the above construction with the canonical 
inclusion $x \hra X$. Let $\omega_x$ be the composite 
$\NfMICn(X/k) \lra \MICn(x/k) \os{\omega}{\lra} \MICn(S/k)$ and 
define $\omega'_x$ similarly, with $\omega$ replaced by $\omega'$. 
The functors $\omega_x, \omega'_x$ are called tangential sections of $X$ at $x$. 
Then we have the following diagram of functors in which the three evident triangles are commutative
\begin{equation}\label{eq:tang2}
\xymatrix{
\NfMIC(X/S) \ar[rd]^{x_{\dR}^*} & 
\NfMICn(X/k) \ar[l]^r \ar[d]^{x_{\dR}^*} \ar@/^1pc/[r]^{\omega_x}& 
\MICn(S/k) \ar[l]^-{f_{\dR}^*} \ar[ld]^{s_{\dR}^*} \\ 
& \Vector_k,  
}
\end{equation}
and a similar commutative diagram with $\omega_x$ replaced by $\omega'_x$. 
If we denote the Tannaka dual of $(\NfMICn(X/k), x_{\dR}^*), 
(\NfMIC(X/S), x_{\dR}^*)$ by $\pi_1^{\dR}(X,x), \pi_1^{\dR}(X/S,x)$ respectively, 
we obtain the split exact sequence 
$$ 
\xymatrix {
1 \ar[r] & \pi_1^{\dR}(X/S,x) \ar[r] & 
\pi^{\dR}_1(X,x) \ar[r]^{f_{x,*}} & \pi^{\dR}_1(S,s) = \G_{a,k} \ar[r] 
\ar @/^6mm/[0,-1]^{\omega_x^*} & 1. 
} 
$$
and a similar split exact sequence with $\omega^*_x$ replaced by ${\omega'}_x^*$, 
by Proposition \ref{prop:HES_t}. If we put 
$\epsilon_x := \omega_x^*(1)$ and $\epsilon'_x := {\omega'}_x^*(1)$, 
${\epsilon'}_x^{-1} \circ \epsilon_x$ is equal to the image of 
$\eta \in \pi_1^{\dR}(x/S,x)$ in $\pi_1^{\dR}(X/S,x)$. 

\begin{rem}\label{rem:direc}
As we see above, there are two tangential sections at a double point $x$ of $X$. 
We will use the following geometric description 
to indicate  precisely one choice out of the two. 
Let $\wt{X} \lra X$ be the normalization of $X^{\circ}$, endowed with the 
pull-back log structure from that on $X$. Then the inverse image of $x$ consists of 
two points $x_1, x_2$ (endowed with pullback log structure) 
which are both isomorphic to  $x$ (as log scheme). If we take one point $x_1$,  
then there is a unique connected component $X_1$ of $\wt{X}$ containing 
$x_1$;  in  our fixed chart $\N^2 \lra \cO_x = \cO_{x_1}$, 
there is exactly one copy of $\N$  which extends to the zero map 
$\N \lra \cO_{X_1}$ in the log structure $\cM_{X_1} \lra \cO_{X_1}$. 
Also, if we take the other point $x_2$, the corresponding copy of 
$\N$ in $\N^2$ is different from the one we obtain from $x_1$. 
Thus, by  indicating ``the tangential section at $x_1$ (or $x_2$)'', we can 
identify precisely one of the two tangential sections at $x$. 
\end{rem}

In the rest of this subsection, we give some calculation of local monodromy. 
Let $f: X \lra S$ be a semistable log curve 
and let $X_1$ be a connected component of the normalization of $X^{\circ}$, 
endowed with the pullback log structure from that on $X$. 
Let $C_1 \subseteq X_1$ be the inverse image of the double points of $X$,  
let $D_1 \subseteq X_1$ be the inverse image of the marked points of $X$. 
Then, etale locally around a point in $C_1$, the underlying morphism 
of the map $f_1: X_1 \lra S$ is given as 
$$ X_1^{\circ} \lra \Spec k[t] \lra \Spec k$$ 
with first morphism etale, and 
the log structure of $f_1$ is associated to the chart 
\begin{align}
& \N^2 \to \cO_{X_1}, \quad 
\N \to k, \quad \N \to \N^2. \label{eq:chartN2} \\ 
& \! (1,0) \mapsto 0 \quad \,\,\, 1 \mapsto 0 \quad 1 \mapsto 
(1,1) \nonumber \\ 
& \! (0,1) \mapsto t \nonumber 
\end{align} 
Let $P$ be the submonoid of $\N^2$ generated by 
$(1,1), (0,1)$ and define the morphisms of log schemes 
$f'_1: X'_1 \lra S$ by changing the log structure around $C_1$ 
to the one associated to the chart 
\begin{align}
& P \to \cO_{X_1}, \quad 
\N \to k, \quad \N \to P. \label{eq:chartP} \\ 
& \! (1,1) \mapsto 0 \quad \,\,\, 1 \mapsto 0 \quad 1 \mapsto 
(1,1) \nonumber \\ 
& \! (0,1) \mapsto t \nonumber 
\end{align} 
Changing the log structure in this way has the effect that 
the points in $C_1$    are changed to marked points of $X'_1$ (viewed with the log structure relative to $S$)
On the other hand, since $P^{\rm gp} = \Z^2$, we have the 
equality of log differential modules $\Omega^1_{X_1/S} = \Omega^1_{X'_1/S}$ 
and so we have the equivalences  
\begin{equation}\label{eq:xx'}
{\mathrm{N}}_{f_1}\MIC(X_1/S) = 
{\mathrm{N}}_{f'_1}\MIC(X'_1/S), \quad {\mathrm{N}}_{f_1}\MICn(X_1/k) = 
{\mathrm{N}}_{f'_1}\MICn(X'_1/k). 
\end{equation}

Next we define the log scheme $X''_1$ to be the scheme $X_1^{\circ} = 
{X'_1}^{\circ}$ endowed with the log structure associated to 
the normal crossing divisor $C_1 \cup D_1$. Then we have the equivalence 
\begin{equation}\label{eq:x'x''}
{\mathrm{N}}_{f'_1}\MIC(X'_1/S) = {\mathrm{N}}_{f''_1}\MIC(X''_1/k), 
\end{equation}
where $f''_1: X''_1 \lra k$ is the structure morphism (no log structure on $k$).

For closed points $x, y$ of $X$ (endowed with the pullback log structure from that one of $X$), let 
$\pi_1^{\dR}(X/S,x,y)$ be the $k$-scheme representing the set of tensor isomorphisms 
from the functor  
$x_{\dR}^*: \NfMIC(X/S) \to \MIC(x/x) = \Vector_k$ to the functor 
$y_{\dR}^*: \NfMIC(X/S) \to \MIC(y/y) = \Vector_k$. 
For closed points $x_1, y_1$ of $X_1$, we define the $k$-scheme 
$\pi_1^{\dR}(X_1/S,x_1,y_1)$ in a similar way, using the category 
${\mathrm{N}}_{f_1}\MIC(X_1/S)$ instead of $\NfMIC(X/S)$. 
By \cite[Theorem 3.2]{DelMil82}, after we extend the field $k$, both $\pi_1^{\dR}(X_1/S,x_1,y_1)$ and $\pi_1^{\dR}(X_1/S,x,y)$  have a $k$-rational point 
 (and we are allowed to make field extensions by Proposition \ref{prop:!!bc}).  
If $x_1, y_1$ are sent to $x, y$ respectively by the map $X_1 \lra X$, 
we have the corresponding morphism 
$\pi_1^{\dR}(X_1/S,x_1,y_1) \lra \pi_1^{\dR}(X/S,x,y)$. 

Then we have the following: 

\begin{prop}\label{prop:comp1}
Let the notations be as above. \\
$(1)$ \, Let $x$ be a double point of $X$ with fixed inverse image $x_1$ in $X_1$, and 
let $\alpha$ be an element of $\pi_1^{\dR}(X/S,x)$ coming from 
$\pi_1^{\dR}(X_1/S,x_1)$. Let $\omega$ be the tangential section at $x_1$ and 
put $\epsilon := \omega^*(1) \in \pi_1^{\dR}(X,x)$. Then we have 
$\epsilon^{-1} \circ \alpha \circ \epsilon = \alpha$. \\ 
$(2)$ \, Let $x, y$ be double points $($possibly equal$)$ of $X$ with fixed distinct inverse images 
$x_1, y_1$ in $X_1$, and 
let $\alpha$ be an element of $\pi_1^{\dR}(X/S,x,y)$ coming from 
$\pi_1^{\dR}(X_1/S,x_1,y_1)$. Let $\omega_x, \omega_y$ be the tangential section at $x_1$,  
$y_1$ respectively and 
put $\epsilon_x := \omega_x^*(1) \in \pi_1^{\dR}(X,x), 
\epsilon_y := \omega_y^*(1) \in \pi_1^{\dR}(X,y)$. Then we have 
$\epsilon_y^{-1} \circ \alpha \circ \epsilon_x = \alpha$. 
\end{prop}

\begin{prop}\label{prop:comp2}
Let the notations be as above. 
Assume that $X_1$ is isomorphic to $\Pr^1_k$ with $D_1 = \emptyset, 
C_1 = \{x_1, y_1\}$ and that the image of $x_1, y_1$ in $X$, which we denote by $x, y$, are 
different. Let $\alpha$ be the element of $\pi_1^{\dR}(X_1/S,x_1,y_1)$ defined by 
$$ E|_x \os{\cong}{\lla} \Gamma(X_1,E) \os{\cong}{\lra} E|_y \quad ((E,\nabla) \in 
{\mathrm{N}}_{f_1}\MIC(X_1/S)) $$
and denote its image in $\pi_1^{\dR}(X/S,x,y)$ by the same letter. 
On the other hand, let $\omega_x$ $($resp. $\omega_y)$ be the tangential section 
at $x_1$ $($resp. $y_1)$ and let $\omega'_x$ $($resp; $\omega'_y)$ be 
the tangential section at $x$ $($resp. $y)$ which is different from 
$\omega_x$ $($resp. $\omega_y)$. Also, let 
$\eta_x$ $($resp. $\eta_y)$ be the image of 
${\omega'}_x^*(1)^{-1} \circ \omega_x^*(1)$ $($resp. ${\omega'}_y^*(1)^{-1} \circ \omega_y^*(1))$
in $\pi_1^{\dR}(X/S,x)$ $($resp,. $\pi_1^{\dR}(X/S,y))$.  
Then we have $\eta_y^{-1} \circ \alpha = \alpha \circ \eta_x$. 
\end{prop}

\begin{rem}\label{rem:20190306}
Note that, when $X_1$ is isomorphic to $\Pr^1_k$, the underlying module
$E$ is a finite direct sum of copies of 
$\cO_{X_1}$ for any $(E,\nabla) \in {\mathrm{N}}_{f_1}\MIC(X_1/S)$ 
because $H^1(X_1, \cO_{X_1}) = 0$ (see \cite[Proposition 12.3]{Del89}). 
Thus the element $\alpha$ of $\pi_1^{\dR}(X_1/S,x_1,y_1)$
in the statement of Proposition \ref{prop:comp2} is well-defined. 
\end{rem}

\begin{proof}[Proof of Proposition \ref{prop:comp1}]
(1) \, First, note that we have the diagram 
\begin{equation}\label{eq:x_1HES}
\xymatrix {
1 \ar[r] & \pi_1^{\dR}(X_1/S,x_1) \ar[r] & 
\pi^{\dR}_1(X_1,x_1) \ar[r]^(.43){f_{1,*}} & \pi^{\dR}_1(S,s) = \G_{a,k} \ar[r] 
\ar @/^6mm/[0,-1]^{\omega^*} & 1. 
} 
\end{equation} If the above exact sequence is exact, it defines the action of  $\pi^{\dR}_1(S,s)$ on $\pi_1^{\dR}(X_1/S,x_1)$.
Since $X_1$ is not log smooth over $S$, it is not immediate that 
the previous sequence is exact. However, because of the equivalences 
\eqref{eq:xx'}, the sequence can be  identified with the sequence 
\begin{equation*}
\xymatrix {
1 \ar[r] & \pi_1^{\dR}(X'_1/S,x_1) \ar[r] & 
\pi^{\dR}_1(X'_1,x_1) \ar[r]^(.43){f_{1,*}} & \pi^{\dR}_1(S,s) = \G_{a,k} \ar[r] & 1, 
} 
\end{equation*}
which is known to be exact because $X'_1$ is log smooth over $S$. 
Thus the diagram \eqref{eq:x_1HES} is a split exact sequence.  So we can define the inner action of $\epsilon=\omega^*(1) \in \pi^{\dR}_1(X_1,x_1)$ on $ \pi_1^{\dR}(X_1/S,x_1)$
and so it suffices to prove the required equality in 
$\pi_1^{\dR}(X_1/S,x_1)$. For an object $(E,\nabla)$ in ${\mathrm{N}}_{f_1}\MICn(X_1/k)$, 
if we denote its image  in $\MICn(x_1/k)$ by $(E|_{x_1}, N,N')$, the action of 
$\epsilon$ on $E|_{x_1}$ is given by $\exp(N)$. Also, since 
$\epsilon^{-1} \circ \alpha \circ \epsilon$ is an element in 
$\pi_1^{\dR}(X_1/S,x_1)$, the action of it on $E|_{x_1}$ depends only on 
the image $(E,\ol{\nabla})$ of $(E,\nabla)$ in ${\mathrm{N}}_{f_1}\MICn(X_1/S)$. 
So we can take an object $(E,\nabla'')$ in ${\mathrm{N}}_{f''_1}\MICn(X''_1/k)$ corresponding to 
$(E,\ol{\nabla})$ by the equivalences of categories \eqref{eq:xx'}, \eqref{eq:x'x''}, and 
replace $(E,\nabla)$ by the image of $(E,\nabla'')$ in 
${\mathrm{N}}_{f_1}\MICn(X_1/k)$. If we do so, we have $N=0$. So we obtain that 
$\epsilon^{-1} \circ \alpha \circ \epsilon = \alpha$. \\ 
(2) \, By the same argument as in (1), it suffices to prove the equality 
$\alpha^{-1} \circ \epsilon_y^{-1} \circ \alpha \circ \epsilon_x = \id$ 
in $\pi_1^{\dR}(X_1/S,x_1)$. 
For an object $(E,\nabla)$ in ${\mathrm{N}}_{f_1}\MICn(X_1/k)$, 
if we denote the image of it in $\MICn(x_1/k)$ (resp.  $\MICn(y_1/k)$) 
by $(E|_{x_1}, N_x,N'_x)$ (resp. $(E|_{y_1}, N_y,N'_y)$), the action of 
$\epsilon_x$ on $E|_{x_1}$ (resp. $\epsilon_y$ on $E|_{y_1}$) is given by 
$\exp(N_x)$ (resp. $\exp(N_y)$). Also, since 
$\alpha^{-1} \circ \epsilon_y^{-1} \circ \alpha \circ \epsilon_x$ is an element in 
$\pi_1^{\dR}(X_1/S,x_1)$, the action of it on $E|_{x_1}$ depends only on 
the image $(E,\ol{\nabla})$ of $(E,\nabla)$ in ${\mathrm{N}}_{f_1}\MICn(X_1/S)$. 
If we replace $(E,\nabla)$ as in (1), we have $N_x=0$, $N_y=0$ and so we obtain that 
$\alpha^{-1} \circ \epsilon_y^{-1} \circ \alpha \circ \epsilon_x = \id$. 
\end{proof}

\begin{proof}[Proof of Proposition \ref{prop:comp2}]
As in the proof of Proposition \ref{prop:comp1}, it suffices to prove the 
equality $\alpha^{-1} \circ \eta_y \circ 
\alpha \circ \eta_x = \id$
in $\pi_1^{\dR}(X_1/S,x_1)$. If we define 
$(E|_{x_1}, N_x,N'_x)$, $(E|_{y_1}, N_y,N'_y)$ as in the proof of Proposition \ref{prop:comp1}(2), 
the action of $\eta_x$ on $E|_{x_1}$ (resp. $\epsilon_y$ on $E|_{y_1}$) is given by 
$\exp(N_x-N'_x)$ (resp. $\exp(N_y-N'_y)$), and if we replace $(E,\nabla)$ as in the proof 
of Proposition \ref{prop:comp1}(2), we have $N_x=0, N_y=0$. Also, via the identification 
by $\alpha$, we have the equality of operators $N'_x + N'_y = 0$ by 
the residue theorem. (Indeed, $N'_x$ (resp. $N'_y$) is the residue of $\nabla$ at $x$ (resp. 
at $y$).) Thus $\exp(-N'_y) = \exp(-N'_x)^{-1}$ and so 
we have the required equality $\alpha^{-1} \circ \eta_y \circ \alpha \circ \eta_x = \id$. 
\end{proof}

\subsection{Calculation of monodromy (I)}

In this subsection, we prove Theorem \ref{!!!} in the case where the dual graph of $X$ has a loop. 

\begin{thm}\label{thm:loop}
Theorem \ref{!!!} holds when the dual graph of the geometric fiber of $f:X \lra S$ has a loop.
\end{thm}

In fact, we can prove the non-triviality of the monodromy action 
on the abelianization $\pi_1^{\dR}(X/S,\iota)^{\rm ab}$. 
Although there  exists a proof which uses only the cohomology theory 
(see \cite{morrison}, for example), 
we give a proof using our theory of  the relative de Rham fundamental groups. 

\begin{proof}
We may assume that $k$ is algebraically closed by Proposition \ref{prop:!!bc} 
and we may change the base point as we like by Proposition \ref{prop:HES_t}. 

Take a loop in the dual graph of $X$. Let $Y_1, ..., Y_n$ be 
the irreducible components which appear as vertices of the loop, 
and let $y_1, ..., y_n$ be 
the double points which appear as edges of the loop. 
When $n \geq 3$, we assume that $y_i \, (1 \leq i \leq n-1)$ is 
in $Y_i \cap Y_{i+1}$ and $y_n$ is $Y_n \cap Y_1$. 
Let $\wt{X} \lra X$ be the normalization of $X^{\circ}$ endowed with 
the pullback log structure of that on $X$, and let $X_i \,(1 \leq i \leq n)$ be 
the connected component of $\wt{X}$ corresponding to $Y_i$. 
Define $x_i, x'_i \, (1 \leq i \leq n)$ in the following way: 
When $n=1$, let $x_1$ be a point of the inverse image of $y_1$ in $X_1$ and 
let $x'_1$ be another point of the inverse image of $y_1$ in $X_1$. 
When $n \geq 2$, let $x_i \,(1 \leq i \leq n)$ be the inverse image of 
$y_i$ in $X_i$, let $x'_i \, (1 \leq i \leq n-1)$ be the inverse image of 
$y_i$ in $X_{i+1}$ and let $x'_n$ be the inverse image of 
$y_n$ in $X_1$. Let 
$\alpha_1$ be an element of $\pi_{\dR}(X_1/S, x'_n, x_1)$ 
and denote its image in $\pi_{\dR}(X/S, y_n, y_1)$ by the same letter. 
Let $\alpha_i \, (2 \leq i \leq n)$ be an element of 
 $\pi_{\dR}(X_i/S, x'_{i-1}, x_i)$ 
and denote its image in $\pi_{\dR}(X/S, y_{i-1}, y_i)$ by the same letter. 
(We may assume the existence of the elements $\alpha_i$ ($1 \leq i \leq n$) 
by extending $k$, which is allowed by Proposition \ref{prop:!!bc}.) 
We will calculate the monodromy action on the element 
\begin{equation}\label{eq:loop1}
\alpha_n \circ \alpha_{n-1} \circ \cdots \circ \alpha_1 
\in \pi_1^{\dR}(X/S,y_n). 
\end{equation}
Let $\omega_i, \omega'_i$ be the tangential section at $x_i, x'_i$ respectively 
and put $\epsilon_i := \omega_i^*(1), \epsilon'_i := {\omega'_i}^*(1), 
\eta_i := {\epsilon'_i}^{-1} \circ \epsilon_i$. Then, by applying the monodromy action 
with respect to ${\omega'_n}^*$ to the 
element \eqref{eq:loop1}, we obtain the element 
\begin{equation}\label{eq:loop2}
{\epsilon'}_n^{-1} \circ \alpha_n \circ \alpha_{n-1} \circ \cdots \circ \alpha_1 \circ 
\epsilon'_n. 
\end{equation}
It is equal to 
\begin{equation*} 
({\epsilon'}_n^{-1} \circ \epsilon_n) \circ (\epsilon_n^{-1} \circ \alpha_n \circ \epsilon'_{n-1}) \circ 
({\epsilon'}_{n-1}^{-1} \circ \epsilon_{n-1}) \circ \cdots \circ 
({\epsilon'}_1^{-1} \circ \epsilon_1) \circ (\epsilon_1^{-1} \circ \alpha_1 \circ \epsilon'_n), 
\end{equation*}
and by Proposition \ref{prop:comp1}(2), it is equal to 
\begin{equation}\label{eq:loop3}
\eta_n \circ \alpha_n \circ \eta_{n-1} \circ \cdots \circ \eta_1 \circ \alpha_1. 
\end{equation}
Since the difference of the elements \eqref{eq:loop1}, \eqref{eq:loop3} in 
the abelianization $\pi_1^{\dR}(X/S,y_n)^{\rm ab}$ is equal to 
\begin{equation}\label{eq:loop4}
\eta_n \cdot \eta_{n-1} \cdot \, \cdots\,  \cdot \eta_1, 
\end{equation}
it suffices to see that the element \eqref{eq:loop4} is nontrivial in $\pi_1^{\dR}(X/S,y_n)^{\rm ab}$. 

Let $\gamma_i$ be an element in $\Gamma(X_i, \Omega^1_{X_i/S})$ whose 
residue at $x_i$ is $1$, whose residue at $x'_{i-1}$ ($x'_n$ when $i=1$) 
is $-1$ and whose residues along other marked points and double points are zero. 
(Such an element exists by the residue theorem and Riemann--Roch.) 
Also, let $X_i \,(n+1 \leq i \leq n')$ be the connected component of $\wt{X}$ other than 
$X_i \,(1 \leq i \leq n)$, and for $n+1 \leq i \leq n'$, put 
$\gamma_i := 0 \in \Gamma(X_i, \Omega^1_{X_i/S})$. Then we see that $\gamma_i$'s glue 
and give an element $\gamma \in \Gamma(X,\Omega^1_{X/S})$. 
We consider the action of the element  \eqref{eq:loop4} on the object 
$$(E,\nabla) := \left(\cO_X^2, d + \begin{pmatrix} 0 & \gamma \\ 0 & 0 \end{pmatrix}\right) \in \NfMIC(X/S). $$
By definition of $\eta_i$ (and the concrete description given in the last subsection), 
it acts on $E|_{y_i} = k^2$ by the matrix $\exp \begin{pmatrix} 0 & 2  \\ 0 & 0 \end{pmatrix}$. 
Thus the action of \eqref{eq:loop4} is given by 
$\exp \begin{pmatrix} 0 & 2n  \\ 0 & 0 \end{pmatrix}$, and since it is not the identity, 
the monodromy action is nontrivial on $\pi_1^{\dR}(X/S,y_n)^{\rm ab}$. So the proof is finished. 
\end{proof}

\subsection{Calculation of monodromy (II)}

In this subsection, we reduce the proof of Theorem \ref{!!!} to the case where the dual graph 
of the geometric fiber is a line and all the components corresponding to non-terminal vertices are 
$\Pr_k^1$. 

To prove it, first we prove a proposition which allows us to reduce the proof of 
Theorem \ref{!!!} to the case with fewer irreducible components. Assume $k$ is 
algebraically closed, 
let $f: X \lra S$ be a semistable log curve, let $X_1, X_2$ be unions of 
some irreducible components of $X$ endowed with pullback log structure from that of $X$ 
and let $x$ be a double point of $X$ such that $X_1 \cup X_2 = X, X_1 \cap X_2 = \{x\}$. 
Etale locally around $x$, each $X_i$ has a log structure of the form 
\eqref{eq:chartN2}. We define the log scheme $X'_i$ by 
changing the log structure around $x$ to the one induced by 
the chart \eqref{eq:chartP}. As before, changing the log structure in this way 
has the effect that $x$ is changed to a marked point. Also, 
let $Q$ be the submonoid of $\N^2$ generated by $(1,1)$ and 
we define the log scheme $\ol{X}_i$ by 
changing the log structure around $x$ to the one induced by 
the chart
\begin{align*}
& Q \to \cO_{X_i}, \quad 
\N \to k, \quad \N \to Q. \\ 
& \! (1,1) \mapsto 0 \quad \,\,\, 1 \mapsto 0 \quad 1 \mapsto 
(1,1) \nonumber  
\end{align*} 
Changing the log structure in this way 
has the effect that $x$ is changed to a smooth point. 
We denote the structure morphisms $X_i \lra S, X'_i \lra S, \ol{X}_i \lra S$ by 
$f_i, f'_i, \ol{f}_i$ respectively. 
There are canonical morphisms $X_i \lra X'_i \lra \ol{X}_i$ over $S$ and, as before, 
we have the equivalences 
\begin{equation}\label{eq:xx'2}
{\mathrm{N}}_{f_i}\MIC(X_i/S) = 
{\mathrm{N}}_{f'_i}\MIC(X'_i/S), \quad {\mathrm{N}}_{f_1}\MICn(X_i/k) = 
{\mathrm{N}}_{f'_i}\MICn(X'_i/k). 
\end{equation}
Let $\omega$ be a tangential section at $x$. 
As in the proof of Proposition \ref{prop:comp1}, we can prove the existence of the 
split exact sequence 
\begin{equation*}
\xymatrix {
1 \ar[r] & \pi_1^{\dR}(X_i/S,x) \ar[r] & 
\pi^{\dR}_1(X_i,x) \ar[r]^{f_{i,*}} & \pi^{\dR}_1(S,s) = \G_{a,k} \ar[r] 
\ar @/^6mm/[0,-1]^{\omega^*} & 1. 
} 
\end{equation*}
by using the equivalence \eqref{eq:xx'2}. Also, since $\ol{f}_i: \ol{X}_i \lra S$ is a 
semistable log curve, we have the split exact sequence 
\begin{equation*}
\xymatrix {
1 \ar[r] & \pi_1^{\dR}(\ol{X}_i/S,x) \ar[r] & 
\pi^{\dR}_1(\ol{X}_i,x) \ar[r]^{\ol{f}_{i,*}} & \pi^{\dR}_1(S,s) = \G_{a,k} \ar[r] 
\ar @/^6mm/[0,-1]^{\omega^*} & 1. 
} 
\end{equation*}
Hence the monodromy action on 
$\pi_1^{\dR}(X_i/S,x)$,  $\pi_1^{\dR}(\ol{X}_i/S,x)$ are defined as 
the conjugate action by $\epsilon := \omega^*(1)$. 
Then we have the following proposition. 

\begin{prop}\label{discard} 
Let the notations be as above. Then, if the monodromy action on 
$\pi_1^{\dR}(\ol{X}_1/S,x)$ is nontrivial as an element in ${\rm Out}(\pi_1^{\dR}(\ol{X}_1/S,x))$, 
that on $\pi_1^{\dR}(X/S,x)$ is nontrivial as an element in ${\rm Out}(\pi_1^{\dR}(X/S,x))$. 
\end{prop}

\begin{proof}
Note that the monodromy action on the de Rham fundamental groups induces that on 
the associated Lie algebras. By functoriality, we have the following commutative diagrams  
of pro-nilpotent Lie algebras compatible with the monodromy action: 
\begin{equation}\label{eq:liemono}
\begin{CD}
L(x/S) @>>> L(X_1/S) \\ 
@VVV @VVV \\ 
L(X_2/S) @>>> L(X/S), 
\end{CD}
\qquad 
\begin{CD}
L(x/S) @>>> L(X_1/S) \\ 
@VVV @VVV \\ 
0 @>>> L(\ol{X}_1/S).  
\end{CD}
\end{equation}
Recall that we indicated by $\NfMIC(X/S)$ 
(resp. $\NfMIC(X_1/S)$, $\NfMIC(X_2/S)$, $\MICn(x/S)$) 
the subcategory of $\MIC(X/S)$ 
(resp. $\MIC(X_1/S)$, $\MIC(X_2/S)$, $\MIC(x/S)$) consisting of 
iterated extensions of trivial objects.  Since to give an object in $\MIC(X/S)$ is equivalent to 
give a pair of objects in $\MIC(X_1/S)$ and $\MIC(X_2/S)$ 
which are compatible in $\MIC(x/S)$, then  we see that the first diagram 
in \eqref{eq:liemono} is a push-out diagram in the category of 
pro-nilpotent Lie algebras compatible with the monodromy action. By the same reason, 
the second diagram 
in \eqref{eq:liemono} is also 
a push-out diagram in the category of 
pro-nilpotent Lie algebras compatible with the monodromy action.
Hence the diagrams in \eqref{eq:liemono} induces the surjective 
morphism of pro-nilpotent Lie algebras 
$$ L(X/S) \lra L(\ol{X}_1/S) $$
which is compatible with the monodromy action. 
Indeed, it is induced by the map between the subdiagrams 
\begin{equation*}
\begin{CD}
L(x/S) @>>> L(X_1/S) \\ 
@VVV \\ 
L(X_2/S), @.
\end{CD}
\qquad 
\begin{CD}
L(x/S) @>>> L(X_1/S) \\ 
@VVV \\ 
0, @. 
\end{CD}
\end{equation*}
which is compatible with the monodromy action. 
Hence we obtain the 
surjective morphism of group schemes 
$$ \pi_1^{\dR}(X/S,x) \lra \pi_1^{\dR}(\ol{X}_1/S,x) $$ 
which is compatible with the monodromy action. Therefore, 
if the monodromy action on 
$\pi_1^{\dR}(\ol{X}_1/S,x)$ is nontrivial as an element in 
${\rm Out}(\pi_1^{\dR}(\ol{X}_1/S,x))$, 
that on $\pi_1^{\dR}(X/S,x)$ is nontrivial 
as an element in ${\rm Out}(\pi_1^{\dR}(X/S,x))$. 
\end{proof}

When $k$ is algebraically closed and $f:X \lra S$ is a semistable log curve, 
an irreducible component of $X$ is called a {\it terminal component} if 
the corresponding vertex in the dual graph of $X$ is terminal 
(meets only one edge). 
Then the following theorem is the main result in this subsection: 

\begin{thm}\label{thm:line}
To prove Theorem \ref{!!!}, it suffices to prove it in the following case$:$ 
$k$ is algebraically closed, the dual graph of $X$ is a line and 
the non-terminal components of $X$ are isomorphic to $\Pr^1_k$ with no marked points. 
\end{thm}

\begin{proof}
Let $\Gamma$ be the dual graph of $X$. We may assume that $\Gamma$ is a tree 
by Theorem \ref{thm:loop}. Then there exist at least two terminal vertices:   
connect  two of them  by (the unique) line $L$. Then the line $L$ consists of vertices 
$v_1, ..., v_n$ and the edges $e_1, ..., e_{n-1}$ such that $e_i \, (1 \leq i \leq n-1)$ 
connects $v_i$ and $v_{i+1}$. Denote the irreducible component 
corresponding to $v_i$ by $X_i$ and the double point corresponding to $e_i$ by $x_i$. 

Suppose that there exists an edge $e$ not in $L$ which touches some vertex $v_i$
in $L$. (Because such a vertex $v_i$ cannot be terminal, we have $2 \leq i \leq n-1$.) 
Since $\Gamma$ is a tree, $\Gamma \setminus \{e\}$ is the disjoint union 
of subgraphs $\Gamma_1 \sqcup \Gamma_2$ with $L \subseteq \Gamma_1$. 
Let $Y_1$ (resp. $Y_2$) be the union of irreducible components of $X$ corresponding to 
the graph $\Gamma_1$ (resp. $\Gamma_2 \cup \{e, v_i\}$). Then we have 
$X = Y_1 \cup Y_2, Y_1 \cap Y_2 = \{x_i\}$. If we define $\ol{Y}_1$ from $Y_1$ as 
before Proposition \ref{discard}, we see by Proposition \ref{discard} that 
the theorem for $X$ is reduced to the theorem for $\ol{Y}_1$. Note that 
$\ol{Y}_1$ remains to be minimal: The only irreducible component in $\ol{Y}_1$ 
which is changed is $X_i$, and since $X_i$ intersects with $X_{i-1}$ and $X_{i+1}$, 
this change does not violate the minimality. Thus we could reduce to the case 
with fewer irreducible components. Repeating this procedure, we can reduce the 
theorem to the case that the dual graph of $X$ is equal to the line $L$. 

Next, let $\ol{X}$ be the log scheme which we obtain by changing the log structure 
around the marked points on $X_i \, (2 \leq i \leq n-1)$ to the pullback log structure 
from that on $S$. (This change has the effect to erase the marked points from 
$X_i \, (2 \leq i \leq n-1)$.) Then, since $\pi_1^{\dR}(X/S,\iota) \lra \pi_1^{\dR}(\ol{X}/S,\iota)$ 
is surjective morphism compatible with monodromy action, we can reduce the theorem 
for $X$ to the theorem for $\ol{X}$. So we may assume that 
$X_i \, (2 \leq i \leq n-1)$ has no marked points ( Proposition \ref{discard}).

Finally, if there is an index $i \, (2 \leq i \leq n-1)$ 
with $X_i$ nonisomorphic to $\Pr^1_k$, we take the minimal $i$ with this condition, 
let $L_1$ (resp. $L_2$) be the subgraph of $L$ whose vertices are 
$v_1, ..., v_i$ (resp. $v_{i+1}, ..., v_n$) and let $Y_1$ (resp. $Y_2$) be 
the union of irreducible components of $X$ corresponding to 
the graph $L_1$ (resp. $L_2 \cup \{e_i, v_i\}$). Then, 
if we define $\ol{Y}_1$ from $Y_1$ as 
before Proposition \ref{discard}, we see by Proposition \ref{discard} that 
the theorem for $X$ is reduced to the theorem for $\ol{Y}_1$ and that 
$\ol{Y}_1$ is a minimal semistable log curve satisfying 
the condition described in the statement of the theorem. So we are done. 
\end{proof}

\subsection{Calculation of monodromy (III)}

In this subsection, assume that $k$ is algebraically closed. Let 
$f:X \lra S$ be a minimal semistable log curve satisfying the condition stated in 
Theorem \ref{thm:line}, namely, a minimal semistable log curve 
such that the dual graph of $X$ is a line and 
the non-terminal components of $X$ are isomorphic to $\Pr^1_k$ with no marked point. 
In this subsection, we calculate the monodromy action on certain elements in 
the de Rham fundamental group of $X$ over $S$. 

Let us fix some notations. Let $Y_i \,(0 \leq i \leq n)$ be 
the irreducible components of $X$ and let $y_i \,(0 \leq i \leq n-1)$ be 
the double points of $X$ such that $Y_i \cap Y_{i+1} = \{y_i\}$. 
Note that, by our assumption, $Y_i \,(1 \leq i \leq n-1)$ are isomorphic to 
$\Pr_k^1$. We assume that $Y_i$'s, $y_i$'s are endowed with the pullback log structure of that on $X$. 
Let $\wt{X} \lra X$ be the normalization of $X$ endowed with the pullback log 
structure of that on $X$ and let $X_i \, (0 \leq i \leq n)$ be 
the connected component of $\wt{X}$ corresponding to $Y_i$. 
Let $x_i \,(0 \leq i \leq n-1)$ be the inverse image of $y_i$ in $X_i$ and 
let $x'_i \, (0 \leq i \leq n-1)$ be the inverse image of $y_i$ in $X_{i+1}$. 
Let $\alpha_0$ (resp. $\alpha_n$) 
be an element of $\pi_1^{\dR}(X_0/S,x_0)$ 
(resp. $\pi_1^{\dR}(X_n/S, x'_{n-1})$) and denote its image in 
$\pi_1^{\dR}(X/S,y_0)$ 
(resp. $\pi_1^{\dR}(X/S, y_{n-1})$) by the same letter. 
Also, let $\alpha_i \,(1 \leq i \leq n-1)$ be the element 
of $\pi_1^{\dR}(X_i/S, x'_{i-1},x_i)$ defined by 
$$ E|_{x'_{i-1}} \os{\cong}{\lla} \Gamma(Y_i, E) \os{\cong}{\lra} 
E|_{x_i} \quad ((E,\nabla) \in {\mathrm{N}}_{f_i}\MIC(X_i/S), f_i:X_i \lra S), $$
and denote its image in $\pi_1(X/S,y_{i-1},y_i)$ by the same letter. 
(The well-definedness of the elements $\alpha_i \,(1 \leq i \leq n-1)$ 
follows from Remark \ref{rem:20190306}.) 
We put $\alpha := \alpha_0$, 
$\beta = 
\alpha_1^{-1} \circ \cdots \circ \alpha_{n-1}^{-1} \circ \alpha_n \circ 
\alpha_{n-1} \circ \cdots \circ \alpha_1$ and 
calculate the monodromy action on the element $\beta \circ \alpha$. 

Let $\omega_i, \omega'_i$ be the tangential section at $x_i, x'_i$ respectively 
and put $\epsilon_i := \omega_i^*(1), \epsilon'_i := {\omega'_i}^*(1), 
\eta_i := {\epsilon'_i}^{-1} \circ \epsilon_i$. Then we have the following: 

\begin{prop}\label{prop:linemono}
Let the notations be as above. If we apply the monodromy action with respect to 
$\omega_0$ to $\beta \circ \alpha$, we obtain the element 
$\eta_0^{-n} \circ \beta \circ \eta_0^n \circ \alpha$. 
\end{prop}

\begin{proof}
By definition of the monodromy action, the element we obtain is 
$$ 
\epsilon_0^{-1} \circ 
\alpha_1^{-1} \circ \cdots \circ \alpha_{n-1}^{-1} \circ \alpha_n \circ 
\alpha_{n-1} \circ \cdots \circ \alpha_1 \circ \alpha_0 \circ \epsilon_0, 
$$
which can be  rewritten as 
\begin{align*}
& 
(\epsilon_0^{-1} \circ \epsilon'_0) \circ 
({\epsilon'}_0^{-1} \circ \alpha_1^{-1} \circ \epsilon_1) \circ 
(\epsilon_1^{-1} \circ \epsilon'_1) \circ 
\cdots \circ ({\epsilon'}_{n-2}^{-1} \circ \alpha_{n-1}^{-1} \circ \epsilon_{n-1}) 
\circ (\epsilon_{n-1}^{-1} \circ \epsilon'_{n-1}) \label{eq:linemono1} \\ & \circ 
({\epsilon'}_{n-1}^{-1} \circ \alpha_n \circ \epsilon'_{n-1}) 
\circ ({\epsilon'}_{n-1}^{-1} \circ \epsilon_{n-1}) \circ  
(\epsilon_{n-1}^{-1} \circ \alpha_{n-1} \circ \epsilon'_{n-2}) \circ 
\cdots \circ (\epsilon_1^{-1} \circ \alpha_1\circ \epsilon'_0) \nonumber \\ & \circ 
({\epsilon'}_0^{-1} \circ \epsilon_0) \circ 
(\epsilon_0^{-1} \circ \alpha_0 \circ \epsilon_0).  \nonumber 
\end{align*}
By Proposition \ref{prop:comp1}, it is equal to 
$$ 
\eta_0^{-1} \circ \alpha_1^{-1} \circ \eta_1^{-1} \circ \cdots \circ 
\alpha_{n-1}^{-1} \circ \eta_{n-1}^{-1} \circ \alpha_n 
\circ \eta_{n-1} \circ \alpha_{n-1} \circ \cdots \circ 
\alpha_1 \circ \eta_0 \circ \alpha_0
$$
and by Proposition \ref{prop:comp2}, it is further equal to 
$$ 
\eta_0^{-n} \circ \alpha_1^{-1} \circ \cdots \circ \alpha_{n-1}^{-1} 
\circ \alpha_n \circ \alpha_{n-1} \circ \cdots \circ \alpha_1 \circ \eta_0^n \circ \alpha_0 
= \eta_0^{-n} \circ \beta \circ \eta_0^n \circ \alpha. 
$$
(Here we apply Proposition \ref{prop:comp2} by replacing 
$X_1, x_1, y_1$ with  $X_i, x'_{i-1},x_i$ respectively, 
and the elements $\eta_x, \eta_y$ in the proposition 
are now  $\eta_{i-1}^{-1}, \eta_i$ respectively. We then obtain the equality $\eta_i \circ \alpha_i = \alpha_i \circ \eta_{i-1}$.) 
\end{proof}

\begin{rem}\label{rem:pi1-yz}
Let the notation be as above and 
$Y := Y_0, Z := \bigcup_{i=1}^n Y_i, y := y_0$ (endowed with the pullback log structure 
of that on $X$). Then the pro-nilpotent Lie algebra $L(X/S)$ fits into the following 
push-out diagram in the category of pro-nilpotent Lie algebras: 
\begin{equation*}
\begin{CD}
L(y/S) @>>> L(Y/S) \\ 
@VVV @VVV \\ 
L(Z/S) @>>> L(X/S).  
\end{CD}
\end{equation*}
To finish the proof of Theorem \ref{!!!} (which will be given in the  next subsection) we need to have a 
 concrete description of the maps $L(y/S) \lra L(Y/S), 
L(y/S) \lra L(Z/S)$. This is the subject of this remark. 

First, if we construct the log scheme $Y'$ from $Y$ using  the methods we  introduced at the beginning  of Subsection 6.5 ($y$ becomes a marked point), 
we have $L(Y/S) \cong L(Y'/S)$. But, since  $Y'$ has $y$ as a marked point, 
we conclude that 
$L(Y/S)$ is a free pro-nilpotent Lie algebra by Proposition \ref{propr>0}. 
We denote its rank by $m$ in this remark. 
(Note that, by minimality assumption on $X$, $m \geq 2$.) 
Next, if we construct the log scheme 
$\ol{Y}$ from $Y$ (again following the notation and recipe of Subsection 6.5),  we have the push-out diagram 
of pro-nilpotent Lie algebras 
\begin{equation*}
\begin{CD}
L(y/S) @>>> L(Y/S) \\ 
@VVV @VVV \\ 
0 @>>> L(\ol{Y}/S),   
\end{CD}
\end{equation*}
namely, $L(\ol{Y}/S)$ is the quotient of $L(Y/S)$ by the ideal generated by 
the image of $1 \in k = L(y/S) \lra L(Y/S)$. 

If $Y$ has a marked point, so does $\ol{Y}$ and the number of marked points 
in $\ol{Y}$ is one less than that in $Y'$. Hence, 
by Proposition \ref{propr>0}, $L(\ol{Y}/S)$ is a free pro-nilpotent Lie algebra 
whose rank is one less than that of $L(Y/S)$. Thus we see that 
there exists an isomorphism  of free pro-nilpotent Lie algebras
\begin{equation}\label{eq:concrete1}
L(Y/S) \cong L(v_1, ..., v_m) 
\end{equation}
(we use the notation as before Proposition \ref{prope}) 
 and we may suppose that   the image of $1 \in k = L(y/S) \lra L(Y/S)$ is $v_1$. 
On the other hand, if $Y$ does not have a marked point, 
neither does $\ol{Y}$. Then we have $m = 2g$ with $g (\geq 1)$ the genus of $\ol{Y}$ 
and there exist compatible isomorphisms  
\begin{equation}\label{eq:concrete2}
L(Y/S) \cong L(Y'/S) \cong L(v_1, ..., v_{2g}), \quad 
L(\ol{Y}/S) \cong L(v_1, ..., v_{2g})/\langle \sum_{i=1}^g[v_{2i-1},v_{2i}] \rangle 
\end{equation}
with the image of $1 \in k = L(y/S) \lra L(Y/S)$
equal to $\sum_{i=1}^g[v_{2i-1},v_{2i}]$, 
by the proof of Proposition \ref{prope}. 

Next, note that we have an isomorphism 
$L(Y_n/S) \os{\cong}{\lra} L(Z/S)$.  Indeed, if we put 
$Z_i := \bigcup_{j=i}^n Y_j$ (in particular:  $Z_n = Y_n, Z_1 = Z$), we have 
the push-out diagram 
of pro-nilpotent Lie algebras 
\begin{equation*}
\begin{CD}
L(y_i/S) @>>> L(Y_i/S) \\ 
@VVV @VVV \\ 
L(Z_{i+1}/S) @>>> L(Z_i/S), 
\end{CD}
\end{equation*}
and since we see easily that the upper horizontal arrow is an isomorphism 
(both are equal to $k$: indeed in the projective line 
$Y_i$ with the induced log-structure we have two marked points and we can use Proposition \ref{propr>0}), we see that the lower horizontal arrow is also 
an isomorphism and  we conclude with an isomorphism 
$L(Y_n/S) \os{\cong}{\lra} L(Z/S)$. 
Using this, we can see that the map 
$L(y/S) \lra L(Z/S)$ may be  described in a similar way as  the map 
$L(y/S) \lra L(Y/S)$. In fact:   
if  $Z$ has a marked point (i.e. so does $Y_n$), there exists an isomorphism  of free pro-nilpotent Lie algebras
\begin{equation}\label{eq:concrete3}
L(Z/S) \cong L(w_1, ..., w_{m'}) 
\end{equation}
for some $m' \geq 2$ 
such that the image of $1 \in k = L(y/S) \lra L(Z/S)$ can be given equal to $w_1$. 
When $Z$ does not have a marked point, 
there exist compatible isomorphisms  
\begin{equation}\label{eq:concrete4}
L(Z/S) \cong L(w_1, ..., w_{2g'}), \quad 
L(\ol{Z}/S) \cong L(w_1, ..., w_{2g'})/\langle \sum_{i=1}^g[w_{2i-1},w_{2i}] \rangle 
\end{equation}
for some $g' \geq 1$  (here $g'$ is the genus of $Y_n$ and $\ol{Z}$ is the log scheme we obtain from $Z$ by the recipe 
in Subsection 6.5) 
with the image of $1 \in k = L(y/S) \lra L(Z/S)$
equal to $- \sum_{i=1}^g[w_{2i-1},w_{2i}]$. (We can adjust the generators 
$w_1, \cdots, w_{2g}$ suitably so that the image of $1 \in L(y/S)$ is 
$- \sum_{i=1}^g[w_{2i-1},w_{2i}]$, not $\sum_{i=1}^g[w_{2i-1},w_{2i}]$.) 
\end{rem}

\subsection{End of the proof}

In this subsection, we finish the proof of the `only if' part of Theorem \ref{!!} and 
Theorem \ref{!!!}. 

\begin{proof}[Proof of  the `only if' part of Theorem \ref{!!} and 
Theorem \ref{!!!}]
By Proposition \ref{!!bu}, the `only if' part of Theorem \ref{!!} follows from 
Theorem \ref{!!!}. Then, by Theorem \ref{thm:line}, it suffices to prove 
Theorem \ref{!!!} in the case where $k$ is algebraically closed,  
the dual graph of $X$ is a line and 
the non-terminal components of $X$ are isomorphic to $\Pr_k^1$ with 
no marked point. In the following, we assume that $X$ satisfies these assumptions. 

Let the notations be as in Subsection 6.6 and, in particular,  the decomposition as  in Remark \ref{rem:pi1-yz}. 
Then to such a decomposition we apply  Proposition \ref{prop:linemono}:  
the monodromy action on $\pi_1^{\dR}(X/S,y)$ sends an element of the particular  form  as in the proposition
$\beta \circ \alpha \, (\alpha \in \pi_1^{\dR}(Y/S,y), \beta \in \pi_{\dR}(Z/S,y))$ to  
$\eta_0^{-n} \circ \beta \circ \eta_0^n \circ \alpha$, where 
$\eta_0$ is a nontrivial element of $\pi_1^{\dR}(y/S,y) \cong \G_{a,k}$. 
(Note that the isomorphism $L(Y_n/S) \os{\cong}{\lra} L(Z/S)$ in Remark \ref{rem:pi1-yz} 
induces the isomorphism 
$$ \pi_1^{\dR}(Y_n/S,y_{n-1}) \os{\cong}{\lra} \pi_{\dR}(Z/S, y); 
\,\,\,\, \alpha_n \mapsto 
\alpha_1^{-1} \circ \cdots \circ \alpha_{n-1}^{-1} \circ \alpha_n \circ 
\alpha_{n-1} \circ \cdots \circ \alpha_1. $$
So the particular  element $\beta$ in Proposition \ref{prop:linemono} 
can be understood as any   element of $\pi_1^{\dR}(Z/S,y)$!) 

We assume that the monodromy action on 
$\pi_1^{\dR}(X/S,y)$ is trivial as an element in  
${\rm Out}(\pi_1^{\dR}(X/S,y))$ and deduce a contradiction. By assumption, 
there exists an element $\delta \in \pi_1^{\dR}(X/S,y)$ such that the monodromy 
action is given by $\gamma \mapsto \delta^{-1} \circ \gamma \circ \delta \, 
(\gamma \in \pi_1^{\dR}(X/S,y))$. By comparing it with the description 
of the monodromy action for $\beta \circ \alpha$ given above, we obtain 
the equality 
$$ 
\eta_0^{-n} \circ \beta \circ \eta_0^n \circ \alpha = 
\delta^{-1} \circ \beta \circ \alpha \circ \delta$$ 
for any $\alpha \in \pi_1^{\dR}(Y/S,y), \beta \in \pi_1^{\dR}(X/S,y)$. 
By considering the case $\beta = \id$ and the case $\alpha = \id$, we obtain 
the following: 
\begin{align}
& \text{For any $\alpha \in \pi_1^{\dR}(Y/S,y), \,\, \alpha = \delta^{-1} \circ \alpha \circ \delta$.} 
\label{eq:final1} \\ 
& \text{For any $\beta \in \pi_1^{\dR}(Z/S,y), \,\, \eta_0^{-n} \circ \beta \circ \eta_0^n = \delta^{-1} \circ \beta \circ \delta$.}
\label{eq:final2} 
\end{align}

We write these assertions in the Lie algebra $L(X/S)/\Fil^4L(X/S)$. 
Let $\{\Gamma^n\}_n$ be the lower central series of $\pi_1^{\dR}(X/S,y)$ (with 
$\Gamma^1 = \pi_1^{\dR}(X/S,y)$) and consider the exponential isomorphism 
$$ \exp: L(X/S)/\Fil^4L(X/S) \os{\cong}{\lra} \pi_1^{\dR}(X/S,y)/\Gamma^4. $$
By Baker--Campbell--Hausdorff formula, we have 
$$ \exp(v)\exp(w) = \exp(v + w + (1/2)[v,w] + (1/12)([v,[v,w]]+[w,[w,v]])) $$ 
for $v, w \in L(X/S)/\Fil^4L(X/S)$. 
In the following, for $n \in \N$, 
we denote the equality in $L(X/S)/\Fil^nL(X/S)$ by $\equ{n}$. 
If we put $\alpha = \exp(a), \delta = \exp(d)$, the equality 
$\alpha = \delta^{-1} \circ \alpha \circ \delta$ implies the equality 
\begin{align*}
a + d + (1/2)[a,d] + & (1/12)([a,[a,d]]+[d,[d,a]]) \\ & \equ{4} 
d + a + (1/2)[d,a] + (1/12)([d,[d,a]]+[a,[a,d]]),  
\end{align*}
hence the equality $[a,d] \equ{4} 0$. Also, if we put 
$\beta = \exp(b), \eta_0^n = \exp(e)$, the equality 
$\eta_0^{-n} \circ \beta \circ \eta_0^n = \delta^{-1} \circ \beta \circ \delta$ 
implies the equality 
$$ \exp(d) \exp(-e) \exp (b) = \exp(b) \exp(d) \exp (-e), $$
which is rewritten as 
$$ \exp(d-e-(1/2)[d,e]+ \cdots ) \exp(b) = \exp(b) \exp(d-e-(1/2)[d,e] + \cdots ) $$
in $\pi_1^{\dR}(X/S,y)/\Gamma^4$, and this implies the equality 
$[b,d] \equ{4} [b,e + (1/2)[d,e]].$ Thus the assertions \eqref{eq:final1}, \eqref{eq:final2} 
imply the following assertions: 
\begin{align}
& \text{For any $a \in L(Y/S), \,\, [a,d] \equ{4} 0$.} 
\label{eq:final3} \\ 
& \text{For any $b \in L(Z/S), \,\, [b,d] \equ{4} [b,e + (1/2)[d,e]]$.}
\label{eq:final4} 
\end{align}
We will prove that these assertions imply a contradiction. 
Here we give a proof suggested by the referee which is shorter than our original proof.  
First we prove that there exists a Lie algebra map 
$h_Y: L(Y/S) \longrightarrow L({\bar v}_1,{\bar v}_2)$ such that 
${\bar v}_1$ is in the image of $h_Y$ and that 
$e \in L(y/S)$  is sent to $[{\bar v}_1,{\bar  v}_2]$, 
where $L({\bar v}_1,{\bar v}_2)$ is the free pro-nilpotent Lie algebra generated by ${\bar v}_1, {\bar v}_2$. 
Indeed, if $Y$ has a marked point, we have an isomorphism 
$L(Y/S) \cong L(v_1, ..., v_{m})$ for some $m \geq 2$ such that $e$ is sent to $v_1$, by 
Remark \ref{rem:pi1-yz} and the nontriviality of $e \in L(y/S)$. 
Then we can define a required map $h_Y$ by   defining 
$h_Y(v_1) := [{\bar v}_1,{\bar v}_2]$ and $h_Y(v_i) := {\bar v}_1 \, (i \geq 2)$. 
If $Y$ does not have a marked point, we have an isomorphism 
$L(Y/S) \cong L(v_1, ..., v_{2g})$ for some $g \geq 1$ such that $e \in L(y/S)$ is sent to $\sum_{i=1}^g[v_{2i-1},v_{2i}]$, 
again by Remark \ref{rem:pi1-yz} and the nontriviality of $e \in L(y/S)$. 
Then we can define a required map $h_Y$ in this case by
$h_Y(v_i) := {\bar v}_i \, (i=1,2)$, $h_Y(v_i) := 0 \, (i \geq 3)$.  

In a similar way, we see that there exists a Lie algebra map $h_Z: L(Z/S) \lra L({\bar w}_1,{\bar w}_2)$
such that ${\bar w}_1$ is in the image of $h_Z$ and $e \in L(y/S)$  is sent to $-[{\bar w}_1,{\bar w}_2]$, 
where $L({\bar w}_1,{\bar w}_2)$ is the free pro-unipotent Lie algebra 
generated by  ${\bar w}_1,{\bar w}_2$. 
It then follows that there exists a Lie algebra map 
$$
h_X: L(X/S) \rightarrow  L := L( {\bar v}_1,{\bar v}_2, {\bar w}_1,{\bar w}_2)/ ([{\bar v}_1,{\bar v}_2] + [{\bar w}_1,{\bar w}_2])
$$
such that ${\bar v}_1$ and ${\bar w}_1$ are in the image of $h_X$. 
If we denote again by $d$ the image of $d$ by $h_X$, 
then the equations \eqref{eq:final3} and \eqref{eq:final4} imply the following equations in $L$: 
\begin{align}
& [{\bar v}_1,d] \equ{4} 0, 
\label{eq:final5} \\ 
&  [{\bar w}_1,d] \equ{4} [{\bar w}_1, [{\bar v}_1, {\bar v}_2] + (1/2)[d,[{\overline v}_1, {\bar v}_2]] \equ{4} [{\bar w}_1, [{\bar v}_1, {\bar v}_2]]. 
\label{eq:final6} 
\end{align}

By \eqref{eq:final5},  it follows in particular that $[{\bar v}_1,d] \equ{3} 0$. Hence, if we write $d\equ{2}C_1{\bar v}_1+  C_2{\bar v}_2+ C_3{\bar w}_1+ C_2{\bar w}_2$, we conclude that $d \equ{2} C_1{\bar v}_1$. But  then \eqref{eq:final6} says that $[{\bar w}_1,d] \equ{3} 0$, and then it implies that $C_1=0$ as well. Hence $d\equ{2}0$, that is, 
$ d \in  \Fil^2L$. Then we consider the map 
$$[v_1, -]: \Fil^2L/\Fil^3L 
\lra \Fil^3L/\Fil^4L. $$ 
 By \cite[Proposition 4]{labute}, 
the dimension of $\Fil^2L/\Fil^3L$ (resp. $\Fil^3L/\Fil^4L$) is 
$5$ (resp. $16$), and a basis is given by 
$$ [{\bar v}_1,{\bar v}_2], \quad [{\bar v}_1,{\bar w}_1], \quad [{\bar v}_1,{\bar w}_2], \quad [{\bar v}_2,{\bar w}_1], \quad [{\bar v}_2,{\bar w}_2] $$
\begin{align*}
(\text{resp. } \, 
& [{\bar v}_1,[{\bar v}_1,{\bar v}_2]], \quad [{\bar v}_1,[{\bar v}_1,{\bar w}_1]], \quad [{\bar v}_1,[{\bar v}_1,{\bar w}_2]], \quad [{\bar v}_1,[{\bar v}_2, {\bar w}_1]], \\ 
& [{\bar v}_1,[{\bar v}_2,{\bar w}_2]], \quad [{\bar v}_2,[{\bar v}_1,{\bar v}_2]], \quad [{\bar v}_2,[{\bar v}_1,{\bar w}_2]], \quad [{\bar v}_2,[{\bar v}_2,{\bar w}_1]], \\ 
& [{\bar v}_2,[{\bar v}_2,{\bar w}_2]], \quad [{\bar w}_1,[{\bar v}_1,{\bar v}_2]], \quad [{\bar w}_1,[{\bar v}_1,{\bar w}_1]], \quad [{\bar w}_1,[{\bar v}_1,{\bar w}_2]], \\ 
& [{\bar w}_1,[{\bar v}_2,{\bar w}_1]], \quad [{\bar w}_1,[{\bar v}_2, {\bar w}_2]], 
\quad [{\bar w}_2,[{\bar v}_1,{\bar w}_2]], \quad [{\bar w}_2,[{\bar v}_2,{\bar w}_2]]).  
\end{align*}
Thus the map 
$$[{\bar v}_1, -]: \Fil^2L/\Fil^3L 
\lra \Fil^3L/\Fil^4L$$ 
is injective and so we see that $d \equ{3} 0$, because $[{\bar v}_1,d] \equ{4} 0$.  

Then, by \eqref{eq:final6}, we have  $0  \equ{4} [{\bar w}_1,d] \equ{4} [{\bar w}_1, [{\bar v}_1, {\bar v}_2]]$. 
But this last element cannot be zero in $\Fil^3L/\Fil^4L$ because it is an element of the basis of $\Fil^3L/\Fil^4L$ 
listed above. This is a contradiction. 
Thus we have proven that the assertions \eqref{eq:final1}, \eqref{eq:final2} 
imply a contradiction. So the proof of the theorems is finished. 
\end{proof}

\subsection{Application to Andreatta--Iovita--Kim theorem} 

In this final part, we review the theorem of Andreatta--Iovita--Kim \cite{aik}
on $p$-adic good reduction criterion for proper hyperbolic curves and 
we give an outline of its proof. Then we give an alternative, purely algebraic proof 
for the transcendental part of their theorem: this will be done by  using the results of the  previous 
subsections. 

First we recall the setting of their article \cite{aik} (with different notation). 
Let $K$ be a complete discrete valuation field of characteristic zero, 
with valuation ring $O_K$ and perfect residue field $F$. 
We fix a uniformizer $\pi$ of $O_K$ and 
an algebraic closure $\ol{K}$ of $K$. 
Let $$f_K: X_K \lra \Spec K$$ 
be a proper smooth geometrically irreducible 
curve of genus $g \geq 2$, and assume that it admits 
a regular semistable model 
$$f_{O_K}^{\circ}: X_{O_K}^{\circ} \lra \Spec O_K. $$
We may and do assume that the special fiber does not have 
$(-1)$-curve. With this setting, we see that $f_K$ has good reduction 
if and only if $f_{O_K}^{\circ}$ is smooth. 
We endow the source and the target of $f_{O_K}^{\circ}$ with
the log structure associated to the  special fiber. Then we obtain the 
proper log smooth morphism of log schemes, which we denote by 
$$ f_{O_K}: X_{O_K} \lra S_{O_K}. $$
Also, let $\iota_{O_K}: S_{O_K} \lra X_{O_K}$ be a section of $f_{O_K}$. 
Let 
$$ f_{\ol{K}}: X_{\ol{K}} \lra \Spec \ol{K}, \quad \iota_{\ol{K}}: \Spec \ol{K} \lra 
X_{\ol{K}} $$ 
be the scalar extension of $f_{O_K}, \iota_{O_K}$ to $\ol{K}$. 
We denote by $\pi_1^{\rm et}(X_{\ol{K}}, \iota_{\ol{K}})$ the Tannaka dual of 
the category of unipotent smooth $\Q_p$-adic etale sheaves on $X_{\ol{K}}$ 
(with respect to the fiber functor of pullback by $\iota_{\ol{K}}$). 
Then  it has been proved in  \cite{aik}  that   $\pi_1^{\rm et}(X_{\ol{K}}, \iota_{\ol{K}})$ is 
semistable as a proalgebraic group over $\Q_p$ endowed with the action of 
${\rm Gal}(\ol{K}/K)$ (see \cite[Theorem 1.4]{aik}), and the main theorem 
of  that  article is  \cite[Theorem 1.6]{aik}: 

\begin{thm}\label{thm:aikmain}
$f_K: X_K \lra \Spec K$ has good reduction if and only if 
$\pi_1^{\rm et}(X_{\ol{K}}, \iota_{\ol{K}})$ is crystalline. 
\end{thm}

For the definition of crystalline-ness, see \cite[Definition 1.5]{aik}.  In order to achieve such a result some geometric constructions have been made.   Let $W := W(F)$ be the Witt ring of $F$ 
and let $k$ be its fraction field (the field $k$ is denoted by $K_0$ in \cite{aik}, which is also the standard notation in $p$-adic Hodge theory). 
We denote the special fiber of $f_{O_K}$ 
(endowed with the pullback log structure) by 
$$ f_F: X_F \lra S_F. $$
Then, by the absence of $(-1)$-curves in $X_F$, 
$f_F$ is a minimal semistable log curve in the sense of 
Section 6.1. We denote the $p$-adic completion of 
$f_{O_K}$ by 
$$ \wh{f}_{O_K}: \wh{X}_{O_K} \lra \wh{S}_{O_K}. $$
Let $S_{W[[z]]}$ be the scheme $\Spec W[[z]]$ endowed with the log structure 
associated to $\N \lra W[[z]]; \, 1 \mapsto z$, and let 
$\wh{S}_{W[[z]]}$ be its $(p,z)$-adic completion. 
We define the exact closed immersion $\wh{S}_{O_K} \hra \wh{S}_{W[[z]]}$ by 
$W[[z]] \lra O_K; z \mapsto \pi$. 
By the unobstructedness in the deformation theory for $\wh{f}_{O_K}$ 
\cite[Proposition 3.14]{ka}, there exists 
a proper log smooth integral lift 
$$ \wh{f}_{W[[z]]}: \wh{X}_{W[[z]]} \lra \wh{S}_{W[[z]]} $$ 
of $\wh{f}_{O_K}$, and we can algebraize it to the proper log smooth integral morphism 
$$ f_{W[[z]]}: X_{W[[z]]} \lra S_{W[[z]]}. $$ 
Let 
$$  f_{k[[z]]}: X_{k[[z]]} \lra S_{k[[z]]} $$
be the scalar extension of $f_{W[[z]]}$ to $k[[z]]$. 

At a non-smooth point of $\wh{X}_{O_K}$, the formal local ring has the form 
$\Spf O'_K[[x,y]]/\allowbreak (xy-\pi)$ for some finite etale extension 
$O'_K$ of $O_K$. By the unicity of the local lift, 
the formal local ring at this point in $X_{W[[z]]}$ has the form 
$\Spf W'[[x,y,z]]/(xy-z)$ for some finite etale extension $W'$ of $W$ 
(see \cite[p.2604]{aik}). 
If we denote 
the fiber at `$z=0$' of the morphisms $f_{W[[z]]}, f_{k[[z]]}$ by 
$$ f_W: X_W \lra S_W, \quad f_k: X_k \lra S_k $$ 
respectively, we have the Cartesian diagram 
\begin{equation*}
\begin{CD}
X_F @>>> X_W @<<< X_k \\ 
@V{f_F}VV @V{f_W}VV @V{f_k}VV \\ 
S_F @>>> S_W @<<< S_k,  
\end{CD}
\end{equation*}
and the formal local ring at a non-smooth point 
of $X_W$ in the closed fiber has the form 
$\Spf W'[[x,y]]/(xy)$ for some finite etale extension $W'$ of $W$. 
Also, the formal local ring at a smooth point 
of $X_W$ in the closed fiber has the form 
$\Spf W'[[x]]$ for some finite extension $W'$ of $W$. 
Thus there exists an etale covering $\sqcup_{i \in I} \wt{X}_{W,i} \lra X_W$ of $X_W$ 
as a scheme such that, for each $i \in I$, $\wt{X}_{W,i}$ is etale (as a scheme) over 
$\Spec W[x,y]/(xy)$ or $\Spec W[x]$. 


Let $X^{(0)}_W$ be the normalization of $X_W$ and put $X^{(1)}_W :=  (X^{(0)}_W \times_{X_W} X^{(0)}_W \setminus \Delta (X^{(0)}_W)) / {\frak S}_2$, where $\Delta: X^{(0)}_W \longrightarrow  X^{(0)}_W \times_{X_W} X^{(0)}_W$ is the diagonal map and the second symmetric group ${\frak S}_2$ acts on $X^{(0)}_W \times_{X_W} X^{(0)}_W$ by permutation of 
the two factors $X^{(0)}_W$.  Then, by the etale local description of ${X_W}$ above, we have 
the following  properties: 
\medskip 

\noindent
$(1)$ \, $X^{(0)}_W$ is smooth over $W$ and $X^{(1)}_W$ is etale over $W$. \\ 
$(2)$ \, The special (resp.  generic)  fiber $(X^{(0)}_W)_F$ (resp. $(X^{(0)}_W)_k$) of $X^{(0)}_W$ is the normalization 
of $X_F$ (resp. $X_k$). \\ 
$(3)$ \,  The special (resp. generic) fiber $(X^{(1)}_W)_F$ (resp. $(X^{(1)}_W)_k$ ) of $X^{(1)}_W$ is made by 
the intersection points (including self-intersection) of the irreducible components of $X_F$ (resp. $X_k$). 
\medskip 

\noindent
Then we see that the connected components of $(X^{(i)}_W)_k$ ($i=0,1$) correspond bijectively to those of $X^{(i)}_W$ and then to those of $(X^{(i)}_W)_F$. Thus we have a canonical bijection between the dual graph of $X_k$ and that of $X_F$. Moreover, this bijection preserves the genus of the normalization of each irreducible component. 
%
%
Therefore, $f_k: X_k \lra S_k$ is again a minimal semistable log curve 
and the underlying morphism of schemes of $f_k$ is smooth if and only if 
so is that of $f_F$. 

Now let  $S_{W[[x]]}$ be the scheme $\Spec W[[x]]$ endowed with the log structure 
associated to $\N \lra W[[x]]; \, 1 \mapsto px$ and let $S_{W[[x]]} \lra S_{W[[z]]}$ 
be the morphism defined by $z \mapsto px$. Denote the pullback of $f_{W[[z]]}$ 
by this morphism by 
$$ f_{W[[x]]}: X_{W[[x]]} \lra S_{W[[x]]}, $$
and let 
$$  f_{k[[x]]}: X_{k[[x]]} \lra S_{k[[x]]} $$
be the scalar extension of $f_{W[[x]]}$ to $k[[x]]$. Note that the 
the fiber at `$x=0$' of the morphism $f_{k[[x]]}$ is the same as the morphism 
 $$ f_k: X_k \lra S_k $$ 
we already defined. Also, beginning from $\iota_{O_K}$, 
we also obtain a section $\iota_{k[[x]]}, \iota_k$ of $f_{k[[x]]}, f_k$. 
For $f_{k[[x]]}$, Andreatta--Iovita--Kim defined 
a projective system $\{(W_n, e_n)\}_n$ of connections, 
which we denote here by  $\{(W_{n,k[[x]]}, e_{n,k[[x]]})\}_n$, 
by the method explained in  Remark \ref{AIK2}. We proved in that remark that their 
definition coincides with our definition (which we gave in Section 3).  Recall also that, by our methods in Section 3, we are able to define a projective system of connections
 $\{(W_n, e_n)\}_n$  for $f_k$, which we denote here by 
$\{(W_{n,k}, e_{n,k})\}_n$, and by Remark \ref{rem:wn-compati}, we see that 
$\{(W_{n,k}, e_{n,k})\}_n$ is the restriction of $\{(W_{n,k[[x]]}, e_{n,k[[x]]})\}_n$ to 
$X_k$. 

By \cite[Proposition 6.2]{aik}, 
$\pi_1^{\rm et}(X_{\ol{K}}, \iota_{\ol{K}})$ is crystalline
if and only if the residue of $\iota_{k[[x]],\dR}^*W_{n,k[[x]]}$ at the locus $x=0$
is zero for every $n$. Until this point, they never use transcendental method. After that, 
to prove the triviality/nontriviality of the residue according to having good/bad  
reduction of $f_{k[[x]]}$, they use transcendental methods  reducing  the claim 
to Oda's result \cite{oda2}, hence giving a trascendental proof of  Theorem \ref{thm:aikmain}. 

Now we give an alternative, purely algebraic,  proof of Theorem \ref{thm:aikmain}. 

\begin{proof}[An algebraic proof of Theorem \ref{thm:aikmain}]
It suffices to give an algebraic proof of the transcendental part of their proof. 
We use the notations given above. 
Since $\{W_{n,k}\}_n$ is the restriction of $\{W_{n,k[[x]]}\}_n$ to 
$X_k$,  the residue of $\iota_{k[[x]],\dR}^*W_{n,k[[x]]}$ at the locus $x=0$ is nothing but 
the nilpotent map on $\iota_{k,\dR}^*W_{n,k}$ which occurs as a part of data 
as an object in $\MICn(S_k/k)$. Because $\pi_1^{\dR}(X_k/S_k,\iota_k)$ is 
defined as the spectrum of $\varinjlim_n (\iota_{k,\dR}^*W_{n,k})^{\vee}$, 
the triviality/nontriviality of the residue is equivalent to 
the triviality/nontriviality of monodromy action on $\pi_1^{\dR}(X_k/S_k,\iota_k)$
as an element in ${\rm Aut}(\pi_1^{\dR}(X_k/S_k,\iota_k))$. 
Because our $X_k$ does not have a marked point (because we started with a 
proper smooth curve $f_K:X_K \lra \Spec K$), $\iota_k$ is necessarily a good section 
by Remark \ref{good_notgood_section}. Hence, if the underlying morphism of $f_k$ is 
smooth, the monodromy action on $\pi_1^{\dR}(X_k/S_k,\iota_k)$ is trivial 
as an element in ${\rm Aut}(\pi_1^{\dR}(X_k/S_k,\iota_k))$ by Proposition 
\ref{not_good_trivial}(1). If the underlying morphism of $f_k$ is 
not smooth, the monodromy action on $\pi_1^{\dR}(X_k/S_k,\iota_k)$ is nontrivial 
as an element in ${\rm Out}(\pi_1^{\dR}(X_k/S_k,\iota_k))$ by Theorem  
\ref{!!!}, and so nontrivial as an element in 
${\rm Aut}(\pi_1^{\dR}(X_k/S_k,\iota_k))$. 
Since we saw, in the construction before the algebraic proof of Theorem \ref{thm:aikmain}, 
that the underlying morphism of schemes of $f_k$ is smooth if and only if 
so is that of $f_F$, this finishes the proof. 
\end{proof}

\begin{rem}
In this remark, we explain how our definition of the monodromy action on 
$\pi_1^{\dR}(X/S, \iota)$, its calculation and our proof of Theorems  
\ref{!!!}, \ref{thm:aikmain} can be  regarded as an algebraic version of the topological 
definition, calculation and proofs as it can be found in  previous (classical) works \cite{oda2}.

Let $f: X \lra S$ be a minimal semistable log curve over the complex numbers $\C$ with a section $\iota$. 
Then, by Proposition \ref{prop:abcde-rs} and Remark \ref{rem:nilpnonclosed2}, we 
are in the situation of Remarks \ref{rem:2020Aug-1} and \ref{rem:2020Aug-2}. 
So we have the canonical commutative diagram 
$$ 
\xymatrix{
X^{\log}_{\an} \ar[r]^{f^{\log}_{\an}} \ar[d]^{\tau_X} & 
S^{\log}_{\an} \ar[d]^{\tau_S} \\ 
X_{\an} \ar[r]^{f_{\an}} & S_{\an}, 
}
$$
as explained in Remark \ref{rem:2020Aug-1}, and $f^{\log}_{\an}$ admits a section 
$\iota^{\log}_{\an}$. Also, for a point $s$ in $S^{\log}_{\an}$ 
and $x := \iota^{\log}_{\an}(s)$, $X^{\log}_{\an, s} := (f^{\log}_{\an})^{-1}(s)$, 
we have the split homotopy exact sequence 
\begin{equation}\label{eq:2020Aug-3-1}
\xymatrix {
1 \ar[r] & \pi_1^{\rm top}(X^{\log}_{\an, s} ,s) \ar[r] & 
\pi^{\rm top}_1(X^{\log}_{\an} ,x) \ar[r]^-{f^{\log}_{\an *}} & \pi^{\rm top}_1(S^{\log}_{\an},s) \ar[r] 
\ar @/^6mm/[0,-1]^{\iota^{\log}_{\an *}} & 1. 
}
\end{equation}
of topological fundamental groups (Remark \ref{rem:2020Aug-2}). 

The space $S^{\log}_{\an}$ is homeomorphic to the unit circle ${\mathbb{S}}^1$ and for a 
double point $y$ in $X_{\an}$, the fiber 
$$ \tau_X^{-1}(y) \lra \tau_S^{-1}(S_{\an}) = S^{\log}_{\an} = {\mathbb{S}}^1 $$ 
is described as 
\begin{equation}\label{eq:2020Aug3-2}
{\mathbb{S}}^1 \times  {\mathbb{S}}^1 \lra  {\mathbb{S}}^1; \quad (a,b) \mapsto ab. 
\end{equation}

The monodromy action in this topological situation is the action 
of $\pi_1^{\rm top}(S^{\log}_{\an}, s) \cong \Z \ni 1$ (the conjugation by 
$\iota^{\log}_{\an *}(1)$). So the monodromy action defined in Section 6.1 
is actually an algebraic version of this topological monodromy action. 

In the topological situation, the tangential sections at a double point $y \in X_{\an}$ 
are given by the following sections of $f^{\log}_{\an}$: 
\begin{align}
& S^{\log}_{\an} = {\mathbb{S}}^1 \lra {\mathbb{S}}^1 \times  {\mathbb{S}}^1 = \tau_X^{-1}(y); 
\quad a \mapsto (a,1), \label{eq:2020Aug3-3} \\ 
& S^{\log}_{\an} = {\mathbb{S}}^1 \lra {\mathbb{S}}^1 \times  {\mathbb{S}}^1 = \tau_X^{-1}(y); 
\quad a \mapsto (1,a). \nonumber 
\end{align}
So, for example, the topological version of the element $\epsilon_y \in \pi_1^{\dR}(X,y)$ 
in Proposition \ref{prop:comp1} is a loop in $X^{\log}_{\an}$ given by the image of 
one of the sections in \eqref{eq:2020Aug3-3}. Considering the illustration  of the map 
$ \tau_X^{-1}(y) \lra \tau_S^{-1}(S_{\an}) = S^{\log}_{\an} = {\mathbb{S}}^1 $ given in 
\eqref{eq:2020Aug3-2}, we see that the local description of the monodromy action in 
Proposition \ref{prop:comp1} is an algebraic version of the Dehn twist with respect to the 
loop around the double point in consideration. Also, we may see how  the monodromy action 
as described in  Sections 6.4, 6.6  can be understood as  an algebraic version of  the topological monodromy action 
of $\pi_1^{\rm top}(S^{\log}_{\an},s)$ on $\pi_1^{\rm top}(X^{\log}_{\an,s},x)$ 
essentially given in \cite[Lemma 1.7]{oda2}: it  is given by certain products of Dehn twists. 

Once we have the description of the monodromy action, we notice that our proof of Theorem \ref{!!!}
by a direct  calculation in Lie algebras is a modified version of the proof given in 
\cite[Proposition 1.10]{oda2}. Then, by replacing Oda's result by our purely algebraic 
result, we can complete a purely algebraic proof of Andreatta-Iovita-Kim's theorem 
(Theorem \ref{thm:aikmain}). 
\end{rem}


\begin{thebibliography}{DMSS00}


\bibitem[AGT16]{agt} A.~Abbes, M.~Gros and T.~Tsuji, 
The $p$-adic Simpson correspondence, Annals of Mathematics Studies {\bf 193}, Princeton University Press, Princeton NJ, 2016. 

\bibitem[AIK15] {aik}  F.~Andreatta, A.~Iovita and M.~Kim, A $p$-adic non-abelian criterion for good reduction of curves, 
Duke Math. J. {\bf 164}(2015), 2597--2642. 

\bibitem[B74]{BE} P.~Berthelot, Cohomologie Cristalline des sch\'emas de caract\'eristique $p$, 
Lecture Note in Math. ~{\bf 407}, Springer, 1974.

\bibitem[BO78]{berthelotogus} P.~Berthelot and A.~Ogus, Notes on crystalline cohomology, 
Mathematical Notes {\bf 21}, Princeton University Press and University of Tokyo Press, 1978.

\bibitem[BK94]{bk} S.~Bloch and I.~Kriz,  Mixed Tate Motives, Annals of Math. {\bf 140}(1994), 557--605. 

\bibitem[C02]{cailotto} M.~Caillotto, Algebraic connections on logarithmic schemes: Residues and local freeness, 
Preprint 06-15 maggio 2002, Math. Dep. of Padova Univ. available at 
{\tt{http://www.math.unipd.it/\textasciitilde maurizio/Math.html}}. 

\bibitem[CF06]{cf} B.~Chiarellotto and M.~Fornasiero, Logarithmic de Rham, infinitesimal and Betti
cohomologies, J. Math. Sci. Univ. Tokyo {\bf 13}(2006), 205--257.


\bibitem[D89]{Del89} P.~Deligne, Le groupe fondamental de la droite projective moins trois points,  
in Galois groups over {${\bf Q}$}, Math. Sci. Res. Inst. Pub. {\bf 16}, Springer, 1989, pp.~79--297. 

\bibitem[D90]{Del90}
P.~Deligne, Cat\'egories tannakiennes, in The {G}rothendieck {F}estschrift {V}ol.\ {II}, 
Progr. Math. {\bf 87}, Birkh\"auser Boston, Boston MA, 1990, pp.~111--195. 

\bibitem[DM82]{DelMil82} P.~Deligne and J.~S.~Milne, Tannakian categories, in 
Hodge Cycles, Motives, and Shimura Varieties, Lecture Note in Math. {\bf 900}, Springer, 1982,  pp.~101--228. 



\bibitem[DG11]{SGA3-1} M.~Demazure and A.~Grothendieck, 
Sch\'emas en groupes. (SGA 3, Tome I) Propri\'et\'es g\'en\'erales des sch\'emas en groupes, 
S\'eminaire de G\'eom\'etrie Alg\'ebrique du Bois Marie 1962-1964, 
Documents math\'ematiques {\bf 7} (2011), Soc. Math. France. 

\bibitem[DMSS00]{dimca} A.~Dimca, F.~Maaref, C.~Sabbah and M.`Saito, 
Dwork cohomology and algebraic $D$-modules, Math. Ann. {\bf 318}(2000), 107--125. 

\bibitem[DP09]{dp} V.~Di~Proietto, On $p$-adic differential equations for semistable vareties, PhD thesis, Universit\`a di  Padova, 2009. Avalaible at \url{http://empslocal.ex.ac.uk/people/staff/vd238/files/Thesis.pdf}

\bibitem[DPS18] {dps} V.~Di~Proietto, A.~Shiho, On the homotopy exact sequence for the log algebraic fundamental group, Documenta Math. {\bf 23} (2018),  543--597. 

\bibitem[EHS08]{ehs} H.~Esnault, P.~H.~Hai and X.~Sun, On Nori's fundamental group scheme, in 
Geometry and Dynamics of Groups and Spaces, 
Progress in Mathematics {\bf 265} Birkh\"auser, 2008, pp.~377--398. 
 
\bibitem[FHT01]{fht} Y.~F\'elix, S.~Halperin and J.-C.~Thomas, 
Rational homotopy theory, Graduate Texts in Mathematics {\bf 205}. 
Springer-Verlag, New York, 2001. 
 
\bibitem[GL76]{gl} R.~G\'erard and A.~H.~M.~Levelt, Sur les connexions \`a singularit\'es r\'eguli\`eres dans le cas
de plusieurs variables, Funkcialaj Ekvacioj, {\bf 19}(1976), 149--173. 



\bibitem[Gro63]{SGA1}
A.~Grothendieck, Rev\^etements \'etales et groupe fondamental, S\'eminaire de
  G\'eom\'etrie Alg\'ebrique, vol. 1960/61, Institut des Hautes \'Etudes
  Scientifiques, Paris (1963).

\bibitem[GM13]{gm} P.~Griffith and J.~Morgan, Rational Homotopy Theory and Differential Forms (Second Ed.), 
Progress in Math. {\bf 16}, Birkhauser, 2013.
 
 \bibitem[Had11]{ha}  M.~Hadian, Motivic Fundamental Groups and Integral Points.
Duke Mathematical Journal, {\bf 160}(2011), 503--565.  

\bibitem[Hai84]{hain-memoir} R.~Hain, Iterated integrals and homotopy periods, 
Mem. Amer. Math. Soc. {\bf 291}(1984). 

\bibitem[Hai87]{hain} R.~Hain, The de Rham homotopy theory of 
complex algebraic varieties I, K-theory {\bf 1}(1987), 271--324. 

\bibitem[HN17]{HN17} Y.~Hoshi and C.~Nakayama, 
Categorical characterization of strict morphisms of fs log schemes, 
Math. J. Okayama Univ. {\bf 59}(2017), 1--19. 

\bibitem[HK94] {hk} O.~Hyodo and K.~Kato, 
Semi-stable reduction and crystalline cohomology with logarithmic poles, in P\'eriodes $p$-adiques, 
Ast\'erisque {\bf 223}(1994), 221--268. 
 

\bibitem[I02]{il} 
L.~Illusie, An overview of the work of K.~Fujiwara, K.~Kato, and C.~Nakayama on
logarithmic \'etale cohomology, in Cohomologies $p$-adiques et applications arithm\'etiques (II), 
Ast\'erisque {\bf 279}(2002), 271--322.


\bibitem[Kf00]{fk} F.~Kato, Log-smooth deformation and moduli of log-smooth curves, Intern. J. Math. 
{\bf 11}(2000), 215--232. 
 

\bibitem[Kk88]{ka} K.~Kato,  
Logarithmic structures of Fontaine-Illusie, in Algebraic Analysis, Geometry
and Number Theory, Johns Hopkins Univ., 1988, pp.~191--224.

\bibitem[Kk94]{ka2} K.~Kato, Toric singularities, 
Amer. J. Math., {\bf 116}(1994), 1073--1099.

\bibitem[KN99]{kn} K.~Kato and C.~Nakayama, 
Log Betti cohomology, log \'etale cohomology, and log de Rham cohmology of log schemes over 
$\mbox{\boldmath $C$}$, Kodai Math. J. {\bf 22}(1999), 161--186. 
 
\bibitem[Ke07]{kedlayaI} K.~S.~Kedlaya, Semistable reduction for overconvergent $F$-isocrystals
I: Unipotence and logarithmic extensions, Compositio Math. {\bf 143}(2007), 1164--1212. 

 \bibitem[Ki05]{ki} M.~Kim, The motivic fundamental group of ${\mathbb P}^1  \setminus \{ 0,1,\infty \}$ and the theorem of Siegel, Inv. Math. {\bf 161} (2005), 629-656.
 
 \bibitem[Kn83]{knudsen} F.~Knudsen, The projectivity of the moduli space of stable curves II. The stack $M_{g,n}$, Mathematica Scandinavica, {\bf 52}(1983), 161--199, (1983). 
 
 \bibitem[Lab70]{labute} J.~P.~Labute, On the descending central series of groups with a single defining relation, J. Algebra 
{\bf 14}(1970), 16--23.  
 
\bibitem[Laz15]{laz} C.~Lazda,  Relative fundamental groups and rational points, 
Rendconti Sem. Mat. Univ. Padova {\bf 134}(2015), 1--45.
 
 

\bibitem[Mi17]{milne} J.~S.~Milne, 
Algebraic groups. 
The theory of group schemes of finite type over a field, Cambridge Studies in Advanced Mathematics 
{\bf 170}, Cambridge University Press, Cambridge, 2017.
 
\bibitem[Mo84]{morrison} D.~Morrison, The Clemens-Schmid exact sequence and applications, 
in Topics in Transcendental Algebraic Geometry, Ann. of Math. Studies {\bf 106}, Princeton 
Univ. Press 1984, pp.~101--120. 

\bibitem[NO10]{no} C.~Nakayama and A.~Ogus, 
Relative rounding in toric and logarithmic geometry, Geometry \& Topology {\bf 14}(2010), 2189--2241. 

\bibitem[NA87]{naho}  V.~Navarro Aznar, Sur la theorie de Hodge-Deligne, Inv. Math. {\bf 90}(1987), 11--76.

\bibitem[NA93]{nagm} V.~Navarro Aznar,  Sur la connexion de Gauss-Manin en homotopie rationelle, 
Ann. Sci. E.N.S. {\bf 26}(1993), 99--148.  

\bibitem[N06]{niziol} W.~Niziol, 
Toric singularities: log-blow-ups and global resolutions, 
J. Alg. Geom. {\bf 15}(2006), 1--29.
 
\bibitem[Od95]{oda2} T.~Oda,  
A note on ramification of the Galois representation on the fundamental group of an algebraic curve II,  
J. Number Theory {\bf 53}(1995), 342--355.

\bibitem[Og03]{ogus-rh} A.~Ogus, 
On the logarithmic Riemann-Hilbert correspondence, 
Documenta Math. Extra Volume Kato (2003), 655--724. 

\bibitem[Og18]{ogu} A.~Ogus, 
Lectures on Logarithmic Algebraic Geometry, 
Cambridge Studies in Advanced Mathematics {\bf 178}, 
Cambridge University Press, 2018.
 
 \bibitem[Sh00]{shihocrysI} A.~Shiho, Crystalline fundamental groups I --- Isocrystals
on log crystalline site and log convergent site, J. Math. Sci. Univ. Tokyo
{\bf 7}(2000), 509--656. 

\bibitem[Sh02]{shihocrysII} A.~Shiho, Crystalline fundamental groups II --- Log
convergent cohomology and rigid cohomology, J. Math. Sci. Univ. Tokyo
{\bf 9}(2002), 1--163.

\bibitem[ST]{stacks}  The Stacks project, \url{https://stacks.math.columbia.edu/}, 2020.

 \bibitem[Su77]{sullivan} D.~Sullivan, Infinitesimal computations in topology,  Publ. Math. IHES {\bf 47}(1977), 269--331.
 
 \bibitem[Te10]{terasoma} T.~Terasoma,  DG-categories and simplicial bar complexes, Moscow Math. Journal {\bf 10}(2010),
231--267.

\bibitem[To]{tong} J.~Tong, Unipotent groups over a discrete valuation ring (after Dolgachev-Weisfeiler), 
Autour des sch\'emas en groupes. Vol. III, 173--225,
Panor. Synth\`eses {\bf 47}, Soc. Math. France, Paris, 2015.

\bibitem[Ts99]{tsuji} T.~Tsuji, Poincar\'e duality for Logarithmic Crystalline Cohomology, Comp. Math. {\bf 118}(1999), 11--41. 

\bibitem[Ts19]{tsuji19} T.~Tsuji, Saturated morphisms of logarithmic schemes, 
Tunisian Journal of Math. {\bf 1}(2019), 185--220. 
  
\bibitem[W97]{w} J.~Wildeshaus, Realizations of Polylogarithms, Lecture Notes in Math. {\bf 1650}, Springer, 1997. 

 \end{thebibliography}
\end{document}